\documentclass[11pt]{article}

\usepackage[a4paper,margin=2.5cm]{geometry}
\usepackage{float}
\usepackage{graphics}
\usepackage{graphicx}
\usepackage{epsfig}
\usepackage{indentfirst}
\usepackage{amsmath, amsfonts, amssymb, amsthm, mathtools}
\usepackage{url}
\usepackage{csquotes}
\usepackage{epstopdf}
\usepackage{pdfpages}
\usepackage{multirow}
\usepackage{subcaption}
\usepackage{enumitem}

\newtheorem{thm}{Theorem}[section]

\newtheorem{definition}{Definition}[section]
\newtheorem{remark}{Remark}[section]

\usepackage[left]{lineno}  % "left" ou "right" define o lado da numeração

\title{Complex Dynamics of Wave-Character Transitions in Radially Symmetric Isentropic Euler Flows: Theory and Numerics}

\author{
Eduardo Abreu\thanks{University of Campinas, Brazil}
\and Geng Chen\thanks{University of Kansas, USA}
\and Faris El-Katri\thanks{University of Kansas, USA}
\and Erivaldo Lima\thanks{Federal University of Roraima, Brazil}
}

\date{}

\begin{document}
\maketitle

%\tableofcontents

\begin{abstract}
We investigate the qualitative dynamics of smooth solutions to the radially 
symmetric isentropic compressible Euler equations, focusing specifically on 
the evolution of rarefactive and compressive wave characters across three 
distinct configurations: the outward supersonic, subsonic, and inward 
supersonic regimes. For each case, we establish structural restrictions on 
wave-character transitions and identify invariant sign domains for gradient 
variables under specific initial data conditions. While our findings refine 
existing invariance properties in the outward supersonic regime, they reveal 
novel asymmetric transition mechanisms in the subsonic and inward regimes 
that are absent in purely supersonic expanding cases. Consequently, we derive 
sufficient conditions for finite-time singularity formation. To complement 
the analytical results where closed-form solutions are unavailable, we provide 
numerical experiments using a Semi-Discrete Lagrangian-Eulerian (SDLE) 
formulation. These simulations reproduce the predicted wave-character dynamics 
and offer qualitative evidence that supports our theoretical findings, 
providing a unified description of wave transitions in radially symmetric 
isentropic gas dynamics. \\



%%%%%%%%%%%%%%%%%%%%%%%%%%%%%%%%%%%%%%%%%%%%%%%%%%%%%%%%%%
% Here is a suggestion:
% Nonlinearity often requires AMS Subject Classifications:
%    35L65: Conservation laws
%    35L67: Shocks and singularities
%    76N15: Gas dynamics, general
%    65M08: Finite volume methods
%%%%%%%%%%%%%%%%%%%%%%%%%%%%%%%%%%%%%%%%%%%%%%%%%%%%%%%%%%

\end{abstract}

% ------------------------------------------------
\section{Introduction}\label{sec:introduction}

The compressible Euler equations constitute a canonical class of nonlinear hyperbolic systems of conservation laws. A defining feature of these systems is their capacity to develop singularities in finite time, even when initialized from smooth data. The mechanisms precipitating gradient blow-up and subsequent shock formation have remained a central inquiry since the foundational developments of hyperbolic theory. For one-dimensional Cartesian flows, the criteria for singularity formation are well-established \cite{lax2, Liu1, CPZ, G9, G10, CCZ, CY, CYZ}, with extensive extensions to multi-dimensional settings provided by \cite{Ali, Schen, Kong, sideris, CCY, Christodoulou1, Christodoulou2, Speck1, Speck2, Speck3}. 

In the presence of radial symmetry, the dynamics are modulated by geometric source terms that encode the divergence of the flow field. These terms introduce a critical dependency on the spatial coordinate $r$, exerting a decisive influence on the dynamics near the coordinate origin and governing the formation of implosive scenarios.

For classical solutions---defined prior to the onset of singularities---the radially symmetric compressible Euler equations in Eulerian coordinates are expressed as:
\begin{align}
    (r^m\rho)_t + (r^m\rho u)_r &= 0, \\
    (r^m\rho u)_t + (r^m\rho u^2)_r + r^m p_r &= 0, \\
    S_t + uS_r &= 0,
\end{align}
supplemented by the equation of state $p = K\rho^\gamma e^{S/c_v}$, where $\rho(r,t)$ is the density, $u(r,t)$ the radial velocity, $S$ the entropy, and $p$ the pressure. The geometric parameter $m \in \{1, 2\}$ distinguishes cylindrical and spherical symmetries, respectively. 

In the present work, we restrict our attention to the isentropic regime, for the entropy is constant. Under this assumption, the pressure law reduces to the polytropic relation \(p=K\rho^\gamma\), and the compressible Euler system under radial 
symmetry takes the form
\begin{align} 
  (r^m\rho)_t + (r^m\rho u)_r &= 0, \label{EulerSystem1} \\
  (r^m\rho u)_t + (r^m\rho u^2)_r + r^m p_r &= 0, \label{EulerSystem2}\\
  p &= K\rho^\gamma. \label{EulerSystem3}
\end{align}
It is worth mentioning, in this regard, we restrict our analysis to the isentropic regime ($S = \text{constant}$), reducing the pressure law to the polytropic relation $p = K\rho^\gamma$. 
For smooth solutions, we utilize the sound speed 
\begin{equation}
    h=\sqrt{K\gamma}\,\rho^{\frac{\gamma-1}{2}},
    \label{eq:h_function}
\end{equation}
%$h = \sqrt{K\gamma}\,\rho^{(\gamma-1)/2}$ 
to define the characteristic speeds $c_1 = u - h$ and $c_2 = u + h$, with associated Riemann invariants:
\begin{equation}\label{riemann_variables}
z = u - \frac{2}{\gamma-1}h, \qquad w = u + \frac{2}{\gamma-1}h.
\end{equation}
The sign configuration of $\{c_1, c_2\}$ delineates the primary dynamical regimes---{outward supersonic}, {subsonic}, and {inward supersonic}---thereby dictating the directional transport of information across the domain boundaries.

Recent analytical advances have provided a refined description of solution behavior up to the point of first singularity. Studies on pre-shock formation have characterized the one-sided gradient catastrophe in related compressible models \cite{Vicol0, Vicol1, Vicol2, Vicol3, Vicol4, Vicol5, Vicol6}. In the radially symmetric framework, these dynamics are further complicated by inward-propagating waves that may focus at the origin, potentially triggering an implosion where physical variables themselves blow up \cite{Vicol7}.

A powerful diagnostic tool in this context is the formulation of {gradient variables} that quantify the compressive or rarefactive nature of each characteristic family. Following \cite{F1}, we define $\alpha$ and $\beta$ (associated with $c_2$ and $c_1$, respectively) as linear combinations of $u_r$ and $h_r$, augmented by symmetry-induced terms. These variables satisfy Riccati-type transport equations along the characteristics. In this framework, the rarefactive character is identified by $\alpha, \beta \geq 0$, whereas negative values signify compressive behavior and serve as precursors to shock formation.

The present paper pursues two complementary objectives:
\begin{enumerate}
    \item \textbf{Analytical Extension:} We refine the analysis of wave-character transitions across the three fundamental regimes. We establish structural constraints on the evolution of rarefaction and compression, proving invariant-sign results for $\alpha$ and $\beta$ within specific domains of influence. This offers a unified perspective that clarifies how the interaction between characteristic families is modulated by subsonic and inward supersonic geometries.
    \item \textbf{Numerical Synthesis:} Motivated by the gradient-variable framework, we perform a series of numerical experiments using a {Semi-Discrete Lagrangian-Eulerian (SDLE)} formulation \cite{ABREU_LE_LINEAR}. This scheme, based on a space-time control-volume approach with {no-flow curves}, allows for the robust capture of nonlinear wave propagation. In the absence of closed-form solutions, our study focuses on identifying observable signatures consistent with the theoretical predictions, specifically the localized concentration of gradients and the spatio-temporal evolution of wave characters.
\end{enumerate}

The paper is organized as follows: Section \ref{sec:formulation} establishes the mathematical formulation and the gradient-variable framework, while Section \ref{sec:analysis} provides a rigorous mathematical analysis of wave-character invariance and regime classification. Section \ref{sec:scheme} details the semi-discrete Lagrangian-Eulerian discretization used to support our theoretical findings, followed by Section \ref{sec:experiments}, which presents extensive computational experiments exploring wave-character transitions in supersonic and subsonic radially symmetric isentropic Euler flows. Finally, Section \ref{sec:conclusions} offers concluding remarks and outlines directions for future research.

\section{Mathematical Formulation}\label{sec:formulation}

We consider the compressible Euler system under radial symmetry, modeling an isentropic fluid governed by the conservation laws of mass and momentum. The physical domain corresponds to the radial interval \( b \leq r \leq L \) and the time interval \( 0 \leq t \leq T \), with \( m = 1 \) for cylindrical symmetry and \( m = 2 \) for spherical symmetry. The governing equations are those introduced in Section~1, namely \eqref{EulerSystem1}–\eqref{EulerSystem3}.

\subsection{Supersonic Regime and Characteristic Variables}

Following standard notation, we introduce the sound speed variable (\ref{eq:h_function}) which allows the system to be rewritten in terms of the variables \((h,u)\). In this formulation, the characteristic speeds are given by
\[
  c_{1} = u - h, \qquad \mbox{and} \qquad c_{2} = u + h,
\]
and the associated Riemann variables are (\ref{riemann_variables}).

These quantities naturally distinguish the main dynamical regimes of the flow, namely outward supersonic, subsonic, and inward supersonic configurations, according to the signs of \(c_1\) and \(c_2\).

Based on this characteristic structure, the authors in \cite{F1} introduce the gradient variables \(\alpha\) and \(\beta\), associated respectively with the characteristic families \(c_{2}\) and \(c_{1}\), defined by
\begin{align}
  \alpha &= u_{r} + \frac{2}{\gamma - 1}h_{r} + \frac{m}{r}\,\frac{h u}{c_{2}}, \label{eq:alpha}\\[4pt]
  \beta  &= u_{r} - \frac{2}{\gamma - 1}h_{r} - \frac{m}{r}\,\frac{h u}{c_{1}}. \label{eq:beta}
\end{align}
The signs of these variables encode the local wave character of each characteristic family. Nonnegative values correspond to rarefactive behavior, while negative values indicate compression and are associated with the formation of shock profiles in finite time.

\subsubsection{Initial Conditions}

The initial data are constructed to be consistent with the supersonic regime and with the gradient-variable framework proposed in \cite{F1}. Given a prescribed initial expansion rate \(v(r_{0}) = v_{a}\) and a sound speed value at the center \(h(r_{0}) = h_{c}\), related to the central density \(\rho_{c}\), the initial profiles are obtained by inverting the relations \eqref{eq:alpha}–\eqref{eq:beta}. More precisely, we set
\begin{align}
    u(r,0) & = u_{0}(r) = v_{a} + \int_{r_{0}}^{r}\left[\frac{1}{2}(\alpha+\beta)
    +\frac{m\,h_{0}(r')^{2}u_{0}(r')}{r'\left(u_{0}(r')^{2}-h_{0}(r')^{2}\right)}\right]dr', \label{eq:u0}\\[6pt]
    h_{0}(r) & = h_{c} + \int_{r_{0}}^{r}\frac{\gamma-1}{4}\left[(\alpha-\beta)
    -\frac{2m\,h_{0}(r')u_{0}(r')^{2}}{r'\left(u_{0}(r')^{2}-h_{0}(r')^{2}\right)}\right]dr', \label{eq:h0}\\[6pt]
    \rho(r,0) & = \left(\frac{h_{0}(r)^{2}}{K\gamma}\right)^{\frac{1}{\gamma-1}}, \label{eq:rho0}\\[6pt]
    p(r,0) & = K\rho(r,0)^{\gamma}. \label{eq:p0}
\end{align}
The integrals in \eqref{eq:u0}–\eqref{eq:h0} are evaluated numerically, yielding initial profiles consistent with the prescribed values of \(\alpha\) and \(\beta\), chosen positive for rarefactive configurations and negative for compressive ones.

\subsubsection{Boundary Conditions}

For the experiments involving rarefactive and compressive profiles, homogeneous Neumann boundary conditions are imposed at \(r=b\) and \(r=L\). In the subsonic configuration, periodic boundary conditions are employed.

The parameters \((v_{a}, h_{c}, \alpha, \beta, m, \gamma, K)\) are taken from paper \cite{F1}, ensuring reproducibility of the simulations and consistency with the physical scenarios discussed in the original work.

\subsection{Subsonic Regime}

We now briefly describe the subsonic regime, which plays a distinct role in the analysis of radially symmetric flows. In this configuration, the characteristic speeds satisfy
\[
c_1 < 0 < c_2,
\]
so that one family of characteristics propagates outward while the other propagates inward. As a consequence, the flow exhibits a mixed hyperbolic structure in which information simultaneously enters and leaves the spatial domain.

In contrast to the outward supersonic regime, where both characteristic families propagate in the same direction, the subsonic regime allows for stronger interactions between waves of different families. In particular, the inward–outward coupling induced by the radial geometry introduces mechanisms that may alter the evolution of the gradient variables even in the absence of shock formation. This feature distinguishes the subsonic case from its one-dimensional Cartesian counterpart and is responsible for several qualitative differences in the behavior of solutions.

From the viewpoint of the gradient variables $\alpha$ and $\beta$, the subsonic regime is characterized by the coexistence of compressive and rarefactive effects associated with different characteristic families. While the system remains strictly hyperbolic for smooth solutions, the sign structure of $c_1$ and $c_2$ leads to asymmetric contributions in the Riccati-type equations governing $\alpha$ and $\beta$. As a result, the preservation of wave character observed in the outward supersonic regime no longer holds in a straightforward manner.

The subsonic regime therefore provides a natural setting in which transitions between rarefaction and compression may occur, even for initially smooth data. This motivates a separate analysis of wave-character evolution in this regime, which is carried out in Section~3, as well as a dedicated set of numerical experiments aimed at illustrating the resulting dynamics.

\section{Mathematical Analysis}\label{sec:analysis}

Throughout this section we consider smooth solutions on a radial interval $\Omega=[a,b]\subset(0,\infty)$ and a time interval $[0,T]$. The main goal of this paper is to provide avenues of pursuit for future results in the field; therefore, we will limit ourselves to the case of $\gamma=3$. However, we note that many of these results can be extended to the case of a general $\gamma>1$.  

We now define the following notation for an arbitrary function $f(x,t)$ by
\begin{align}
	f\circ\eta := f(\eta(x,t),t).
\end{align} 
We define the flow maps corresponding to each family of characters 
\begin{align}
	&\partial_t\xi(r_0,t) = c_1(\xi(r_0,t),t) = c_1\circ\xi,
	\\
	&\partial_t\psi(r_0,t) = c_2(\psi(r_0,t),t) = c_2\circ\psi,
	\\
	&\xi(r_0,0) = \psi(r_0,0) = r_0.
\end{align} 
Therefore the following equations hold 
\begin{align}
	&\partial_t(c_1\circ\xi) = (\partial_1c_1)\circ\xi = \frac{m}{4r}(c_2^2-c_1^2) \circ\xi, 
	\\
	&\partial_t(c_2\circ\psi) = (\partial_2c_2)\circ\psi = - \frac{m}{4r}(c_2^2-c_1^2) \circ\psi, 
	\\
	&\partial_t(\beta\circ\xi) = (\partial_1\beta)\circ\xi =  \left(-\beta^2 + \frac{mc_2^2}{2rc_1}(\alpha-\beta)\right)\circ\xi, \label{Eqn:beta_evolve}
	\\
	&\partial_t(\alpha\circ\psi) = (\partial_2\alpha)\circ\psi = \left(-\alpha^2 + \frac{mc_1^2}{2rc_2}(\beta-\alpha) \right)\circ\psi. \label{Eqn:alpha_evolve} 
\end{align} 

We now define the domain of influence of our initial data.  
\begin{definition}
Let $\Omega := [a,b] \subset \mathbb{R}$, and $c_1(r,0), c_2(r,0) \in C^1(\Omega)$. We define the domain of influence, $D(\Omega,T)$, as the set of all $(r,t) \in \mathbb{R} \times [0,T]$ such that there exists\ some $r_1,r_2\in\Omega$ satisfying $\psi(r_1,t)=\xi(r_2,t)=r$. 
\end{definition} 

\begin{remark}
	For classical solutions to Equations \ref{EulerSystem1}--\ref{EulerSystem3}, the flow maps remain invertible; therefore, their domain of dependence may be equivalently determined by the set below. 
	\begin{align}
		D(\Omega,T) = \{(r,t) : \psi(a,t) \leq r \leq \xi(b, t) \}.
	\end{align}
\end{remark}

We now prove there exists an invariant domain over our gas velocities for supersonic expanding data. 

\begin{thm}
	Let $c_1(r,0), c_2(r,0) \in C^1(\Omega)$. Further suppose that $c_1(r,0) > 0$ for all $r\in\Omega$, and that $c_1(r,t), c_2(r,t) \in C^1(D(\Omega,T))$ for some $T>0$. Then, for all $(r,t) \in D(\Omega,T)$, we have that $c_1(r,t)\geq \min(c_1(r,0))$, and $c_2(r,t) \leq ||c_2(r,0)||_{L^\infty}$.
\end{thm} 

\begin{proof}
This follows from a bootstrap argument. Suppose that $c_1(r,t) \geq 0$ for all $(r,t) \in D$. Then
\begin{align}
	&\partial_t(c_1\circ\xi) = \frac{m}{4r}(c_2^2-c_1^2)\circ\xi > 0,
	\\
	&c_1\circ\xi = c_1(r,0) + \int_0^t \frac{m}{4r}(c_2^2-c_1^2)\circ\xi ds \geq \min_{r \in \Omega}c_1(r,0) > 0. 
\end{align} 

Now we improve our bootstrap assumption. We then apply this to the second claim: 
\begin{align}
	&\partial_t(c_2\circ\psi) = -\frac{m}{4r}(c_2^2-c_1^2) \circ \psi < 0, 
	\\
	&c_2\circ\psi = c_2(r,0) - \int_0^t \frac{m}{4r}(c_2^2-c_1^2)\circ\psi ds \leq ||c_2(r,0)||_{L^\infty}. 
\end{align} 
This completes the proof. 
\end{proof} 

It is known by previous analysis in \cite{F1} that for the expanding supersonic wave (defined by $u>h$), we have an invariant domain over rarefactive data ($\alpha \geq  0$, $\beta \geq 0$). Here, we seek to quantify the occasions in which rarefactive characters can become compressive, or vice versa for all three cases (the supersonic expanding wave, the subsonic wave, and the supersonic inward wave).

We first define the Rarefaction/Compression ($R/C$) character below.

\begin{definition}
	For any smooth solution of Equations \eqref{EulerSystem1}--\eqref{EulerSystem3}, we 
	call the solution as 1-rarefaction (compression) if $\beta>0$ ($\beta<0$), and 2-rarefaction (compression) if $\alpha>0$ ($\alpha<0$).
\end{definition}

\subsection{Case 1: supersonic expanding wave, when $0<c_1<c_2$}

\begin{thm}
	Consider the possible change of rarefaction ($R$) and compression ($C$) character in the smooth solution of \eqref{EulerSystem1}--\eqref{EulerSystem3}, in Case 1 when $0<c_1<c_2$.
	
	If the 1-wave (2-wave) is rarefaction ($R$), then the 2-wave  (1-wave) can only change from $C$ to $R$.
	
	If the 1-wave (2-wave) is compression ($C$), then the 2-wave  (1-wave) can only change from $R$ to $C$.
	\label{thm:inward_wave_supersonic}
\end{thm} 

\begin{proof}
The proof follows from Equations \eqref{Eqn:beta_evolve}--\eqref{Eqn:alpha_evolve}. 
\\
Case 1: Suppose by way of contradiction that the 1-wave is rarefactive and the 2-wave changes from rarefactive to compressive. This implies that there exists some pair $r_0^*,t^*$ and $\epsilon>0$ such that the following hold for all $\tau \in (0,\epsilon)$. 
\begin{align}
	&\beta(\psi(r_0^*,t^*),t^*) > 0, 
	\\
	&\alpha(\psi(r_0^*,t^*),t^*)=0, 
	\\
	&\alpha(\psi(r_0^*,t^*-\tau),t^*-\tau) > 0, 
	\\
	&\alpha(\psi(r_0^*,t^*+\tau),t^*+\tau) < 0.
\end{align} 
We then compute the following. 
\begin{align}
	\partial_t(\alpha\circ\psi(r_0,t))|_{t=t^*} = \frac{mc_1^2}{2rc_2}\beta\circ\psi > 0. 
\end{align} 
Thus implying that there exists some $\delta>0$ such that for all $\tau\in(0,\delta)$, the following hold. 
\begin{align}
	&\alpha(\psi(r_0^*,t^*-\tau),t^*-\tau) < 0, 
	\\
	&\alpha(\psi(r_0^*,t^*+\tau),t^*+\tau) > 0. 
\end{align} 
This contradicts our assumption. 
\bigskip
\\
Case 2: Suppose by way of contradiction that the 2-wave is rarefactive and the 2-wave changes from rarefactive to compressive. This implies that there exists some pair $r_0^*,t^*$ and $\epsilon>0$ such that the following hold for all $\tau \in (0,\epsilon)$. 
\begin{align}
	&\alpha(\xi(r_0^*,t^*),t^*) > 0, 
	\\
	&\beta(\xi(r_0^*,t^*),t^*)=0, 
	\\
	&\beta(\xi(r_0^*,t^*-\tau),t^*-\tau) > 0, 
	\\
	&\beta(\xi(r_0^*,t^*+\tau),t^*+\tau) < 0. 
\end{align} 
We then compute the following. 
\begin{align}
	\partial_t(\beta\circ\xi(r_0,t))|_{t=t^*} = \frac{mc_2^2}{2rc_1}\beta\circ\xi > 0. 
\end{align} 
Thus implying that there exists some $\delta>0$ such that for all $\tau\in(0,\delta)$, the following hold. 
\begin{align}
	&\beta(\xi(r_0^*,t^*-\tau),t^*-\tau) < 0, 
	\\
	&\beta(\xi(r_0^*,t^*+\tau),t^*+\tau) > 0. 
\end{align} 
This contradicts to our assumption.

The proofs for the cases of the compressive characters are symmetric and will thus be omitted. 
\end{proof} 

The following result indicates to us that there exist invariant domains over $\alpha>0,\beta>0$ and $\alpha<0,\beta<0$. 

%%% Denote D omega T better. Bold all "Proof". Also no more "simple". Also make each case a subsection. 
\begin{thm}
	Suppose that our initial data satisfy the following over $\Omega=[a,b]$. 
	\begin{align}
		&\alpha(r,0) > 0, 
		\\
		&\beta(r,0) > 0.
	\end{align} 
	Then the following hold for all $(r,t) \in D(\Omega,T)$. 
	\begin{align}
		&\alpha(r,t) > 0, 
		\\
		&\beta(r,t) > 0. 
	\end{align} 
	Similarly, if our initial data satisfy the following, 
	\begin{align}
		&\alpha(r,0) < 0 ,
		\\
		&\beta(r,0) < 0.
	\end{align} 
	Then the following hold for all $(r,t) \in D(\Omega,T)$. 
	\begin{align}
		&\alpha(r,t) < 0, 
		\\
		&\beta(r,t) < 0. 
	\end{align} 
\end{thm} 

\begin{proof}
The proof follows from a bootstrap argument. 
\\
Case 1: 

\begin{align}
	&\alpha(r,0) > 0, 
	\\
	&\beta(r,0) > 0.
\end{align} 

Assume that the following bootstrap assumption holds for all $(r, t)$ within our domain of influence. 
\begin{align}
	&\alpha(r,t) \geq 0, 
	\\
	&\beta(r,t) \geq 0.
\end{align} 
We then apply Gronwall's inequality to our Riccati Equations. 
\begin{align}
	&\partial_t(\beta\circ\xi) \geq -\beta\left(\beta + \frac{mc_2^2}{2rc_1}\beta\right) \circ\xi, 
	\\
	&\beta\circ\xi \geq \beta_0 e^{-\int_0^t \left( \beta + \frac{mc_2^2}{2rc_1} \right)\circ\xi ds} > 0, 
	\\
	&\partial_t(\alpha\circ\psi) \geq -\alpha\left(\alpha + \frac{mc_1^2}{2rc_2}\alpha\right) \circ\psi, 
	\\
	&\alpha\circ\psi \geq \alpha_0 e^{-\int_0^t \left( \alpha + \frac{mc_1^2}{2rc_2} \right)\circ\psi ds} > 0. 
\end{align} 
This improves our bootstrap assumptions and completes the proof. 
\bigskip
\\
Case 2: 
\begin{align}
	&\alpha(r,0) < 0, 
	\\
	&\beta(r,0) < 0.
\end{align} 
Assume that the following bootstrap assumption holds for all $(r, t)$ within our domain of influence. 
\begin{align}
	&\alpha(r,t) \leq 0, 
	\\
	&\beta(r,t) \leq 0.
\end{align} 
We then apply Gronwall's inequality to our Riccati Equations. 
\begin{align}
	&\partial_t(\beta\circ\xi) \leq -\beta\left(\beta + \frac{mc_2^2}{2rc_1}\beta\right) \circ\xi, 
	\\
	&\beta\circ\xi \leq \beta_0 e^{-\int_0^t \left( \beta + \frac{mc_2^2}{2rc_1} \right)\circ\xi ds} < 0, 
	\\
	&\partial_t(\alpha\circ\psi) \leq -\alpha\left(\alpha + \frac{mc_1^2}{2rc_2}\alpha\right) \circ\psi, 
	\\
	&\alpha\circ\psi \leq \alpha_0 e^{-\int_0^t \left( \alpha + \frac{mc_1^2}{2rc_2} \right)\circ\psi ds} < 0. 
\end{align} 
Thus improving upon our bootstrap assumptions and completing the proof. 
\end{proof} 

%%%
%
%%%
Next,  we give a singularity formation result when the initial data includes strong compression.

\begin{thm}
	Suppose that we have initial data over $r\in[a,\infty)=:\Omega$ with both families of gradient variables compressive. 	
	\begin{align}
		&0<\inf_{r\in[a,\infty]} S_1 := {c_1(r,0)} <||c_2(r,0)||_{L^\infty} =: S_2 < \infty, 
		\\
		&\alpha(r,0) < 0, 
		\\
		&\beta(r,0) < 0. 
	\end{align} 
	
	Further suppose that for some $r^* \in \Omega$, a strong compression, as defined below. 
	\begin{align}
		\alpha(r^*, 0) < -\frac{mS_2^2}{2r^*S_1} =: -M.
	\end{align}	
	Then, singularity formation occurs in finite time. 
\end{thm} 
\begin{proof}
We proceed still by contradiction. Suppose that for any time, $T>0$, the classical solution exists over $D(\Omega,T)$. Then, the following hold for all $(r,t) \in D(\Omega,T)$. 
\begin{align}
	S_1 \leq c_1(r,t) < c_2(r,t) \leq S_2.
\end{align} 

We now claim that $\alpha(\psi(r^*,t),t) \leq \alpha(r^*,0)$. Indeed, using this as our bootstrap assumption gives us the following. 
\begin{align}
	\partial_t(\alpha\circ\psi) < -\alpha(\alpha + M)\circ\psi < 0. 
\end{align} 
Thus improving on our bootstrap assumption. By construction, there then exists some $\epsilon > 0$ such that the following holds for all time. 
\begin{align}
	&\partial_t(\alpha\circ\psi) \leq (-\epsilon\alpha^2 
	- (1-\epsilon)\alpha^2 + M\alpha)\circ\psi \leq -\epsilon\alpha^2 \circ\psi. 
\end{align} 

Separating and integrating gives us the following upper bound on the time of singularity formation. 
\begin{align}
	t^* \leq -\frac{1}{\epsilon\alpha_0(r^*)}.
\end{align}

\end{proof} 

\subsection{Case 2: subsonic expanding wave, when $c_1<0<c_2$}

\begin{thm}
	Consider the possible change of rarefaction ($R$) and compression ($C$) character in the smooth solution of \eqref{EulerSystem1}--\eqref{EulerSystem3}, in Case 2 when $c_1<0<c_2$.
	
	If the 1-wave is rarefaction ($R$), then the 2-wave  can only change from $C$ to $R$.
	
	If the 1-wave is compression ($C$), then the, 2-wave can only change from $R$ to $C$.
	
	If the 2-wave is rarefaction ($R$), then the 1-wave  can only change from $R$ to $C$.
	
	If the 2-wave is compression ($C$), then the, 1-wave can only change from $C$ to $R$.
	\label{Theorem:subsonic}
\end{thm} 

\begin{proof}
Case 1: Suppose by way of contradiction that the 1-wave is rarefactive and the 2-wave changes from rarefactive to compressive. This implies that there exists some pair $r_0^*,t^*$ and $\epsilon>0$ such that the following hold for all $\tau \in (0,\epsilon)$. 
\begin{align}
	&\beta(\psi(r_0^*,t^*),t^*) > 0, 
	\\
	&\alpha(\psi(r_0^*,t^*),t^*)=0, 
	\\
	&\alpha(\psi(r_0^*,t^*-\tau),t^*-\tau) > 0, 
	\\
	&\alpha(\psi(r_0^*,t^*+\tau),t^*+\tau) < 0. 
\end{align} 
We then compute the following. 
\begin{align}
	\partial_t(\alpha\circ\psi(r_0^*,t))|_{t=t^*} = \frac{mc_1^2}{2rc_2}\beta\circ\psi > 0. 
\end{align} 
Thus implying that there exists some $\delta>0$ such that for all $\tau\in(0,\delta)$, the following hold. 
\begin{align}
	&\alpha(\psi(r_0^*,t^*-\tau),t^*-\tau) < 0, 
	\\
	&\alpha(\psi(r_0^*,t^*+\tau),t^*+\tau) > 0. 
\end{align} 
This contradicts to our assumption. 
\vspace{.2cm}
\\
Case 2: Suppose by way of contradiction that the 1-wave is compressive and the 2-wave changes from compressive to rarefactive. This implies that there exists some pair $r_0^*,t^*$ and $\epsilon>0$ such that the following hold for all $\tau \in (0,\epsilon)$. 
\begin{align}
	&\beta(\psi(r_0^*,t^*),t^*) < 0, 
	\\
	&\alpha(\psi(r_0^*,t^*),t^*)=0, 
	\\
	&\alpha(\psi(r_0^*,t^*-\tau),t^*-\tau) < 0, 
	\\
	&\alpha(\psi(r_0^*,t^*+\tau),t^*+\tau) > 0. 
\end{align} 
We then compute the following. 
\begin{align}
	\partial_t(\alpha\circ\psi(r_0^*,t))|_{t=t^*} = \frac{mc_1^2}{2rc_2}\beta\circ\psi < 0. 
\end{align} 
Thus implying that there exists some $\delta>0$ such that for all $\tau\in(0,\delta)$, the following hold. 
\begin{align}
	&\alpha(\psi(r_0^*,t^*-\tau),t^*-\tau) > 0, 
	\\
	&\alpha(\psi(r_0^*,t^*+\tau),t^*+\tau) < 0. 
\end{align} 
This contradicts to our assumption. 
\\
The proofs for the cases of the 2-characters are symmetric and will thus be omitted. 
\end{proof} 
\bigskip

We observe in this case an asymmetry not previously observed in the supersonic outward case. This strongly indicates the possibility of determining a periodic solution for the subsonic case, as was found for the 1-D case in \cite{TY_Periodic}. This indication is further reinforced by the near-periodic nature of the numerical simulations.

\subsection{Case 3: supersonic inward wave, when $c_1<c_2<0$}

\begin{thm}
	Consider the possible change of rarefaction ($R$) and compression ($C$) character in the smooth solution of \eqref{EulerSystem1}--\eqref{EulerSystem3}, in Case 3 when $c_1<c_2<0$.
	
	If the 1-wave (2-wave) is rarefaction ($R$), then the 2-wave  (1-wave) can only change from $R$ to $C$.
	
	If the 1-wave (2-wave) is compression ($C$), then the 2-wave  (1-wave) can only change from $C$ to $R$.
	\label{Theorem:inward_wave}
\end{thm} 
\begin{proof}
Case 1: Suppose by way of contradiction that the 1-wave is rarefactive and the 2-wave changes from compressive to rarefactive. This implies that there exists some pair $r_0^*,t^*$ and $\epsilon>0$ such that the following hold for all $\tau \in (0,\epsilon)$. 
\begin{align}
	&\beta(\psi(r_0^*,t^*),t^*) > 0, 
	\\
	&\alpha(\psi(r_0^*,t^*),t^*)=0, 
	\\
	&\alpha(\psi(r_0^*,t^*-\tau),t^*-\tau) < 0, 
	\\
	&\alpha(\psi(r_0^*,t^*+\tau),t^*+\tau) > 0. 
\end{align} 
We then compute the following. 
\begin{align}
	\partial_t(\alpha\circ\psi(r_0^*,t))|_{t=t^*} = \frac{mc_1^2}{2rc_2}\beta\circ\psi < 0. 
\end{align} 
Thus it implies that there exists some $\delta>0$ such that for all $\tau\in(0,\delta)$, then the following hold. 
\begin{align}
	&\alpha(\psi(r_0^*,t^*-\tau),t^*-\tau) > 0, 
	\\
	&\alpha(\psi(r_0^*,t^*+\tau),t^*+\tau) < 0. 
\end{align} 
This contradicts our assumption. 
\bigskip
\\
Case 2: Suppose by way of contradiction that the 2-wave is rarefactive and the 1-wave changes from compressive to rarefactive. This implies that there exists some pair $r_0^*,t^*$ and $\epsilon>0$ such that the following hold for all $\tau \in (0,\epsilon)$. 
\begin{align}
	&\alpha(\xi(r_0^*,t^*),t^*) > 0, 
	\\
	&\beta(\xi(r_0^*,t^*),t^*)=0, 
	\\
	&\beta(\xi(r_0^*,t^*-\tau),t^*-\tau) < 0, 
	\\
	&\beta(\xi(r_0^*,t^*+\tau),t^*+\tau) > 0. 
\end{align} 
We then compute the following. 
\begin{align}
	\partial_t(\beta\circ\xi(r_0,t))|_{t=t^*} = \frac{mc_2^2}{2rc_1}\beta\circ\xi < 0. 
\end{align} 
Thus implying that there exists some $\delta>0$ such that for all $\tau\in(0,\delta)$, the following hold. 
\begin{align}
	&\beta(\xi(r_0^*,t^*-\tau),t^*-\tau) > 0, 
	\\
	&\beta(\xi(r_0^*,t^*+\tau),t^*+\tau) < 0. 
\end{align} 

This contradicts our assumption. 
\bigskip
\\
The proofs for the cases of the compressive characters are symmetric and will thus be omitted. 
	\end{proof} 

The following theorem indicates to us that there exist invariant domains over $\alpha>0,\beta>0$ and $\alpha<0,\beta<0$. 

\begin{thm}
	Suppose that our initial data satisfy the following for all $r \in \Omega$. 
	\begin{align}
		&\alpha(r,0) > 0, 
		\\
		&\beta(r,0) < 0.
	\end{align} 
	Then the following hold for all $(r,t) \in D(\Omega,T)$. 
	\begin{align}
		&\alpha(r,t) > 0, 
		\\
		&\beta(r,t) < 0. 
	\end{align} 
	Similarly, if the initial data satisfy the following for all $r \in \Omega$
	\begin{align}
		&\alpha(r,0) < 0, 
		\\
		&\beta(r,0) > 0.
	\end{align} 
	Then the following hold for all $(r,t) \in D(\Omega,T)$.  
	\begin{align}
		&\alpha(r,t) < 0,
		\\
		&\beta(r,t) > 0. 
	\end{align} 
\end{thm}

\begin{proof}
This follows from a bootstrap argument. 
\\
Case 1: 

\begin{align}
	&\alpha(r,0) > 0, 
	\\
	&\beta(r,0) < 0.
\end{align} 
Assume that the following bootstrap assumption holds for all $(r, t) \in D(\Omega,T)$. 
\begin{align}
	&\alpha(r,t) \geq 0, 
	\\
	&\beta(r,t) \leq 0.
\end{align} 
We then apply Gronwall's inequality to our Riccati Equations. 
\begin{align}
	&\partial_t(\beta\circ\xi) \leq -\beta\left(\beta + \frac{mc_2^2}{2rc_1}\beta\right) \circ\xi, 
	\\
	&\beta\circ\xi \leq \beta_0 e^{-\int_0^t \left( \beta + \frac{mc_2^2}{2rc_1} \right)\circ\xi ds} < 0, 
	\\
	&\partial_t(\alpha\circ\psi) \geq -\alpha\left(\alpha + \frac{mc_1^2}{2rc_2}\alpha\right) \circ\psi, 
	\\
	&\alpha\circ\psi \geq \alpha_0 e^{-\int_0^t \left( \alpha + \frac{mc_1^2}{2rc_2} \right)\circ\psi ds} > 0. 
\end{align} 
This improves upon our bootstrap assumptions and completes the proof. 
\bigskip
\\
Case 2: 
\begin{align}
	&\alpha(r,0) < 0, 
	\\
	&\beta(r,0) > 0.
\end{align} 
Assume that the following bootstrap assumption holds for all $(r, t) \in D(\Omega,T)$. 
\begin{align}
	&\alpha(r,t) \leq 0, 
	\\
	&\beta(r,t) \geq 0.
\end{align} 
We then apply Gronwall's inequality to our Riccati Equations. 
\begin{align}
	&\partial_t(\beta\circ\xi) \geq -\beta\left(\beta + \frac{mc_2^2}{2rc_1}\beta\right) \circ\xi, 
	\\
	&\beta\circ\xi \geq \beta_0 e^{-\int_0^t \left( \beta + \frac{mc_2^2}{2rc_1} \right)\circ\xi ds} > 0, 
	\\
	&\partial_t(\alpha\circ\psi) \leq -\alpha\left(\alpha + \frac{mc_1^2}{2rc_2}\alpha\right) \circ\psi, 
	\\
	&\alpha\circ\psi \leq \alpha_0 e^{-\int_0^t \left( \alpha + \frac{mc_1^2}{2rc_2} \right)\circ\psi ds} < 0. 
\end{align} 
This improves upon our bootstrap assumptions and completes the proof. 
\end{proof} 

Our result in \cite{F2} shows that if the $1$-wave has a strong compression, it will form a shock. From the current paper, in this case, a strong $1$-compression may induce a wave change in the $2$-character from $C$ to $R$. Once this occurs, a $2$-rarefaction will only make $1$-compression stronger, as evidenced by our analysis on the invariant domains. 

% ------------------------------------------------
\section{Numerical Scheme: Semi-Discrete Lagrangian-Eulerian Formulation}\label{sec:scheme}

In this section, we delineate the numerical architecture employed to approximate solutions of the radially symmetric compressible Euler system. The computational framework utilizes the Semi-Discrete Lagrangian-Eulerian (SDLE) formulation introduced in \cite{ABREU_LE_LINEAR}. This approach constitutes a hybrid methodology that integrates a Lagrangian transport stage with an Eulerian reconstruction step, maintaining a strictly conservative discretization of the underlying conservation laws.

A fundamental component of the SDLE framework is the utilization of no-flow curves, which induce a dynamic, time-dependent partition of the spatial domain. These curves delineate moving control volumes that track the local convective transport of the fluid. This Lagrangian-like motion allows the scheme to resolve nonlinear wave propagation - such as the steepening of compressive fronts - without the computational overhead of explicit spectral decompositions or the necessity of local Riemann solvers.

By evolving the boundaries of the cells in alignment with the physical flux, the SDLE scheme naturally adapts to the local characteristic speeds. This property is particularly advantageous in the radially symmetric context, where the $m/r$ geometric source terms and varying radial velocities create a non-uniform transport field. The resulting formulation ensures the preservation of the conservative structure for mass and momentum while maintaining robustness across subsonic and supersonic regimes involving high-gradient compression or rarefaction.

For numerical implementation, the Euler system is cast into a system of conservation 
laws for the weighted variables. 
\[
s = r^{m}\rho,
\qquad
w = r^{m}\rho u.
\]
This choice of variables facilitates the treatment of the radial geometry by absorbing the rm factor into the conserved quantities, effectively transforming the balance laws into a form amenable to high-resolution finite-volume techniques. Indeed, in terms of $(s,w)$, the system takes the form
\begin{align}
    s_{t} + (H_{1}(s))_{r} & = 0, \label{eq:sdle_s}\\
    w_{t} + (H_{2}(w))_{r} & = -\,r^{m}p_{r}, \label{eq:sdle_w}\\
    p & = K\left(\frac{s}{r^{m}}\right)^{\gamma}, \label{eq:sdle_p}
\end{align}
where $H_{1}(s)=su$ and $H_{2}(w)=wu$, with $u=w/s$.

The pressure-gradient contribution $r^{m}p_{r}$ is treated as a source term and is discretized separately from the convective fluxes. This splitting allows the Lagrangian transport to be determined solely by the convective part of the system, while the pressure effects are incorporated as internal forces at the semi-discrete level.

\subsection*{Semi-Discrete Lagrangian-Eulerian Formulation}

Let $\{\sigma_j^n(t)\}$ denote the family of no-flow curves introduced in \cite{ABREU_LE_LINEAR}, which delimit the moving control volumes over the time interval $[t_n,t_{n+1}]$. Denoting by $s_j(t)$ and $w_j(t)$ the cell averages associated with these volumes, the semi-discrete SDLE formulation reads
\begin{equation}
\begin{cases}
\dfrac{d s_j(t)}{dt}
= -\dfrac{1}{\Delta r}\Big[\mathcal{H}_{1}(s_j,s_{j+1})
- \mathcal{H}_{1}(s_{j-1},s_j)\Big] = 0, \\[2mm]
\dfrac{d w_j(t)}{dt}
= -\dfrac{1}{\Delta r}\Big[\mathcal{H}_{2}(w_j,w_{j+1})
- \mathcal{H}_{2}(w_{j-1},w_j)\Big]
- r_j^{m}(p_r)_j,
\end{cases}
\label{eq:sdle-semidiscrete}
\end{equation}
where $(p_r)_j$ denotes a discrete approximation of the radial pressure derivative.

The numerical fluxes $\mathcal{H}_1$ and $\mathcal{H}_2$ are constructed following \cite{ABREU_LE_LINEAR} as
\begin{align}
\mathcal{H}_{1}(s_j,s_{j+1})
&= \frac{1}{4}\Big[
b^{s}_{j+\frac12}\big(s^{-}_{j+\frac12}-s^{+}_{j+\frac12}\big)
+ (f_j+f_{j+1})\big(s^{-}_{j+\frac12}+s^{+}_{j+\frac12}\big)
\Big], \label{eq:fluxle1}\\
\mathcal{H}_{2}(w_j,w_{j+1})
&= \frac{1}{4}\Big[
b^{w}_{j+\frac12}\big(w^{-}_{j+\frac12}-w^{+}_{j+\frac12}\big)
+ (g_j+g_{j+1})\big(w^{-}_{j+\frac12}+w^{+}_{j+\frac12}\big)
\Big], \label{eq:fluxle2}
\end{align}
where $f_j=H_1(s_j)/s_j$, $g_j=H_2(w_j)/w_j$, and $b^{s}_{j+\frac12}$, $b^{w}_{j+\frac12}$ are order approximations of the local transport velocity $\mathcal{O}(H_{1}(s)/s)$ and $\mathcal{O}(H_{2}(w)/w)$, respectively, defined by
\begin{align}
    b^{s}_{j+\frac12} = \displaystyle\max_{j}\left\{\vert \zeta_{1}f_{j} + \zeta_{2}f_{j+1} \vert\right\} \quad \mbox{and} \quad b^{w}_{j+\frac12} = \displaystyle\max_{j}\left\{\vert \zeta_{1}g_{j} + \zeta_{2}g_{j+1} \vert\right\}.
\end{align}

The interface values are obtained by linear reconstruction,
\begin{align*}
s^{-}_{j+\frac12} = s_j + \frac{\Delta r}{4}(s_r)_j,
\qquad
s^{+}_{j+\frac12} = s_{j+1} - \frac{\Delta r}{4}(s_r)_{j+1},\\
w^{-}_{j+\frac12} = w_j + \frac{\Delta r}{4}(w_r)_j,
\qquad
w^{+}_{j+\frac12} = w_{j+1} - \frac{\Delta r}{4}(w_r)_{j+1},
\end{align*}
where the slopes are limited using a minmod-type operator,
\begin{equation}
    (s_{r})_{j} = minmod \left(\varrho\theta\dfrac{s_{j}-s_{j-1}}{\Delta r}, \varrho\dfrac{s_{j+1}-s_{j-1}}{2\Delta r}, \varrho\theta\dfrac{s_{j+1}-s_{j}}{\Delta r} \right), \quad 1 \leqslant \theta \leqslant 2,
    \label{midmod}
\end{equation}
with $\varrho$ an adjustable parameter, and the $minmod$ function is defined by $minmod(a,b,c) = \mbox{MM}(\mbox{MM}(a,b),c)$,
with
\begin{equation*}
    \mbox{MM}(a,b) = \frac{1}{2}\left[sign(a)+sign(b)\right]\min(\vert a \vert,\vert b \vert).
    \label{MM_mod}
\end{equation*}
An analogous expression is used for $(w_r)_j$.

Time integration of the semi-discrete system \eqref{eq:sdle-semidiscrete} is performed using a Runge-Kutta-Shu scheme \cite{shu1988,shu1989}.

\subsection*{Source-Term Discretization}

The pressure-gradient source term $-r_j^{m}(p_r)_j$ is discretized explicitly and coupled with the thermodynamic variables through the equation of state. The same slope-limiting strategy is employed in the approximation of $p_r$, which contributes to the stability of the scheme in regions where strong gradients are present.

\subsection*{Stability Condition}

The time step $\Delta t^n$ is chosen to satisfy the CFL condition
\[
\frac{\Delta t^n}{\Delta r}
\max\!\left\{\max_j |f_j|,\max_j |g_j|\right\}
%%\leqslant 
< \frac{1}{2}.
\]
This condition prevents the intersection of no-flow curves between successive time steps and ensures the stability of the numerical method.

%----------------------------------------------------
\section{Numerical Experiments}\label{sec:experiments}

In this section, we present a series of computational experiments for the radially symmetric isentropic Euler system \eqref{EulerSystem1}--\eqref{EulerSystem3}, meticulously designed to corroborate the theoretical wave-character transitions established in \cite{F1}. The simulations are executed using the Semi-Discrete Lagrangian-Eulerian (SDLE) framework detailed in Section \ref{sec:scheme}, formulated in terms of the weighted conservative variables $s=r^m\rho$ and $w=r^m\rho u$. Given the absence of closed-form analytical solutions for the multidimensional regimes under consideration, our validation strategy focuses on identifying qualitative signatures that are strictly consistent with the underlying hyperbolic theory:
\begin{itemize}
\item[1)] \textbf{Spatio-temporal Regularity:} The monitoring of localized steepening or the preservation of smoothness in the primitive and thermodynamic fields $(\rho,u,p,h)$.

\item[2)] \textbf{Phase-Space Topology:} The evolution of the invariant manifolds (the invariant curve) within the $(u,h)$-plane.

\item[3)] \textbf{Gradient Dynamics:} The spatio-temporal evolution of the gradient variables $\alpha$ and $\beta$, specifically tracking the onset of a gradient catastrophe as defined by the concentration of negative values along characteristic trajectories.
\end{itemize}

\textbf{Computational Parameters.} To ensure high-fidelity resolution of nonlinear wave structures, all simulations utilize a uniform discretization of $N=8192$ cells. For domains posed on the radial interval $[10,20]$, this corresponds to a spatial resolution of $\Delta r = 8.308900\times 10^{-3}$. The temporal evolution is governed by the CFL condition specified in Section \ref{sec:scheme}, with $CFL=0.1$ (the Courant number) to maintain stability in high-gradient regions. The linear reconstruction employs a slope-limiting framework with parameters $\varrho = \theta=1$, and the dissipation coefficients are set to $\zeta_1 = \zeta_{2} = 2$ to prevent spurious oscillations in the numerical solutions.

The initial profiles are synthesized from prescribed wave-character data via the mappings \eqref{eq:u0}--\eqref{eq:p0}, adhering to the parameter sets specified in \cite{F1}. For each test case, we provide temporal snapshots of the state variables $(\rho,u,p,h)$, the corresponding phase-space invariant curves, and heat maps illustrating the evolution of $\alpha$ and $\beta$ in the $(r,t)$-plane. These diagnostics offer a comprehensive view of the interplay between radial geometry and the nonlinear transition mechanisms.

\subsection*{Case 1: strong supersonic compression}

We initiate our numerical study with an outward-propagating supersonic configuration characterized by a strongly compressive state across both characteristic families. The system parameters and initial data are prescribed as follows:
\[
m=1,\quad K=7.75\times 10^{4},\quad \gamma=1.4,\quad \alpha=\beta=-3,\quad h_c=1,\qquad v_a=10,\quad r\in[10,20].
\]
This initialization establishes a genuinely compressive regime where the nonlinear steepening of the wave fronts is expected to dominate. According to the analytical criteria established in \cite{F1}, such initial data must precipitate a gradient catastrophe and the subsequent formation of shocks within a finite temporal horizon.

The numerical simulations corroborate this theoretical prediction with high fidelity. The evolution of the primitive and thermodynamic variables $(\rho,u,p,h)$ exhibits a rapid steepening of the profiles, indicating a transition from smooth to discontinuous states. This behavior is precisely reflected in the gradient variables $\alpha$ and $\beta$, which exhibit a sharp localized growth in negative magnitude along the characteristic curves.

In the present computation, the onset of shock formation is resolved within the interval $t \in [1.5,10]$ (see Figures \ref{fig:density_case1}--\ref{fig:invariant-curve_case1}). Furthermore, the spatio-temporal diagnostics provided by the heat maps in Figure \ref{fig:heat-map_case1} delineate the morphology of the gradient concentration in the $(r,t)$-plane, allowing for a precise visualization of the coordinate loci where the classical solution ceases to exist and the singularity emerges.

% ------------------------------------------------------------------------

\begin{figure}[H]
  \centering
  \begin{subfigure}{0.32\textwidth}
    \centering
    \includegraphics[width=1.0\textwidth]{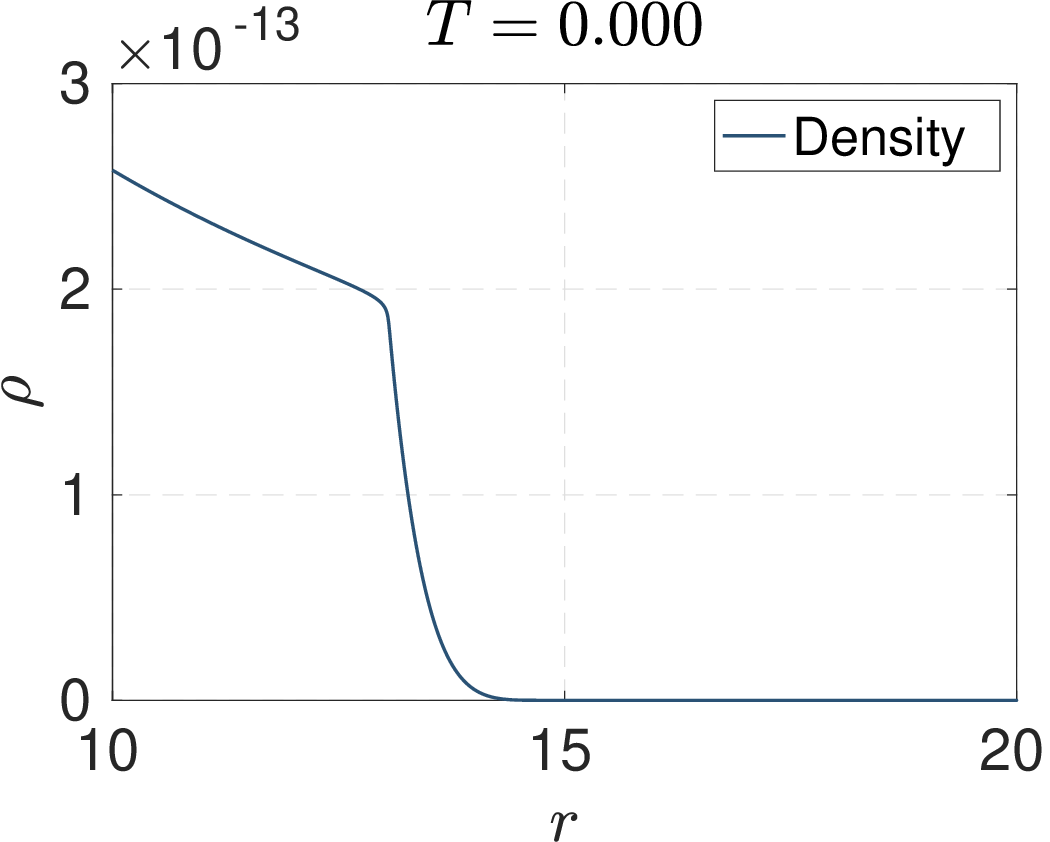}
  \end{subfigure}
    \begin{subfigure}{0.32\textwidth}
    \centering
    \includegraphics[width=1.0\textwidth]{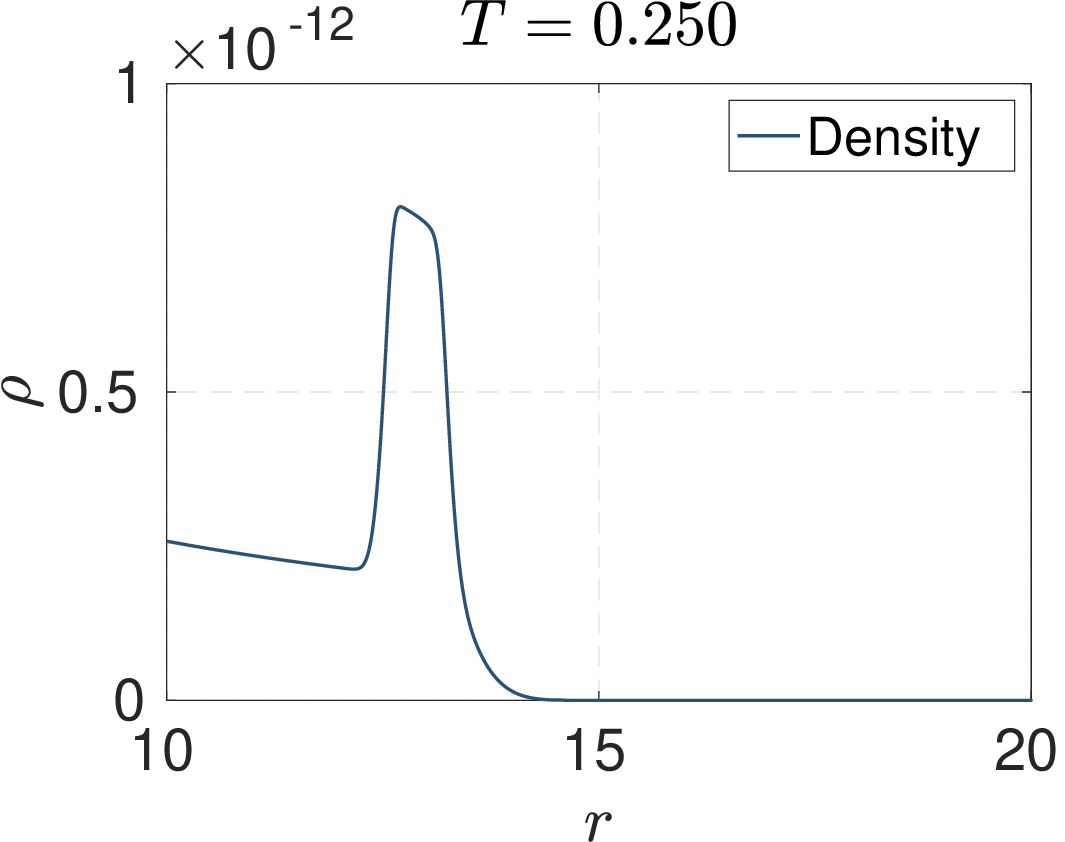}
  \end{subfigure}
  \begin{subfigure}{0.32\textwidth}
    \centering
    \includegraphics[width=1.0\textwidth]{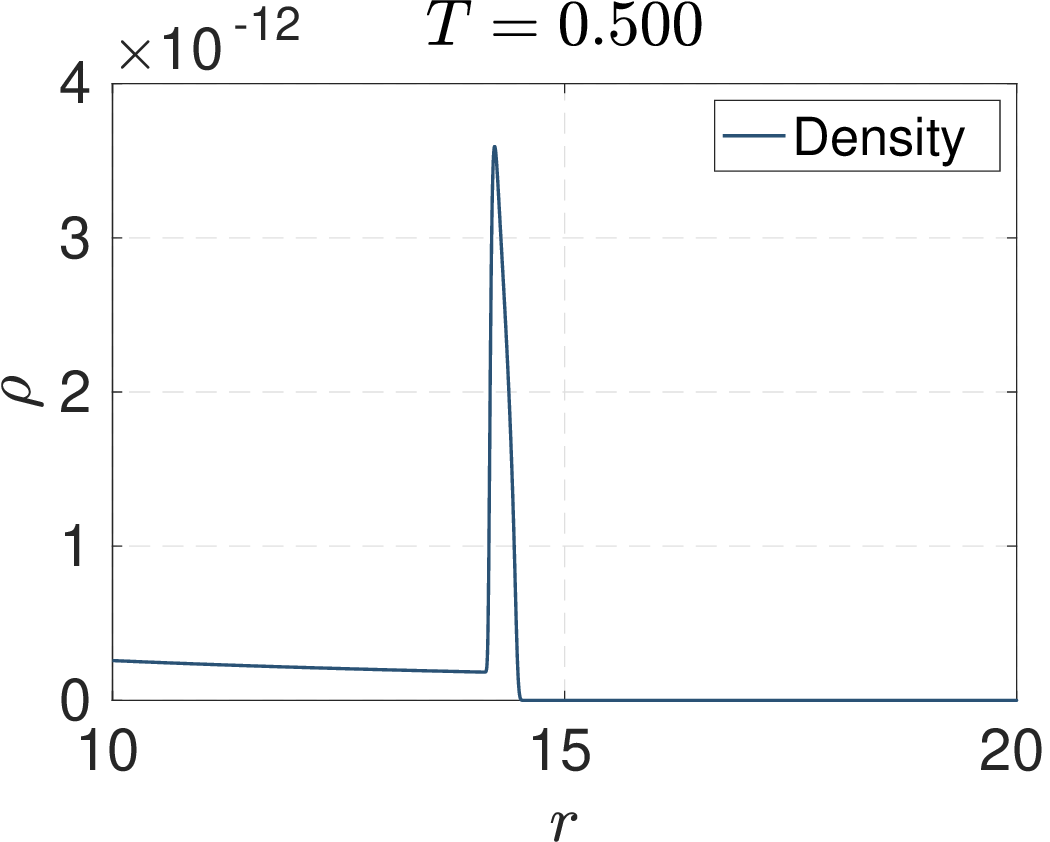}
  \end{subfigure}
  \caption{The initial \textit{density} is shown on the left and its time-evolved state on the center and on the right.}
  \label{fig:density_case1}
\end{figure}

% ------------------------------------------------------------------------

\begin{figure}[H]
  \centering
  \begin{subfigure}{0.32\textwidth}
    \centering
    \includegraphics[width=1.0\textwidth]{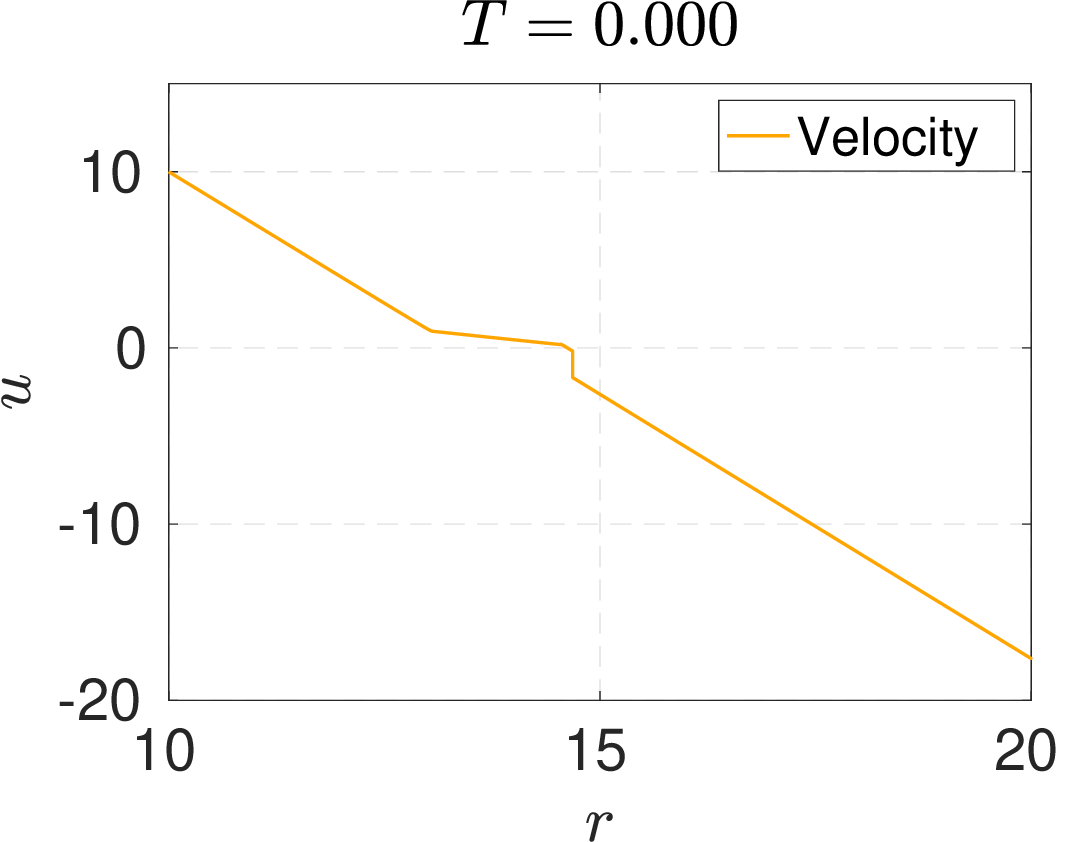}
  \end{subfigure}
    \begin{subfigure}{0.32\textwidth}
    \centering
    \includegraphics[width=1.0\textwidth]{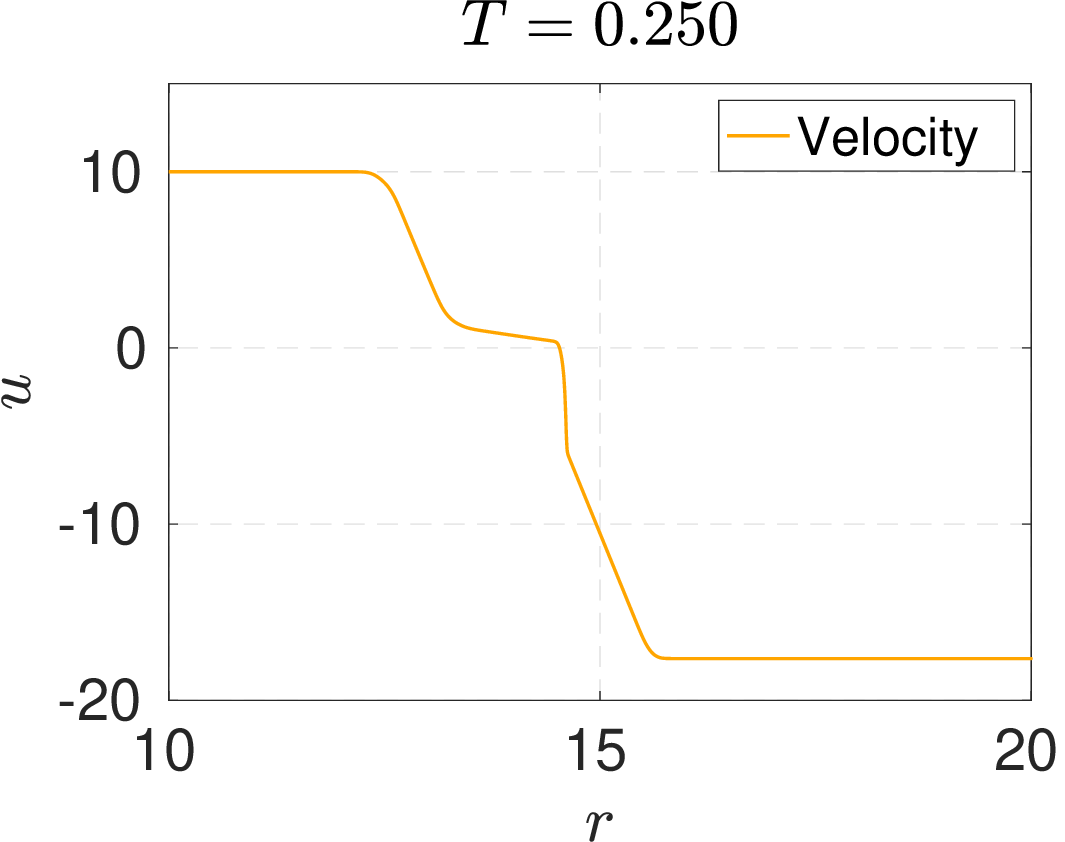}
  \end{subfigure}
  \begin{subfigure}{0.32\textwidth}
    \centering
    \includegraphics[width=1.0\textwidth]{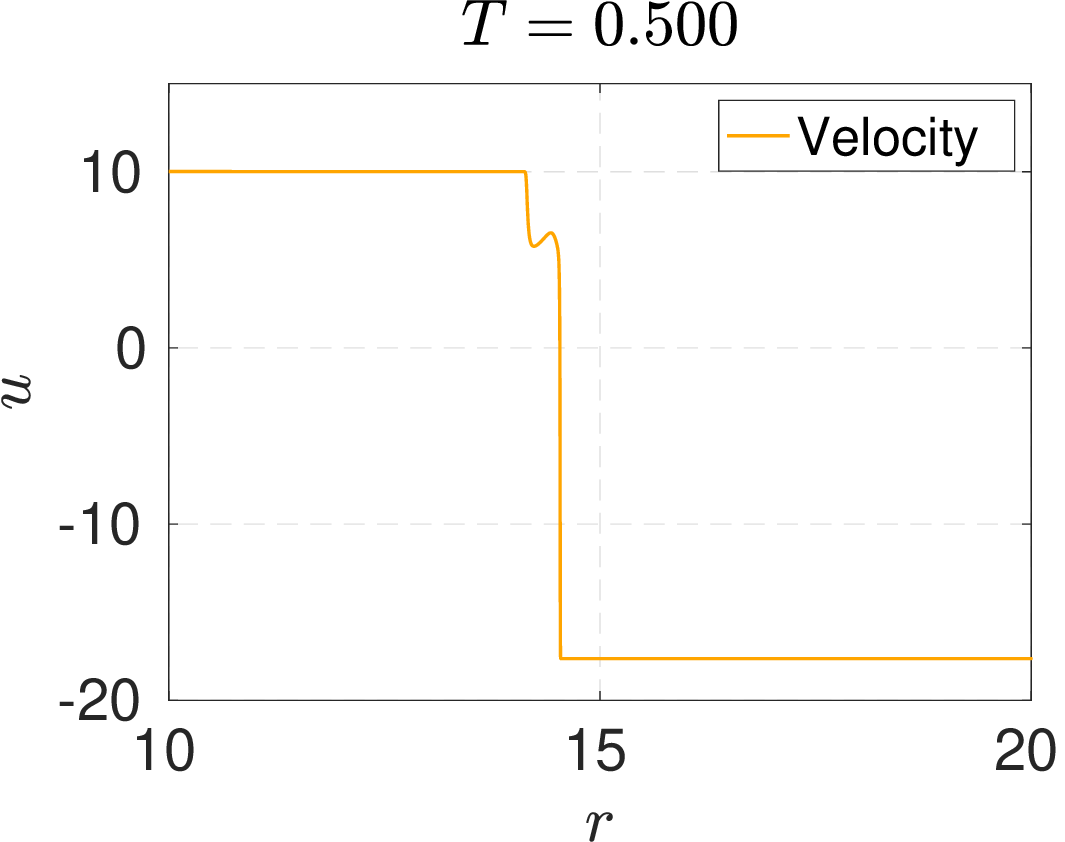}
  \end{subfigure}
  \caption{The initial \textit{velocity} is shown on the left and its time-evolved state on the center and on the right.}
  \label{fig:velocity_case1}
\end{figure}

% ------------------------------------------------------------------------

\begin{figure}[H]
  \centering
  \begin{subfigure}{0.32\textwidth}
    \centering
    \includegraphics[width=1.0\textwidth]{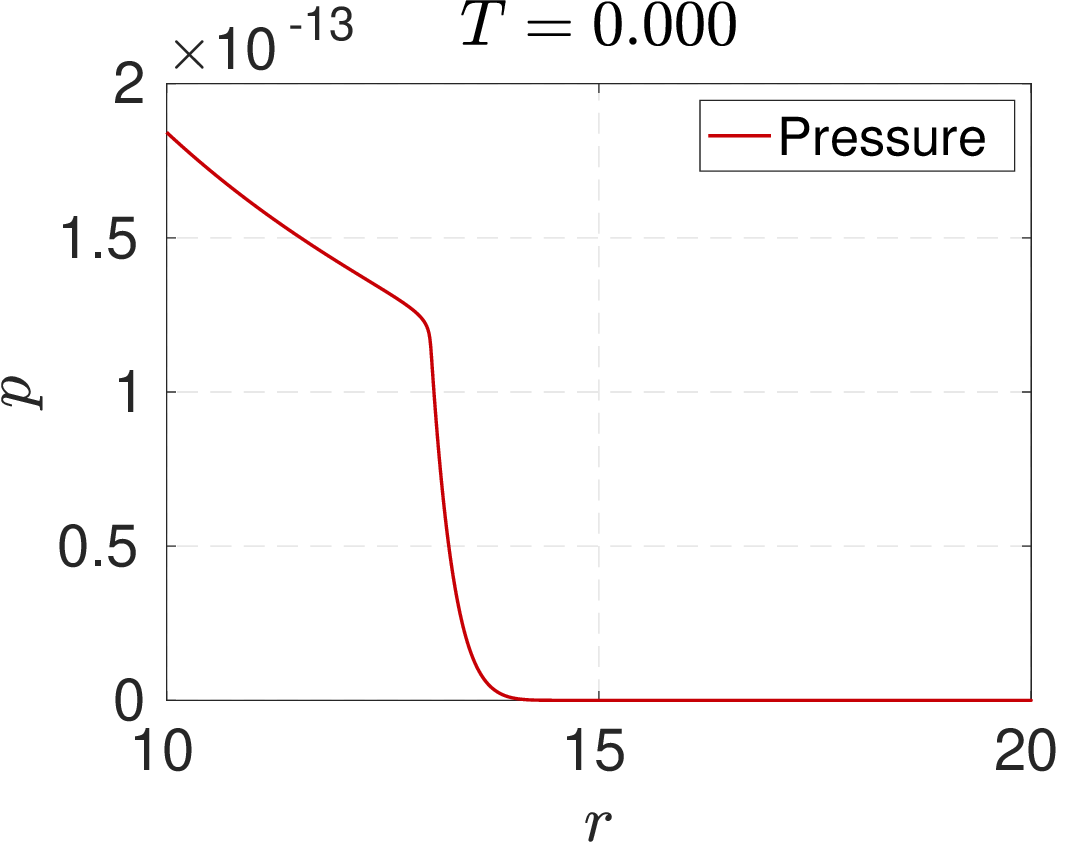}
  \end{subfigure}
    \begin{subfigure}{0.32\textwidth}
    \centering
    \includegraphics[width=1.0\textwidth]{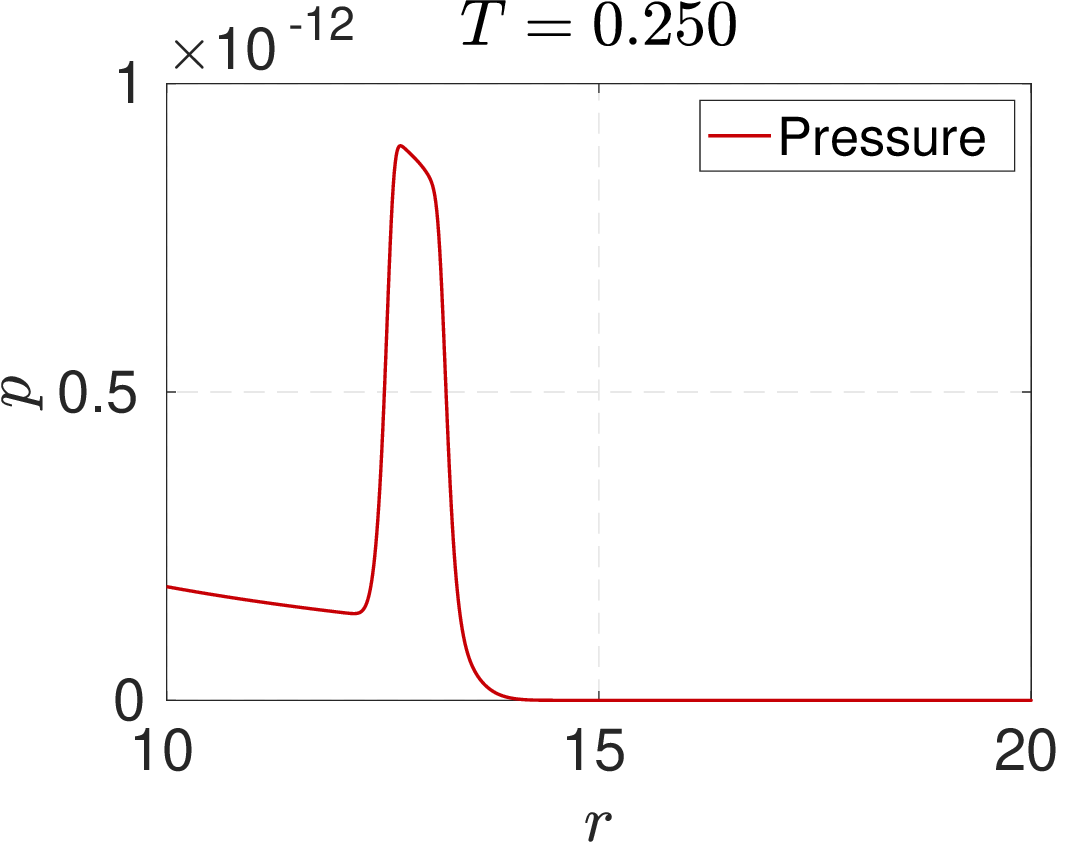}
  \end{subfigure}
  \begin{subfigure}{0.32\textwidth}
    \centering
    \includegraphics[width=1.0\textwidth]{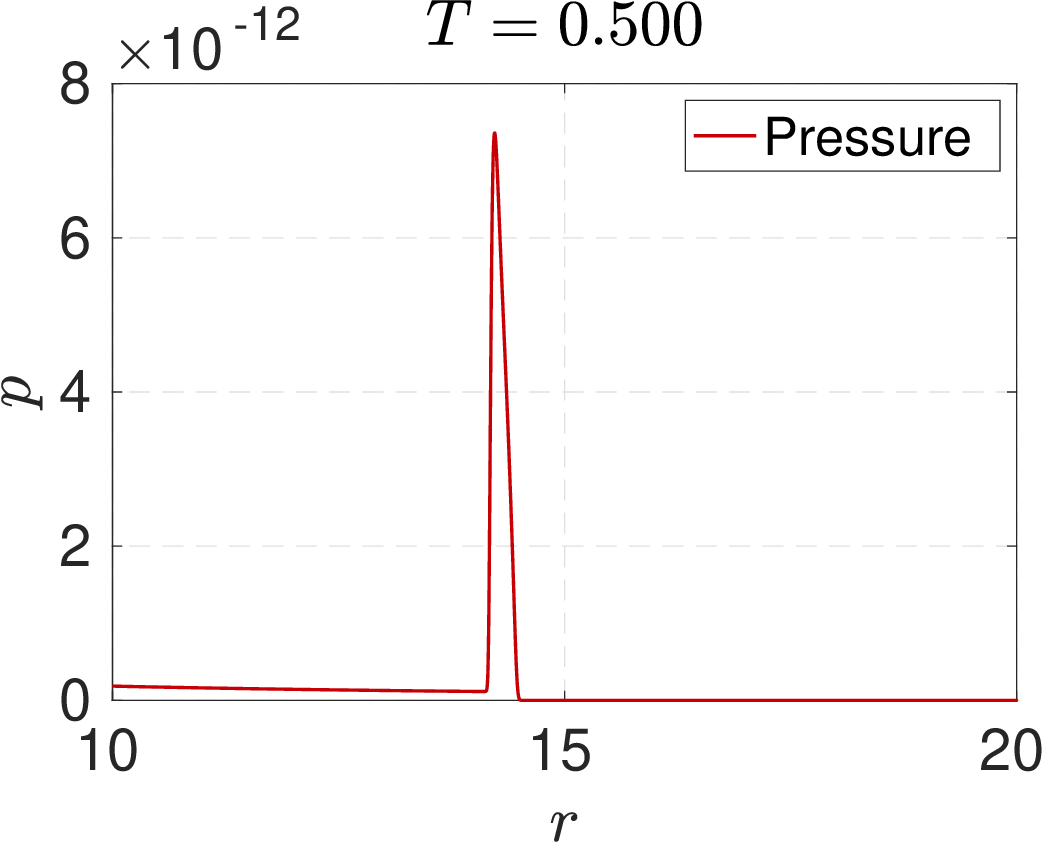}
  \end{subfigure}
  \caption{The initial \textit{pressure} is shown on the left and its time-evolved state on the center and on the right.}
  \label{fig:pressure_case1}
\end{figure}

% ------------------------------------------------------------------------

\begin{figure}[H]
  \centering
  \begin{subfigure}{0.32\textwidth}
    \centering
    \includegraphics[width=1.0\textwidth]{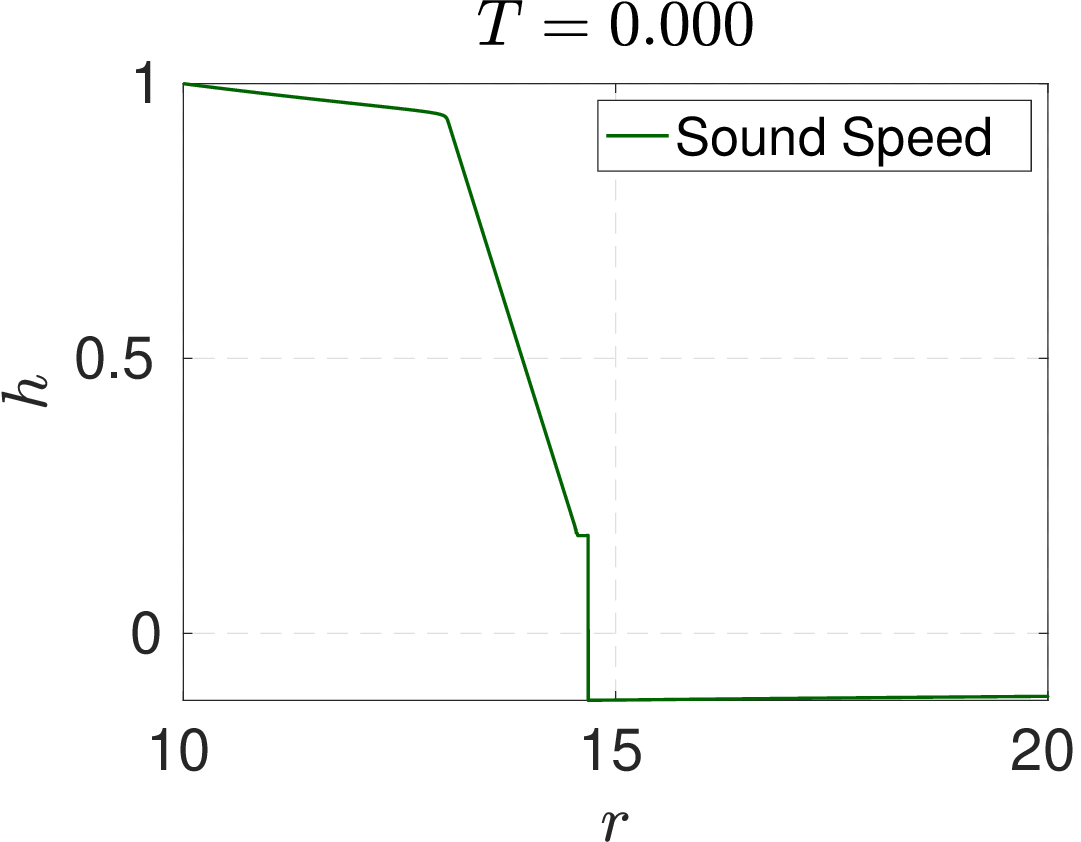}
  \end{subfigure}
    \begin{subfigure}{0.32\textwidth}
    \centering
    \includegraphics[width=1.0\textwidth]{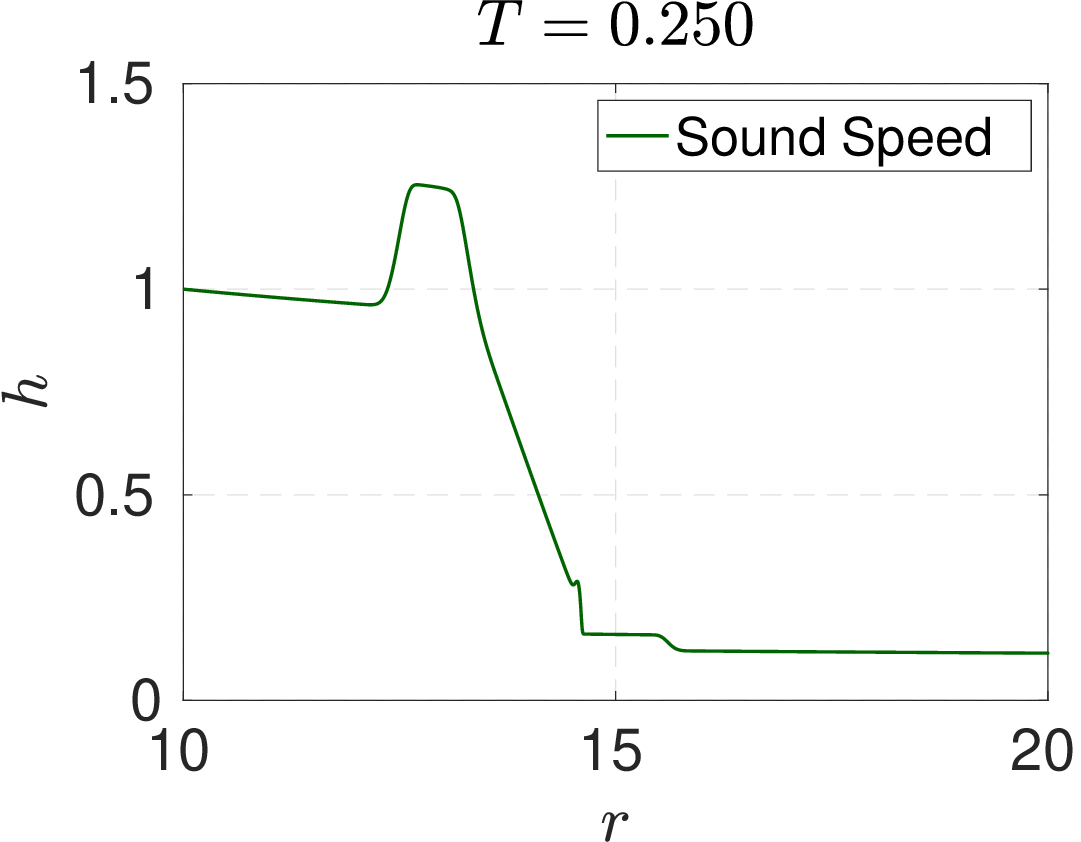}
  \end{subfigure}
  \begin{subfigure}{0.32\textwidth}
    \centering
    \includegraphics[width=1.0\textwidth]{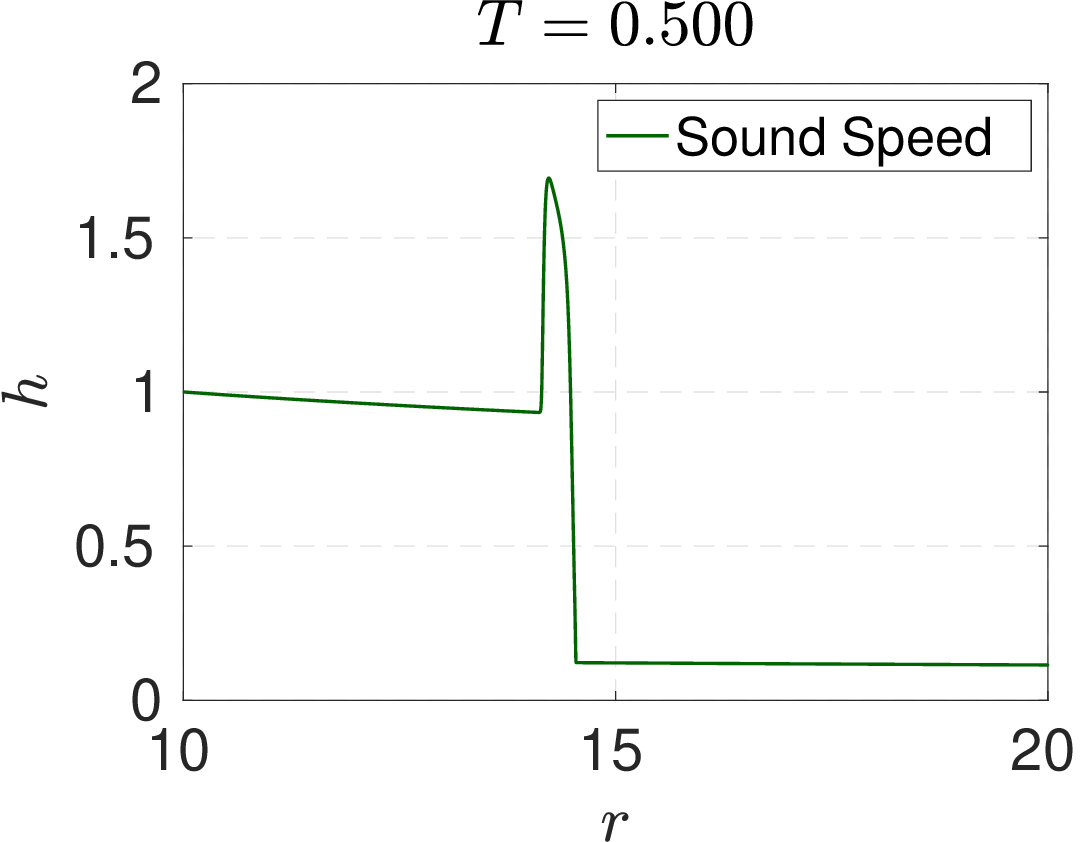}
  \end{subfigure}
  \caption{The initial \textit{sound speed} is shown on the left and its time-evolved state on the center and on the right.}
  \label{fig:sound-speed_case1}
\end{figure}

% ------------------------------------------------------------------------

\begin{figure}[H]
  \centering
  \begin{subfigure}{0.32\textwidth}
    \centering
    \includegraphics[width=1.0\textwidth]{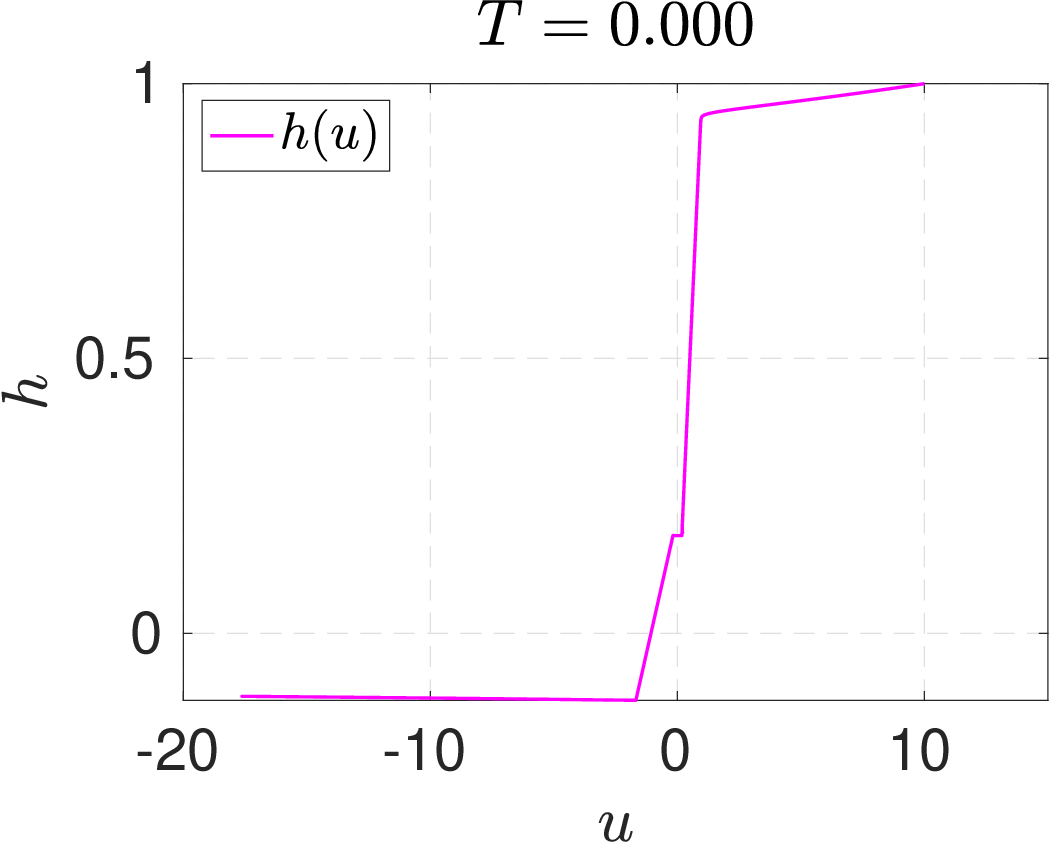}
  \end{subfigure}
    \begin{subfigure}{0.32\textwidth}
    \centering
    \includegraphics[width=1.0\textwidth]{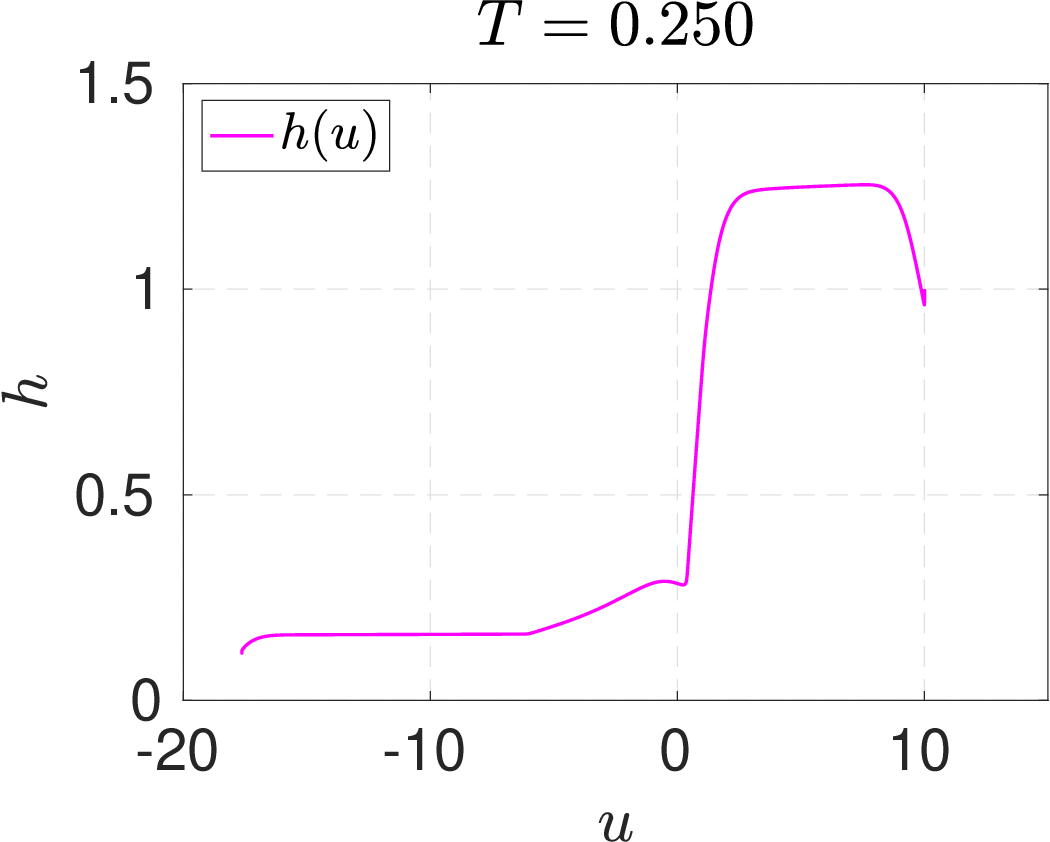}
  \end{subfigure}
  \begin{subfigure}{0.32\textwidth}
    \centering
    \includegraphics[width=1.0\textwidth]{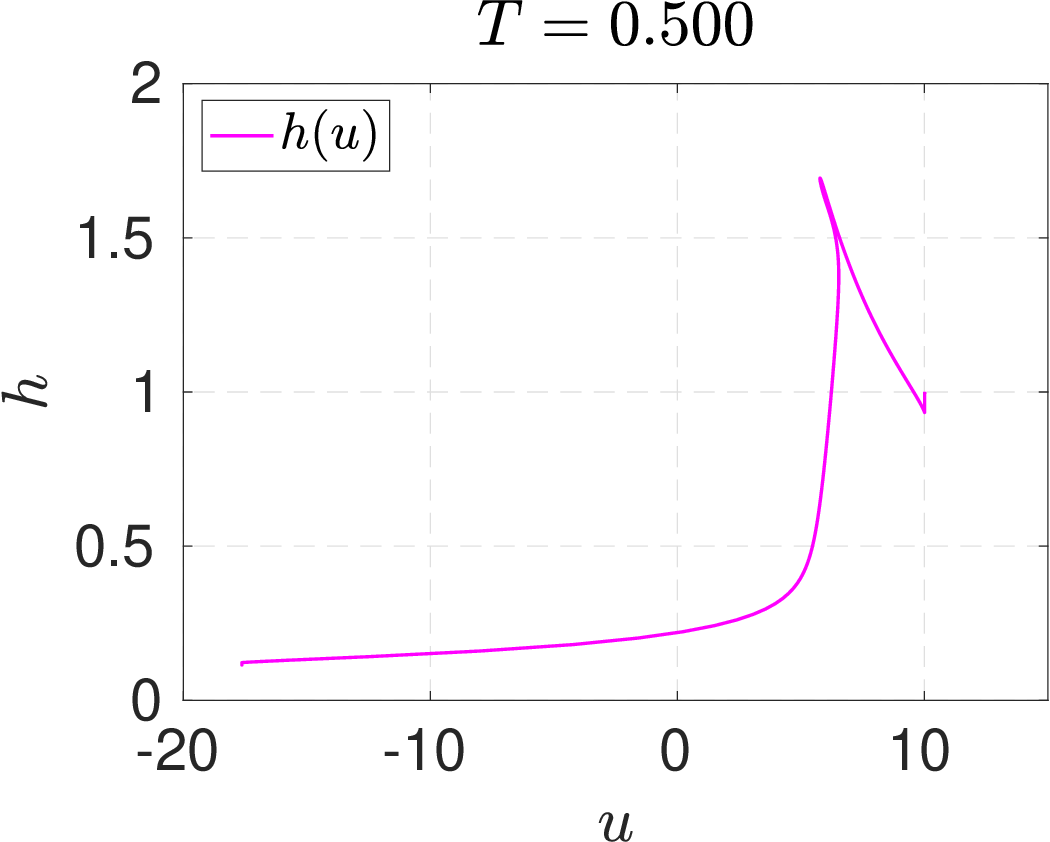}
  \end{subfigure}
  \caption{The initial \textit{invariant curve in \((u,h)-\)plane} is shown on the left and its time-evolved state on the center and on the right.}
  \label{fig:invariant-curve_case1}
\end{figure}

% ------------------------------------------------------------------------

\begin{figure}[H]
  \centering
  \begin{subfigure}{0.49\textwidth}
    \centering
    \includegraphics[width=1.0\textwidth]{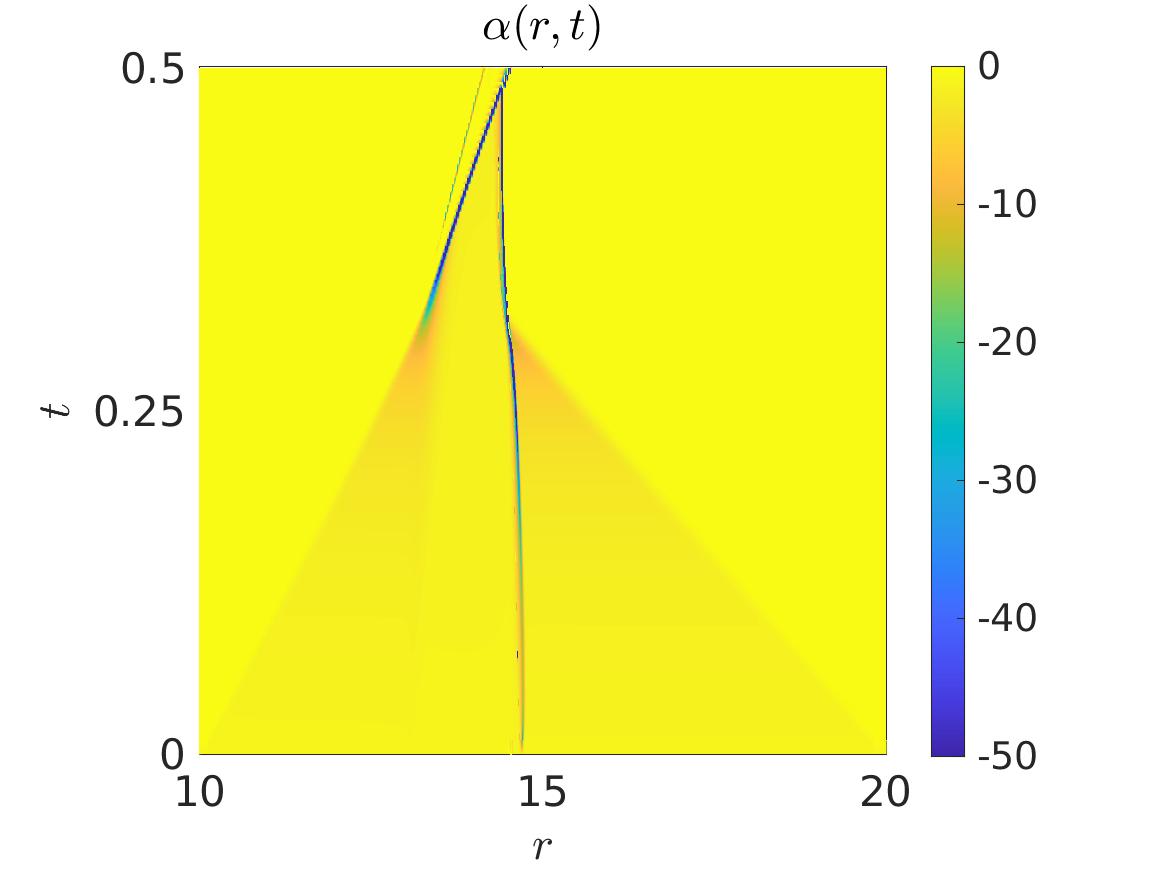}
  \end{subfigure}
  \begin{subfigure}{0.49\textwidth}
    \centering
	\includegraphics[width=1.0\textwidth]{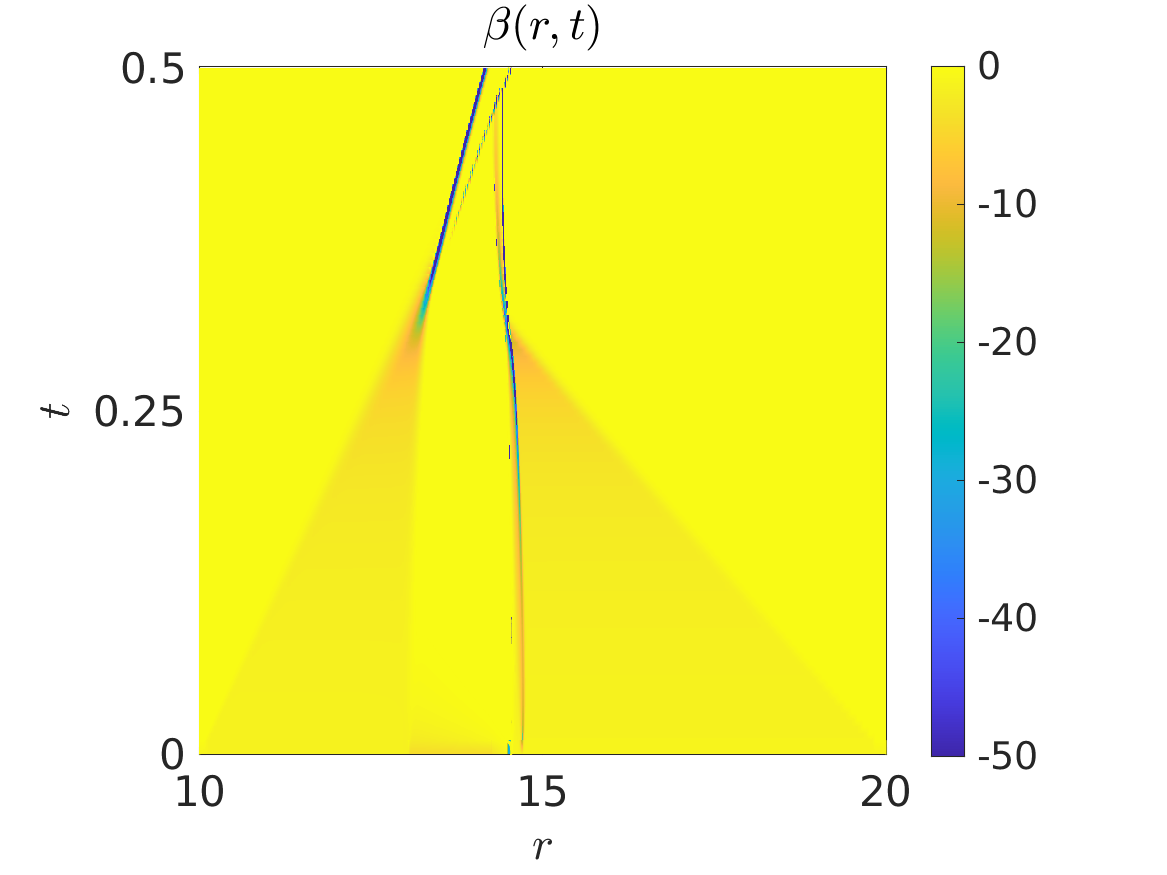}
  \end{subfigure}
  \caption{The \textit{heat map of $\alpha$ in \((r,t)-\)plane} is shown on the left and the heat map of $\beta$ on the right.}
  \label{fig:heat-map_case1}
\end{figure}

% ------------------------------------------------------------------------

\subsection*{Case 2: supersonic rarefaction}

We next reverse the wave character while keeping the same outward supersonic setting:
\[
m=1,\quad K=7.75\times 10^{4},\quad \gamma=1.4,\qquad
\alpha=\beta=3,\quad h_c=1,\quad v_a=10,\quad r\in[10,20].
\]
In this case both characteristic families are initially rarefactive, and the theory predicts persistence of smoothness (in particular, no shock formation triggered by compressive steepening). The numerical solution remains regular on the simulated time interval: the fields $(\rho,u,p,h)$ evolve smoothly (figures \ref{fig:density_case2}-\ref{fig:sound-speed_case2}), the invariant curve in the $(u,h)$-plane exhibits a coherent deformation without folding (Figure \ref{fig:invariant-curve_case2}), and the heat maps confirm that $\alpha$ and $\beta$ stay nonnegative throughout the domain (Figure \ref{fig:heat-map_case2}). 

% ------------------------------------------------------------------------

\begin{figure}[H]
  \centering
  \begin{subfigure}{0.32\textwidth}
    \centering
    \includegraphics[width=1.0\textwidth]{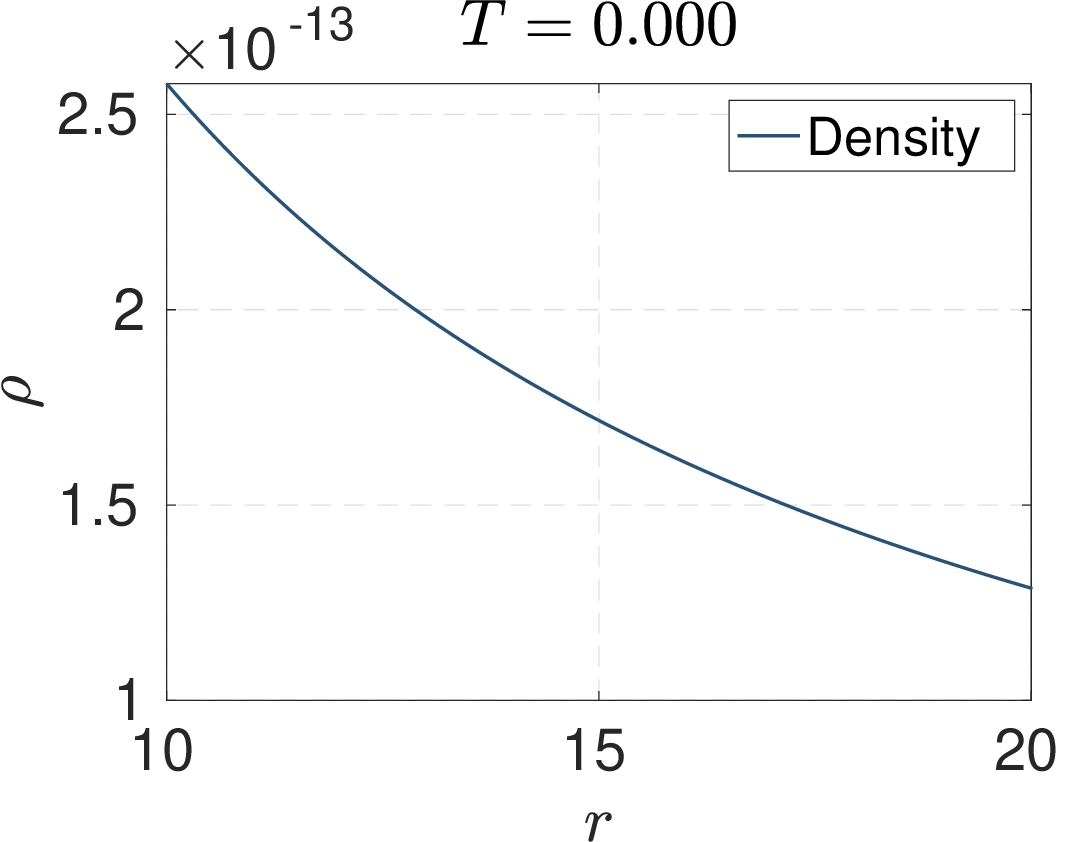}
  \end{subfigure}
    \begin{subfigure}{0.32\textwidth}
    \centering
    \includegraphics[width=1.0\textwidth]{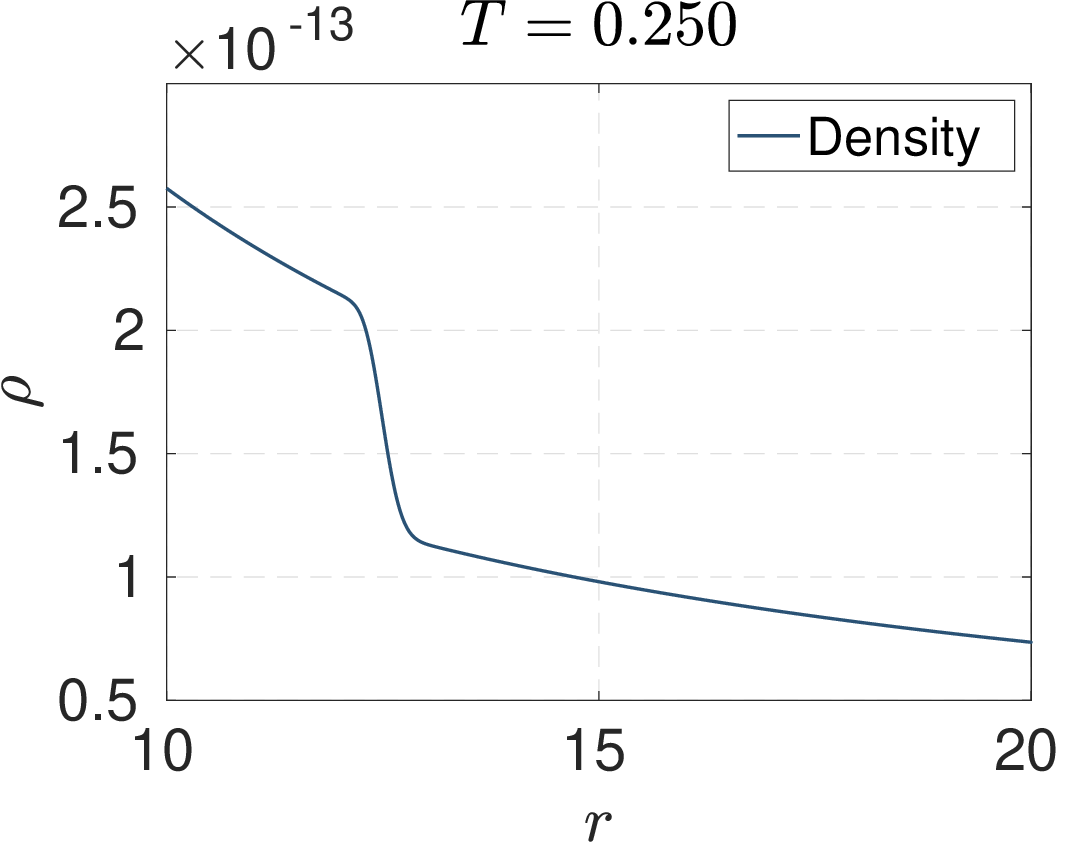}
  \end{subfigure}
  \begin{subfigure}{0.32\textwidth}
    \centering
    \includegraphics[width=1.0\textwidth]{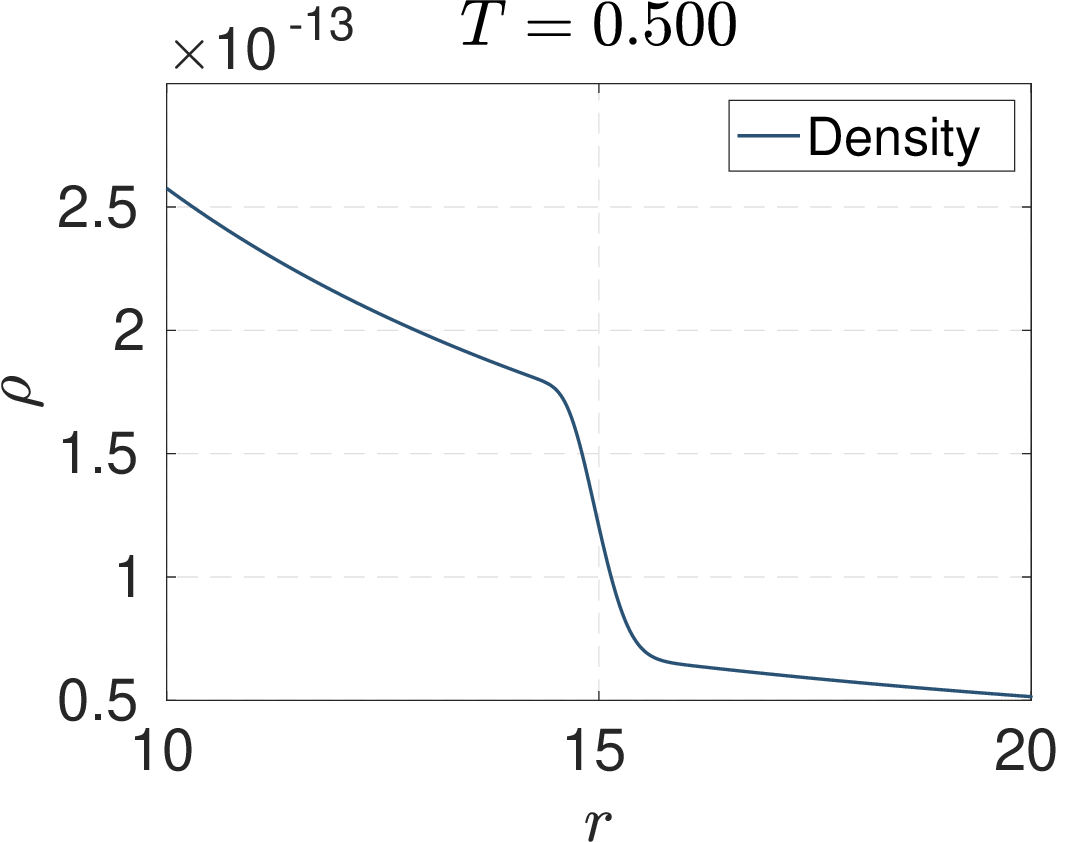}
  \end{subfigure}
  \caption{The initial \textit{density} is shown on the left and its time-evolved state on the center and on the right.}
  \label{fig:density_case2}
\end{figure}

% ------------------------------------------------------------------------

\begin{figure}[H]
  \centering
  \begin{subfigure}{0.32\textwidth}
    \centering
    \includegraphics[width=1.0\textwidth]{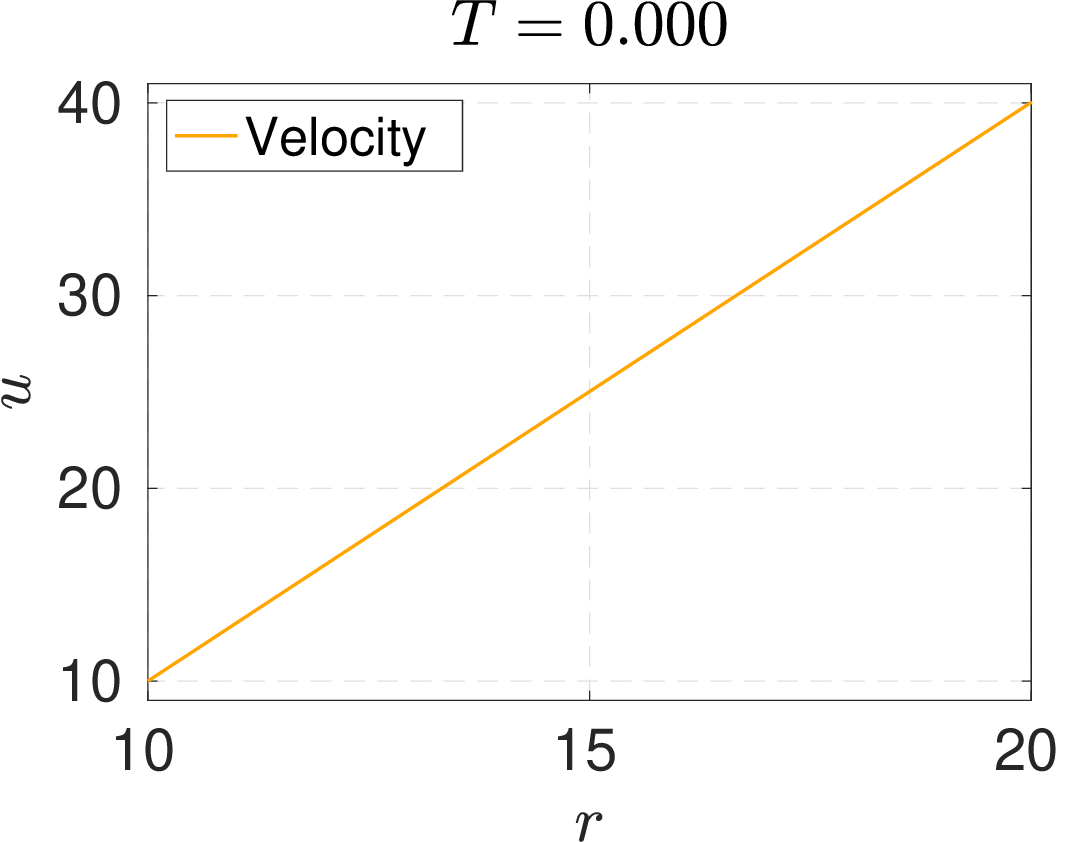}
  \end{subfigure}
    \begin{subfigure}{0.32\textwidth}
    \centering
    \includegraphics[width=1.0\textwidth]{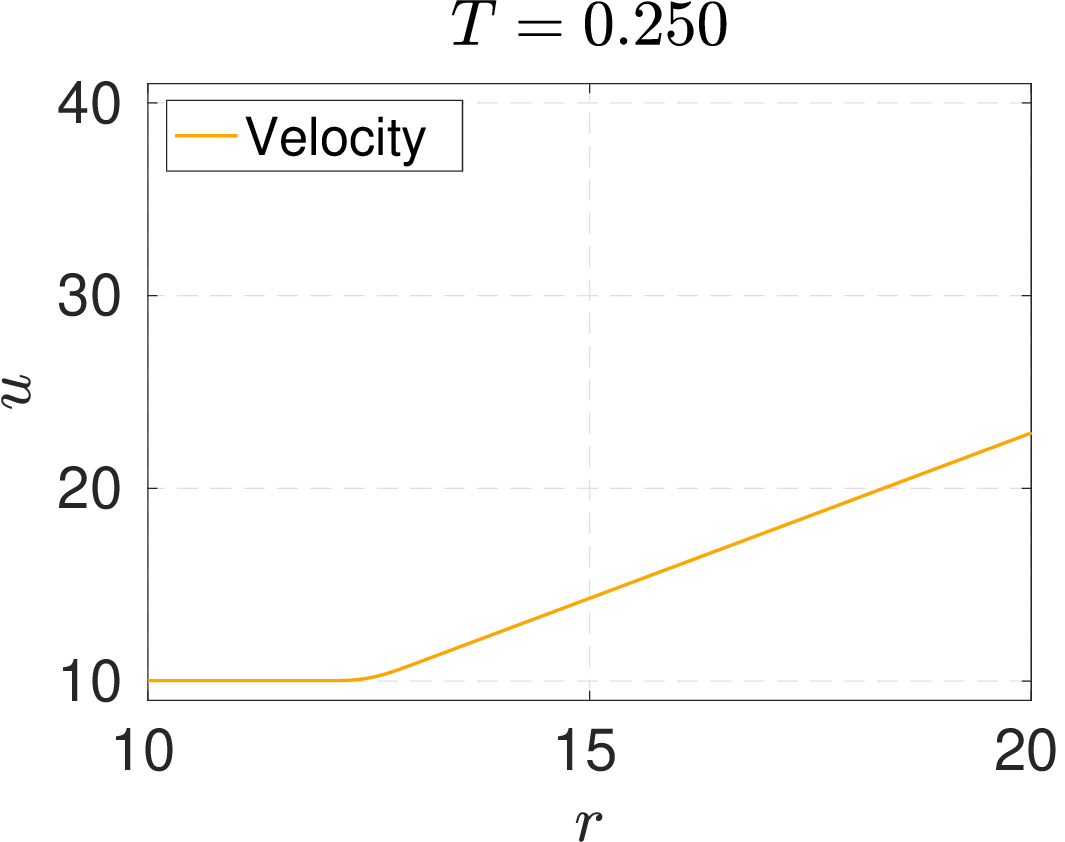}
  \end{subfigure}
  \begin{subfigure}{0.32\textwidth}
    \centering
    \includegraphics[width=1.0\textwidth]{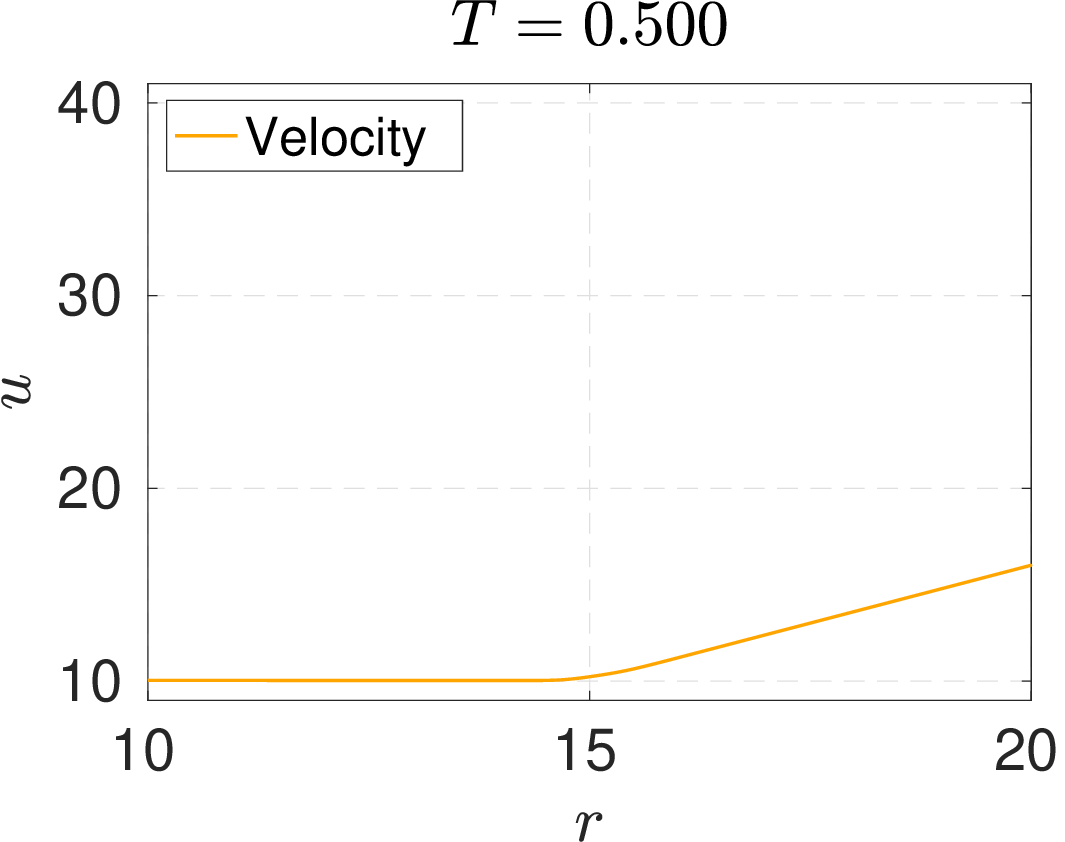}
  \end{subfigure}
  \caption{The initial \textit{velocity} is shown on the left and its time-evolved state on the center and on the right.}
  \label{fig:velocity_case2}
\end{figure}

% ------------------------------------------------------------------------

\begin{figure}[H]
  \centering
  \begin{subfigure}{0.32\textwidth}
    \centering
    \includegraphics[width=1.0\textwidth]{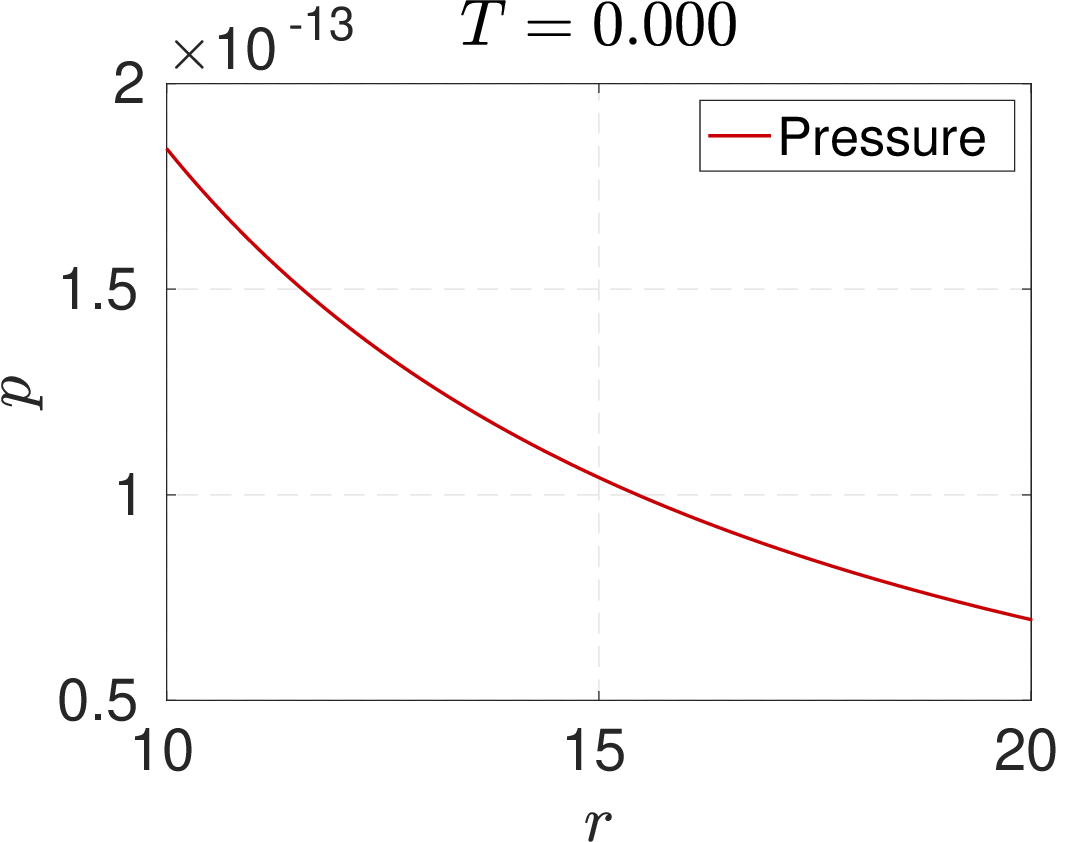}
  \end{subfigure}
    \begin{subfigure}{0.32\textwidth}
    \centering
    \includegraphics[width=1.0\textwidth]{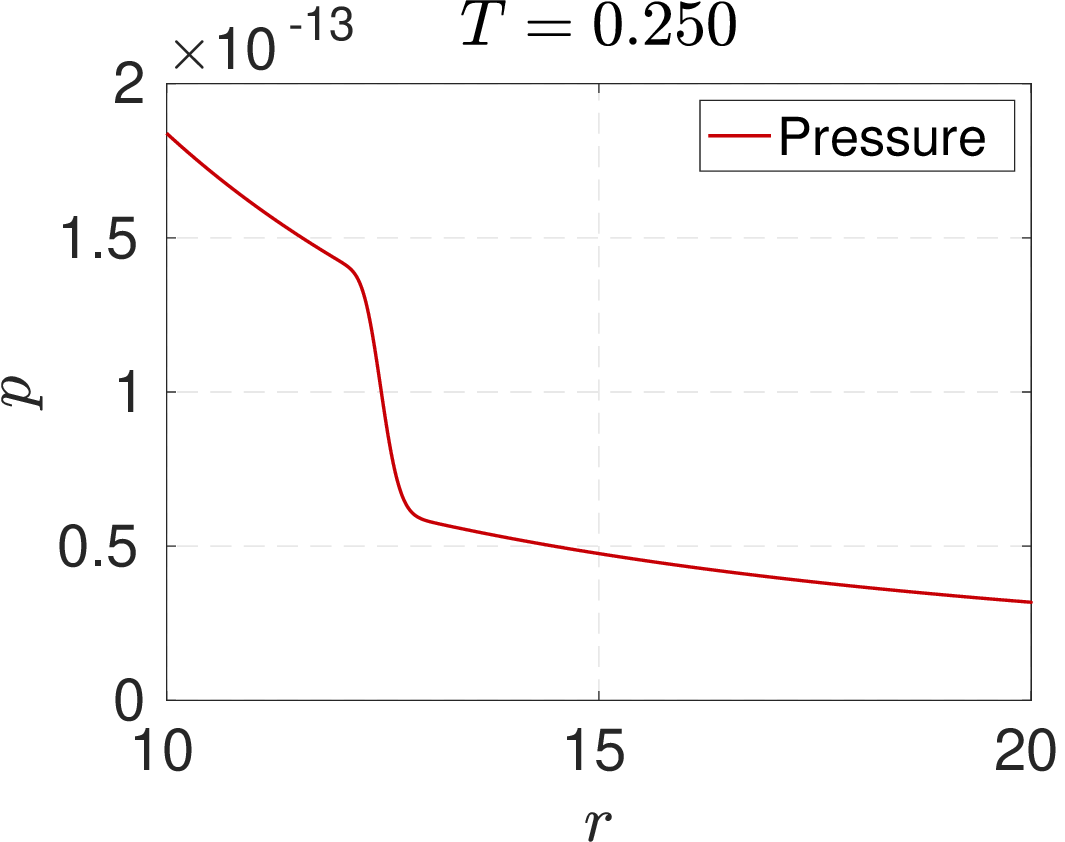}
  \end{subfigure}
  \begin{subfigure}{0.32\textwidth}
    \centering
    \includegraphics[width=1.0\textwidth]{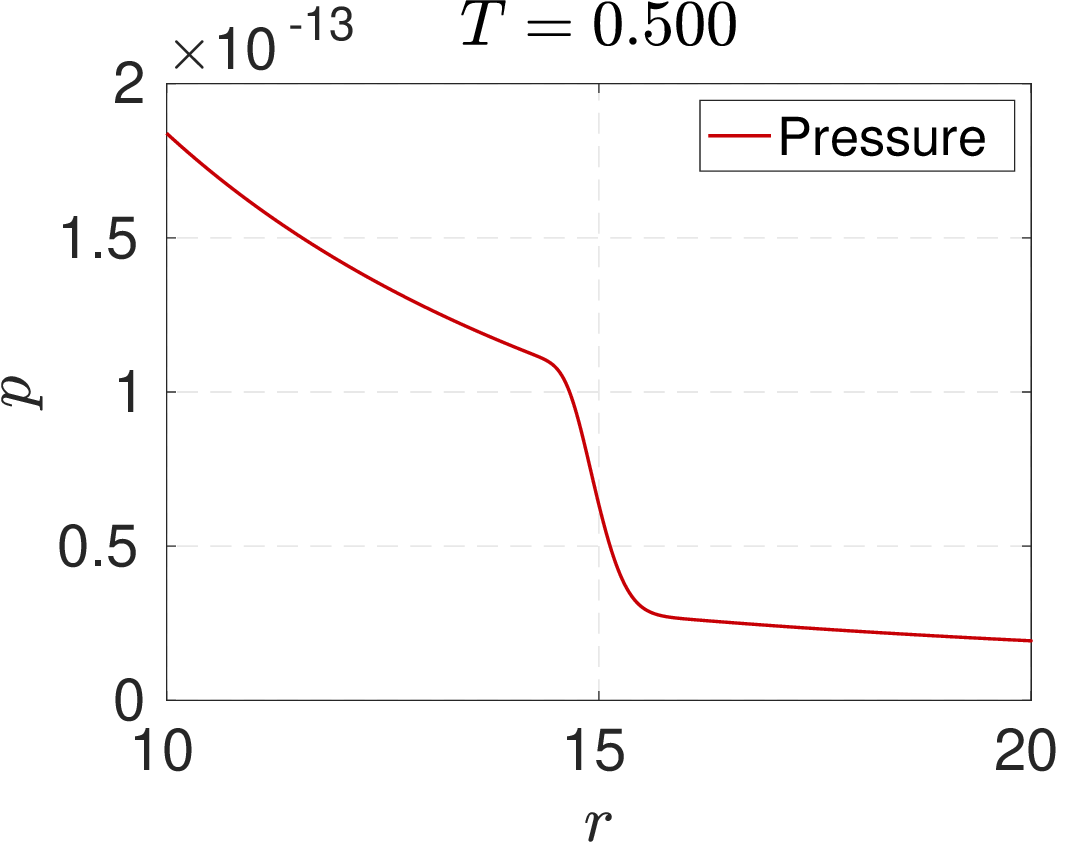}
  \end{subfigure}
  \caption{The initial \textit{pressure} is shown on the left and its time-evolved state on the center and on the right.}
  \label{fig:pressure_case2}
\end{figure}

% ------------------------------------------------------------------------

\begin{figure}[H]
  \centering
  \begin{subfigure}{0.32\textwidth}
    \centering
    \includegraphics[width=1.0\textwidth]{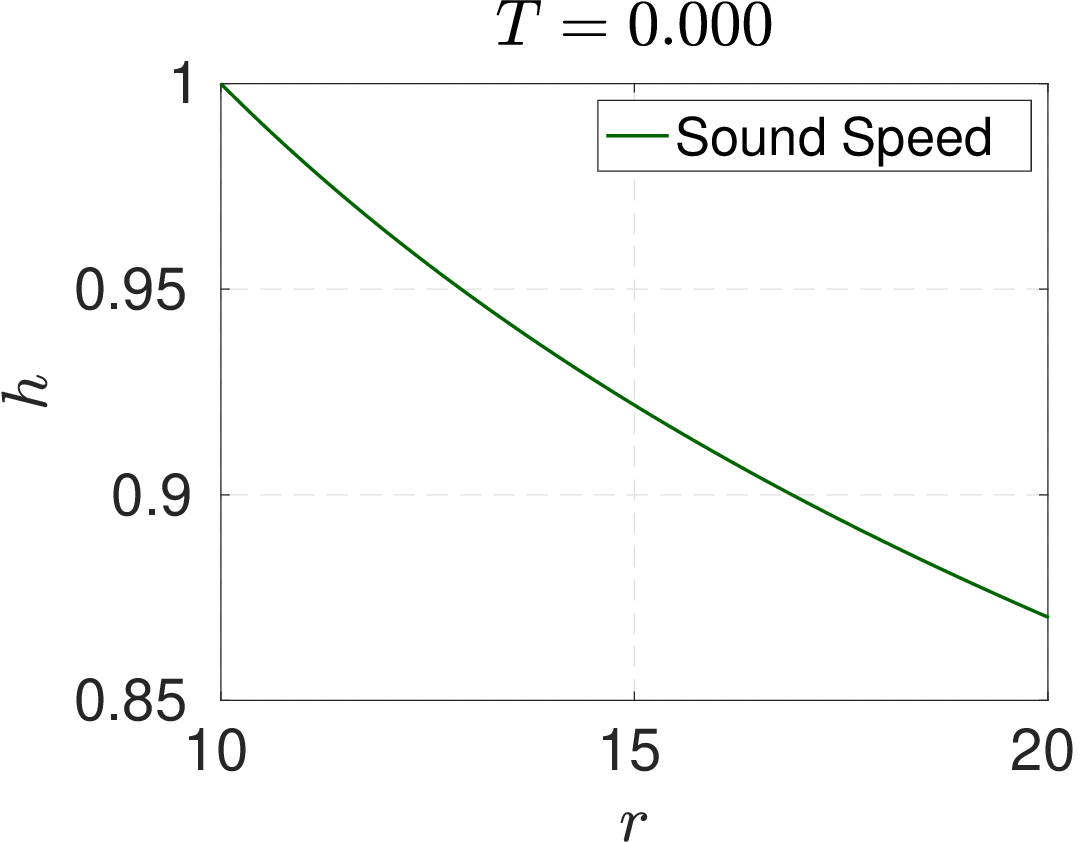}
  \end{subfigure}
    \begin{subfigure}{0.32\textwidth}
    \centering
    \includegraphics[width=1.0\textwidth]{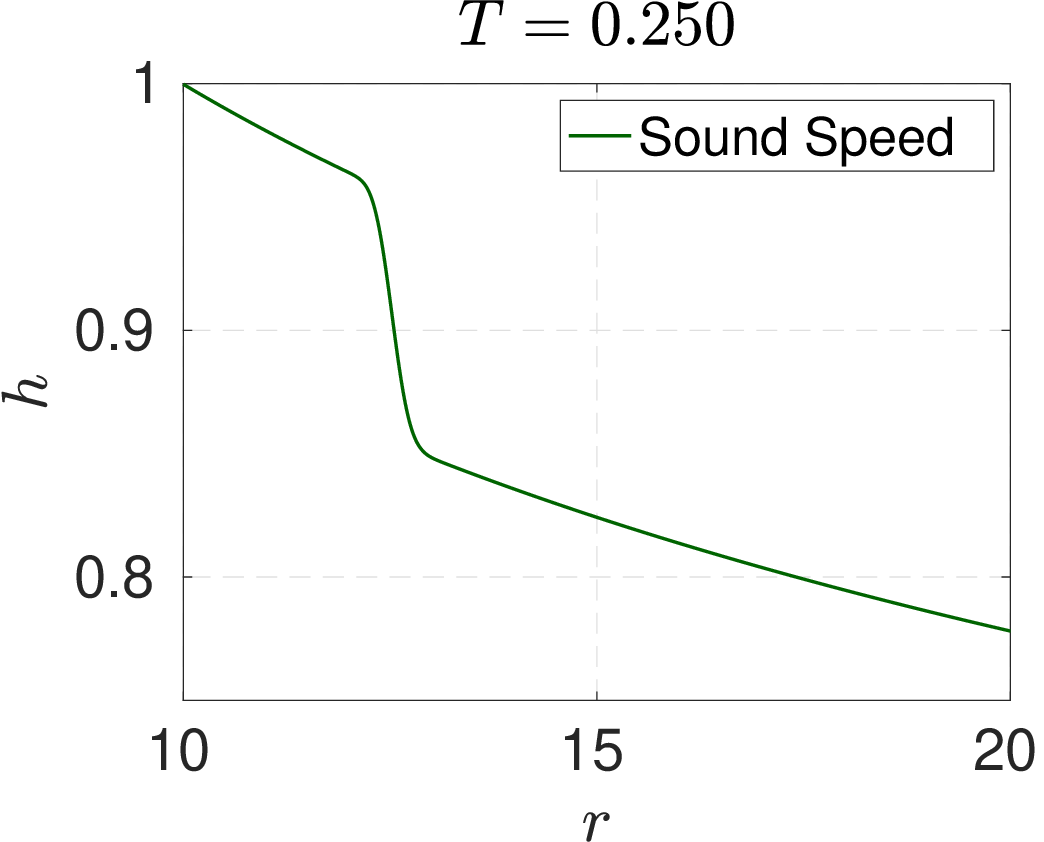}
  \end{subfigure}
  \begin{subfigure}{0.32\textwidth}
    \centering
    \includegraphics[width=1.0\textwidth]{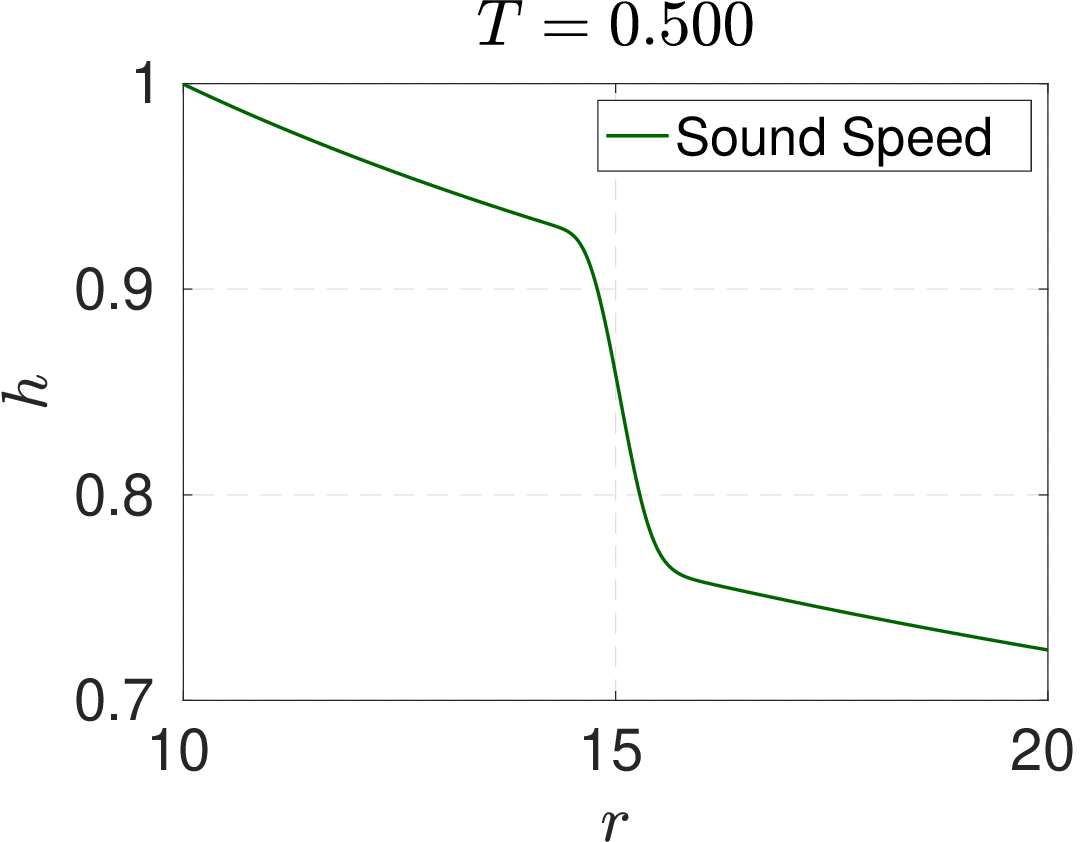}
  \end{subfigure}
  \caption{The initial \textit{sound speed} is shown on the left and its time-evolved state on the center and on the right.}
  \label{fig:sound-speed_case2}
\end{figure}

% ------------------------------------------------------------------------

\begin{figure}[H]
  \centering
  \begin{subfigure}{0.32\textwidth}
    \centering
    \includegraphics[width=1.0\textwidth]{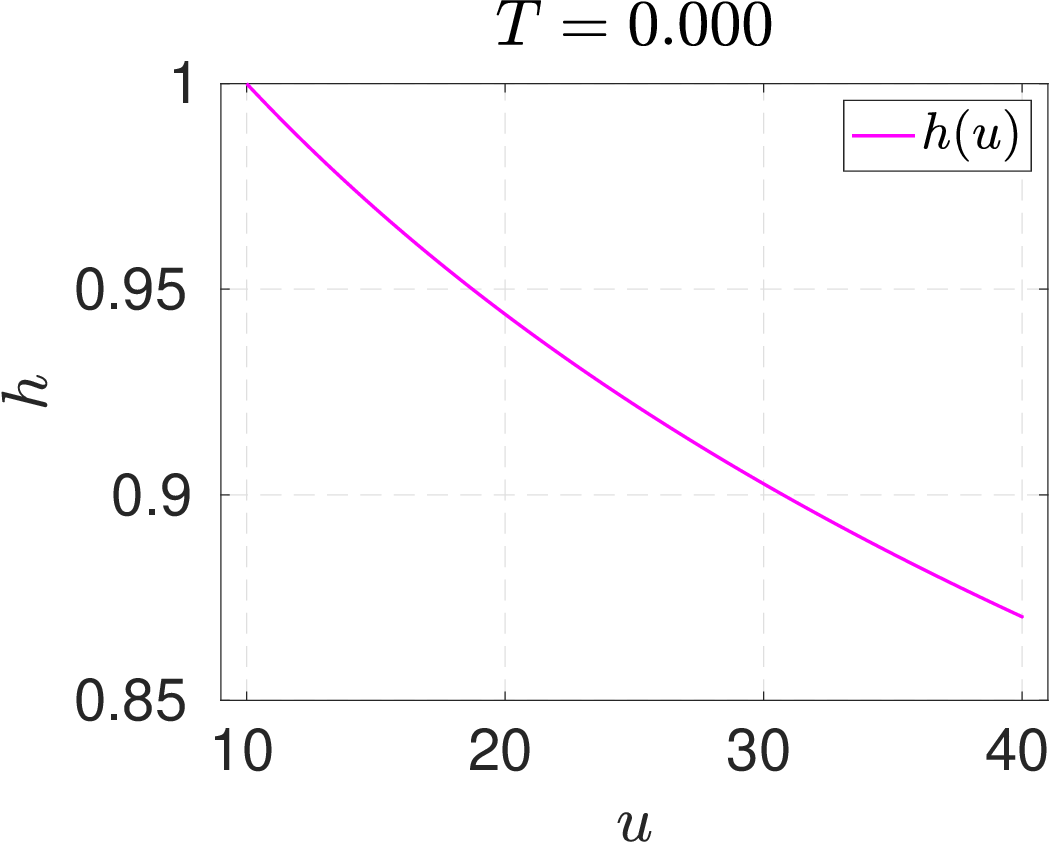}
  \end{subfigure}
    \begin{subfigure}{0.32\textwidth}
    \centering
    \includegraphics[width=1.0\textwidth]{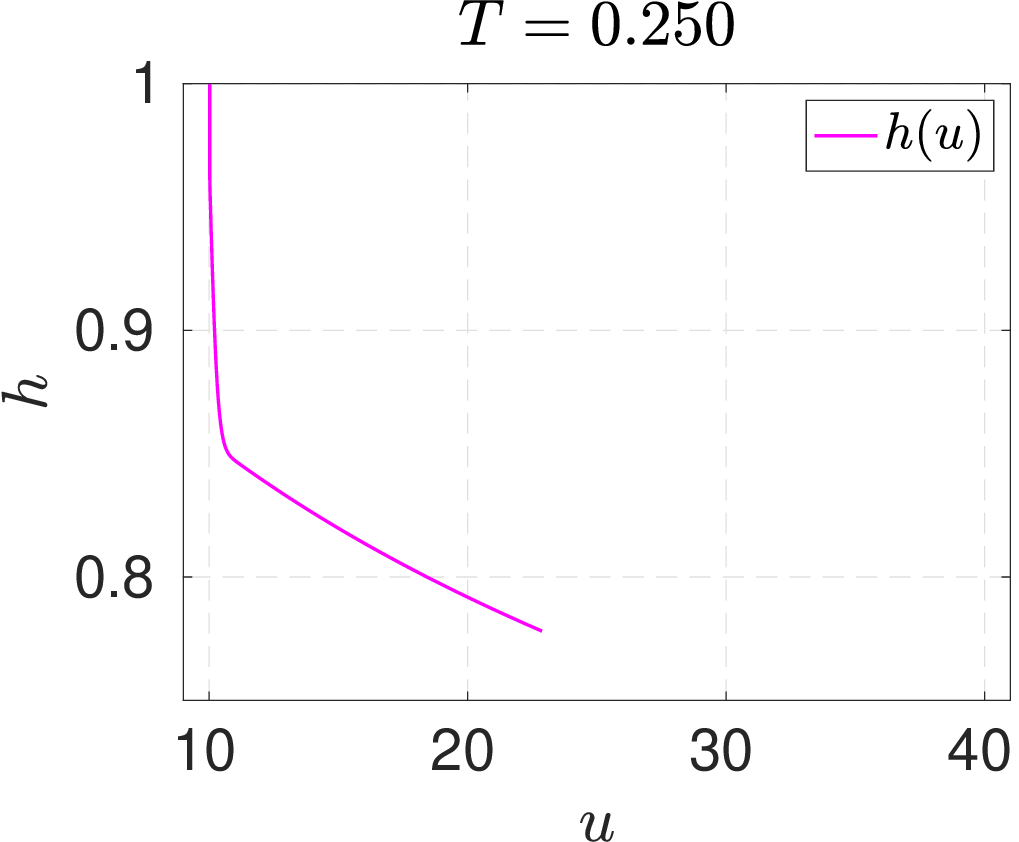}
  \end{subfigure}
  \begin{subfigure}{0.32\textwidth}
    \centering
    \includegraphics[width=1.0\textwidth]{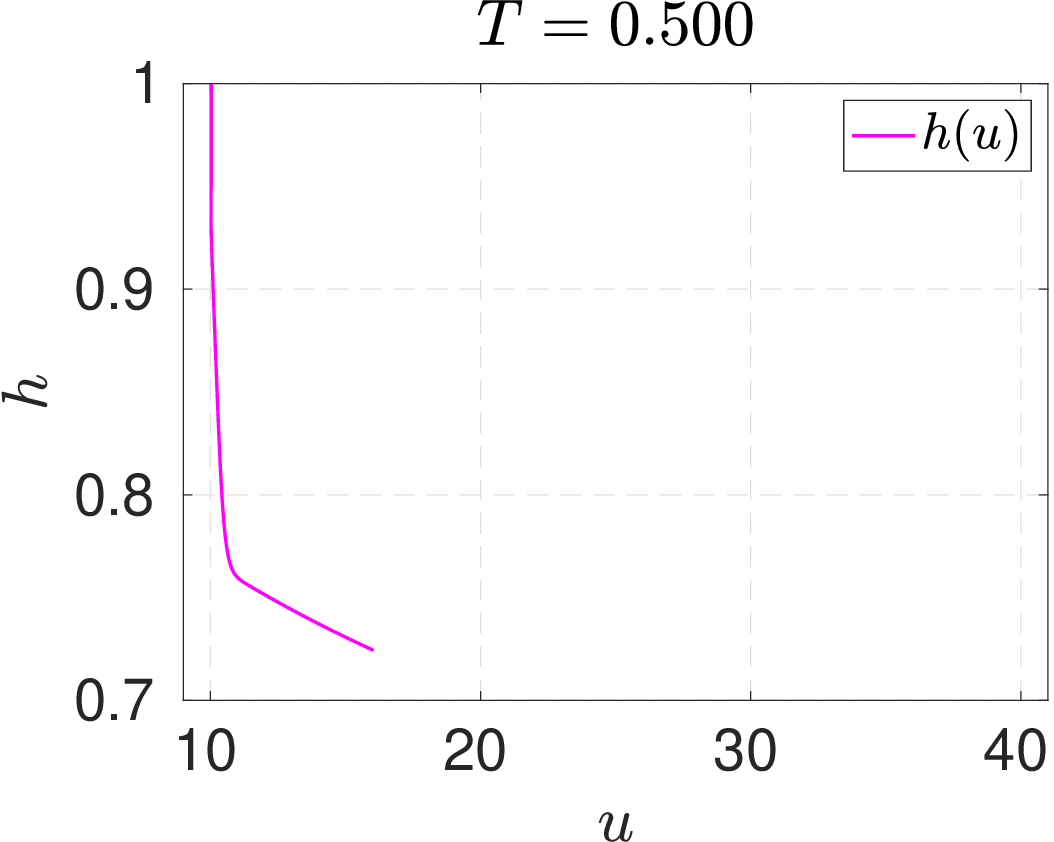}
  \end{subfigure}
  \caption{The initial \textit{invariant curve in \((u,h)-\)plane} is shown on the left and its time-evolved state on the center and on the right.}
  \label{fig:invariant-curve_case2}
\end{figure}

% ------------------------------------------------------------------------

\begin{figure}[H]
  \centering
  \begin{subfigure}{0.49\textwidth}
    \centering
    \includegraphics[width=1.0\textwidth]{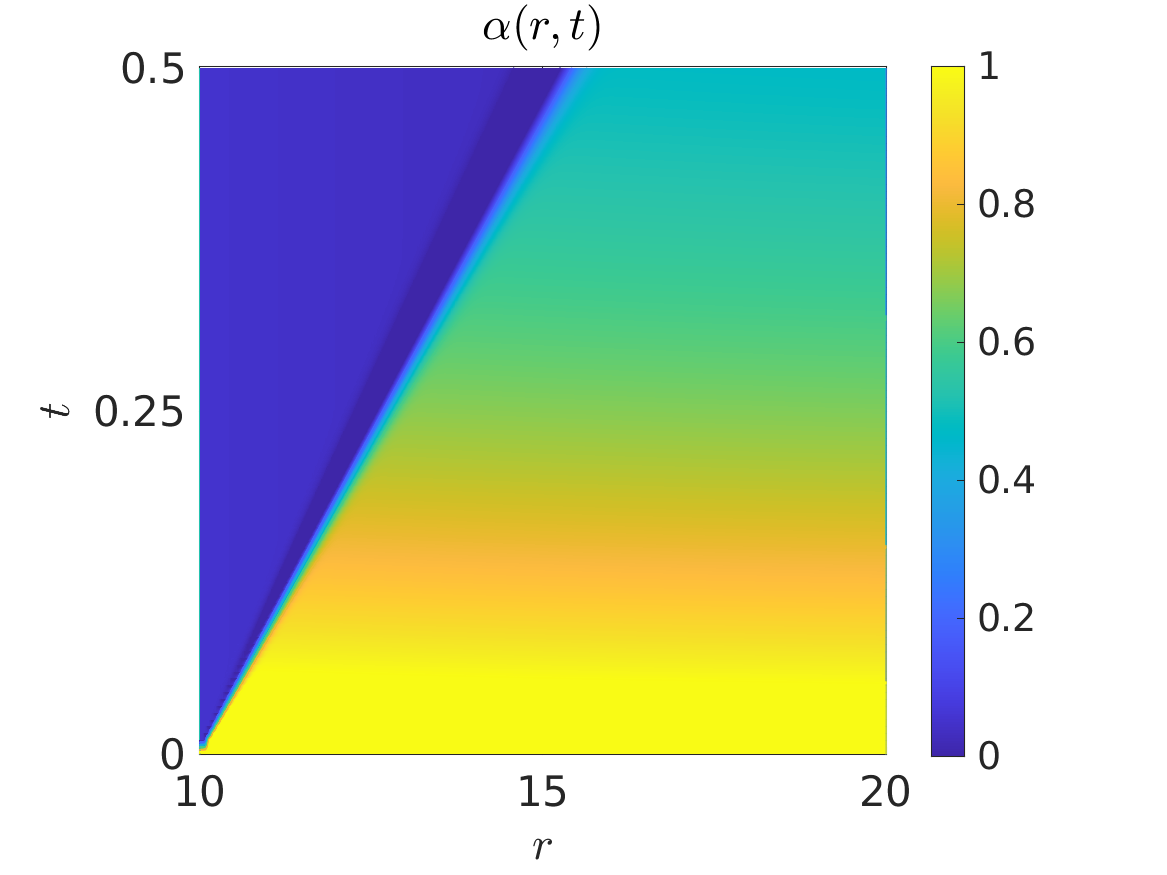}
  \end{subfigure}
  \begin{subfigure}{0.49\textwidth}
    \centering
	\includegraphics[width=1.0\textwidth]{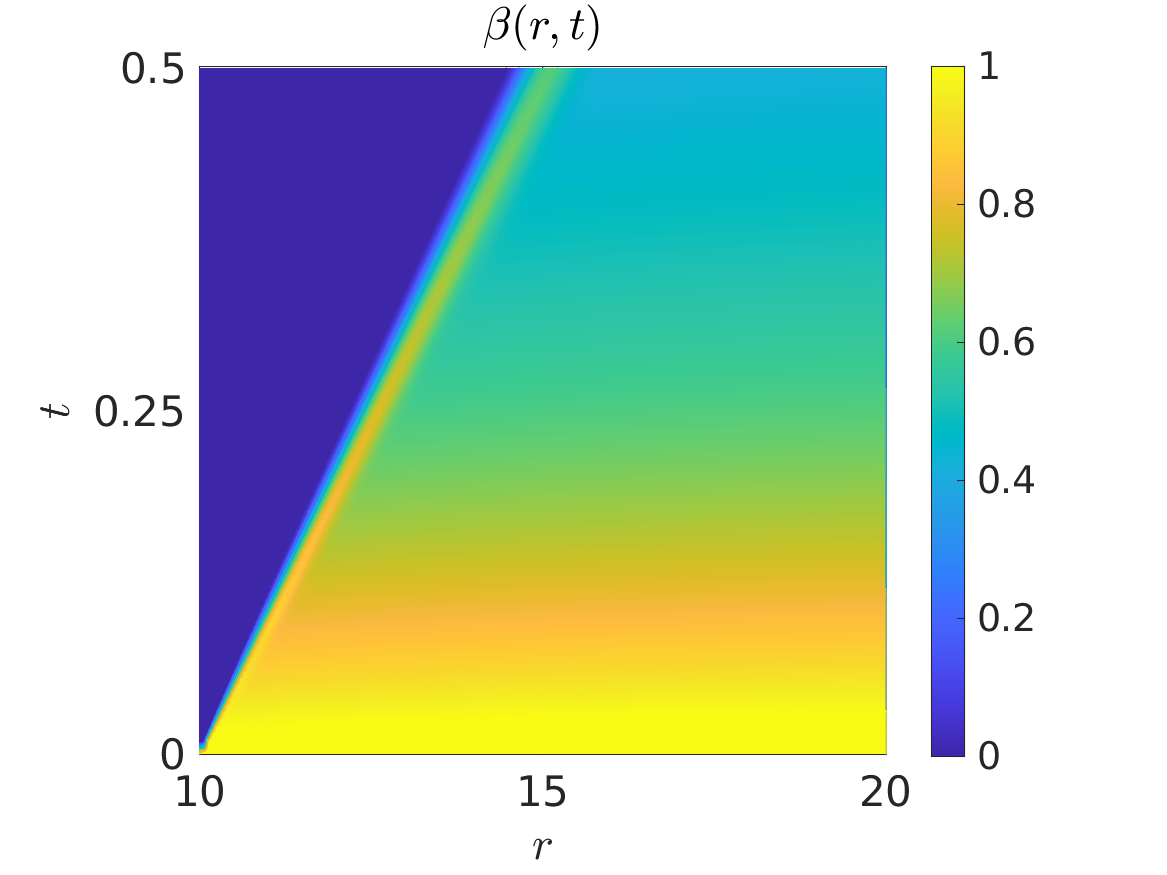}
  \end{subfigure}
  \caption{The \textit{heat map of $\alpha$ in \((r,t)-\)plane} is shown on the left and the heat map of $\beta$ on the right.}
  \label{fig:heat-map_case2}
\end{figure}

% ------------------------------------------------------------------------

\subsection*{Case 3: subsonic oscillatory regime (periodic boundary conditions)}

To probe the subsonic mechanism and the near-periodic behaviour suggested in \cite{F1}, we consider
\[
K=1,\quad \gamma=3,\quad \rho(r,0)=5,\quad u(r,0)=\frac{1}{\varepsilon}\sin(r-10),\quad r\in[10,10+2\pi],
\]
with periodic boundary conditions. We test three magnitudes of the oscillatory perturbation, $\varepsilon\in\{10.0,1.0,0.1\}$, to assess robustness with respect to amplitude.

For $\varepsilon=10$ (Case 3.1) and $\varepsilon=1$ (Case 3.2), the solution remains smooth and oscillatory (for Case 1, see figures \ref{fig:density_case3a}-\ref{fig:invariant-curve_case3a} and for Case 2, see figures \ref{fig:density_case3b}-\ref{fig:invariant-curve_case3b}), and the heat maps show bounded, structured patterns for $\alpha$ and $\beta$ compatible with sustained subsonic dynamics (see figures \ref{fig:heat-map_case3a} and \ref{fig:heat-map_case3b}). In contrast, for $\varepsilon=0.1$ (Case 3.3) the increased amplitude induces strong steepening and the onset of a shock-like breakdown in finite time (see the solutions in figures \ref{fig:density_case3c}-\ref{fig:invariant-curve_case3c}); this transition is visible both in the physical variables and in the corresponding concentration of negative gradients captured by $\alpha$ and $\beta$ (see Figure \ref{fig:heat-map_case3c}).

\subsubsection*{Case 3.1: $\varepsilon = 10$.}

% ------------------------------------------------------------------------

\begin{figure}[H]
  \centering
  \begin{subfigure}{0.32\textwidth}
    \centering
    \includegraphics[width=1.0\textwidth]{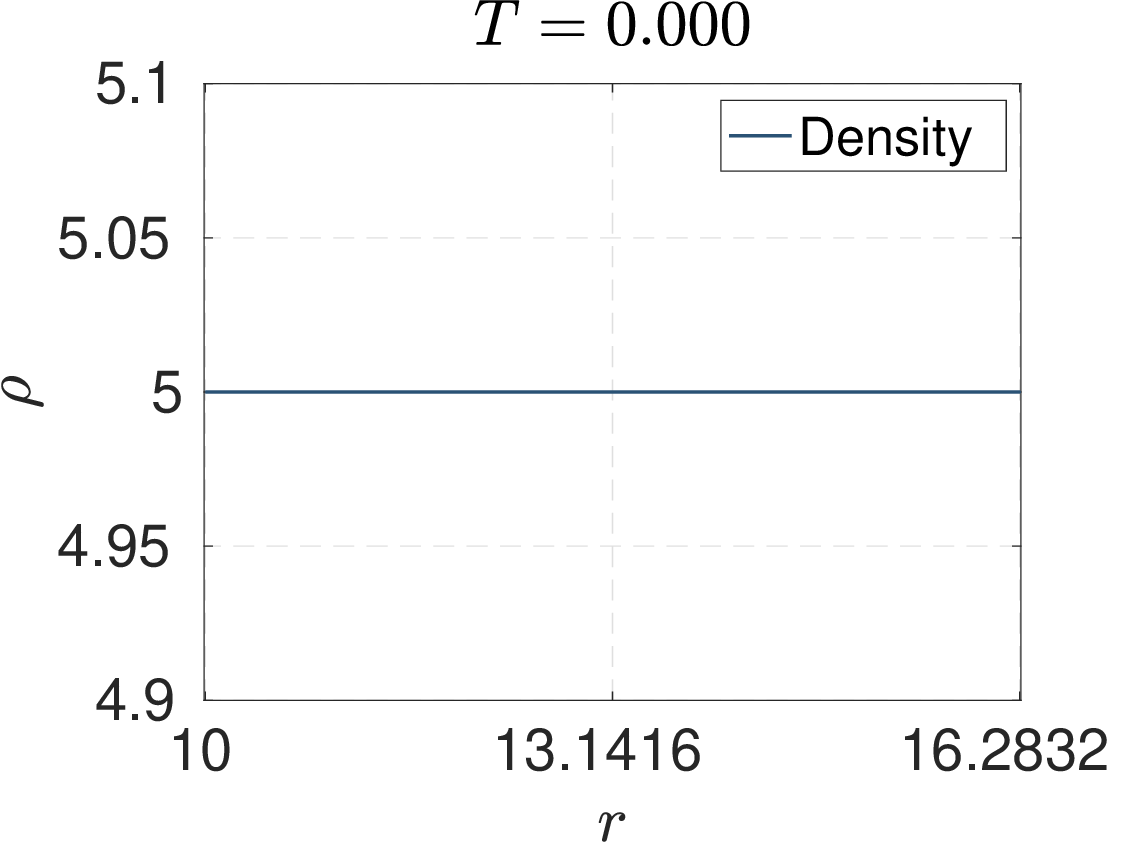}
  \end{subfigure}
    \begin{subfigure}{0.32\textwidth}
    \centering
    \includegraphics[width=1.0\textwidth]{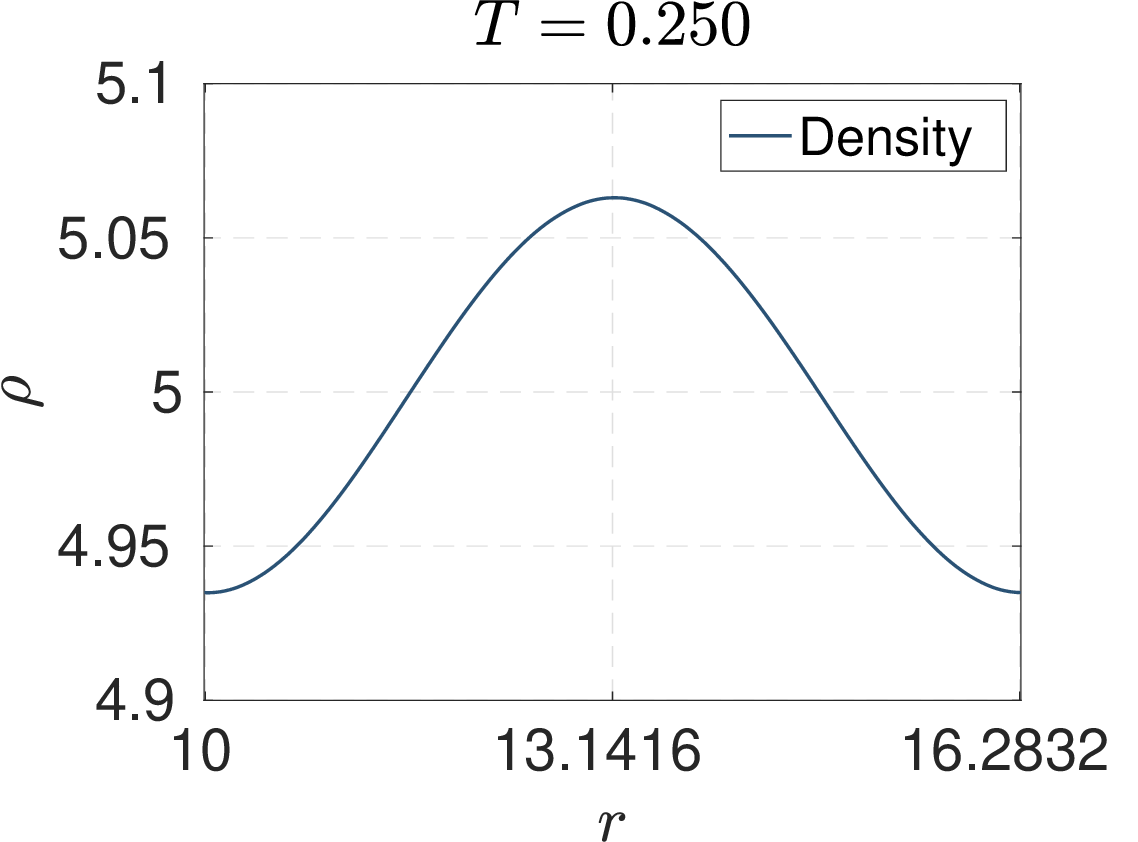}
  \end{subfigure}
  \begin{subfigure}{0.32\textwidth}
    \centering
    \includegraphics[width=1.0\textwidth]{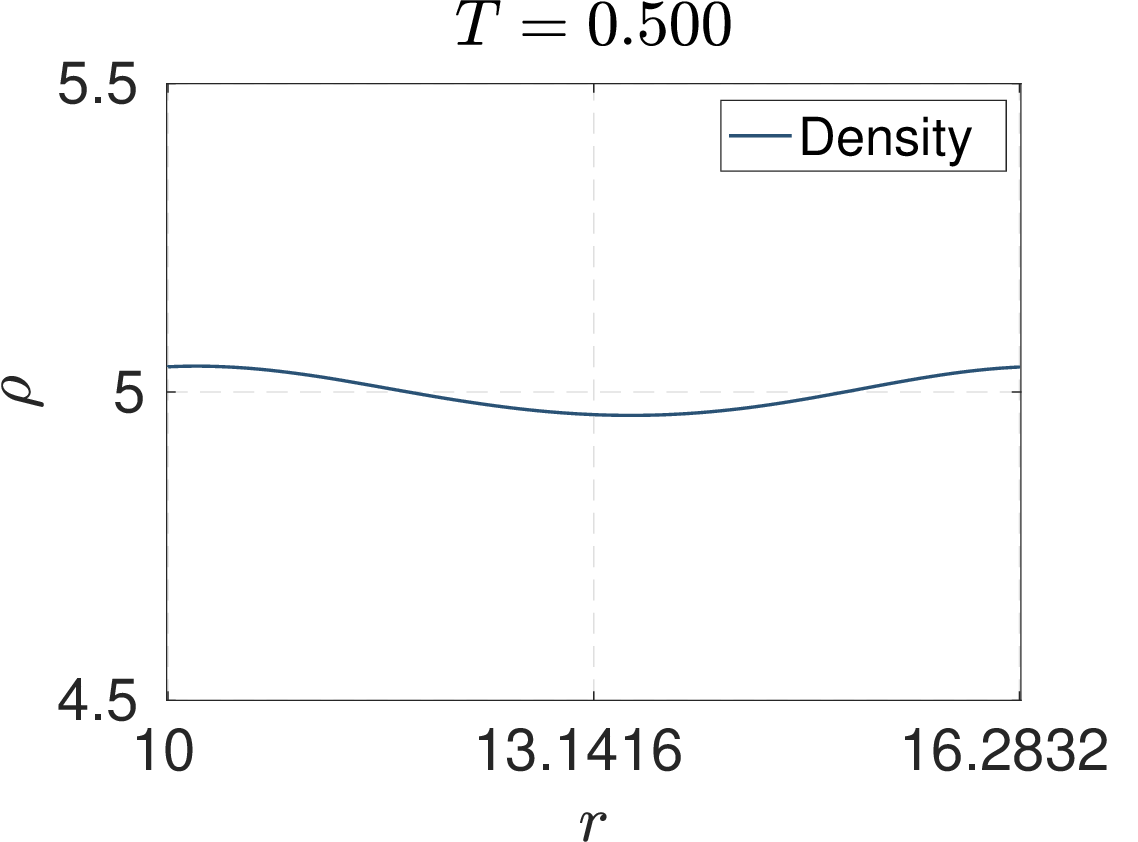}
  \end{subfigure}
  \caption{The initial \textit{density} is shown on the left and its time-evolved state on the center and on the right.}
  \label{fig:density_case3a}
\end{figure}

% ------------------------------------------------------------------------

\begin{figure}[H]
  \centering
  \begin{subfigure}{0.32\textwidth}
    \centering
    \includegraphics[width=1.0\textwidth]{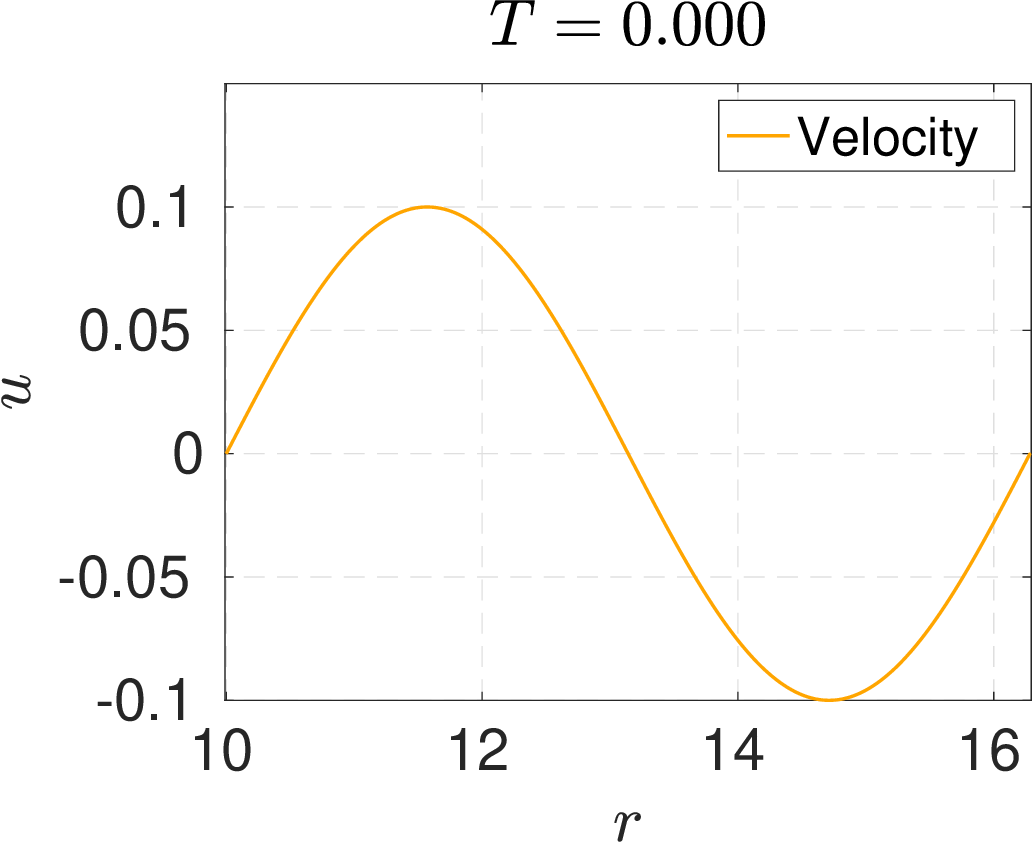}
  \end{subfigure}
    \begin{subfigure}{0.32\textwidth}
    \centering
    \includegraphics[width=1.0\textwidth]{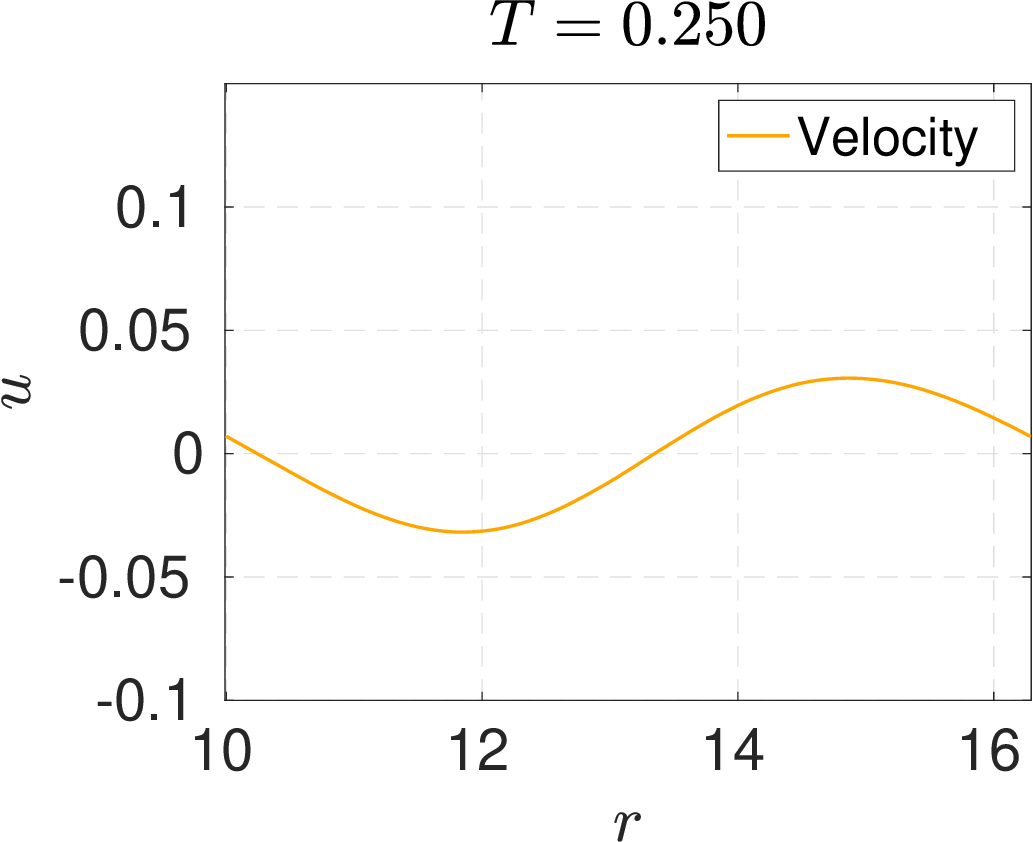}
  \end{subfigure}
  \begin{subfigure}{0.32\textwidth}
    \centering
    \includegraphics[width=1.0\textwidth]{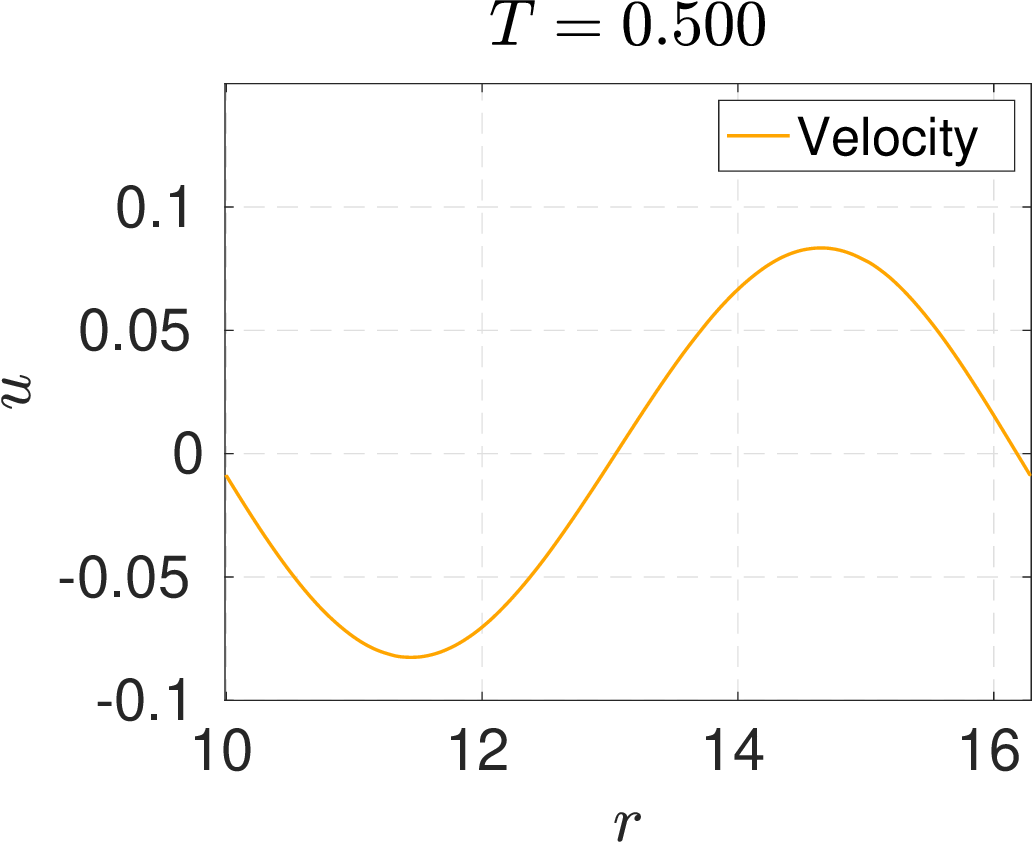}
  \end{subfigure}
  \caption{The initial \textit{velocity} is shown on the left and its time-evolved state on the center and on the right.}
  \label{fig:velocity_case3a}
\end{figure}

% ------------------------------------------------------------------------

\begin{figure}[H]
  \centering
  \begin{subfigure}{0.32\textwidth}
    \centering
    \includegraphics[width=1.0\textwidth]{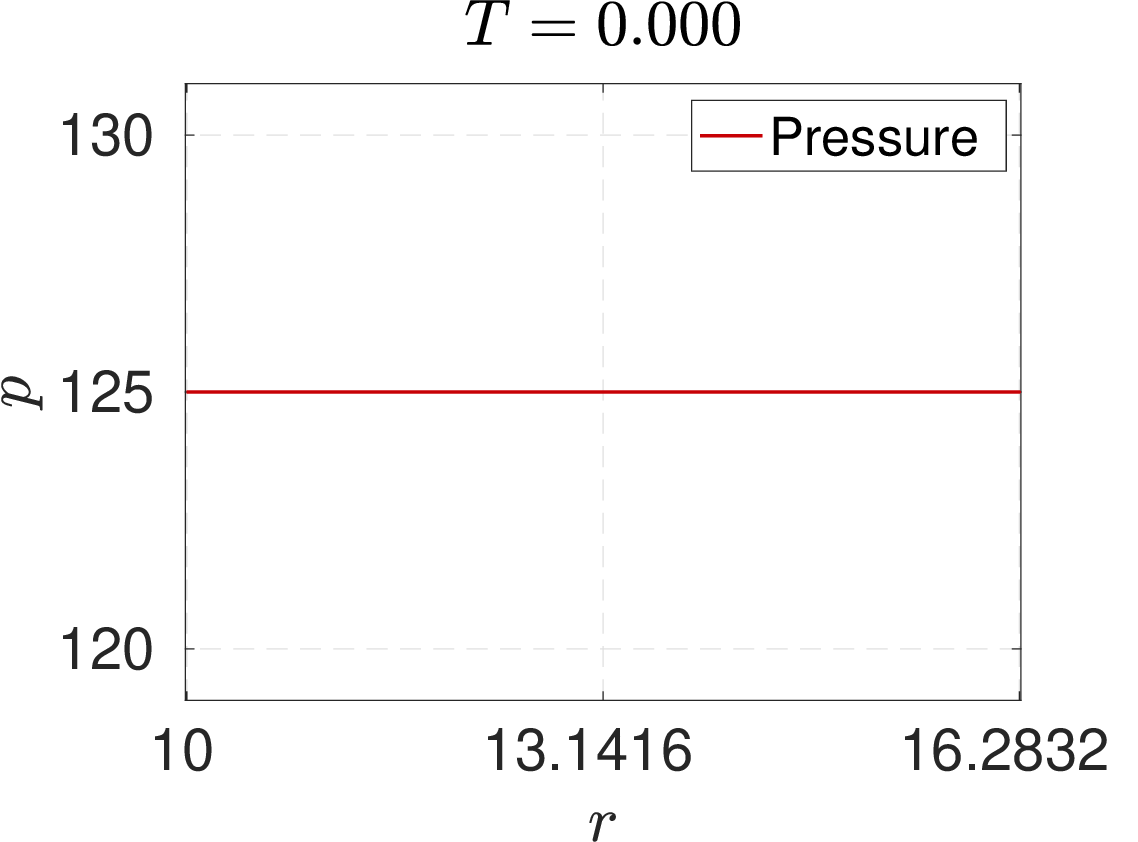}
  \end{subfigure}
    \begin{subfigure}{0.32\textwidth}
    \centering
    \includegraphics[width=1.0\textwidth]{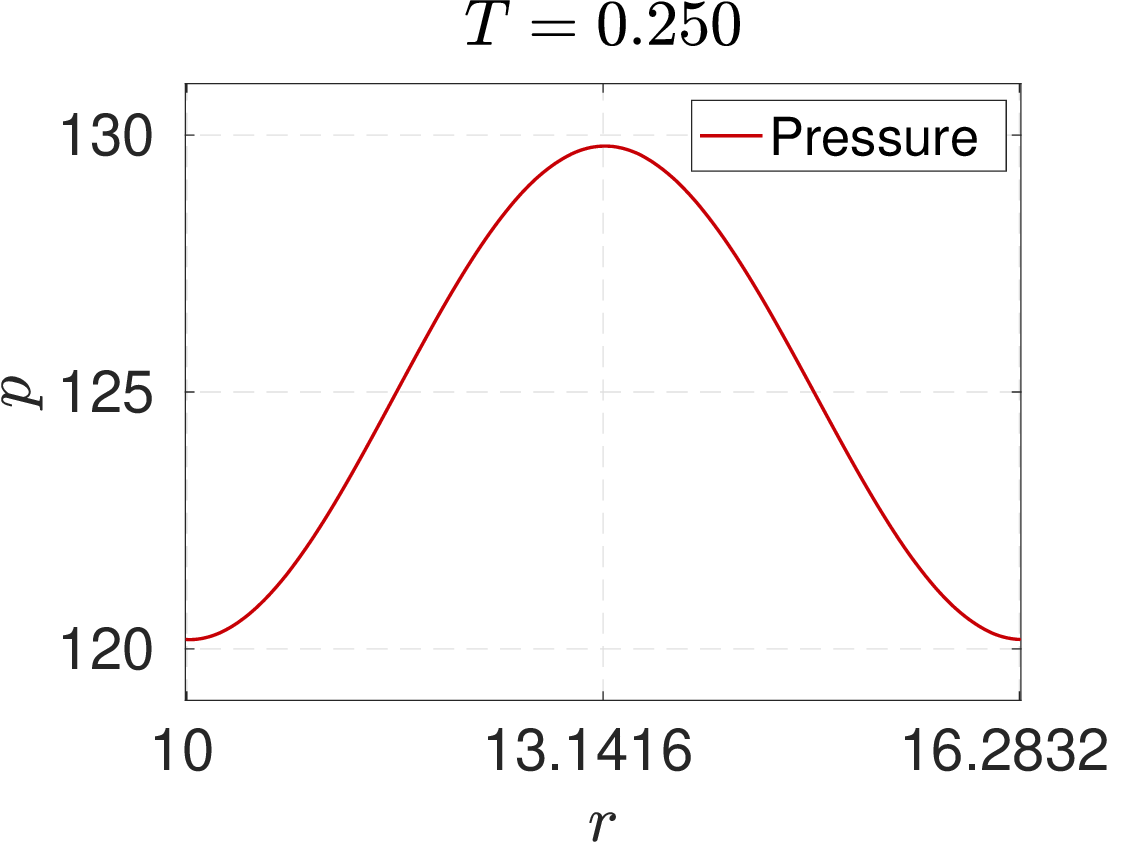}
  \end{subfigure}
  \begin{subfigure}{0.32\textwidth}
    \centering
    \includegraphics[width=1.0\textwidth]{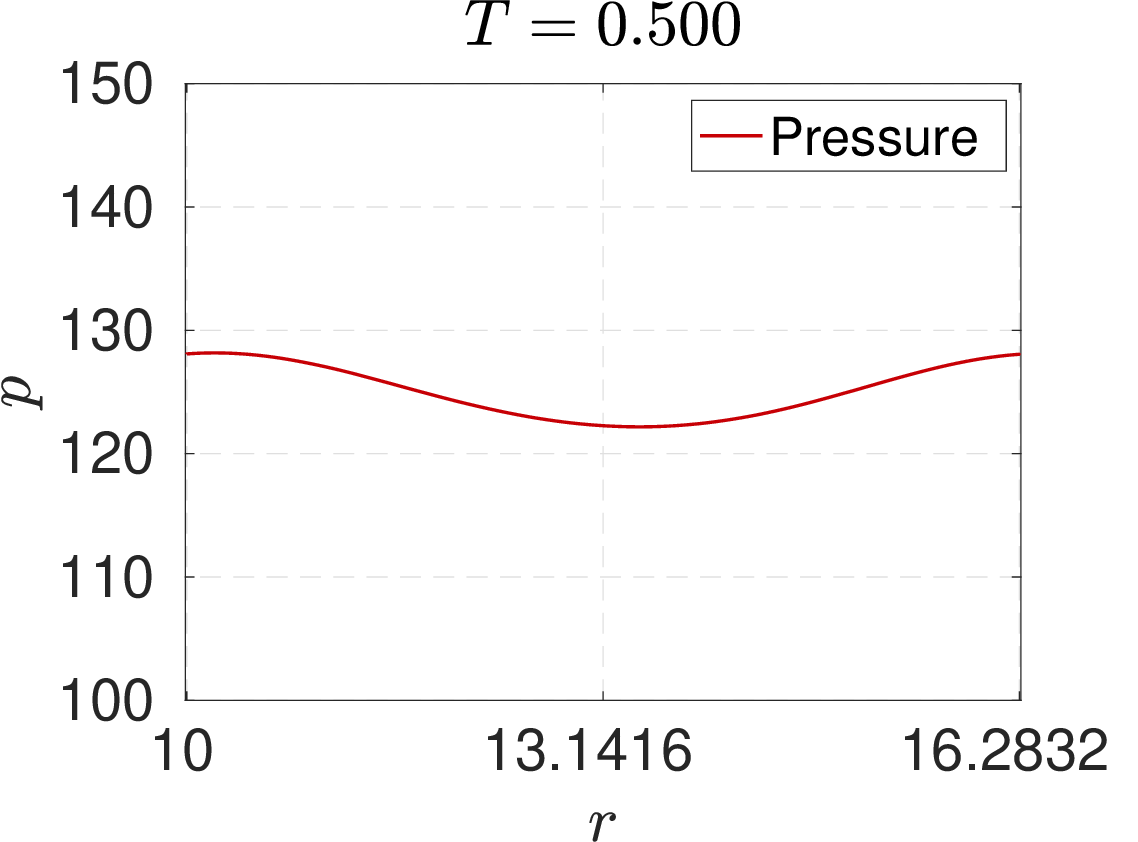}
  \end{subfigure}
  \caption{The initial \textit{pressure} is shown on the left and its time-evolved state on the center and on the right.}
  \label{fig:pressure_case3a}
\end{figure}

% ------------------------------------------------------------------------

\begin{figure}[H]
  \centering
  \begin{subfigure}{0.32\textwidth}
    \centering
    \includegraphics[width=1.0\textwidth]{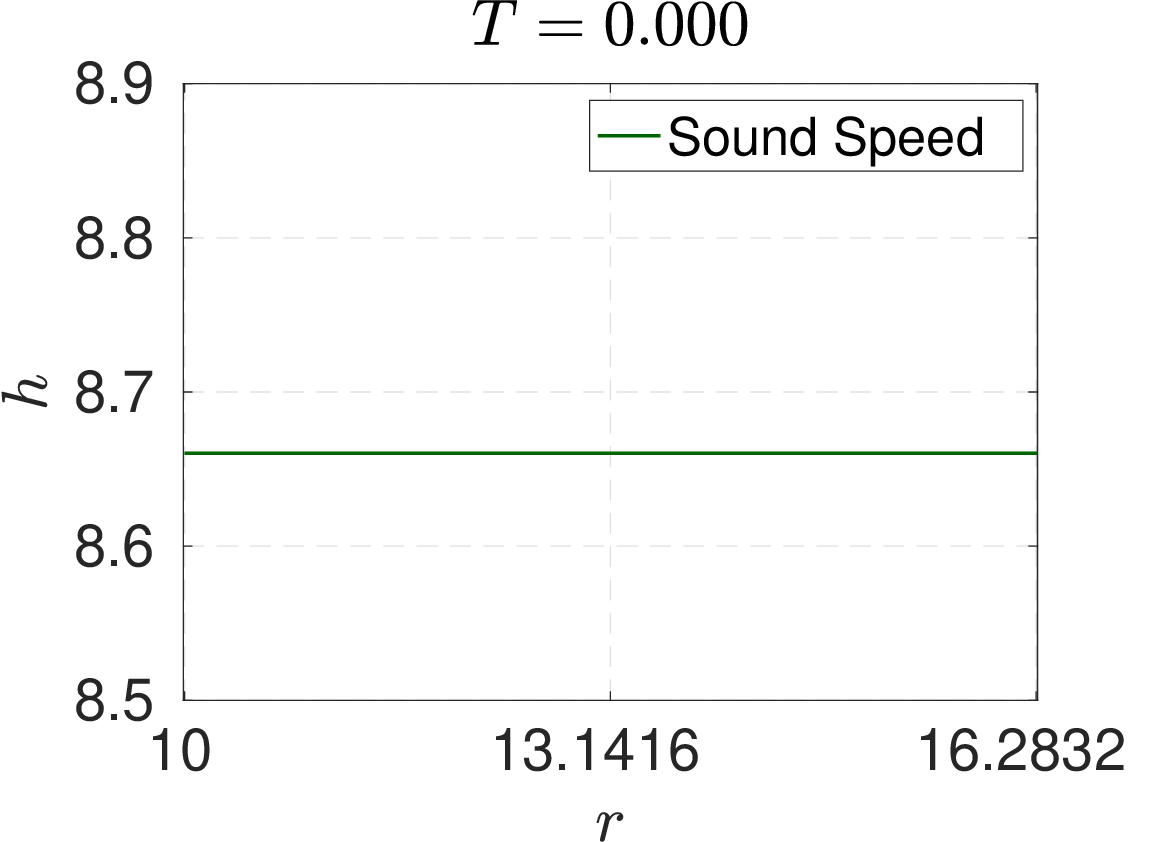}
  \end{subfigure}
    \begin{subfigure}{0.32\textwidth}
    \centering
    \includegraphics[width=1.0\textwidth]{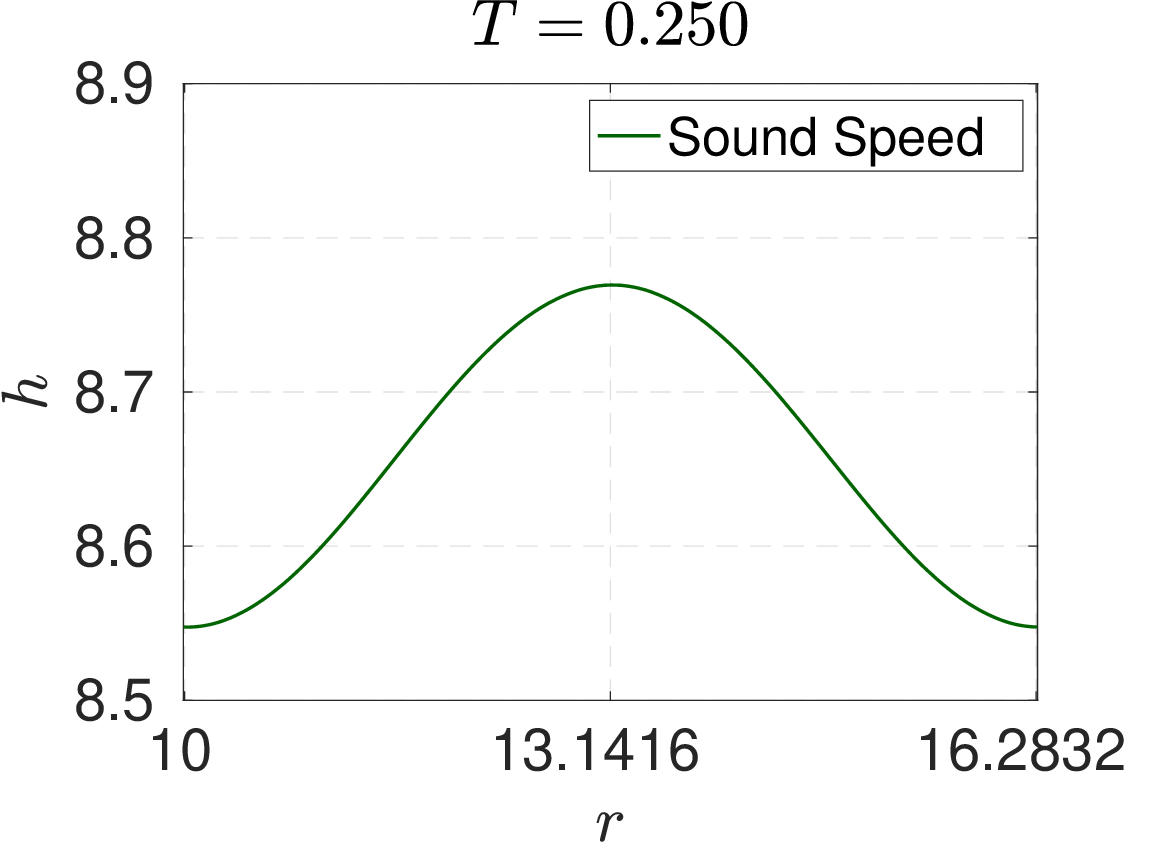}
  \end{subfigure}
  \begin{subfigure}{0.32\textwidth}
    \centering
    \includegraphics[width=1.0\textwidth]{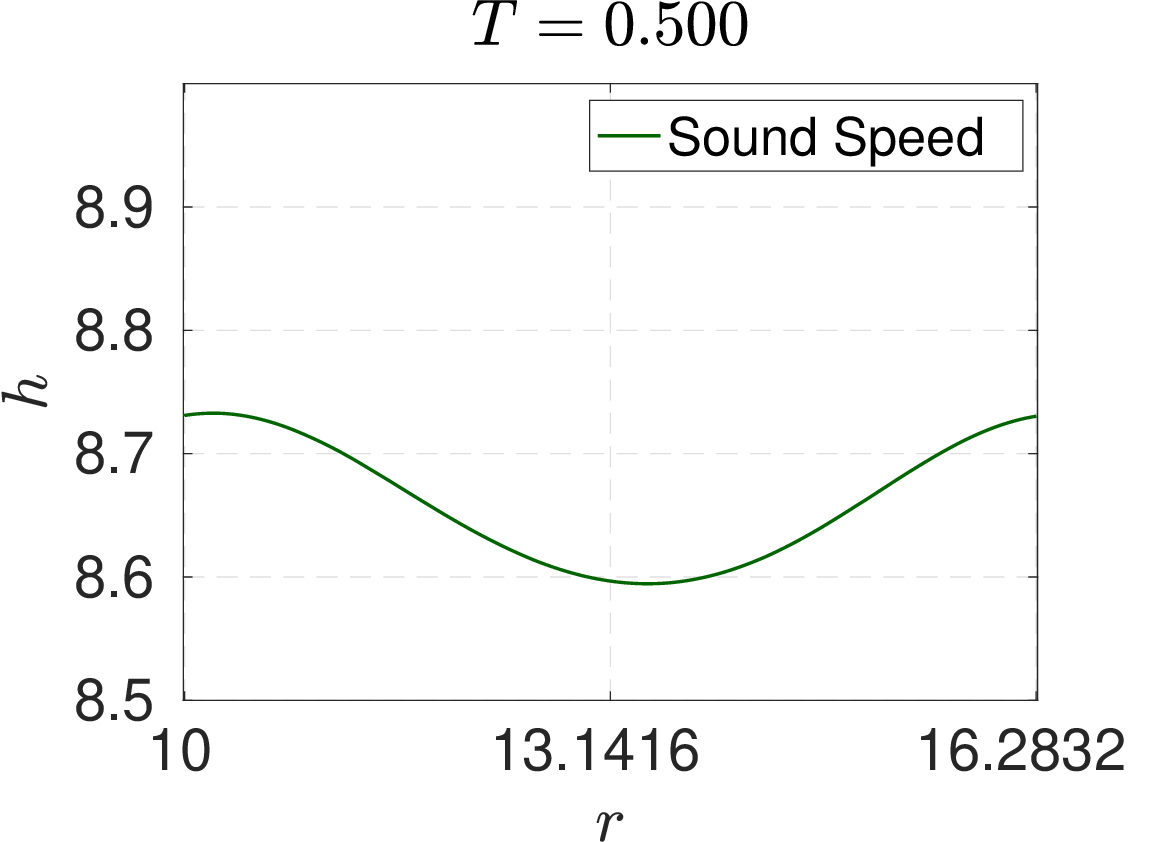}
  \end{subfigure}
  \caption{The initial \textit{sound speed} is shown on the left and its time-evolved state on the center and on the right.}
  \label{fig:sound-speed_case3a}
\end{figure}

% ------------------------------------------------------------------------

\begin{figure}[H]
  \centering
  \begin{subfigure}{0.32\textwidth}
    \centering
    \includegraphics[width=1.0\textwidth]{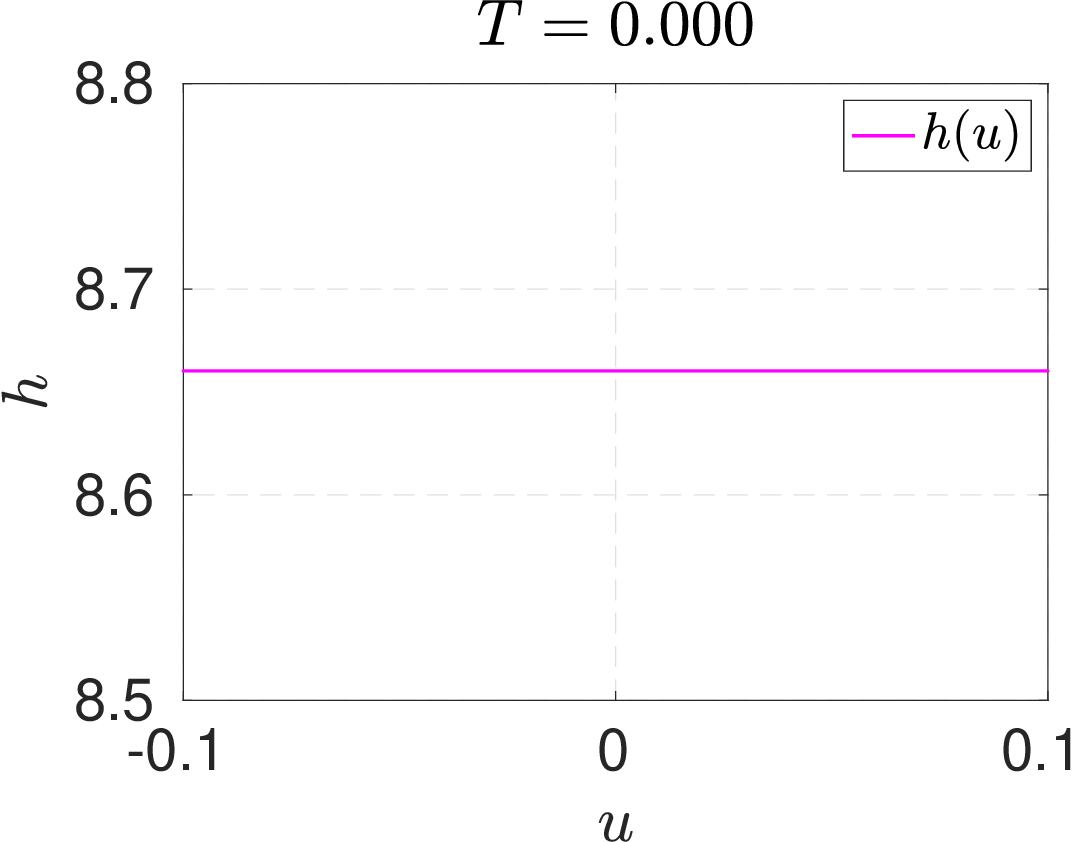}
  \end{subfigure}
    \begin{subfigure}{0.32\textwidth}
    \centering
    \includegraphics[width=1.0\textwidth]{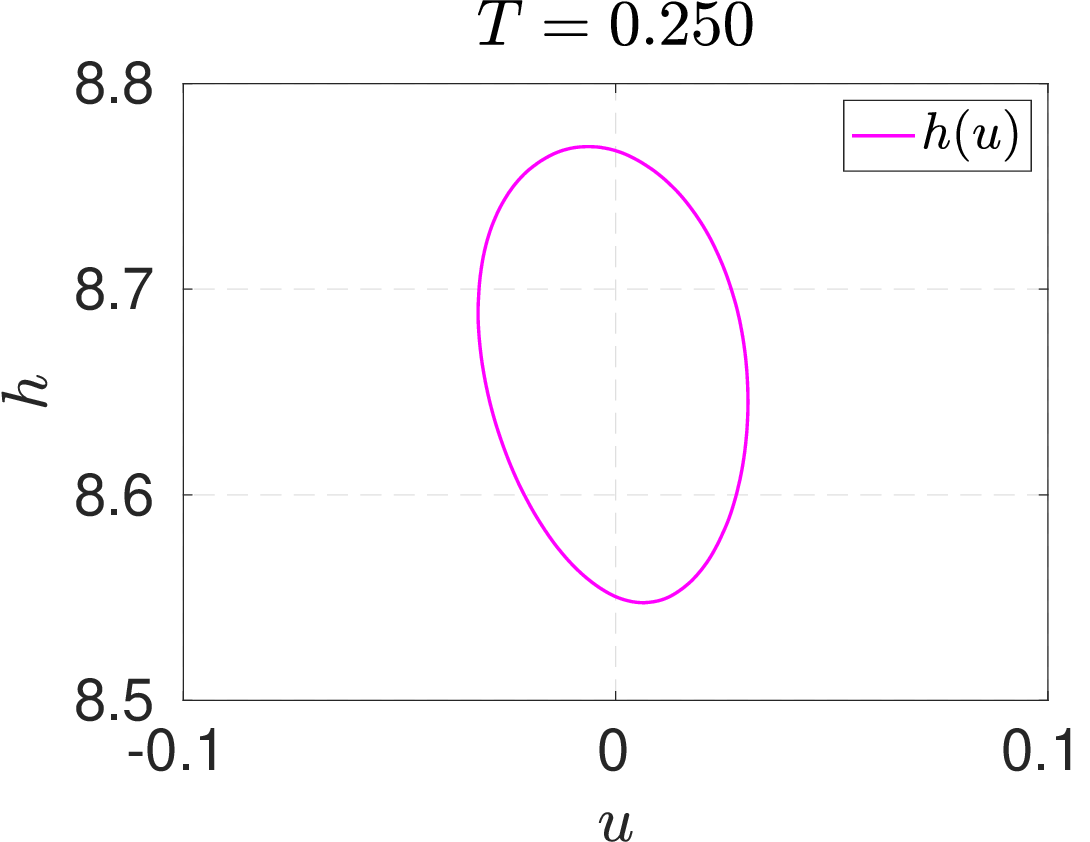}
  \end{subfigure}
  \begin{subfigure}{0.32\textwidth}
    \centering
    \includegraphics[width=1.0\textwidth]{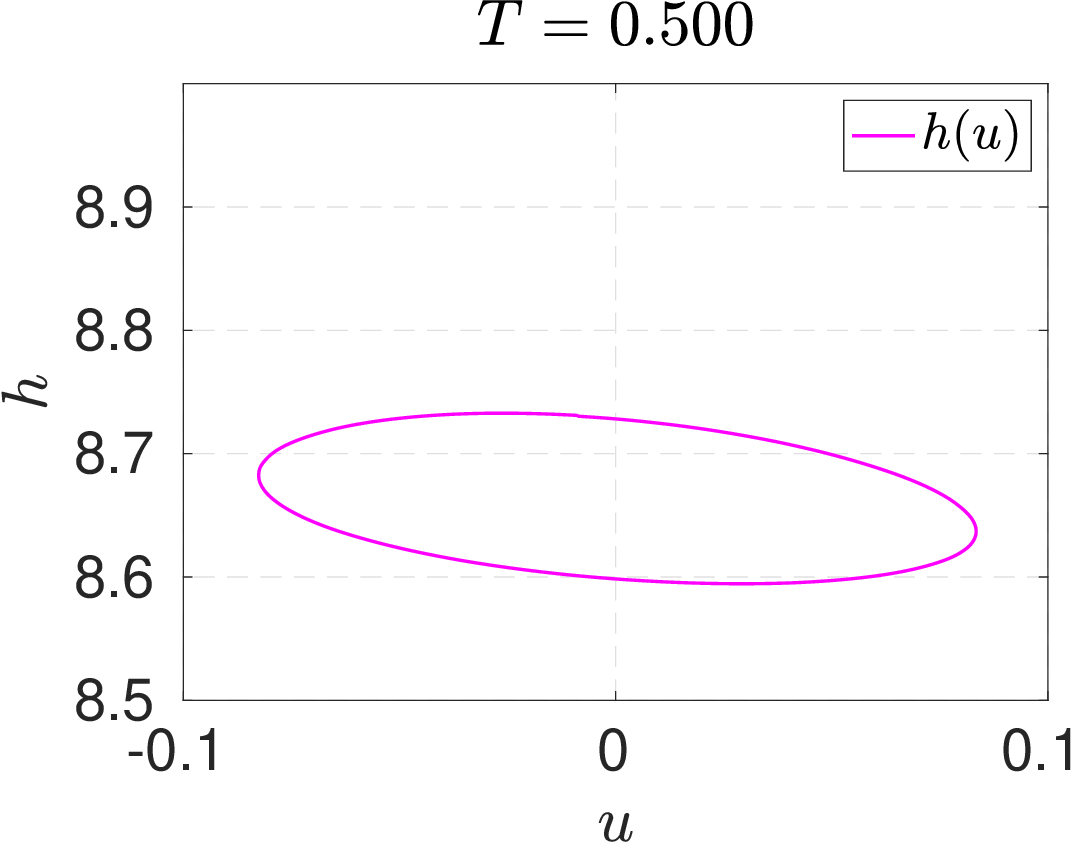}
  \end{subfigure}
  \caption{The initial \textit{invariant curve in \((u,h)-\)plane} is shown on the left and its time-evolved state on the center and on the right.}
  \label{fig:invariant-curve_case3a}
\end{figure}

% ------------------------------------------------------------------------

\begin{figure}[H]
  \centering
  \begin{subfigure}{0.49\textwidth}
    \centering
    \includegraphics[width=1.0\textwidth]{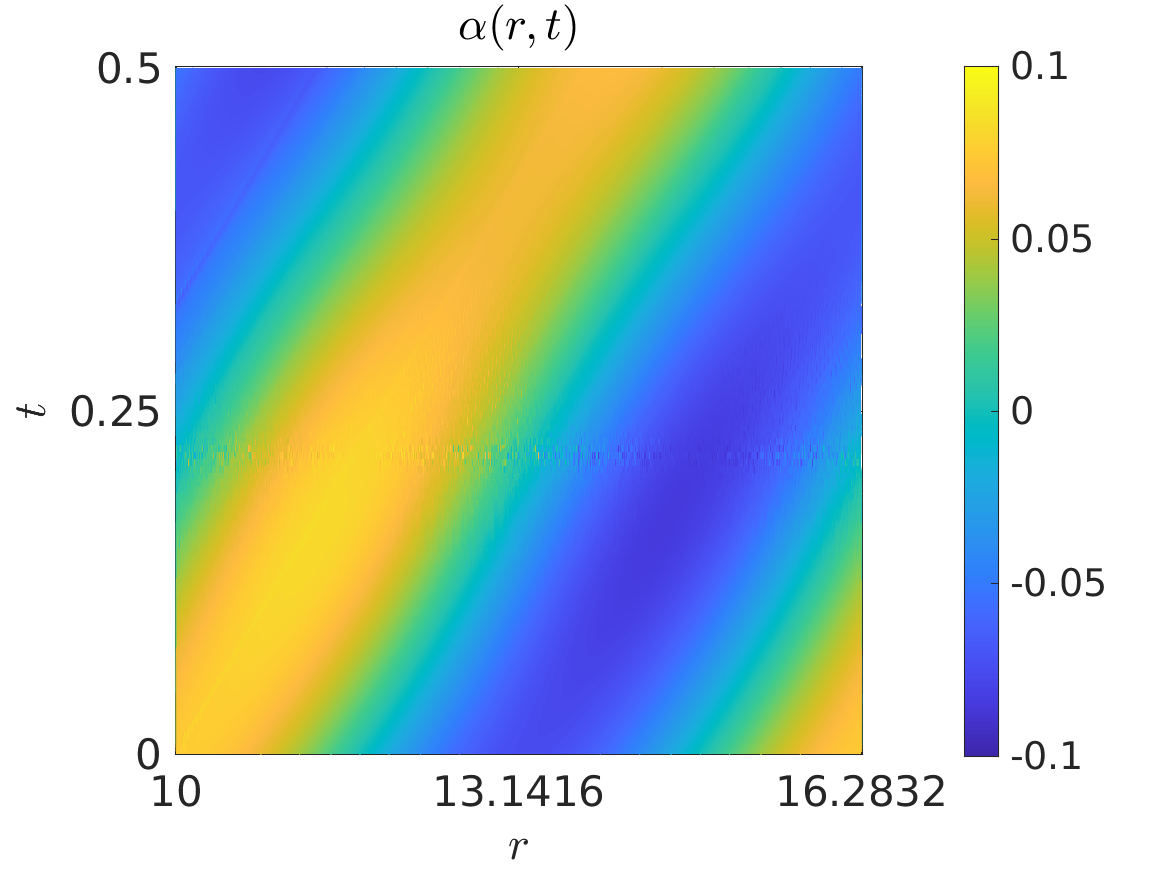}
  \end{subfigure}
  \begin{subfigure}{0.49\textwidth}
    \centering
	\includegraphics[width=1.0\textwidth]{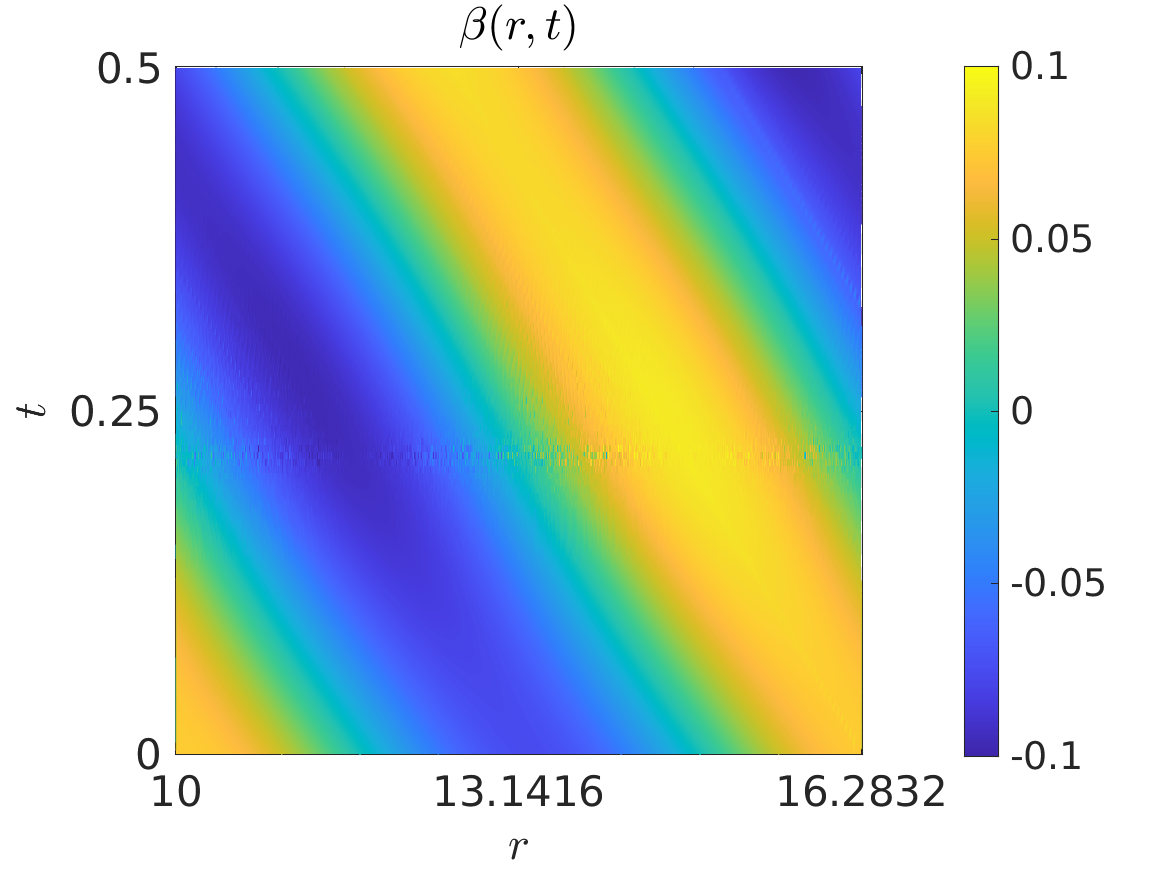}
  \end{subfigure}
  \caption{The \textit{heat map of $\alpha$ in \((r,t)-\)plane} is shown on the left and the heat map of $\beta$ on the right.}
  \label{fig:heat-map_case3a}
\end{figure}

% ------------------------------------------------------------------------

\subsubsection*{Case 3.2: $\varepsilon = 1.0$.}

% ------------------------------------------------------------------------

\begin{figure}[H]
  \centering
  \begin{subfigure}{0.32\textwidth}
    \centering
    \includegraphics[width=1.0\textwidth]{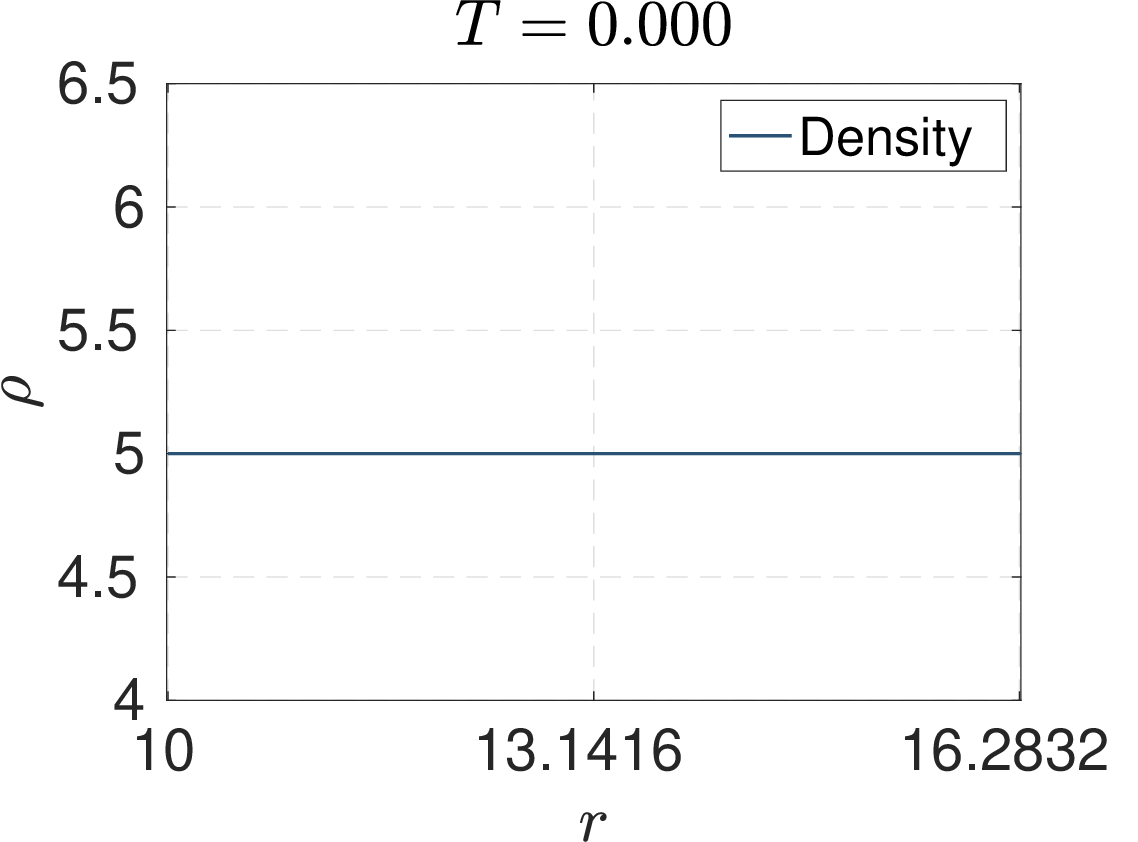}
  \end{subfigure}
    \begin{subfigure}{0.32\textwidth}
    \centering
    \includegraphics[width=1.0\textwidth]{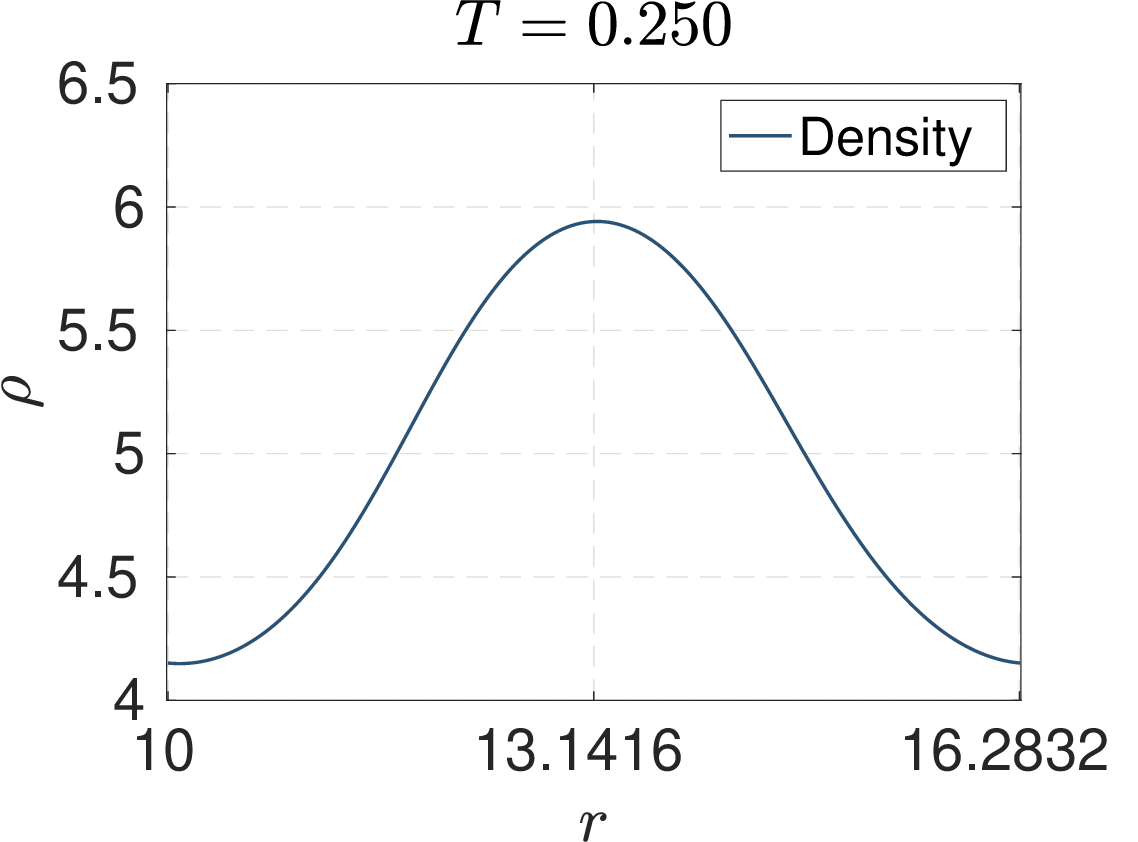}
  \end{subfigure}
  \begin{subfigure}{0.32\textwidth}
    \centering
    \includegraphics[width=1.0\textwidth]{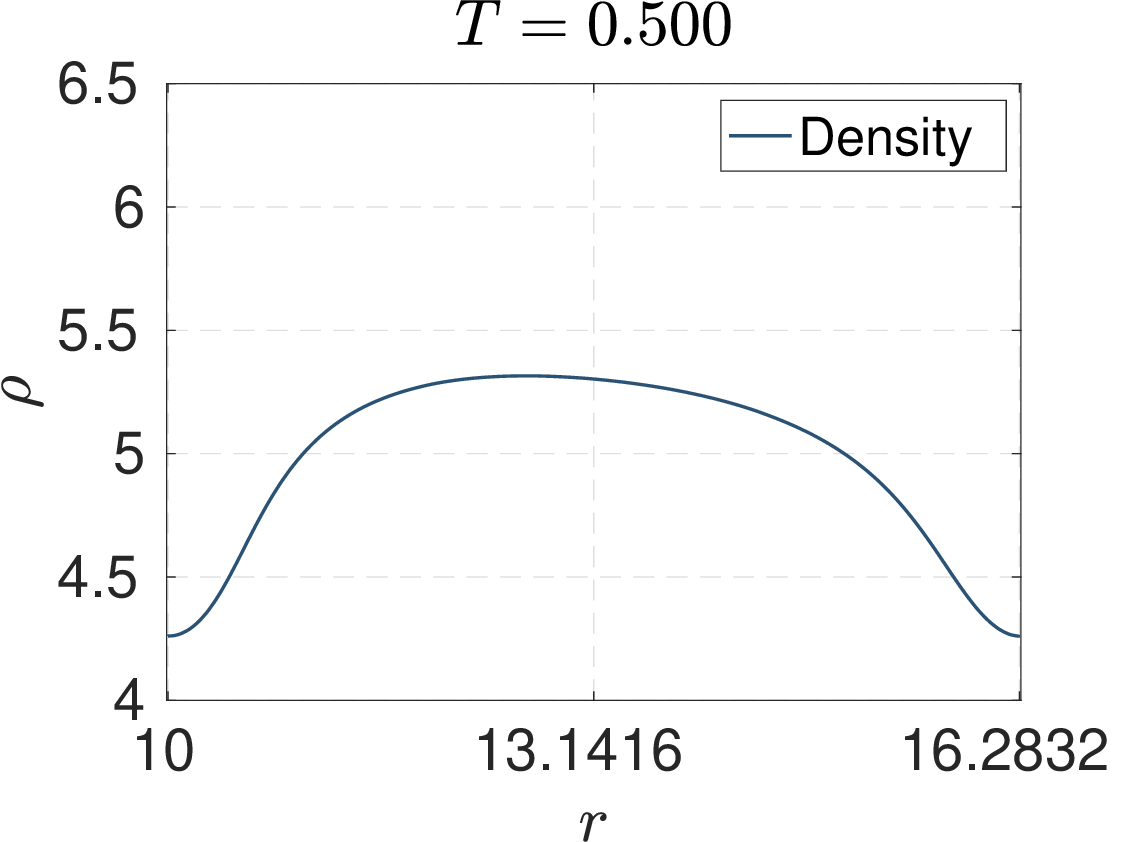}
  \end{subfigure}
  \caption{The initial \textit{density} is shown on the left and its time-evolved state on the center and on the right.}
  \label{fig:density_case3b}
\end{figure}

% ------------------------------------------------------------------------

\begin{figure}[H]
  \centering
  \begin{subfigure}{0.32\textwidth}
    \centering
    \includegraphics[width=1.0\textwidth]{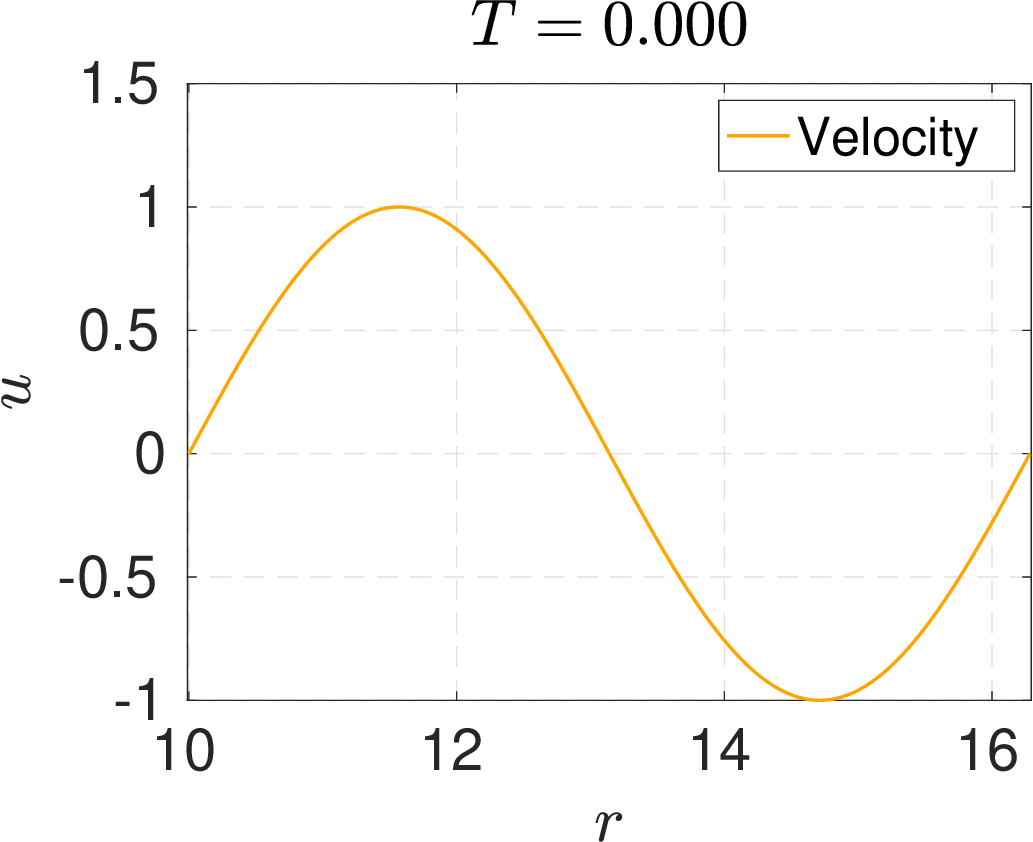}
  \end{subfigure}
    \begin{subfigure}{0.32\textwidth}
    \centering
    \includegraphics[width=1.0\textwidth]{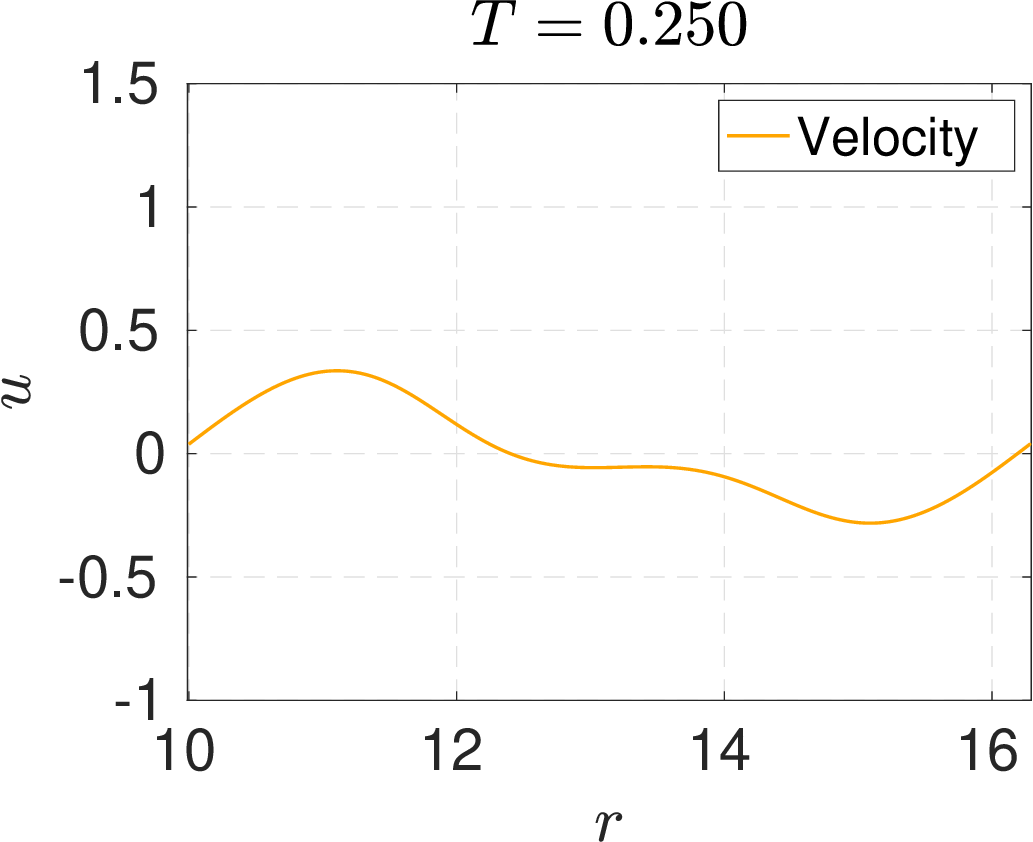}
  \end{subfigure}
  \begin{subfigure}{0.32\textwidth}
    \centering
    \includegraphics[width=1.0\textwidth]{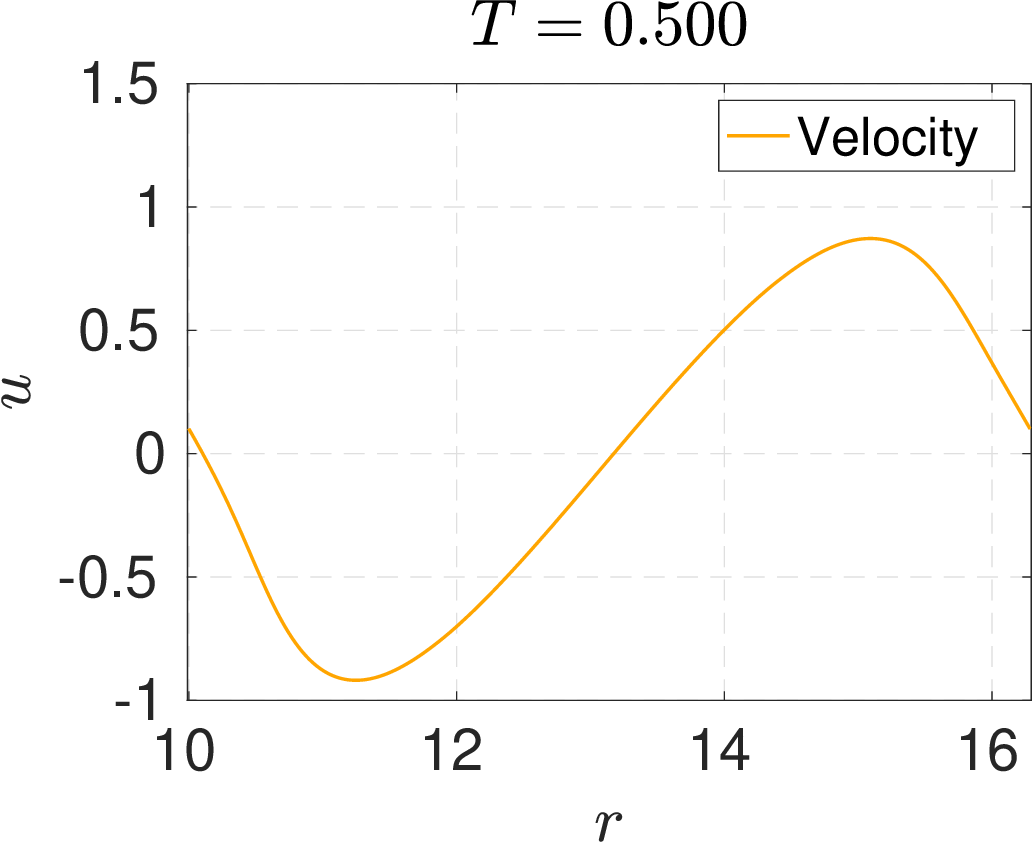}
  \end{subfigure}
  \caption{The initial \textit{velocity} is shown on the left and its time-evolved state on the center and on the right.}
  \label{fig:velocity_case3b}
\end{figure}

% ------------------------------------------------------------------------

\begin{figure}[H]
  \centering
  \begin{subfigure}{0.32\textwidth}
    \centering
    \includegraphics[width=1.0\textwidth]{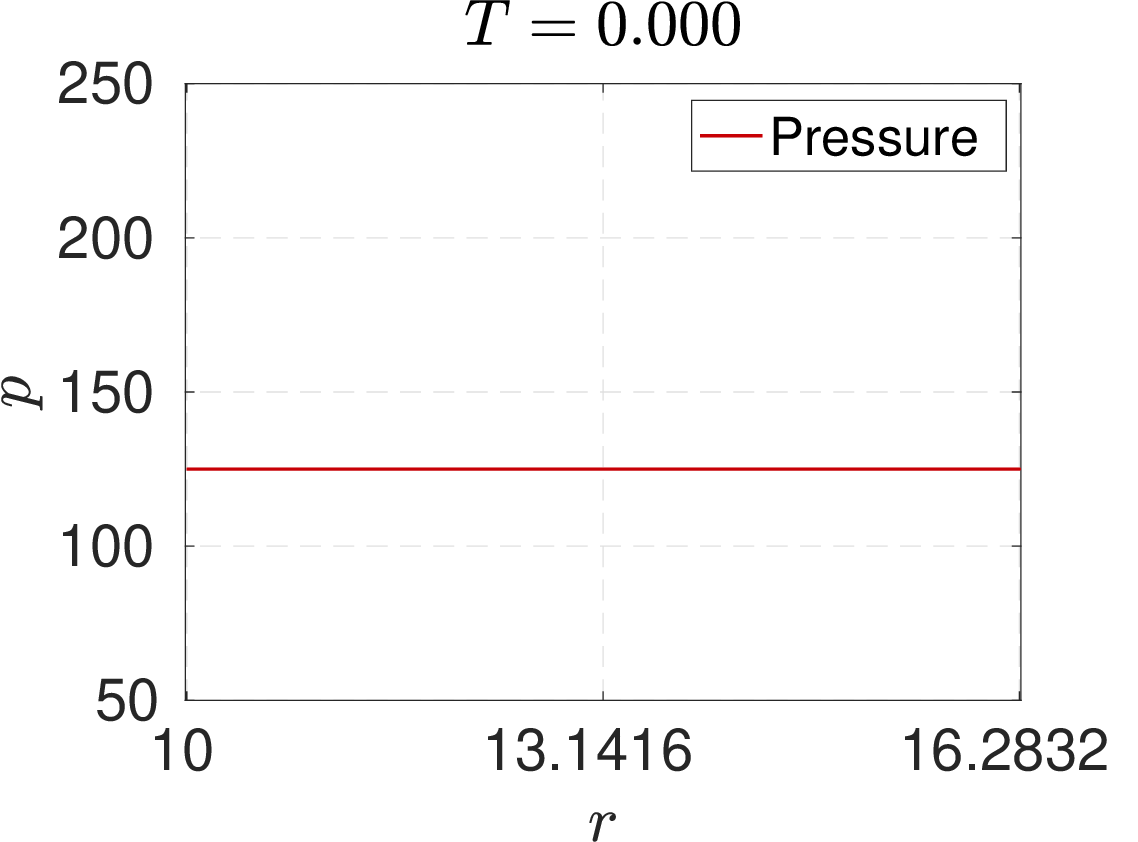}
  \end{subfigure}
    \begin{subfigure}{0.32\textwidth}
    \centering
    \includegraphics[width=1.0\textwidth]{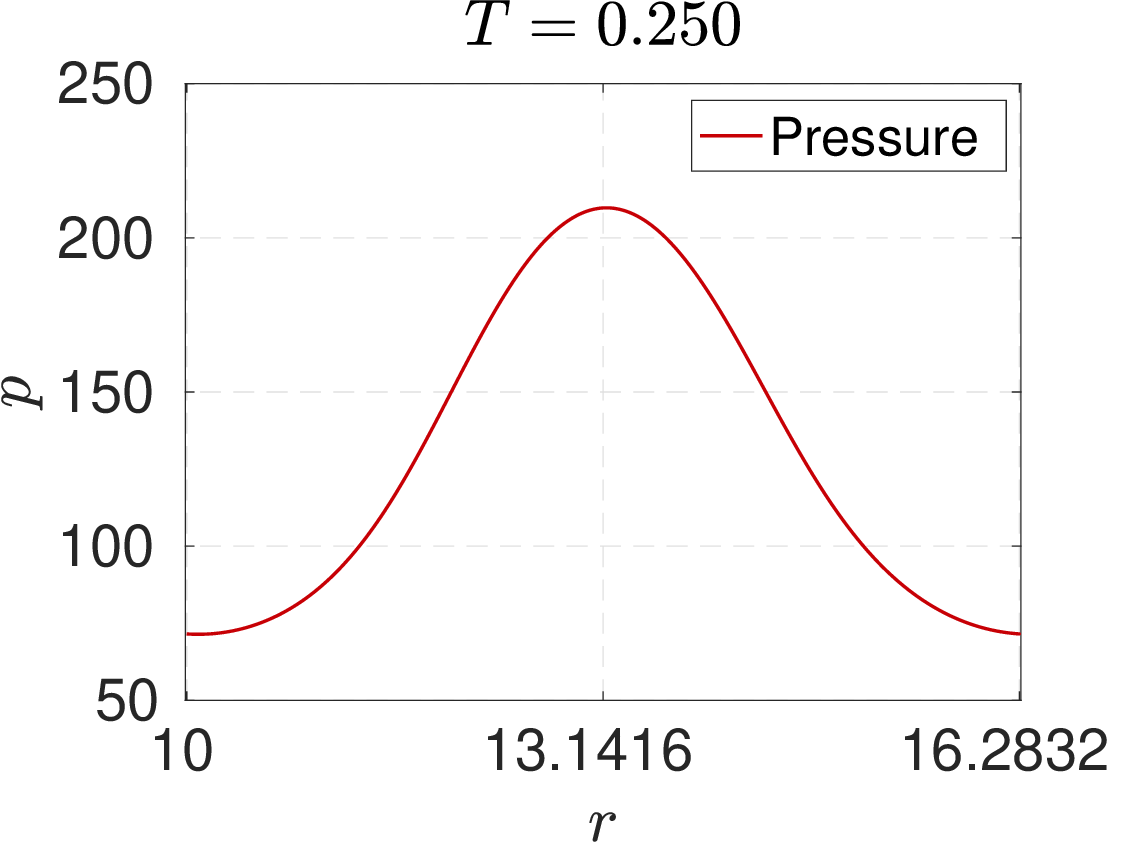}
  \end{subfigure}
  \begin{subfigure}{0.32\textwidth}
    \centering
    \includegraphics[width=1.0\textwidth]{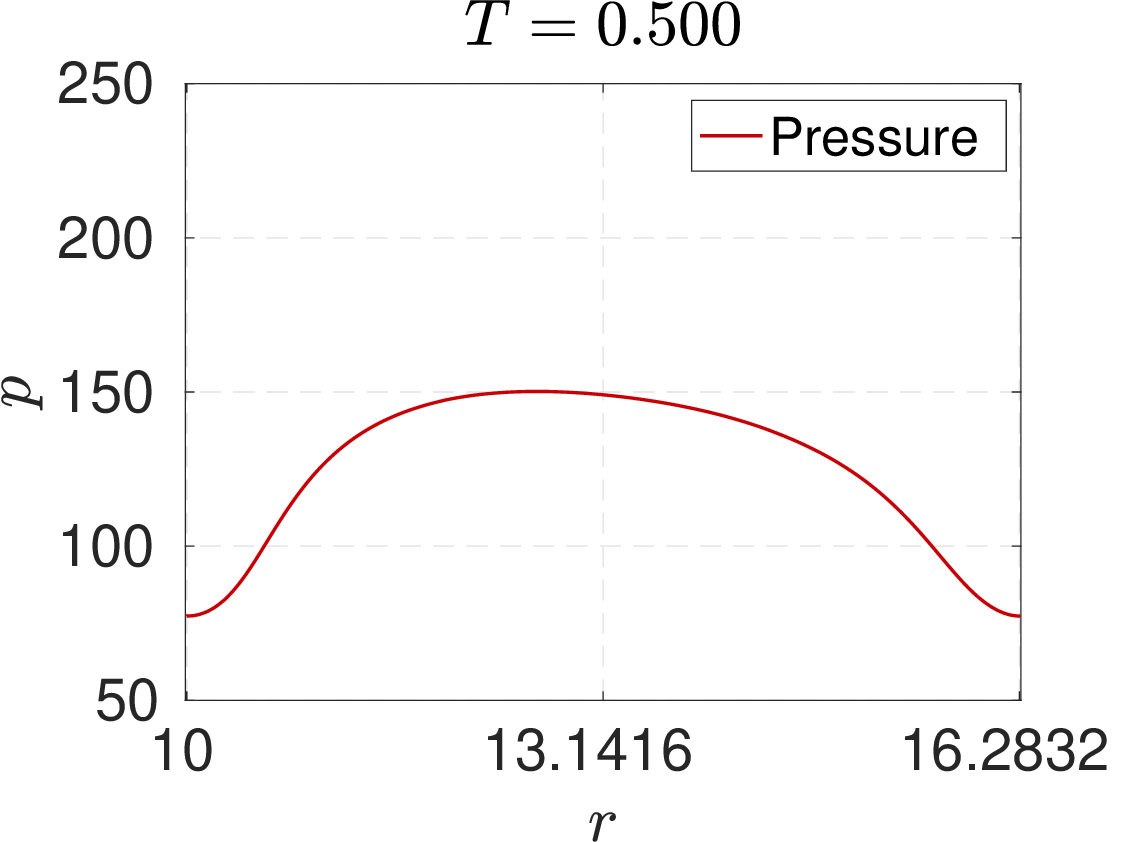}
  \end{subfigure}
  \caption{The initial \textit{pressure} is shown on the left and its time-evolved state on the center and on the right.}
  \label{fig:pressure_case3b}
\end{figure}

% ------------------------------------------------------------------------

\begin{figure}[H]
  \centering
  \begin{subfigure}{0.32\textwidth}
    \centering
    \includegraphics[width=1.0\textwidth]{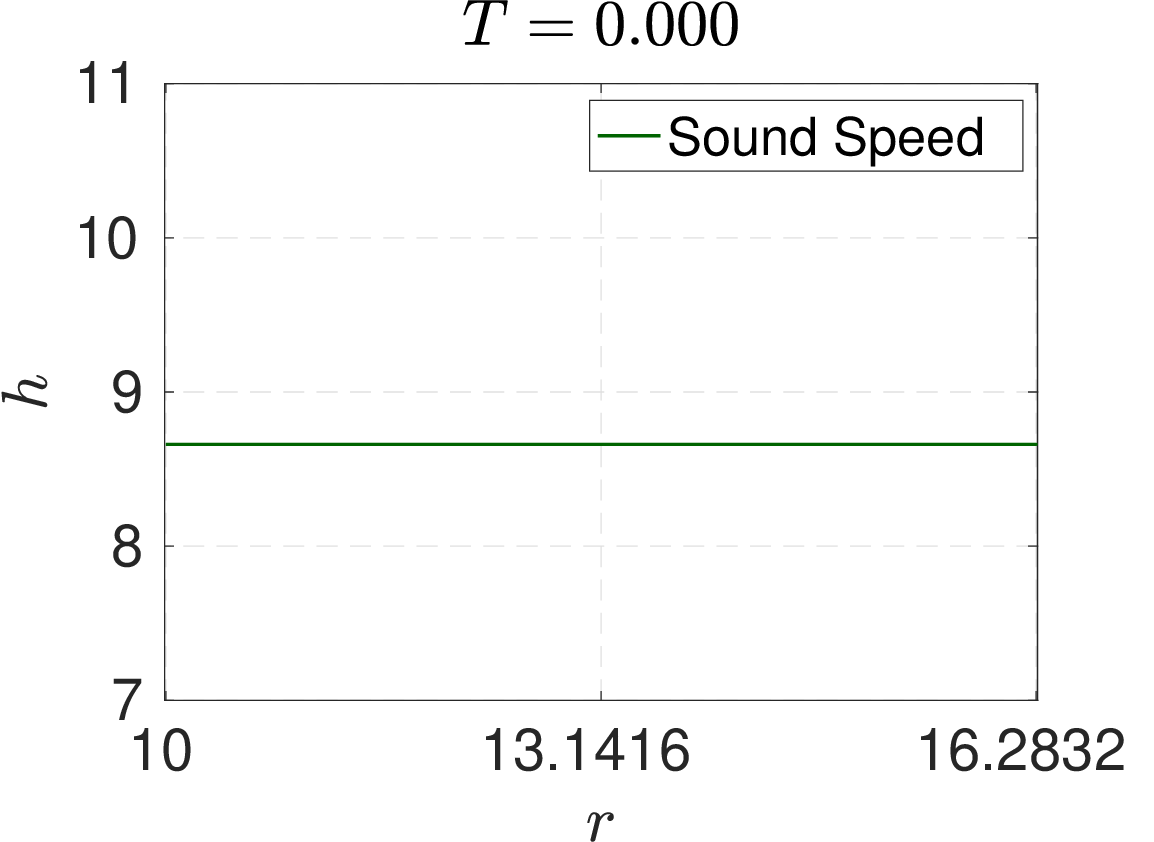}
  \end{subfigure}
    \begin{subfigure}{0.32\textwidth}
    \centering
    \includegraphics[width=1.0\textwidth]{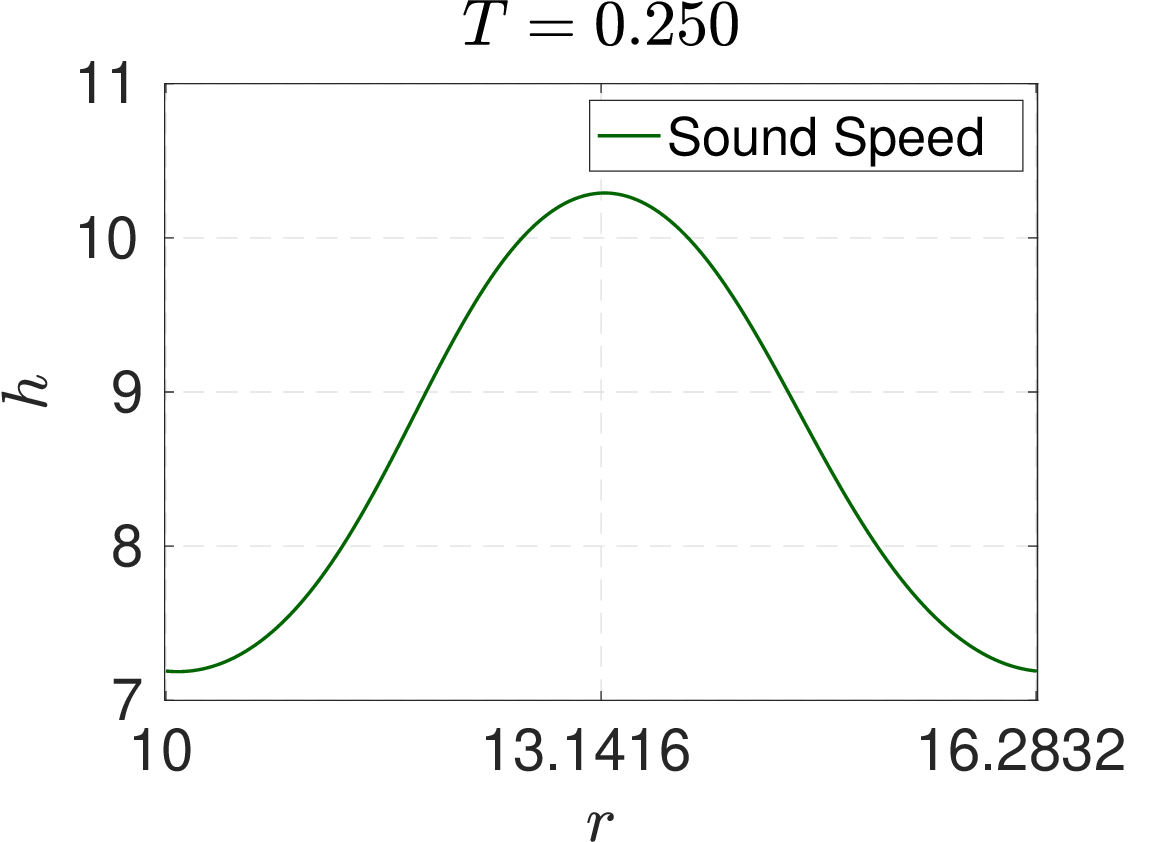}
  \end{subfigure}
  \begin{subfigure}{0.32\textwidth}
    \centering
    \includegraphics[width=1.0\textwidth]{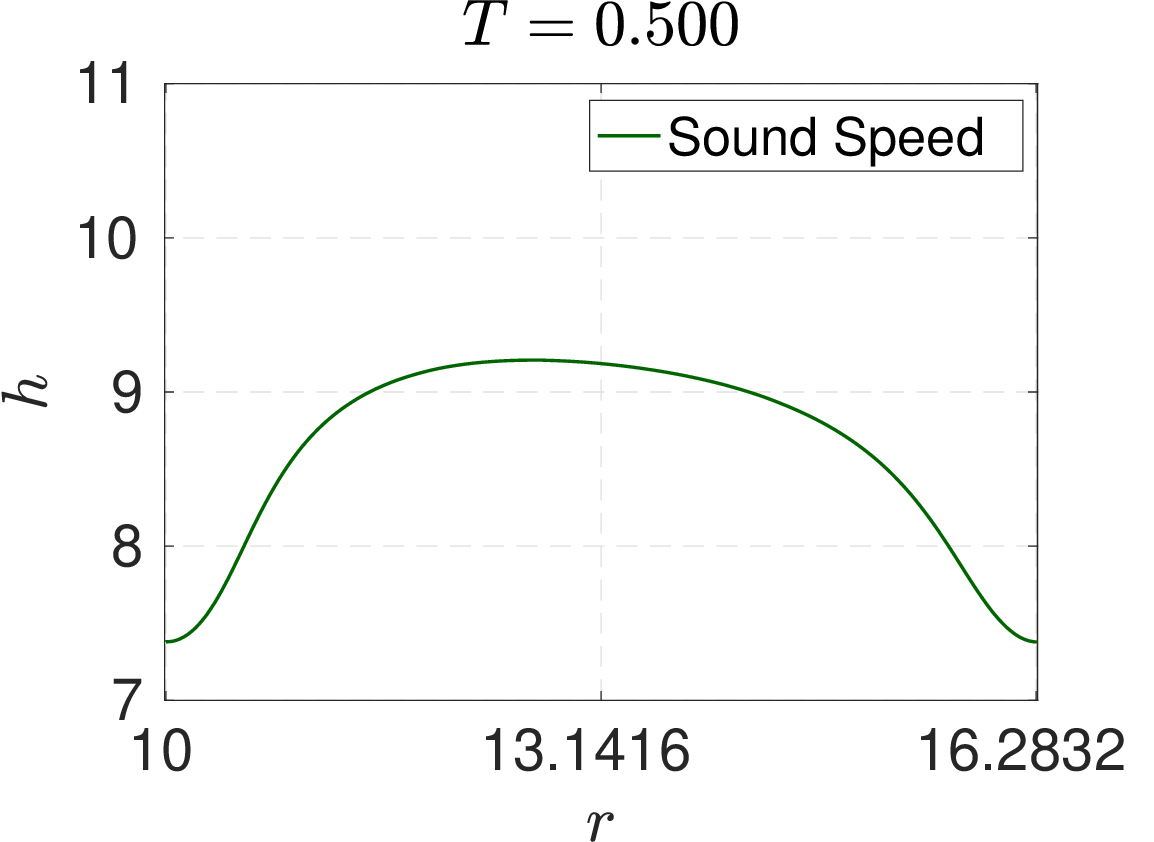}
  \end{subfigure}
  \caption{The initial \textit{sound speed} is shown on the left and its time-evolved state on the center and on the right.}
  \label{fig:sound-speed_case3b}
\end{figure}

% ------------------------------------------------------------------------

\begin{figure}[H]
  \centering
  \begin{subfigure}{0.32\textwidth}
    \centering
    \includegraphics[width=1.0\textwidth]{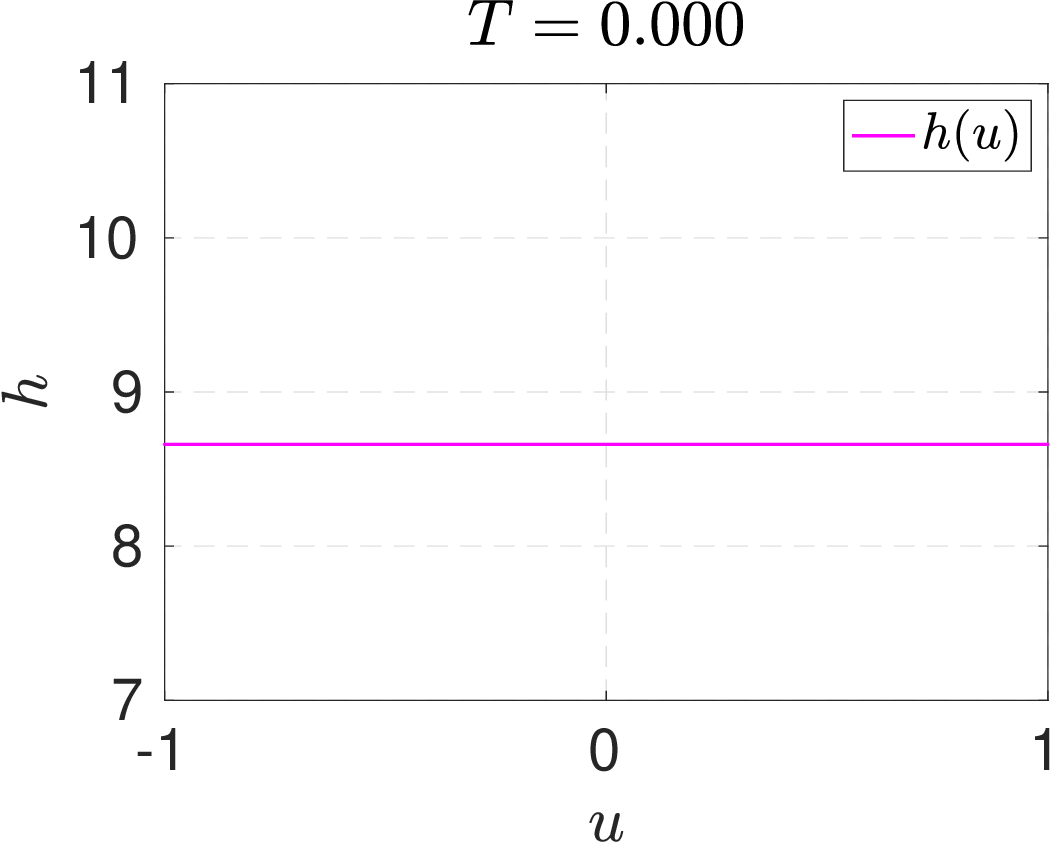}
  \end{subfigure}
    \begin{subfigure}{0.32\textwidth}
    \centering
    \includegraphics[width=1.0\textwidth]{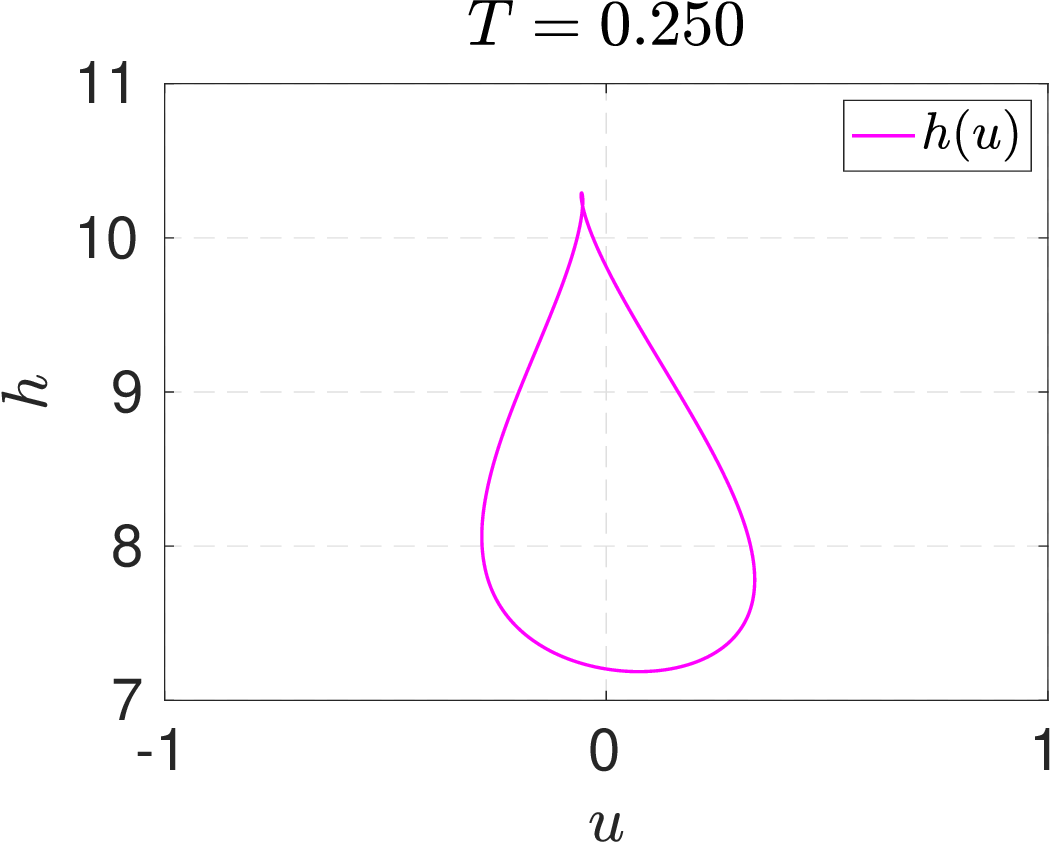}
  \end{subfigure}
  \begin{subfigure}{0.32\textwidth}
    \centering
    \includegraphics[width=1.0\textwidth]{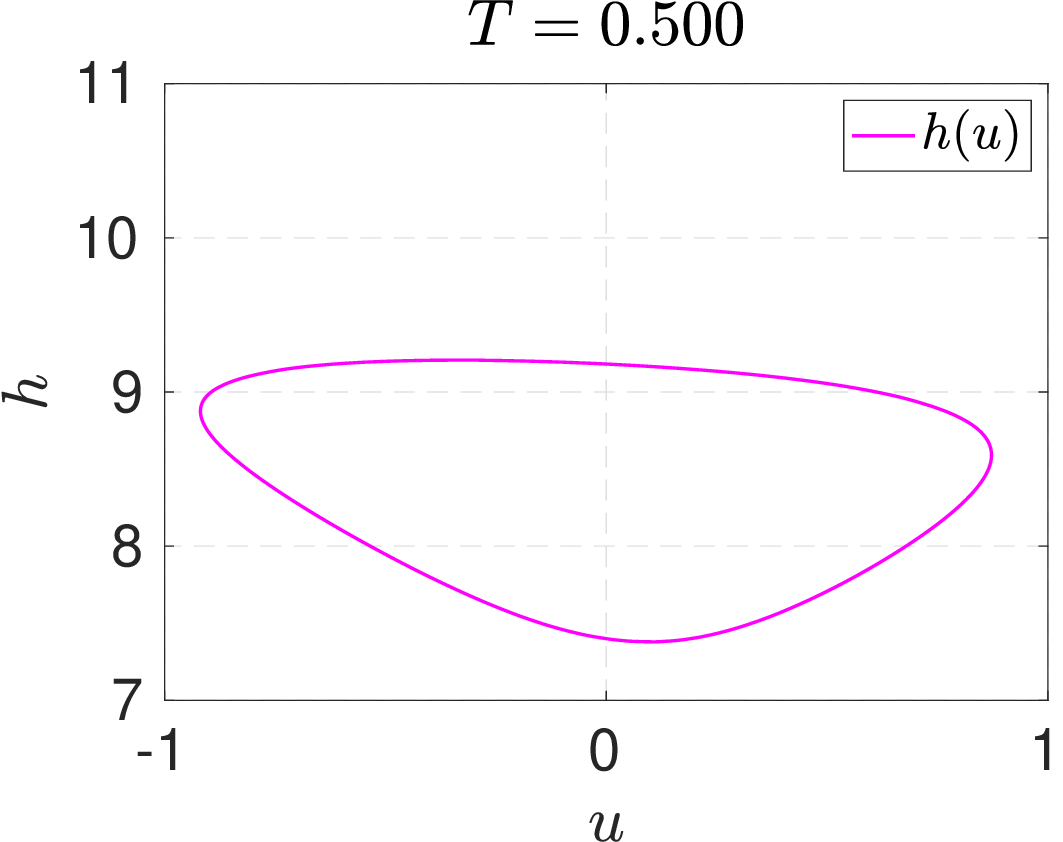}
  \end{subfigure}
  \caption{The initial \textit{invariant curve in \((u,h)-\)plane} is shown on the left and its time-evolved state on the center and on the right.}
  \label{fig:invariant-curve_case3b}
\end{figure}

% ------------------------------------------------------------------------

\begin{figure}[H]
  \centering
  \begin{subfigure}{0.49\textwidth}
    \centering
    \includegraphics[width=1.0\textwidth]{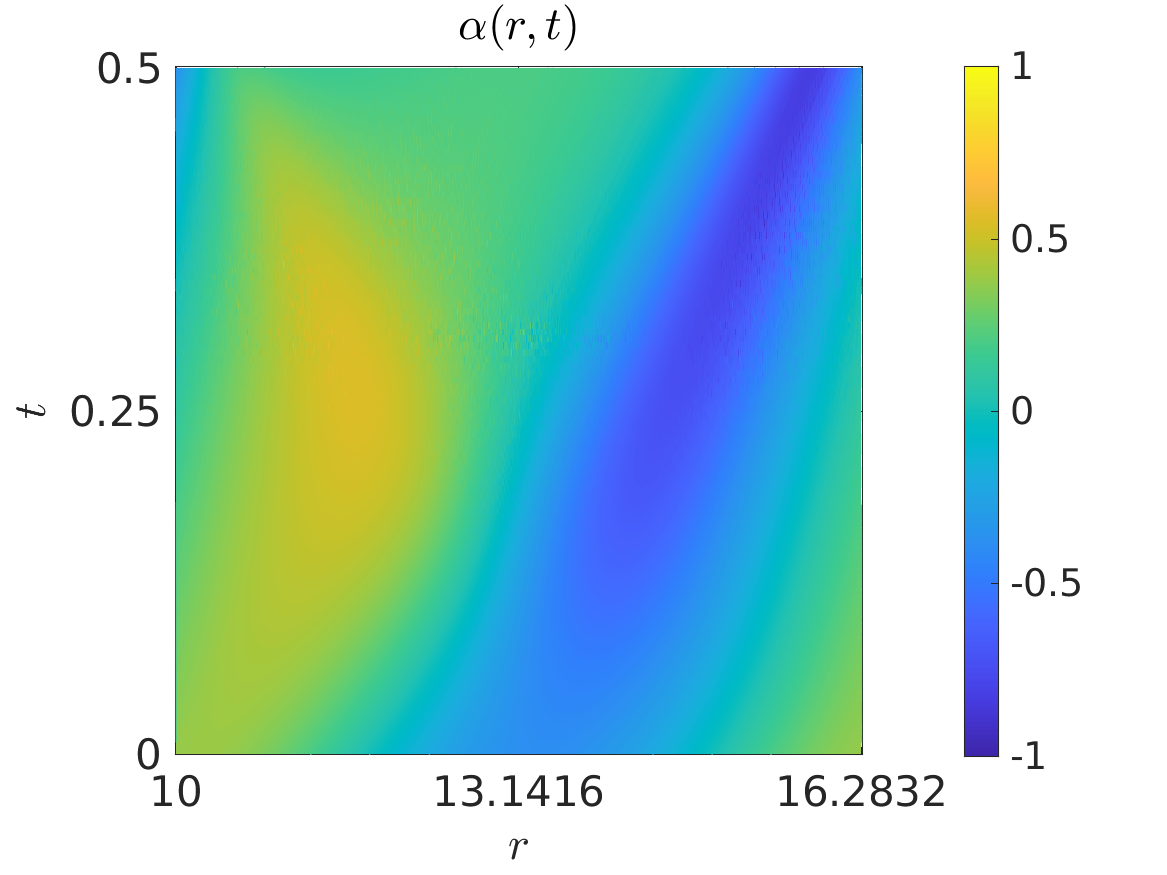}
  \end{subfigure}
  \begin{subfigure}{0.49\textwidth}
    \centering
	\includegraphics[width=1.0\textwidth]{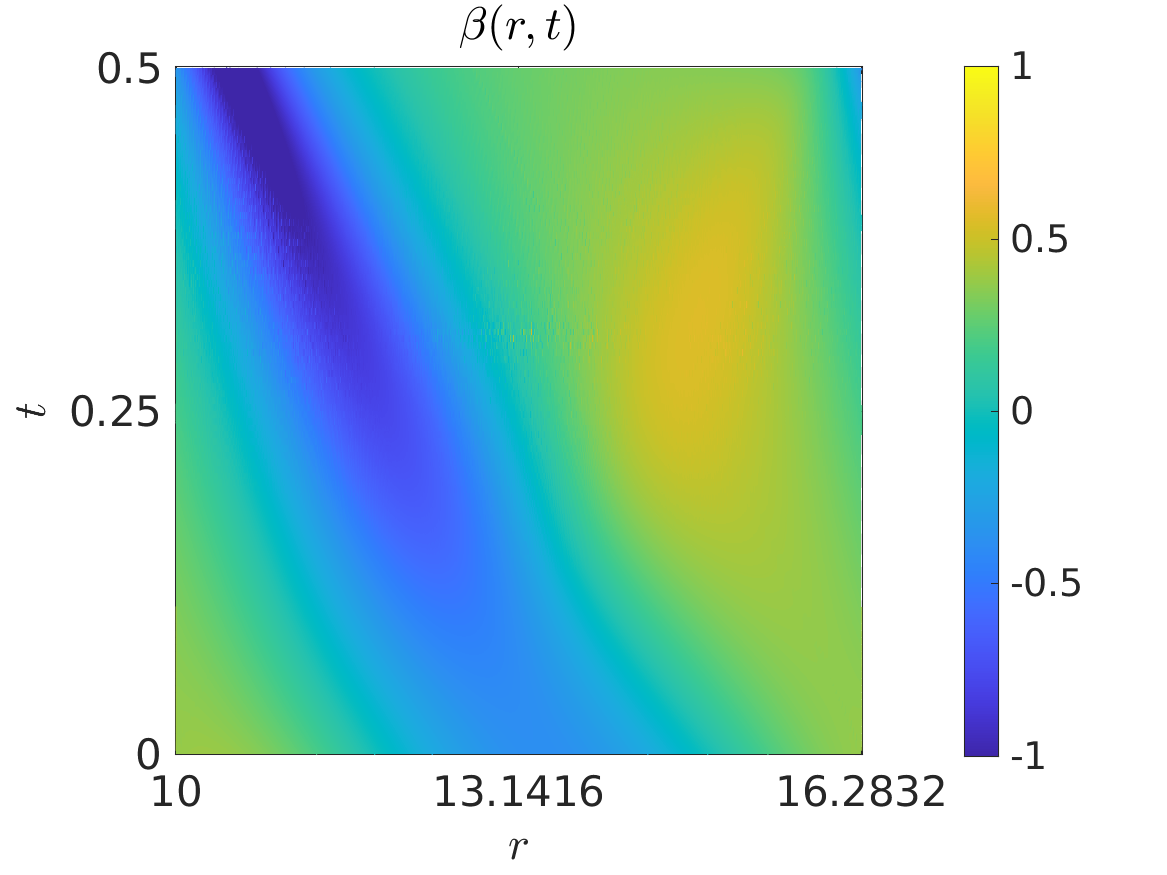}
  \end{subfigure}
  \caption{The \textit{heat map of $\alpha$ in \((r,t)-\)plane} is shown on the left and the heat map of $\beta$ on the right.}
  \label{fig:heat-map_case3b}
\end{figure}

% ------------------------------------------------------------------------

\subsubsection*{Case 3.3: $\varepsilon = 0.1$.}

% ------------------------------------------------------------------------

\begin{figure}[H]
  \centering
  \begin{subfigure}{0.32\textwidth}
    \centering
    \includegraphics[width=1.0\textwidth]{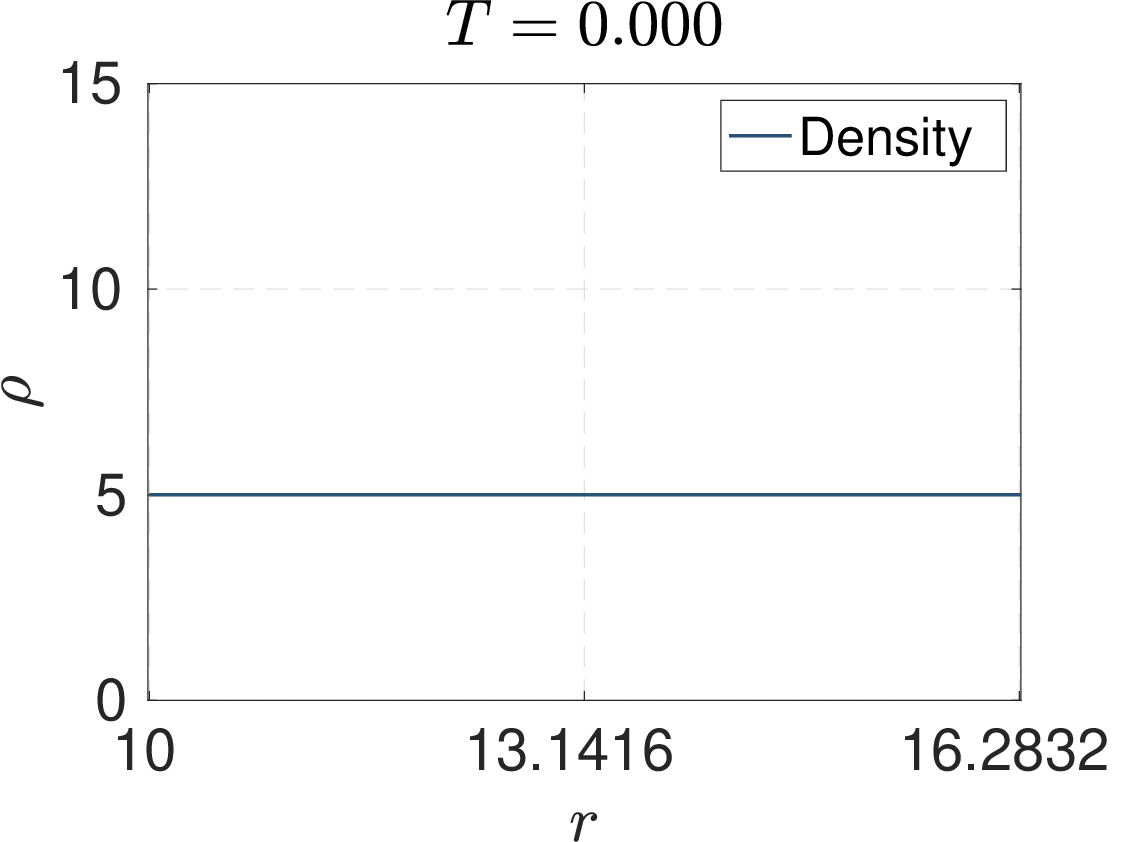}
  \end{subfigure}
    \begin{subfigure}{0.32\textwidth}
    \centering
    \includegraphics[width=1.0\textwidth]{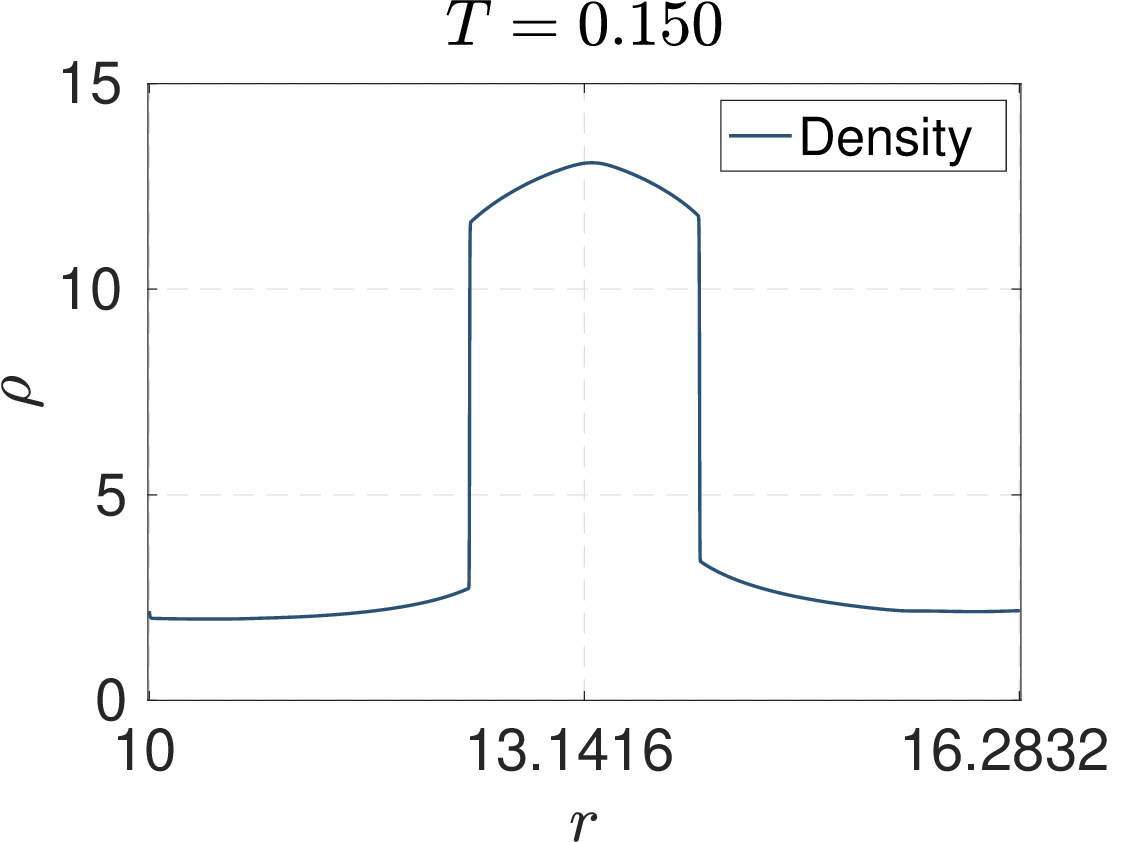}
  \end{subfigure}
  \begin{subfigure}{0.32\textwidth}
    \centering
    \includegraphics[width=1.0\textwidth]{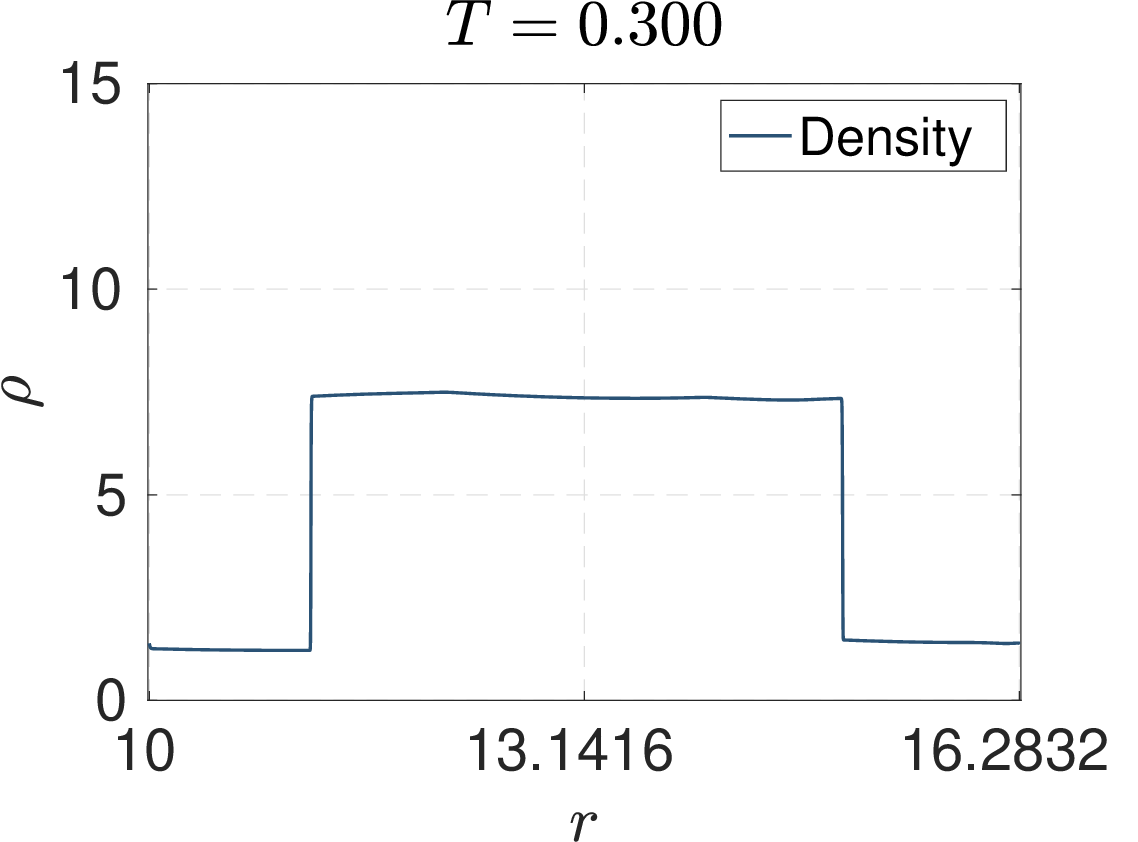}
  \end{subfigure}
  \caption{The initial \textit{density} is shown on the left and its time-evolved state on the center and on the right.}
  \label{fig:density_case3c}
\end{figure}

% ------------------------------------------------------------------------

\begin{figure}[H]
  \centering
  \begin{subfigure}{0.32\textwidth}
    \centering
    \includegraphics[width=1.0\textwidth]{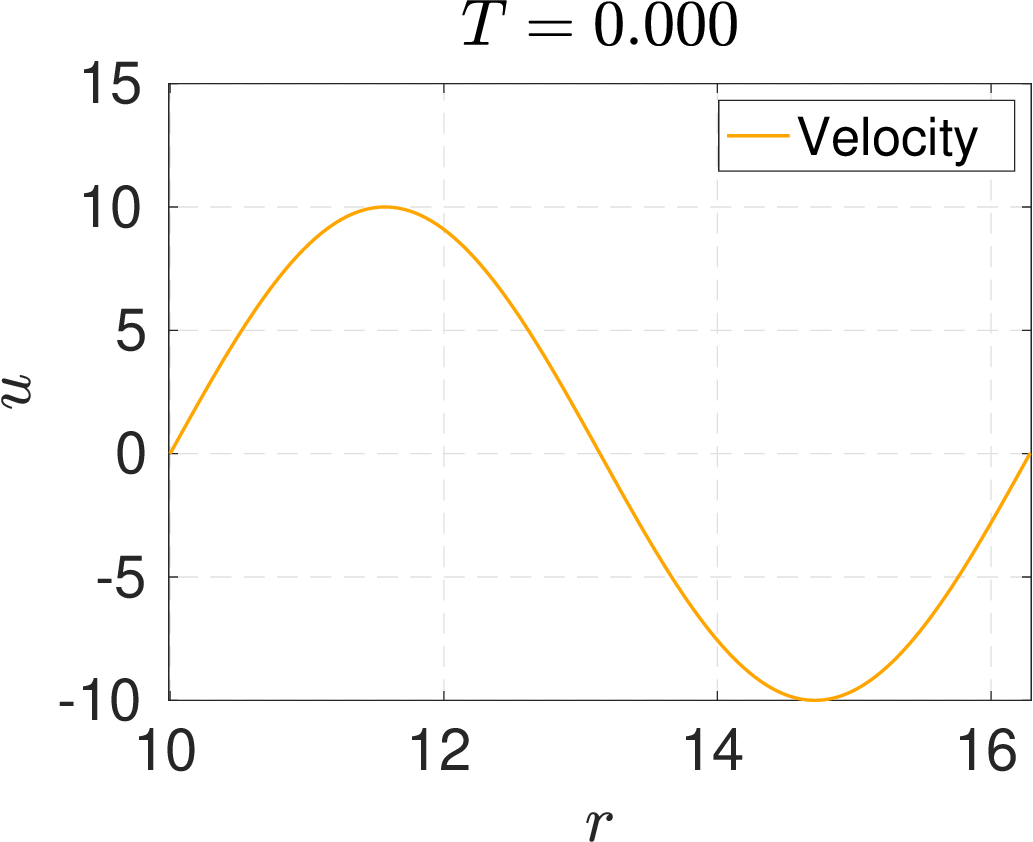}
  \end{subfigure}
    \begin{subfigure}{0.32\textwidth}
    \centering
    \includegraphics[width=1.0\textwidth]{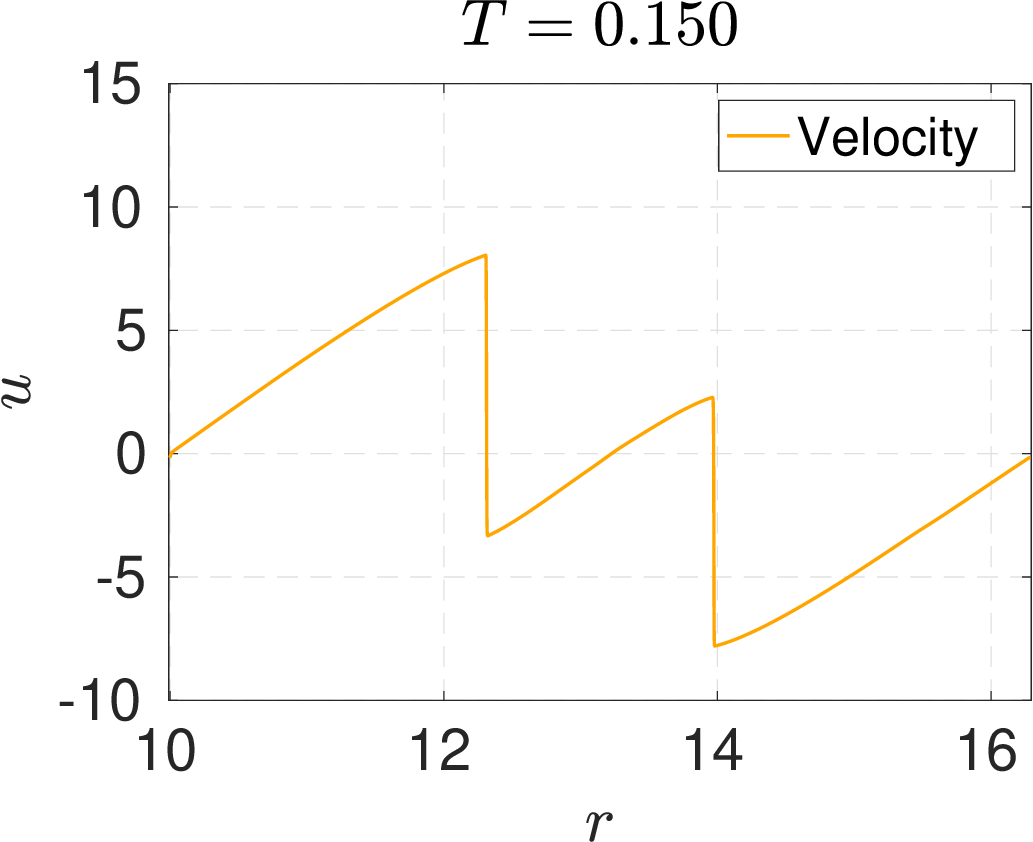}
  \end{subfigure}
  \begin{subfigure}{0.32\textwidth}
    \centering
    \includegraphics[width=1.0\textwidth]{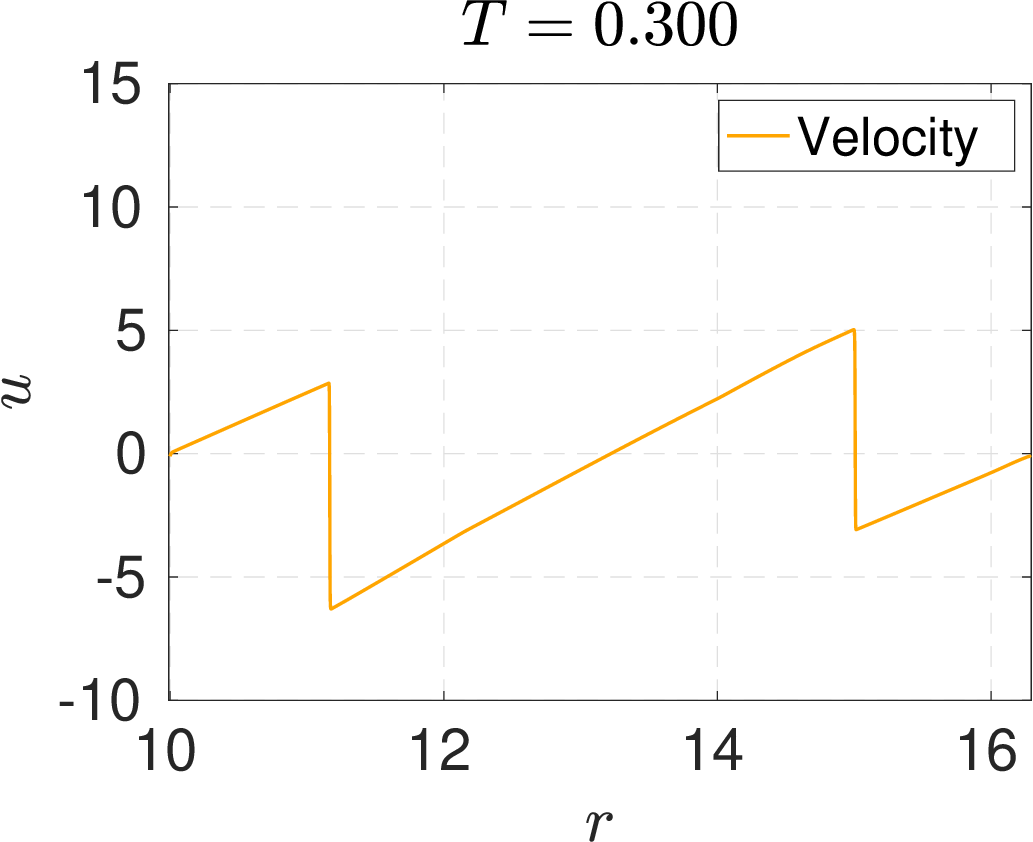}
  \end{subfigure}
  \caption{The initial \textit{velocity} is shown on the left and its time-evolved state on the center and on the right.}
  \label{fig:velocity_case3c}
\end{figure}

% ------------------------------------------------------------------------

\begin{figure}[H]
  \centering
  \begin{subfigure}{0.32\textwidth}
    \centering
    \includegraphics[width=1.0\textwidth]{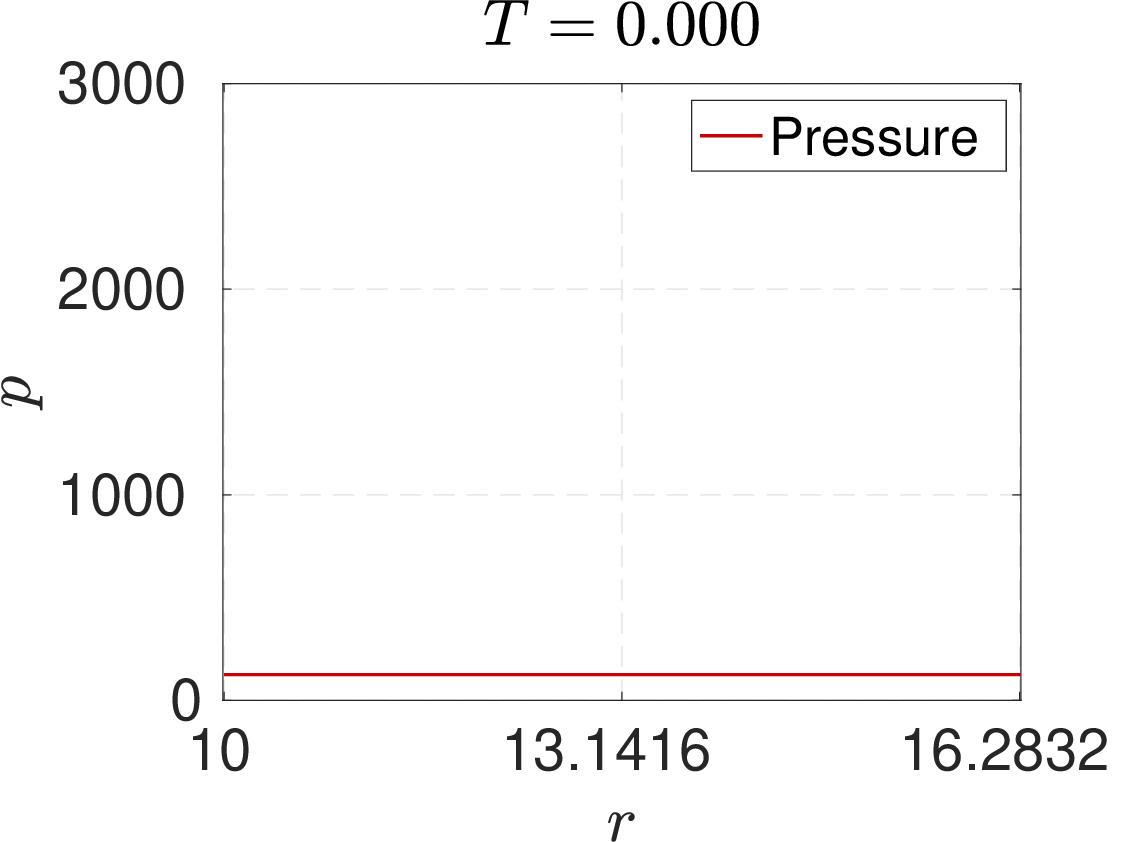}
  \end{subfigure}
    \begin{subfigure}{0.32\textwidth}
    \centering
    \includegraphics[width=1.0\textwidth]{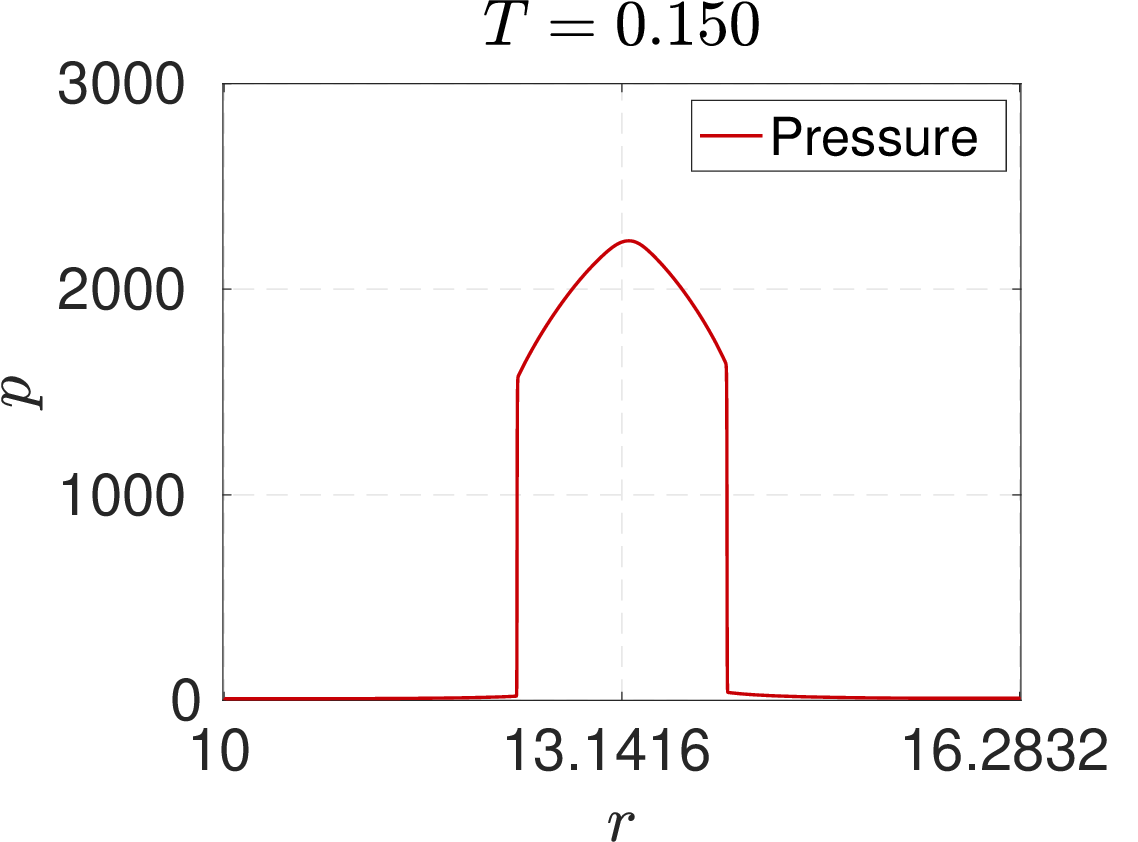}
  \end{subfigure}
  \begin{subfigure}{0.32\textwidth}
    \centering
    \includegraphics[width=1.0\textwidth]{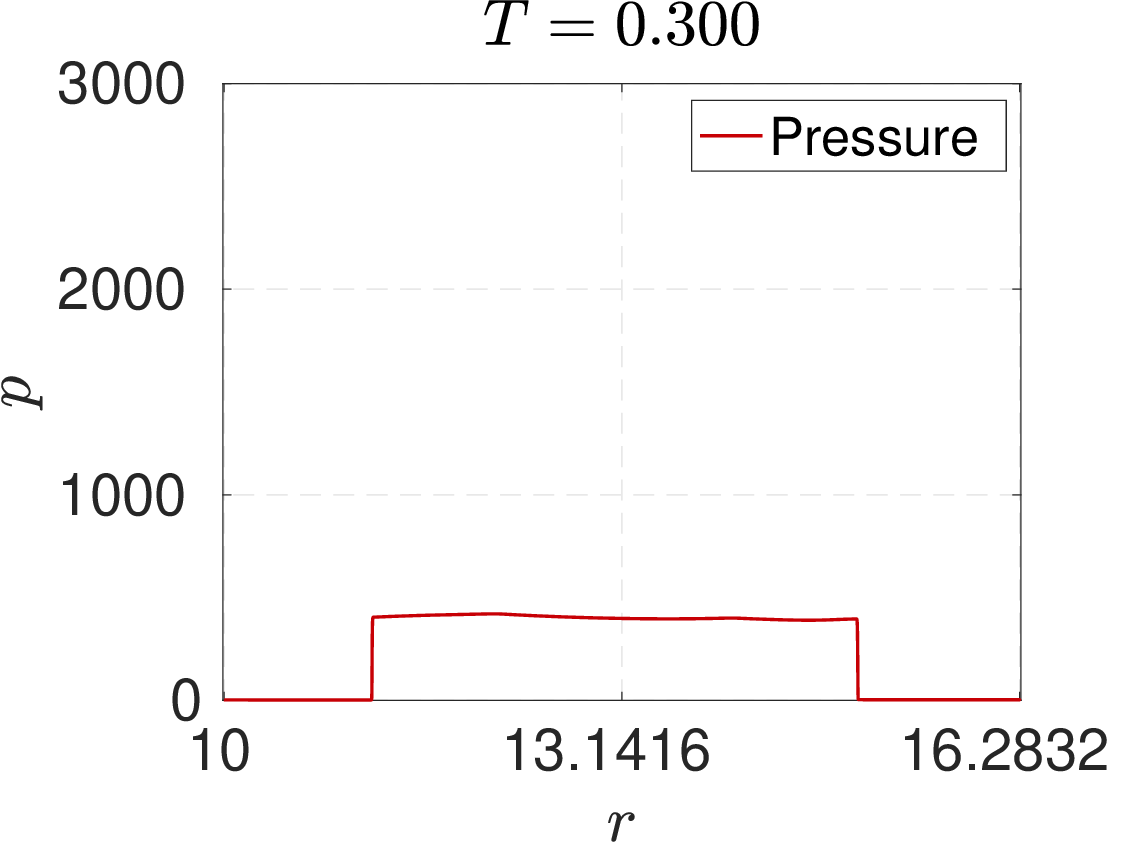}
  \end{subfigure}
  \caption{The initial \textit{pressure} is shown on the left and its time-evolved state on the center and on the right.}
  \label{fig:pressure_case3c}
\end{figure}

% ------------------------------------------------------------------------

\begin{figure}[H]
  \centering
  \begin{subfigure}{0.32\textwidth}
    \centering
    \includegraphics[width=1.0\textwidth]{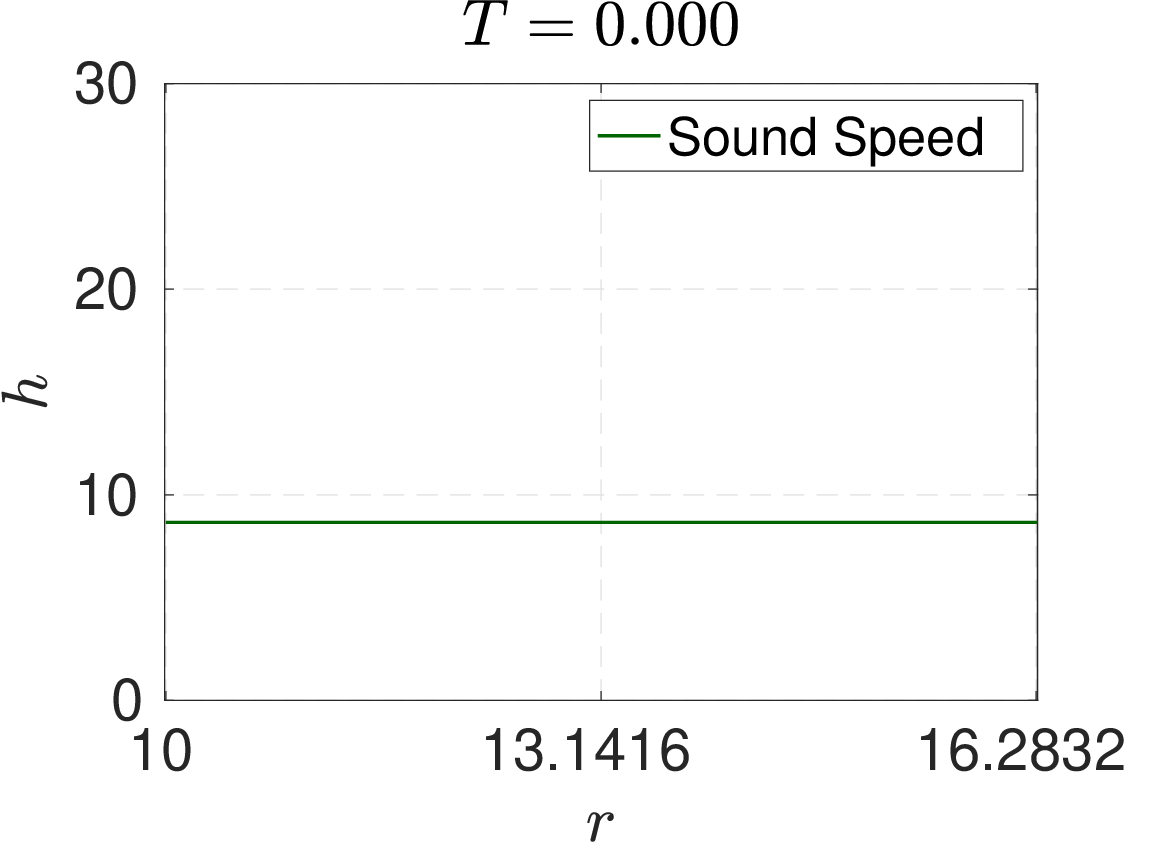}
  \end{subfigure}
    \begin{subfigure}{0.32\textwidth}
    \centering
    \includegraphics[width=1.0\textwidth]{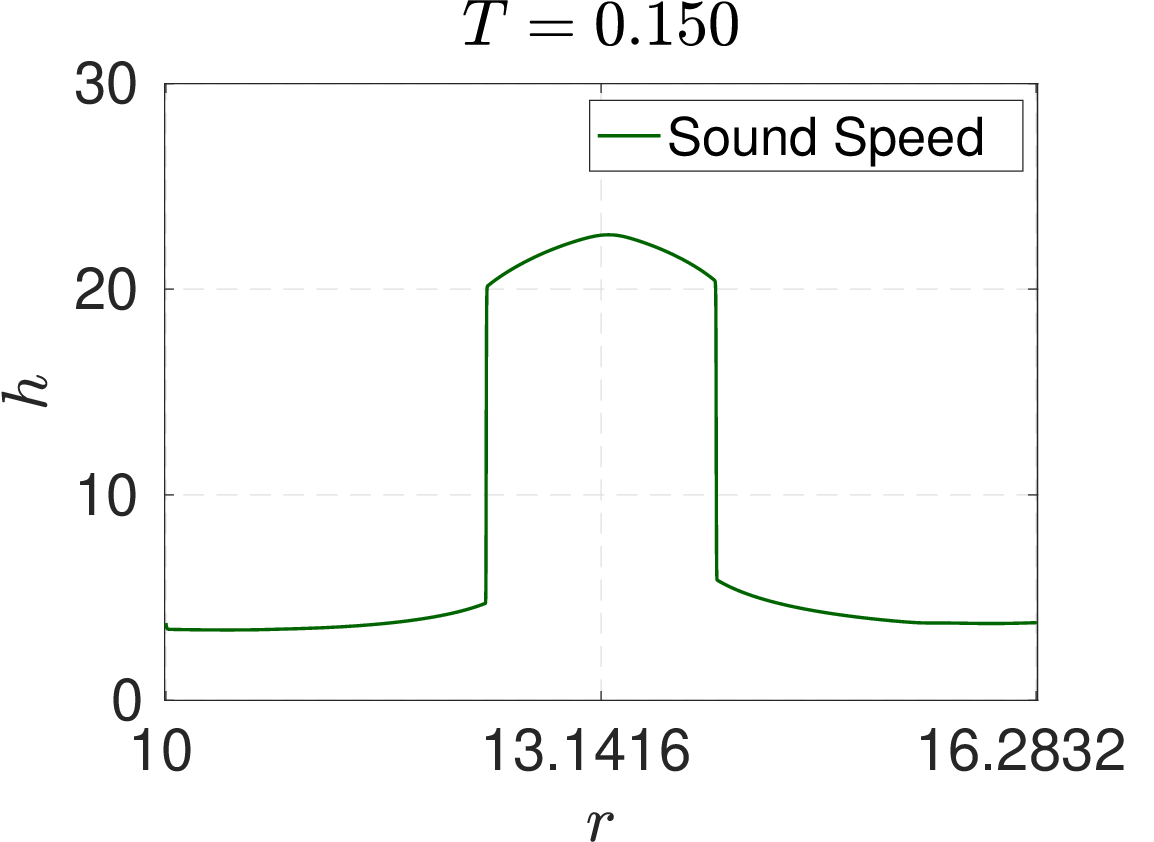}
  \end{subfigure}
  \begin{subfigure}{0.32\textwidth}
    \centering
    \includegraphics[width=1.0\textwidth]{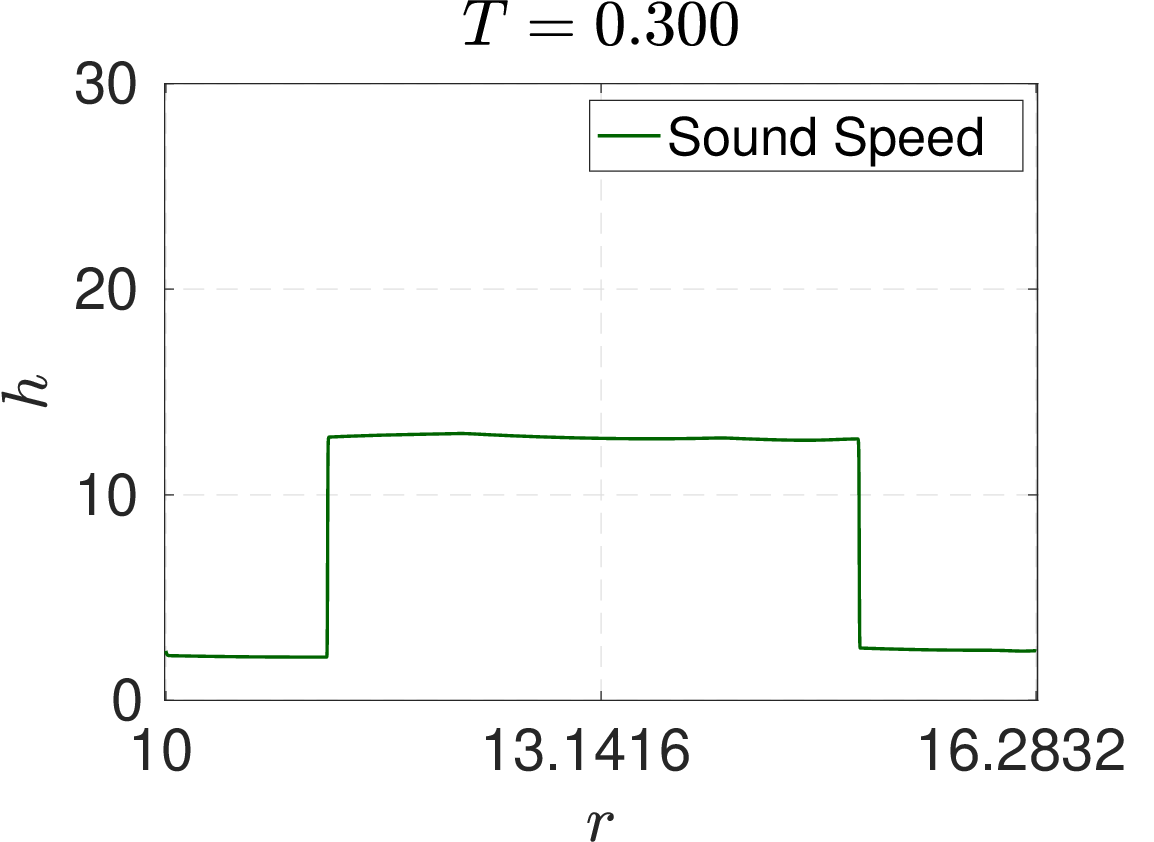}
  \end{subfigure}
  \caption{The initial \textit{sound speed} is shown on the left and its time-evolved state on the center and on the right.}
  \label{fig:sound-speed_case3c}
\end{figure}

% ------------------------------------------------------------------------

\begin{figure}[H]
  \centering
  \begin{subfigure}{0.32\textwidth}
    \centering
    \includegraphics[width=1.0\textwidth]{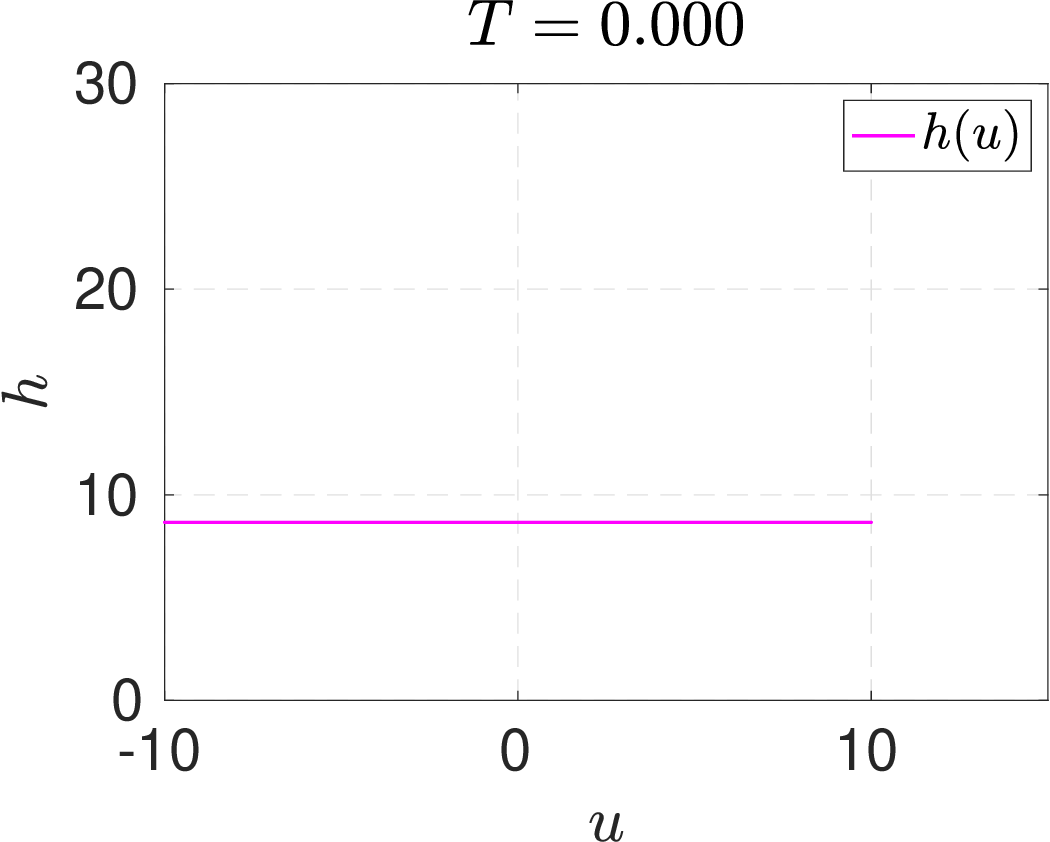}
  \end{subfigure}
    \begin{subfigure}{0.32\textwidth}
    \centering
    \includegraphics[width=1.0\textwidth]{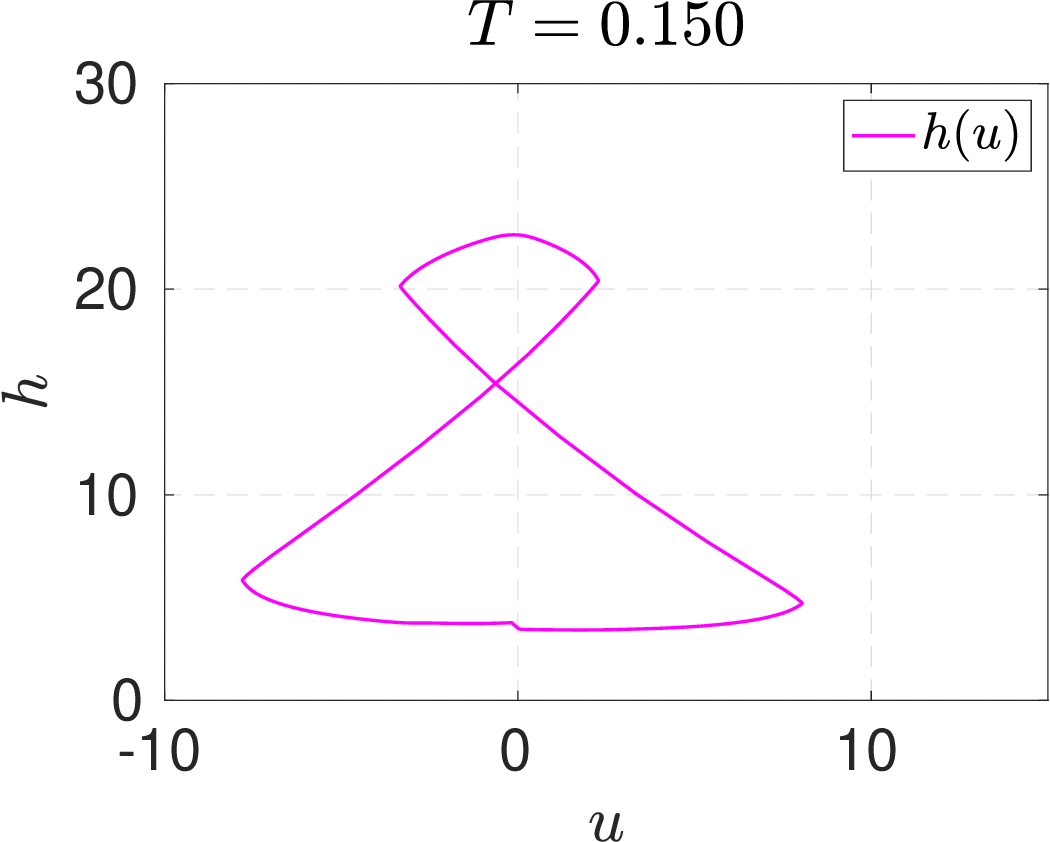}
  \end{subfigure}
  \begin{subfigure}{0.32\textwidth}
    \centering
    \includegraphics[width=1.0\textwidth]{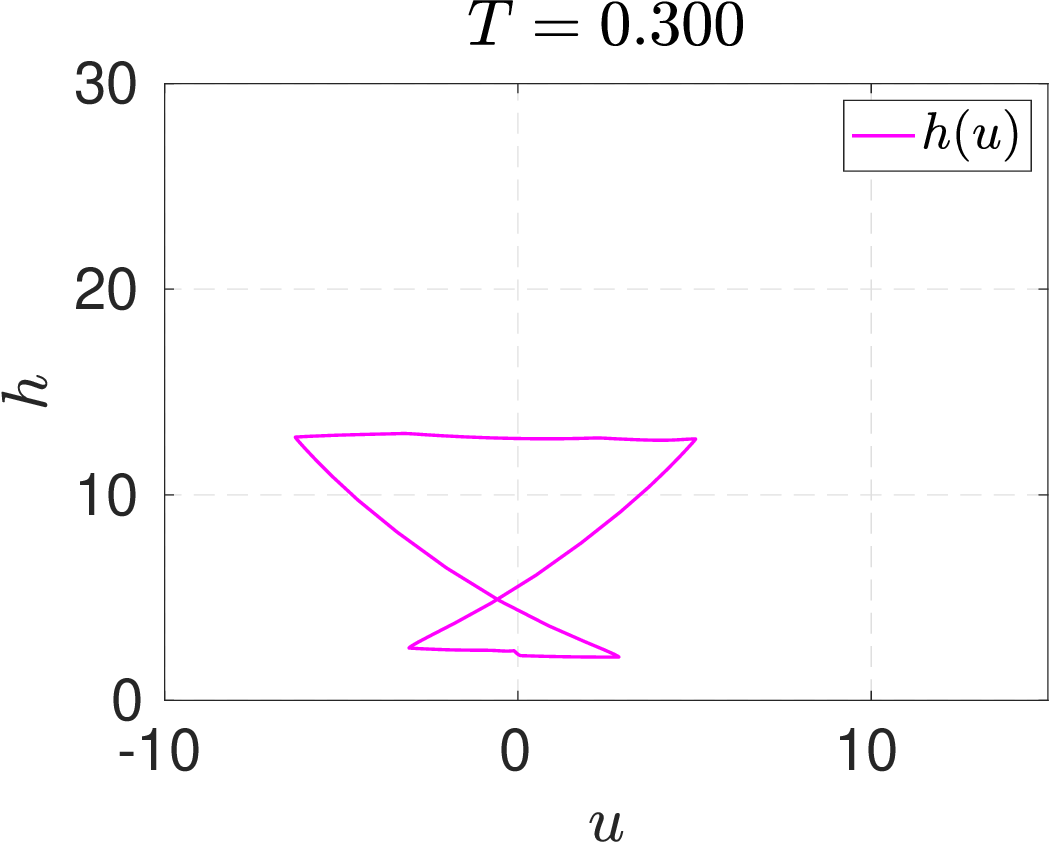}
  \end{subfigure}
  \caption{The initial \textit{invariant curve in \((u,h)-\)plane} is shown on the left and its time-evolved state on the center and on the right.}
  \label{fig:invariant-curve_case3c}
\end{figure}

% ------------------------------------------------------------------------

\begin{figure}[H]
  \centering
  \begin{subfigure}{0.49\textwidth}
    \centering
    \includegraphics[width=1.0\textwidth]{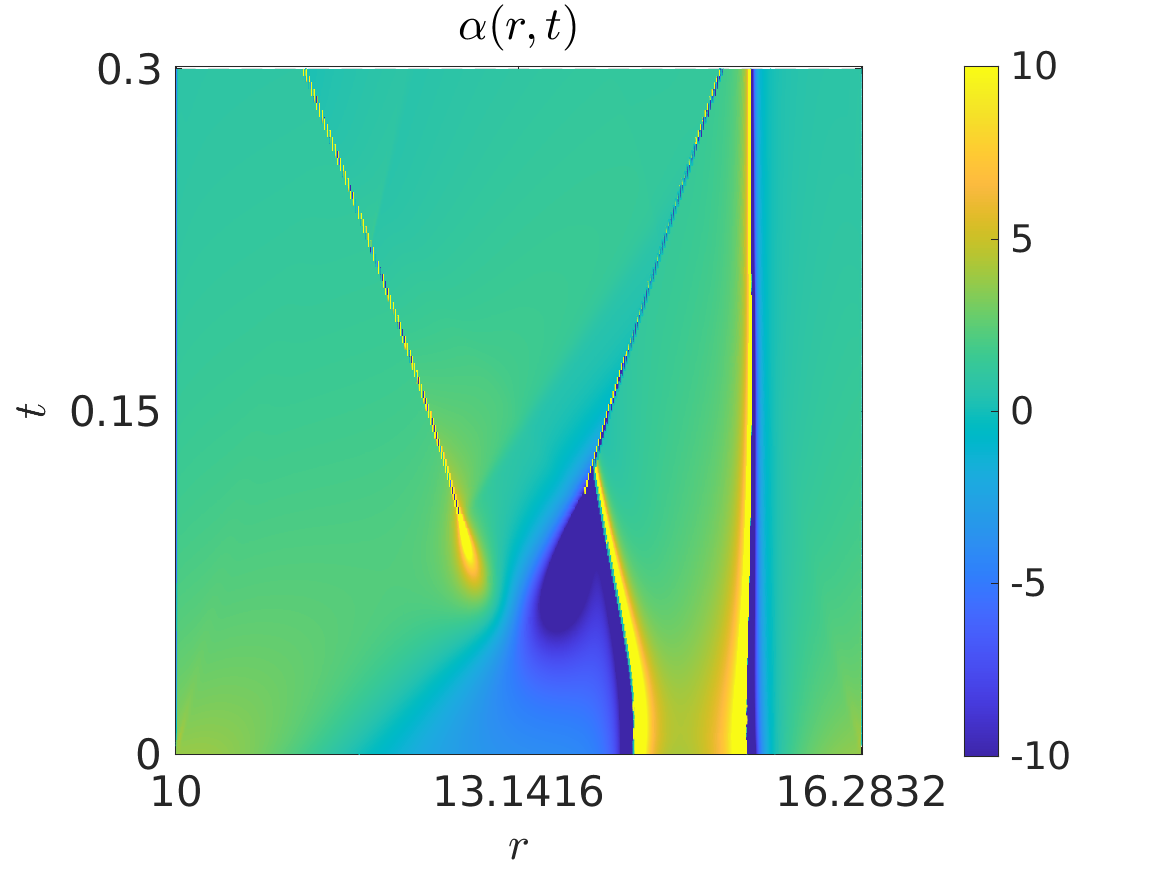}
  \end{subfigure}
  \begin{subfigure}{0.49\textwidth}
    \centering
	\includegraphics[width=1.0\textwidth]{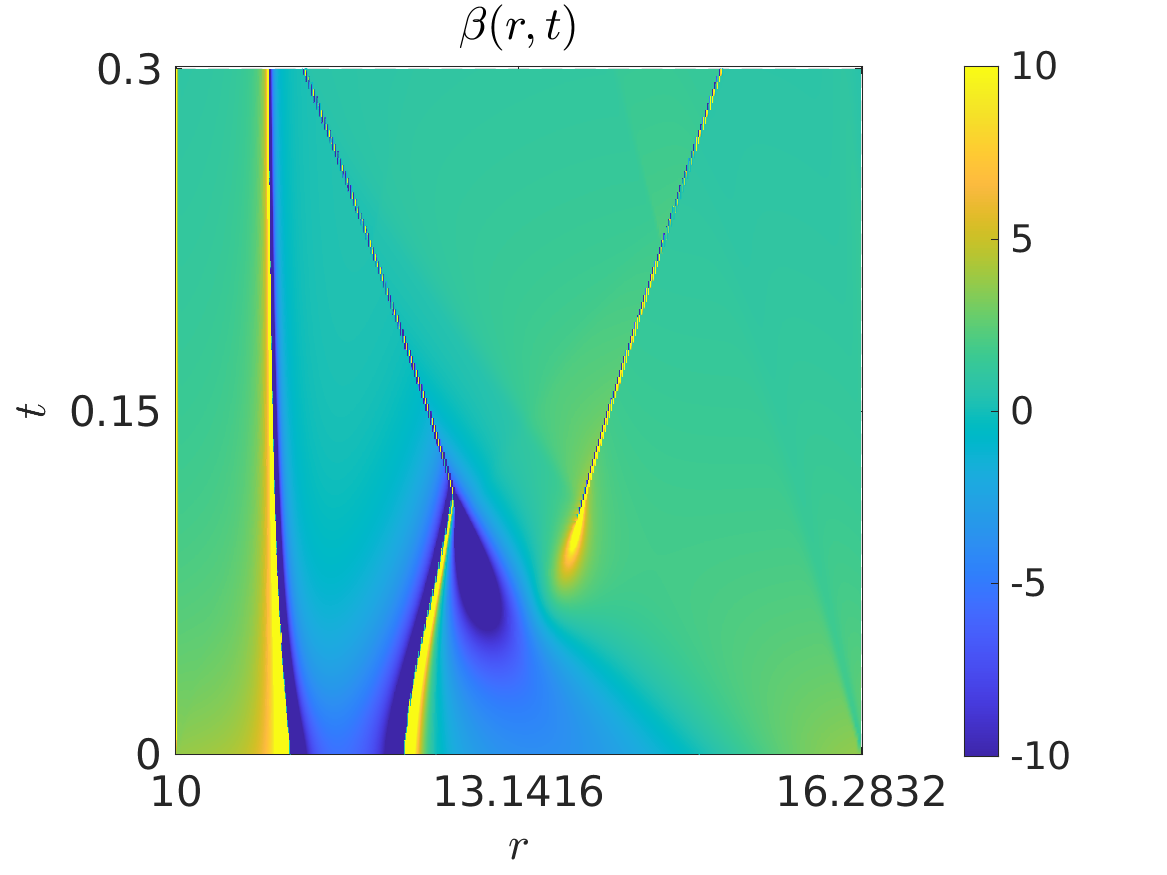}
  \end{subfigure}
  \caption{The \textit{heat map of $\alpha$ in \((r,t)-\)plane} is shown on the left and the heat map of $\beta$ on the right.}
  \label{fig:heat-map_case3c}
\end{figure}

% ------------------------------------------------------------------------

\subsection*{Case 4: realistic compression (sea-level scaling)}

To probe the subsonic mechanisms and the near-periodic behavior characteristic of restricted domains, we examine a periodic configuration with:
\[
K=7.75\times 10^{4},\quad \gamma=1.4,\quad h_{c}=343\ \mathrm{m/s},\quad v_a=3400 \ \mathrm{m/s},\quad r\in[1,5].
\]
We investigate the structural stability of the solution across three perturbation scales, $\epsilon \in {10.0,1.0,0.1}$, to assess the influence of the nonlinearity on the global regularity of the flow.

For the small-to-moderate amplitude regimes, $\epsilon=10$ (Case 3.1) and $\epsilon=1$ (Case 3.2), the system exhibits sustained oscillatory dynamics. The numerical diagnostics (Figures \ref{fig:density_case3a}--\ref{fig:invariant-curve_case3b}) confirm the persistence of smooth functional states, while the spatio-temporal heat maps for $\alpha$ and $\beta$ reveal bounded, well-structured patterns consistent with a stable subsonic equilibrium (Figures \ref{fig:heat-map_case3a}--\ref{fig:heat-map_case3b}).

Conversely, for the large-amplitude regime where $\epsilon=0.1$ (Case 3.3), the enhanced nonlinearity overwhelms the dispersive-like effects of the subsonic transport, precipitating a rapid steepening of the profile. This culminates in a gradient catastrophe and finite-time shock-like breakdown (see Figures \ref{fig:density_case3c}--\ref{fig:invariant-curve_case3c}). This transition is rigorously captured by the gradient variables, which exhibit localized concentration and blow-up in the negative domain, as illustrated in the phase-space diagnostics in Figure \ref{fig:heat-map_case3c}. This case underscores the existence of an amplitude threshold for regularity in the presence of periodic boundary constraints and geometric effects.

% ------------------------------------------------------------------------

\begin{figure}[H]
  \centering
  \begin{subfigure}{0.32\textwidth}
    \centering
    \includegraphics[width=1.0\textwidth]{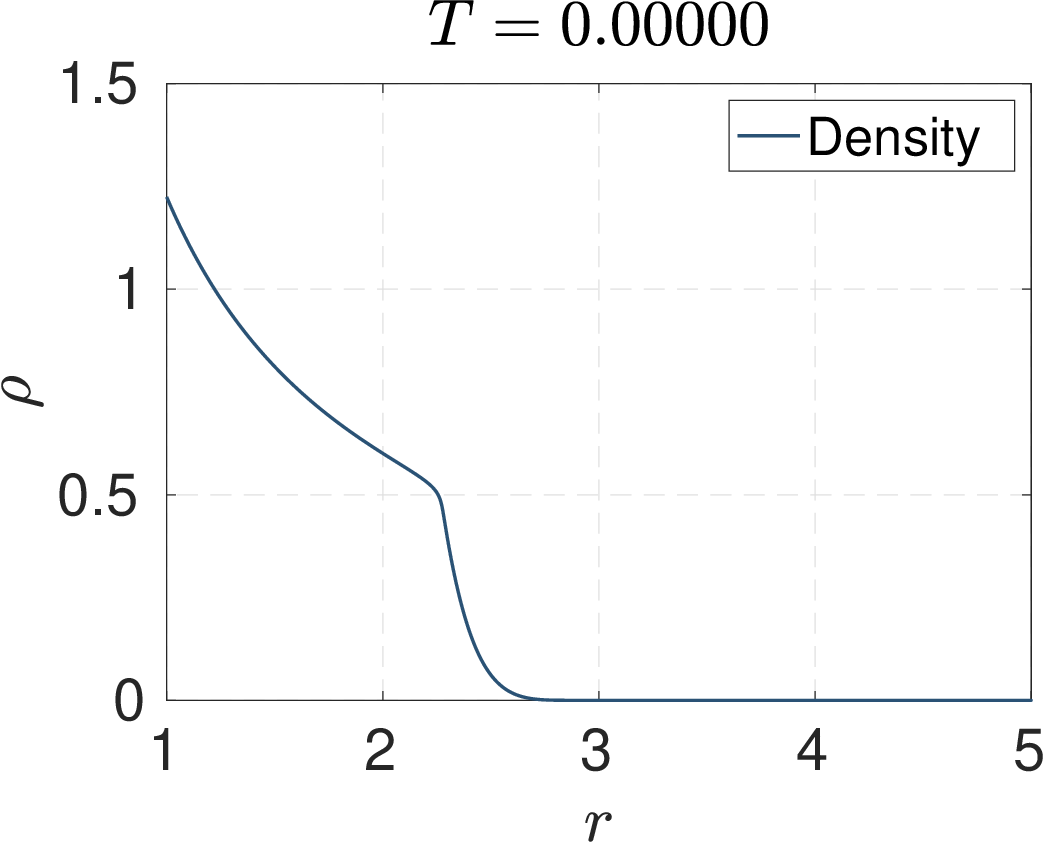}
  \end{subfigure}
    \begin{subfigure}{0.32\textwidth}
    \centering
    \includegraphics[width=1.0\textwidth]{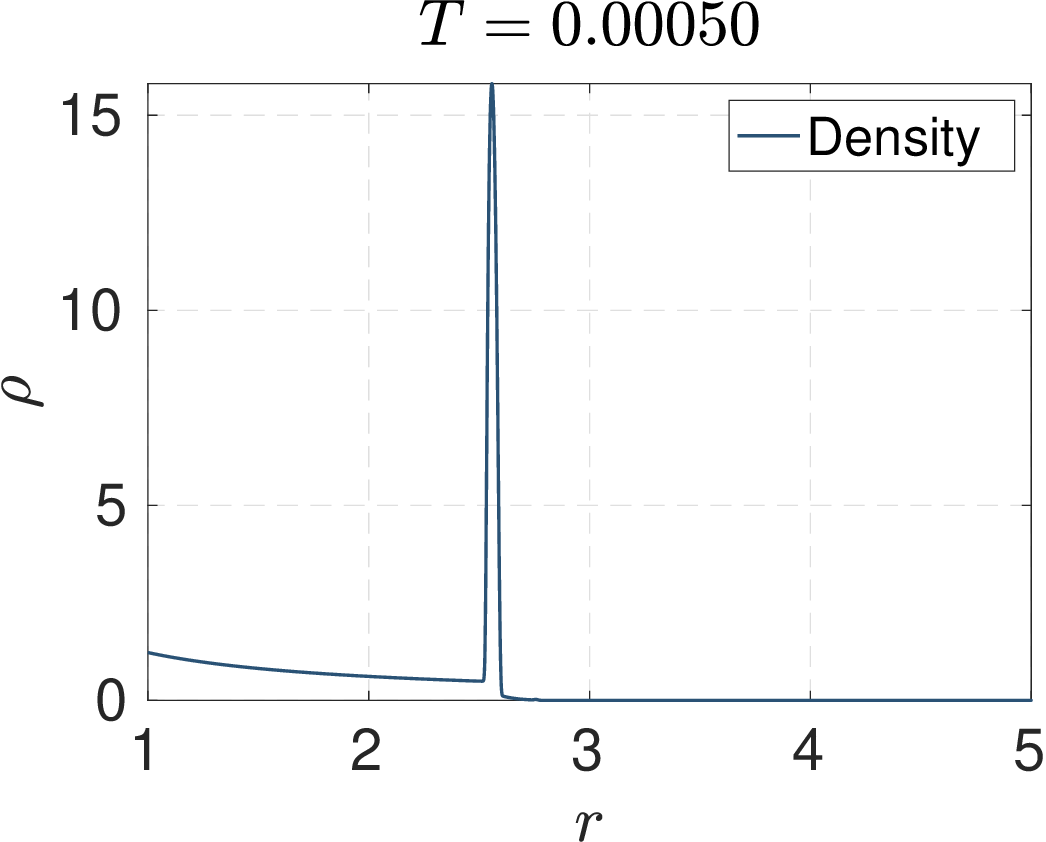}
  \end{subfigure}
  \begin{subfigure}{0.32\textwidth}
    \centering
    \includegraphics[width=1.0\textwidth]{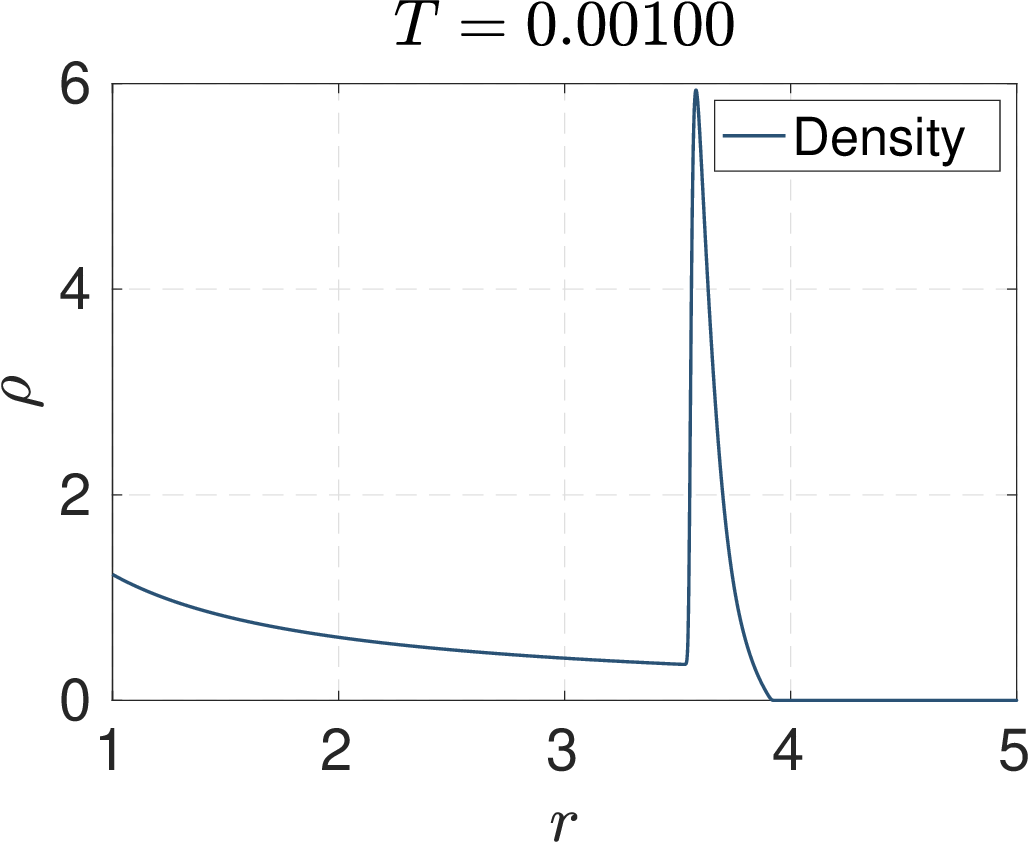}
  \end{subfigure}
  \caption{The initial \textit{density} is shown on the left and its time-evolved state on the center and on the right.}
  \label{fig:density_case4}
\end{figure}

% ------------------------------------------------------------------------

\begin{figure}[H]
  \centering
  \begin{subfigure}{0.32\textwidth}
    \centering
    \includegraphics[width=1.0\textwidth]{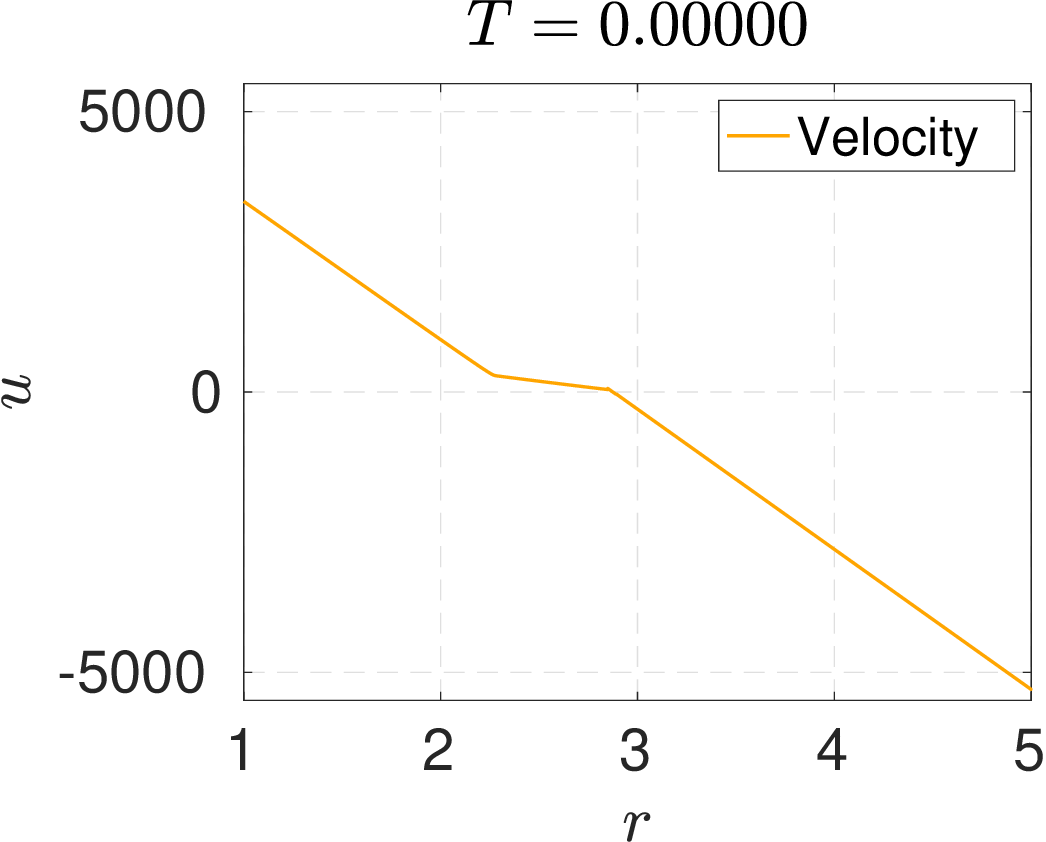}
  \end{subfigure}
    \begin{subfigure}{0.32\textwidth}
    \centering
    \includegraphics[width=1.0\textwidth]{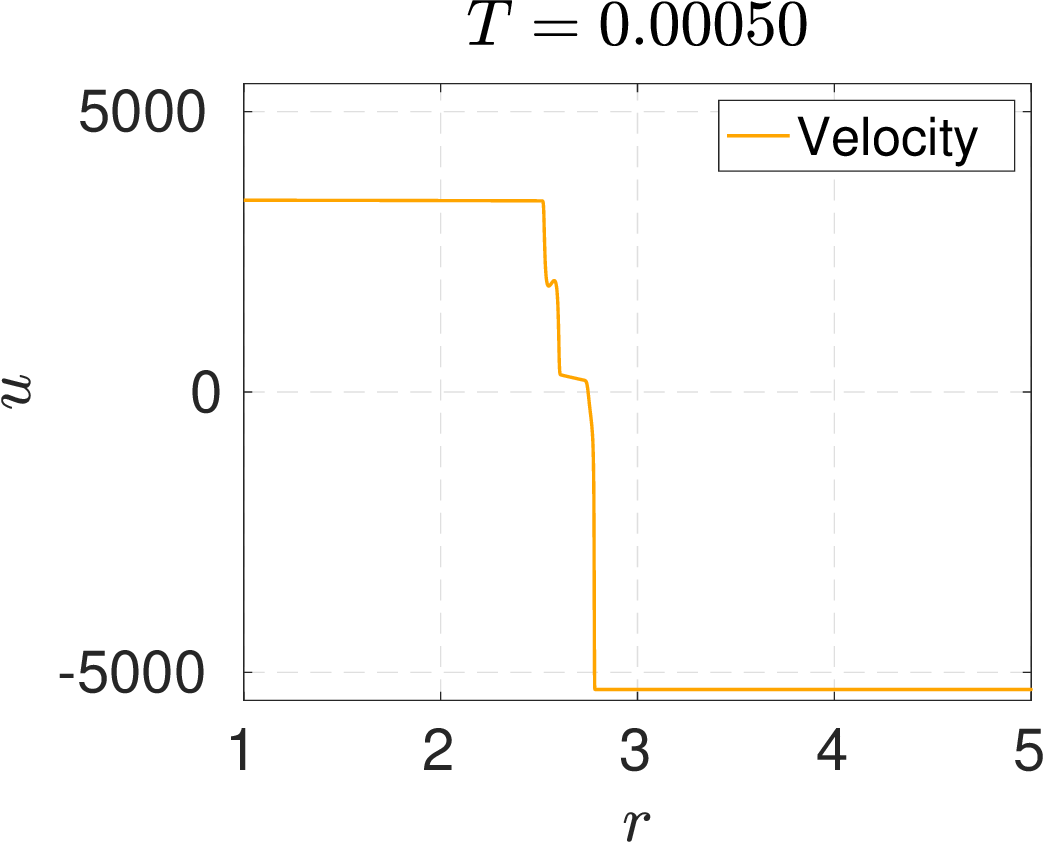}
  \end{subfigure}
  \begin{subfigure}{0.32\textwidth}
    \centering
    \includegraphics[width=1.0\textwidth]{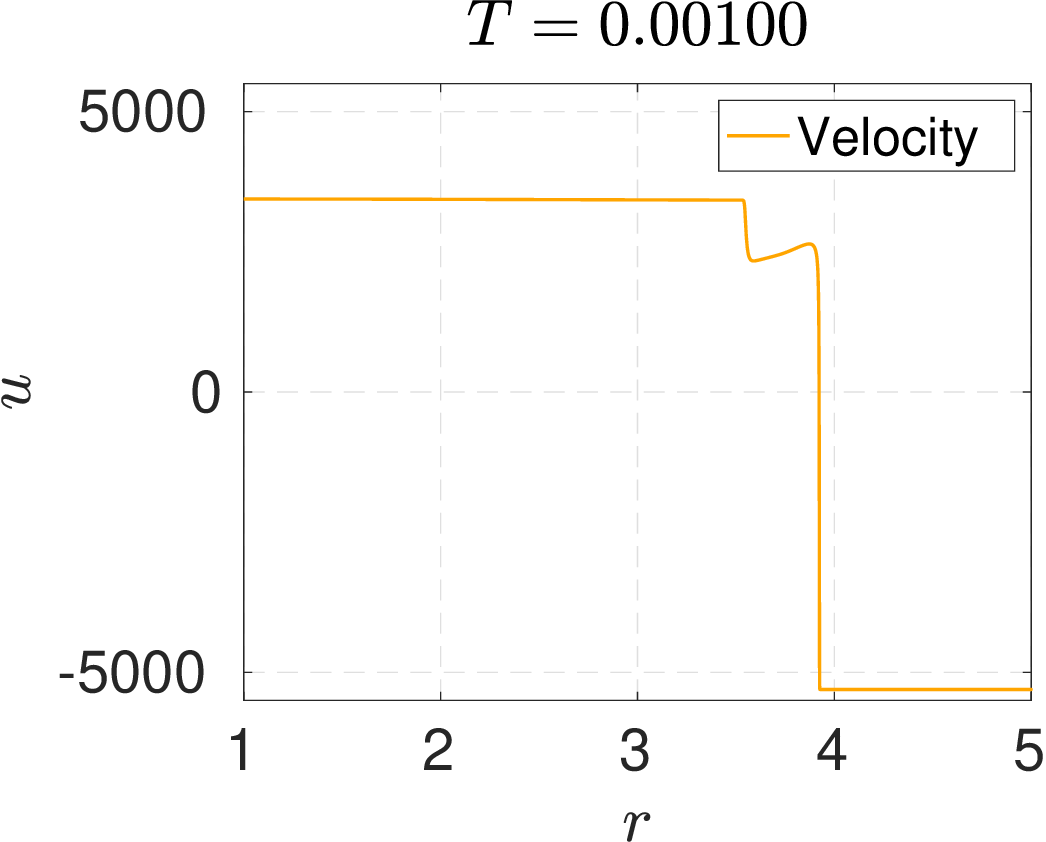}
  \end{subfigure}
  \caption{The initial \textit{velocity} is shown on the left and its time-evolved state on the center and on the right.}
  \label{fig:velocity_case4}
\end{figure}

% ------------------------------------------------------------------------

\begin{figure}[H]
  \centering
  \begin{subfigure}{0.32\textwidth}
    \centering
    \includegraphics[width=1.0\textwidth]{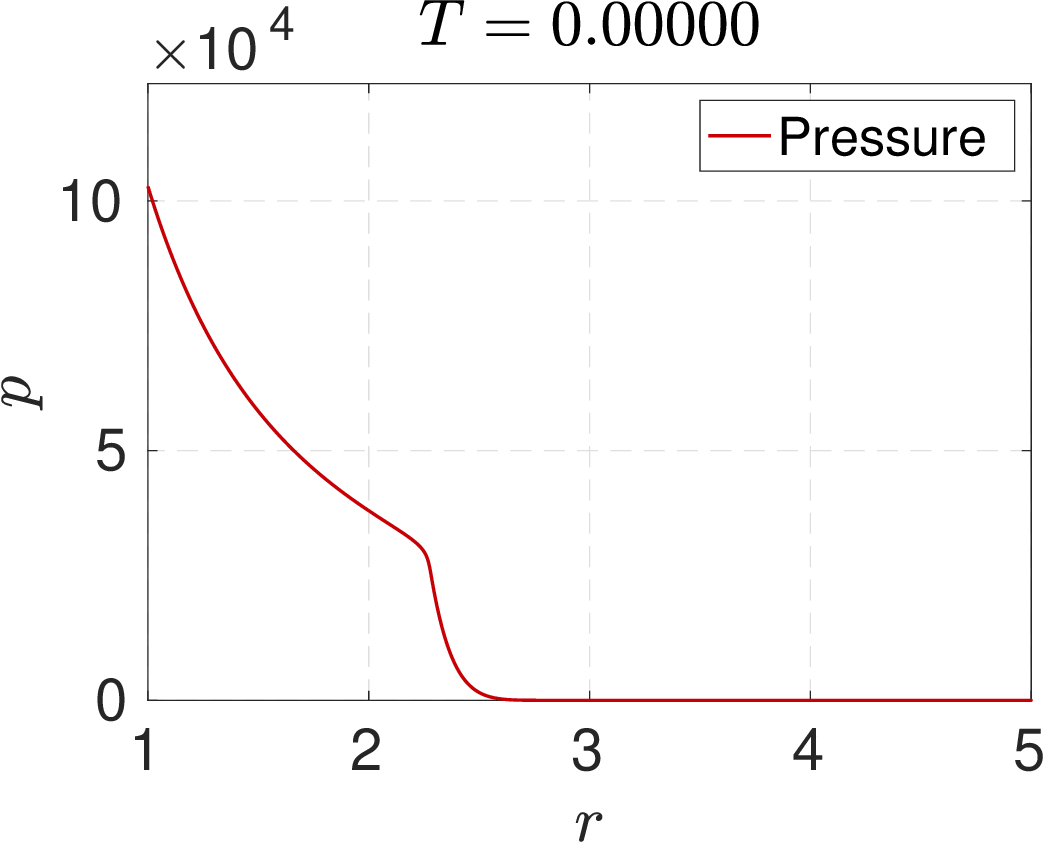}
  \end{subfigure}
    \begin{subfigure}{0.32\textwidth}
    \centering
    \includegraphics[width=1.0\textwidth]{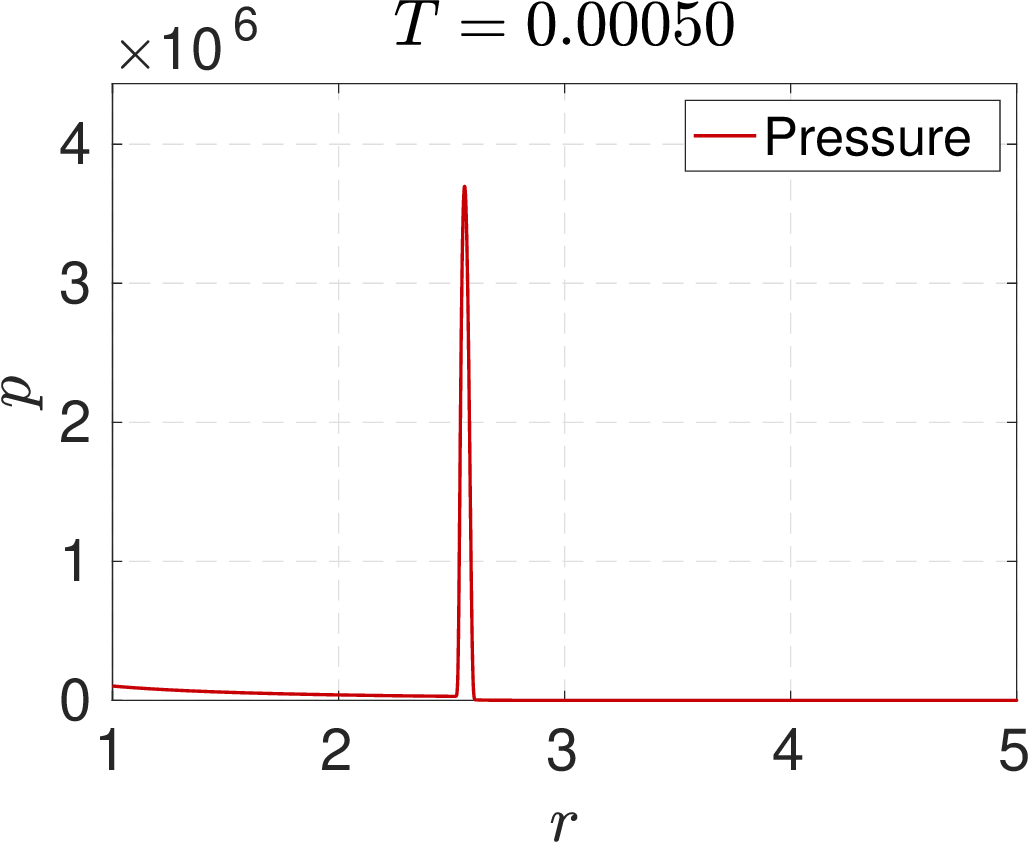}
  \end{subfigure}
  \begin{subfigure}{0.32\textwidth}
    \centering
    \includegraphics[width=1.0\textwidth]{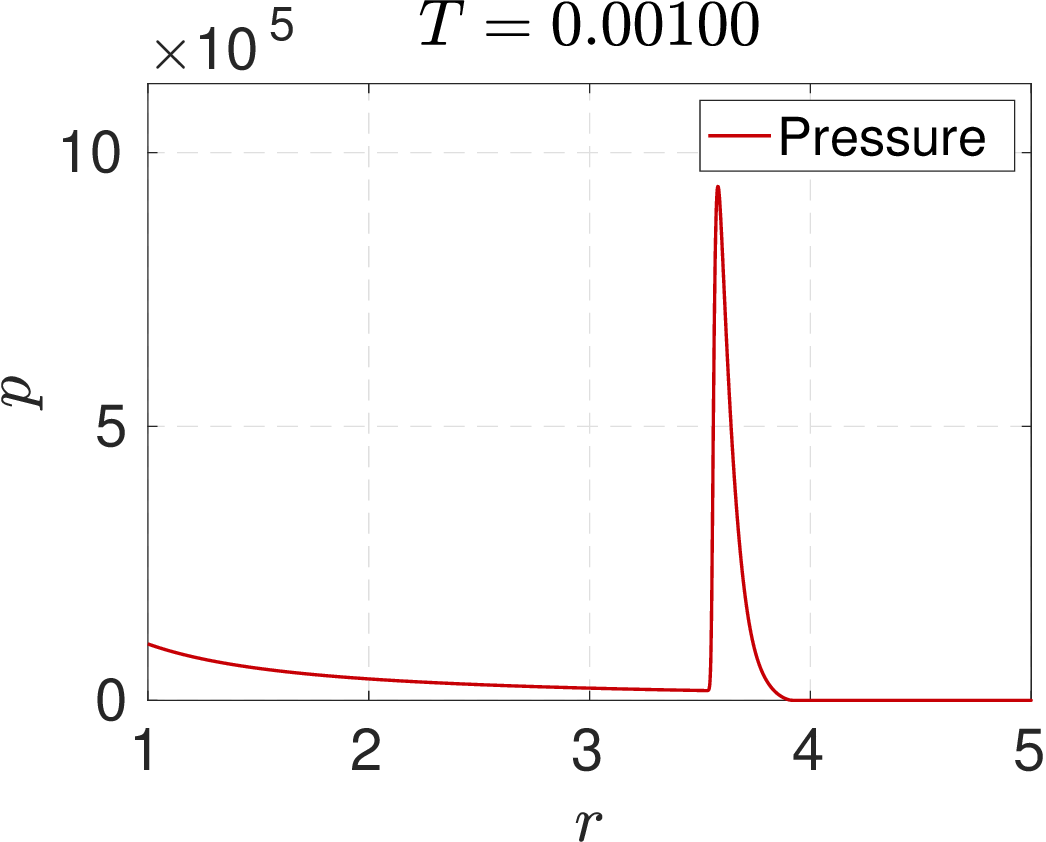}
  \end{subfigure}
  \caption{The initial \textit{pressure} is shown on the left and its time-evolved state on the center and on the right.}
  \label{fig:pressure_case4}
\end{figure}

% ------------------------------------------------------------------------

\begin{figure}[H]
  \centering
  \begin{subfigure}{0.32\textwidth}
    \centering
    \includegraphics[width=1.0\textwidth]{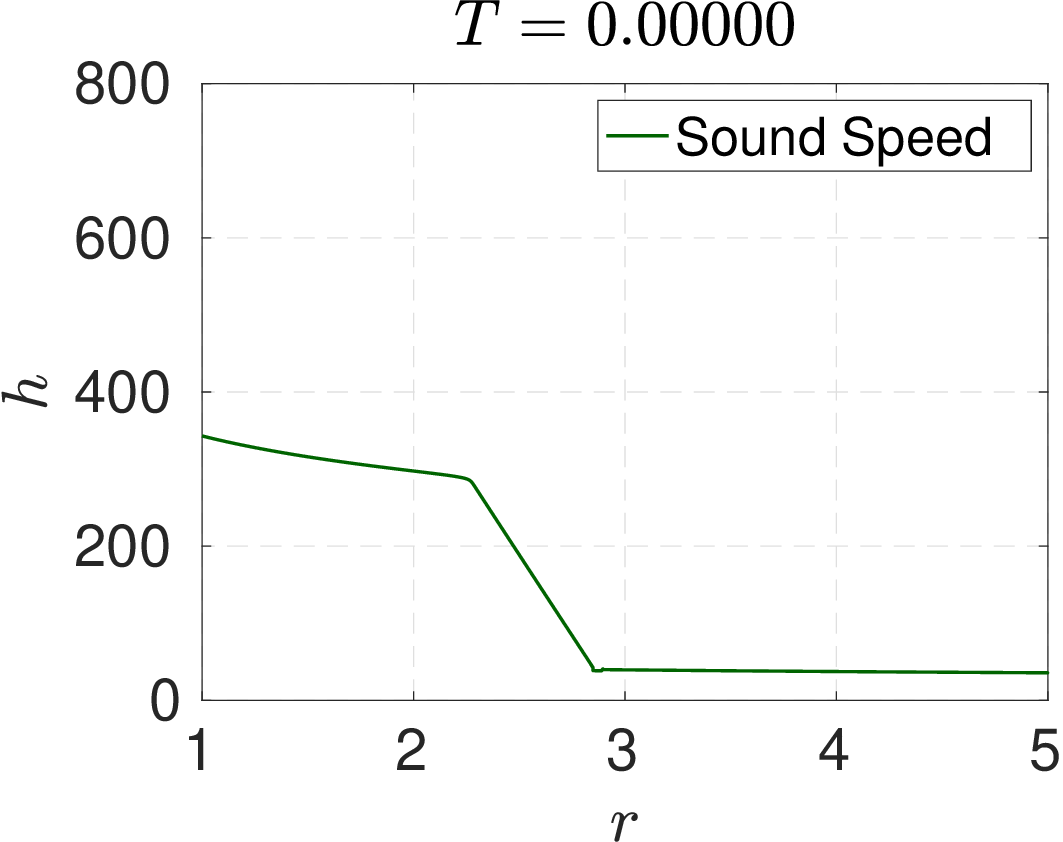}
  \end{subfigure}
    \begin{subfigure}{0.32\textwidth}
    \centering
    \includegraphics[width=1.0\textwidth]{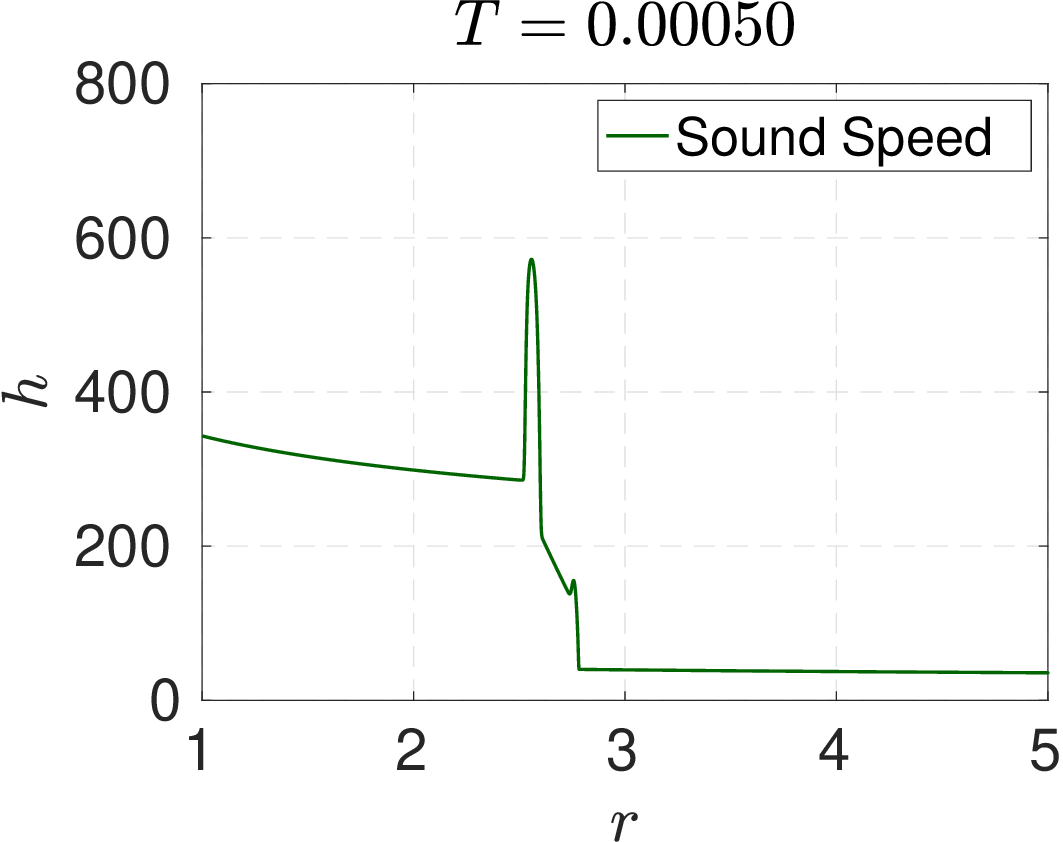}
  \end{subfigure}
  \begin{subfigure}{0.32\textwidth}
    \centering
    \includegraphics[width=1.0\textwidth]{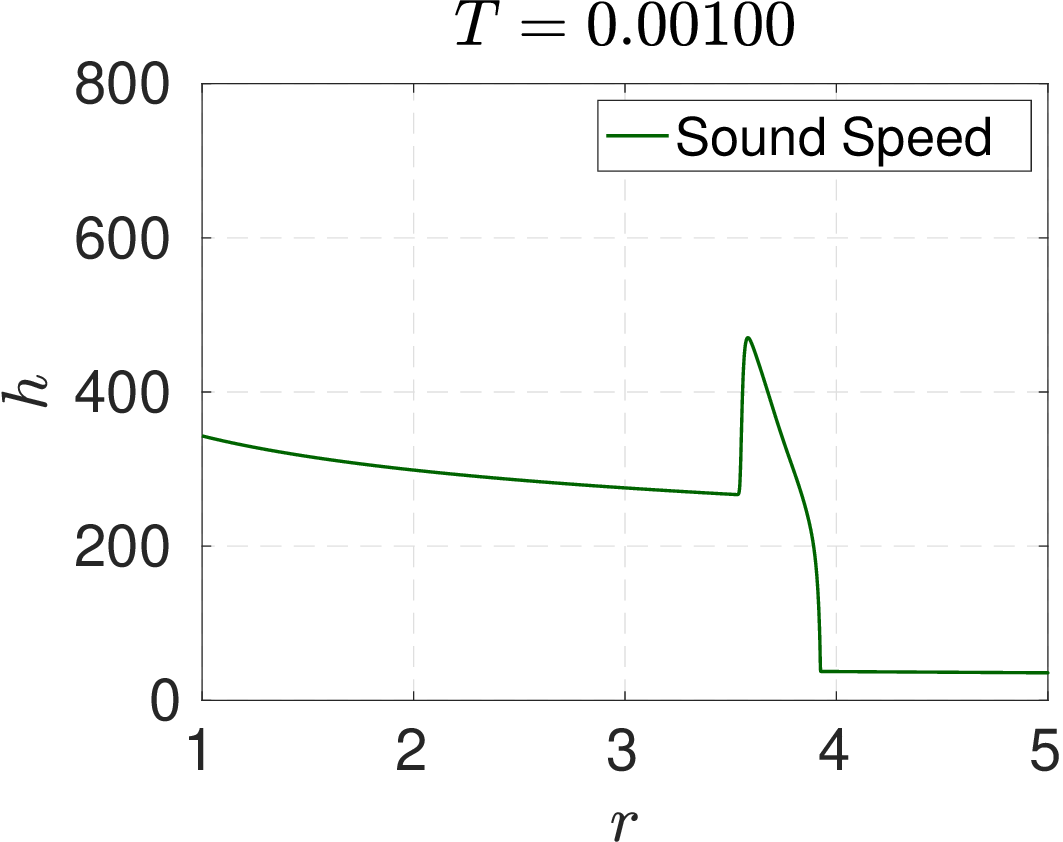}
  \end{subfigure}
  \caption{The initial \textit{sound speed} is shown on the left and its time-evolved state on the center and on the right.}
  \label{fig:sound-speed_case4}
\end{figure}

% ------------------------------------------------------------------------

\begin{figure}[H]
  \centering
  \begin{subfigure}{0.32\textwidth}
    \centering
    \includegraphics[width=1.0\textwidth]{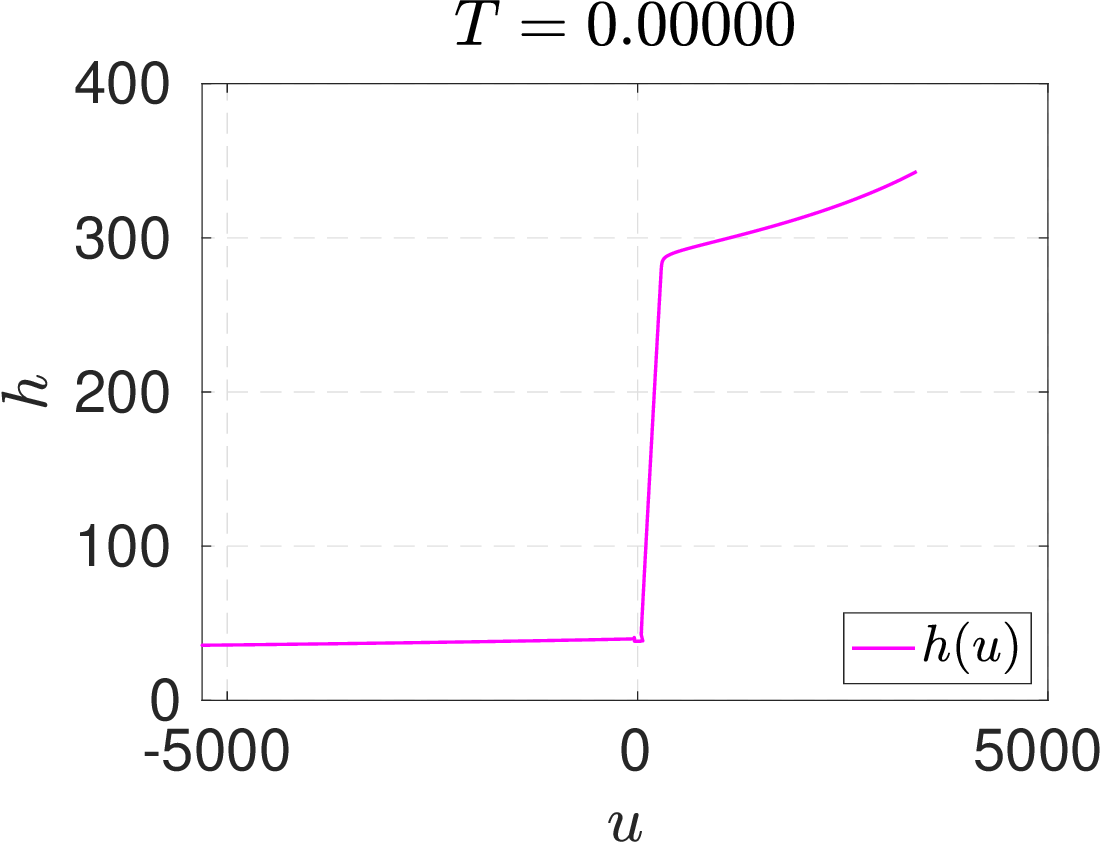}
  \end{subfigure}
    \begin{subfigure}{0.32\textwidth}
    \centering
    \includegraphics[width=1.0\textwidth]{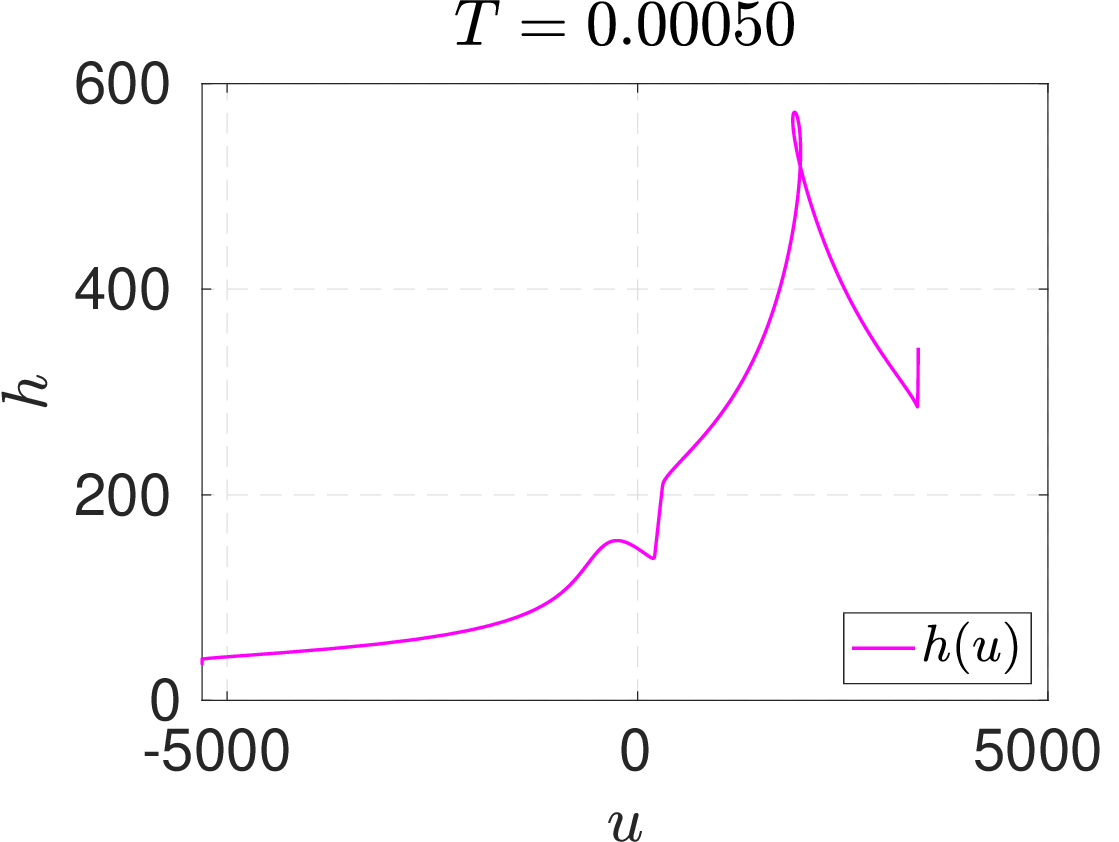}
  \end{subfigure}
  \begin{subfigure}{0.32\textwidth}
    \centering
    \includegraphics[width=1.0\textwidth]{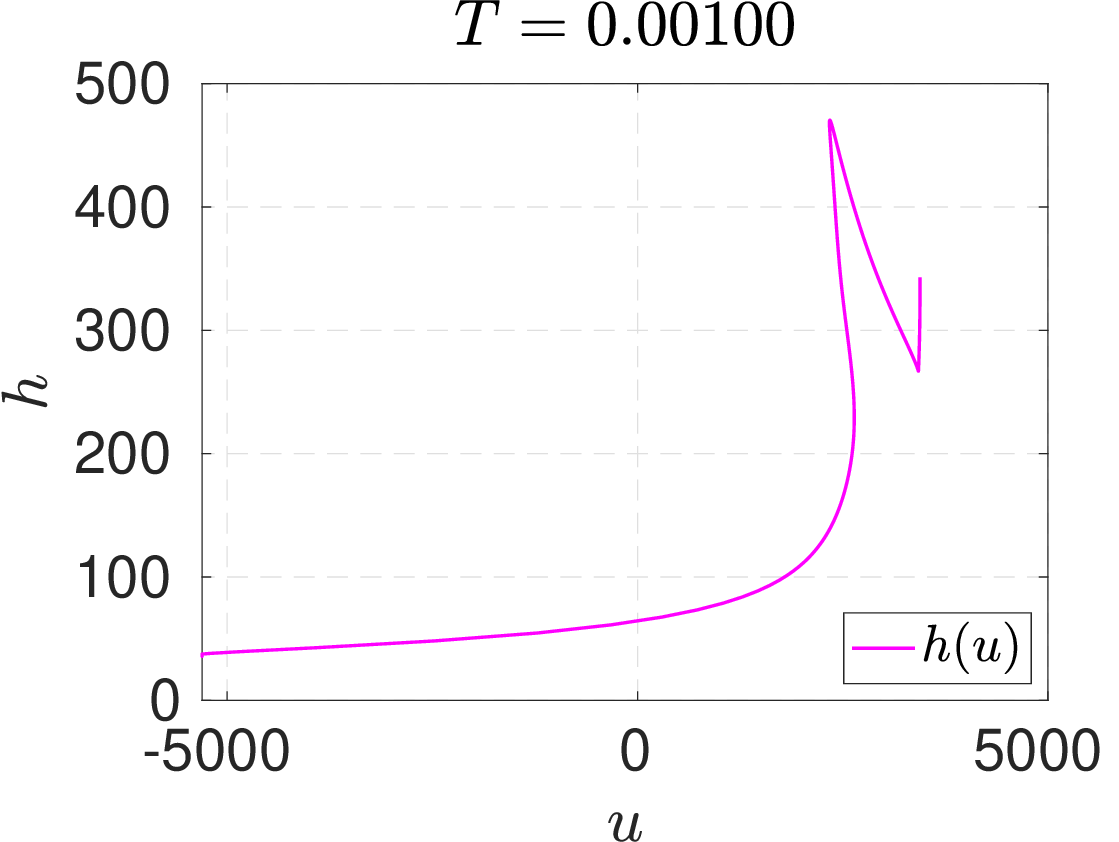}
  \end{subfigure}
  \caption{The initial \textit{invariant curve in \((u,h)-\)plane} is shown on the left and its time-evolved state on the center and on the right.}
  \label{fig:invariant-curve_case4}
\end{figure}

% ------------------------------------------------------------------------

\begin{figure}[H]
  \centering
  \begin{subfigure}{0.49\textwidth}
    \centering
    \includegraphics[width=1.0\textwidth]{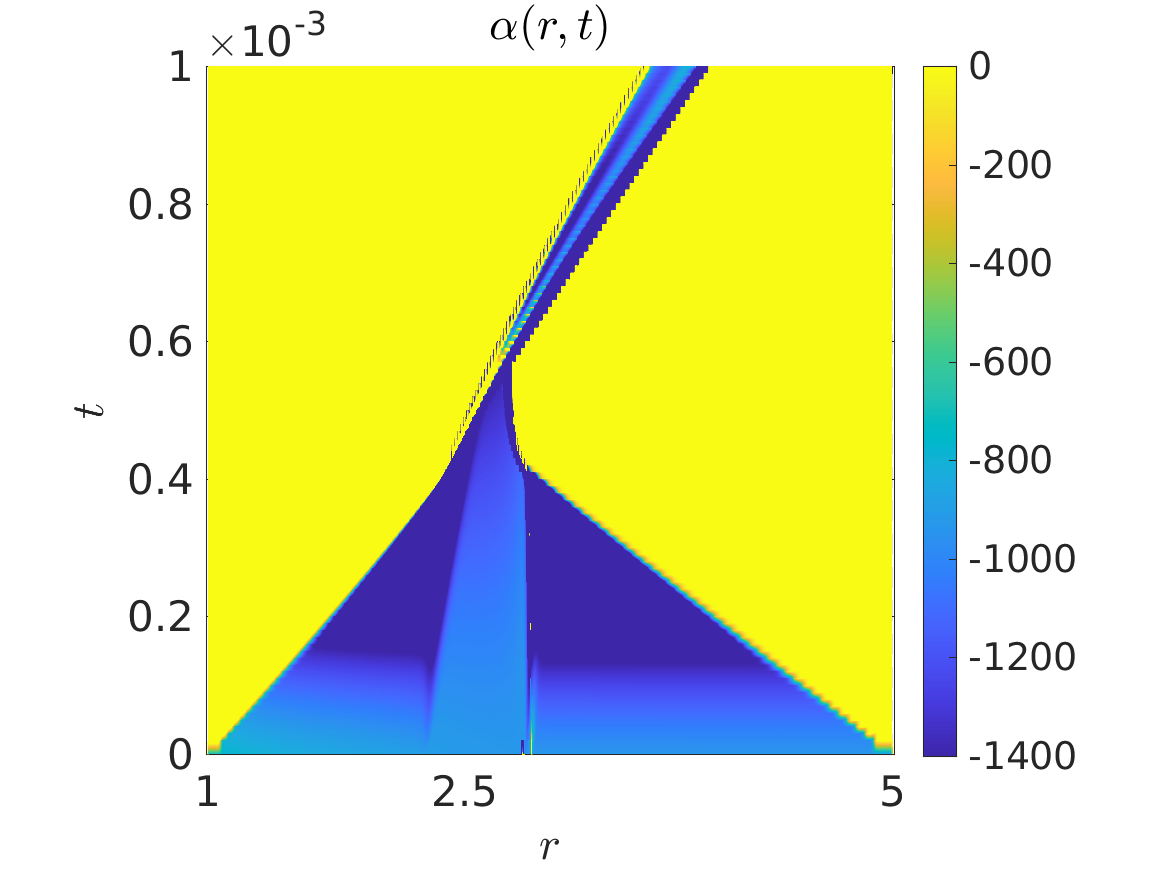}
  \end{subfigure}
  \begin{subfigure}{0.49\textwidth}
    \centering
	\includegraphics[width=1.0\textwidth]{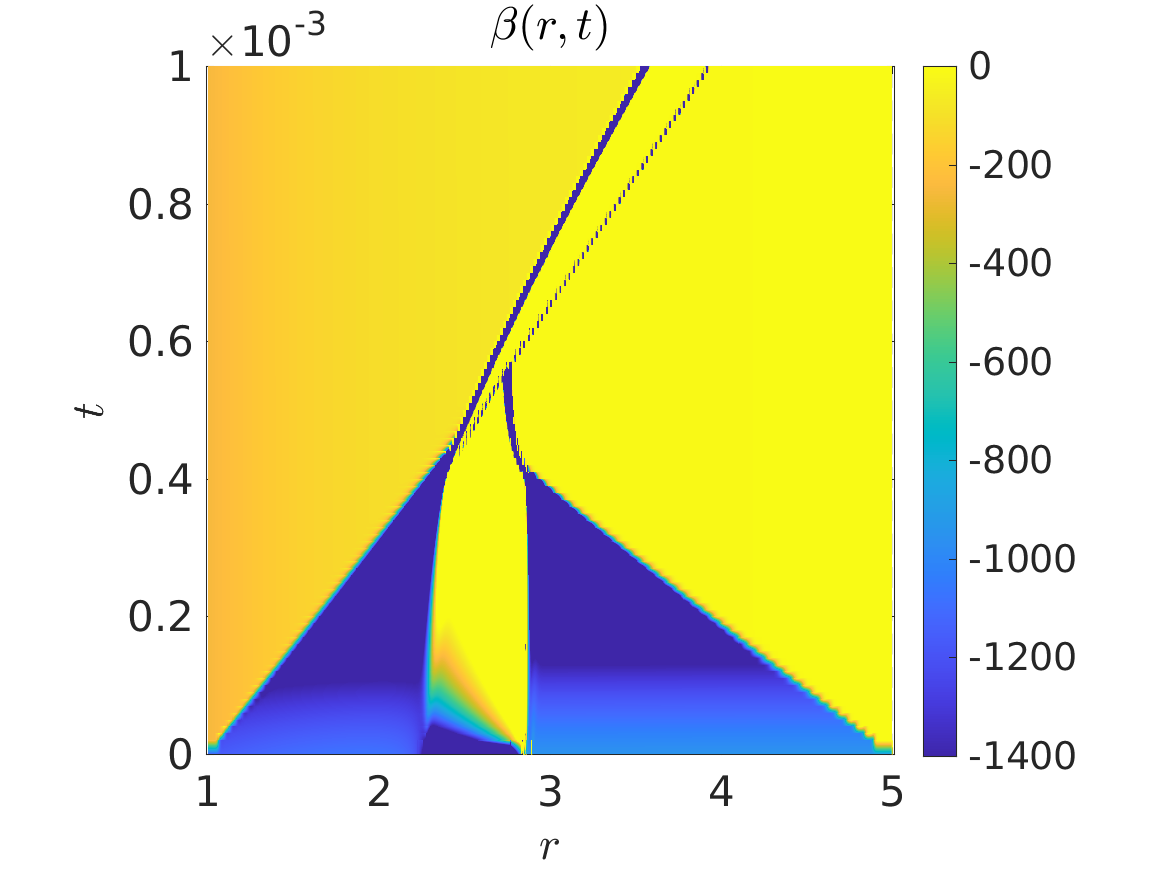}
  \end{subfigure}
  \caption{The \textit{heat map of $\alpha$ in \((r,t)-\)plane} is shown on the left and the heat map of $\beta$ on the right.}
  \label{fig:heat-map_case4}
\end{figure}

% ------------------------------------------------------------------------

\subsection*{Case 5: realistic rarefaction}

Keeping the same physical scaling and domain as in Case 4, we switch to rarefactive initial data by setting
\[
\alpha=\beta=1300.
\]
The evolution is qualitatively analogous to Case 2: the solution remains smooth on the simulated interval, the invariant curve evolves without the compressive folding typical of shock development, and the gradient variables remain predominantly nonnegative (see Figures \ref{fig:density_case5}-\ref{fig:heat-map_case5}).

% ------------------------------------------------------------------------

\begin{figure}[H]
  \centering
  \begin{subfigure}{0.32\textwidth}
    \centering
    \includegraphics[width=1.0\textwidth]{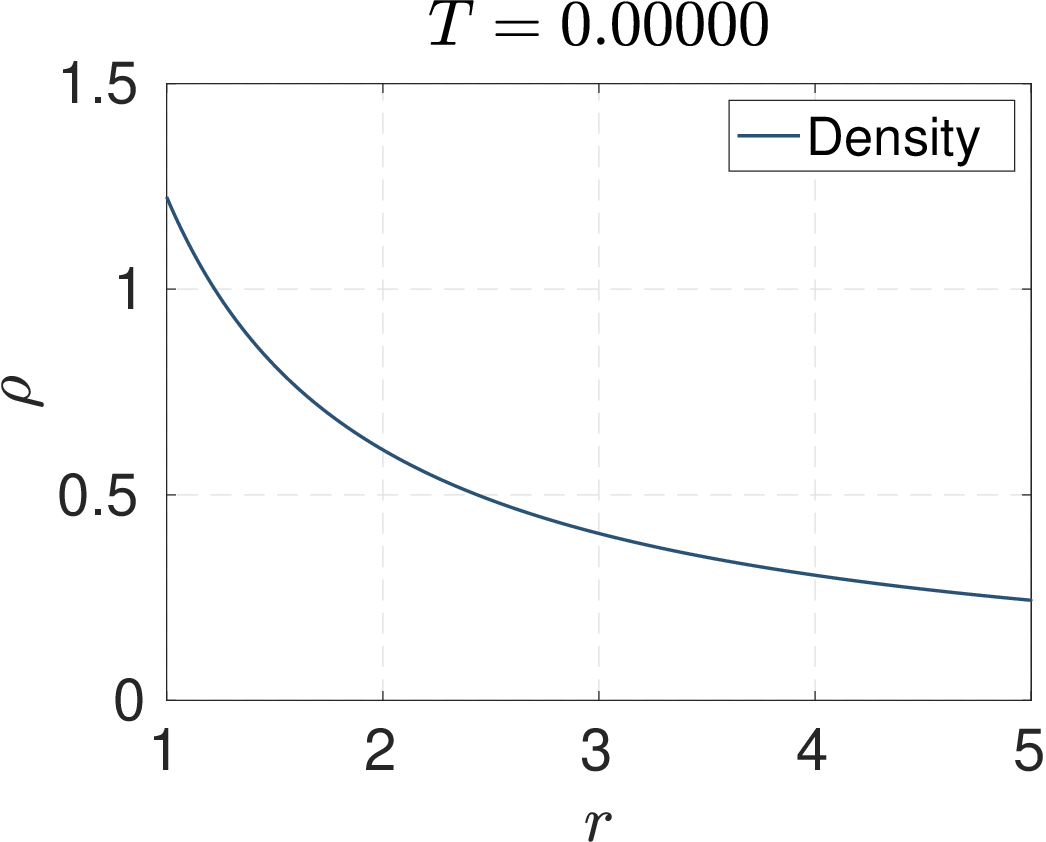}
  \end{subfigure}
    \begin{subfigure}{0.32\textwidth}
    \centering
    \includegraphics[width=1.0\textwidth]{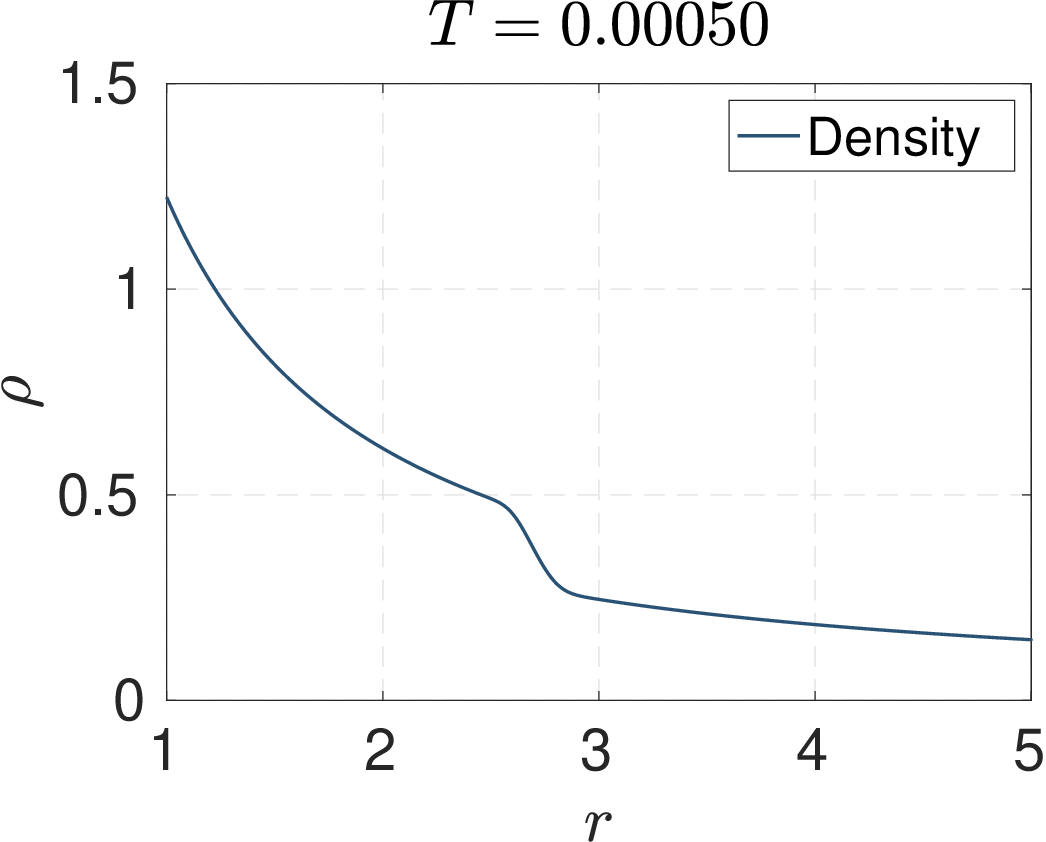}
  \end{subfigure}
  \begin{subfigure}{0.32\textwidth}
    \centering
    \includegraphics[width=1.0\textwidth]{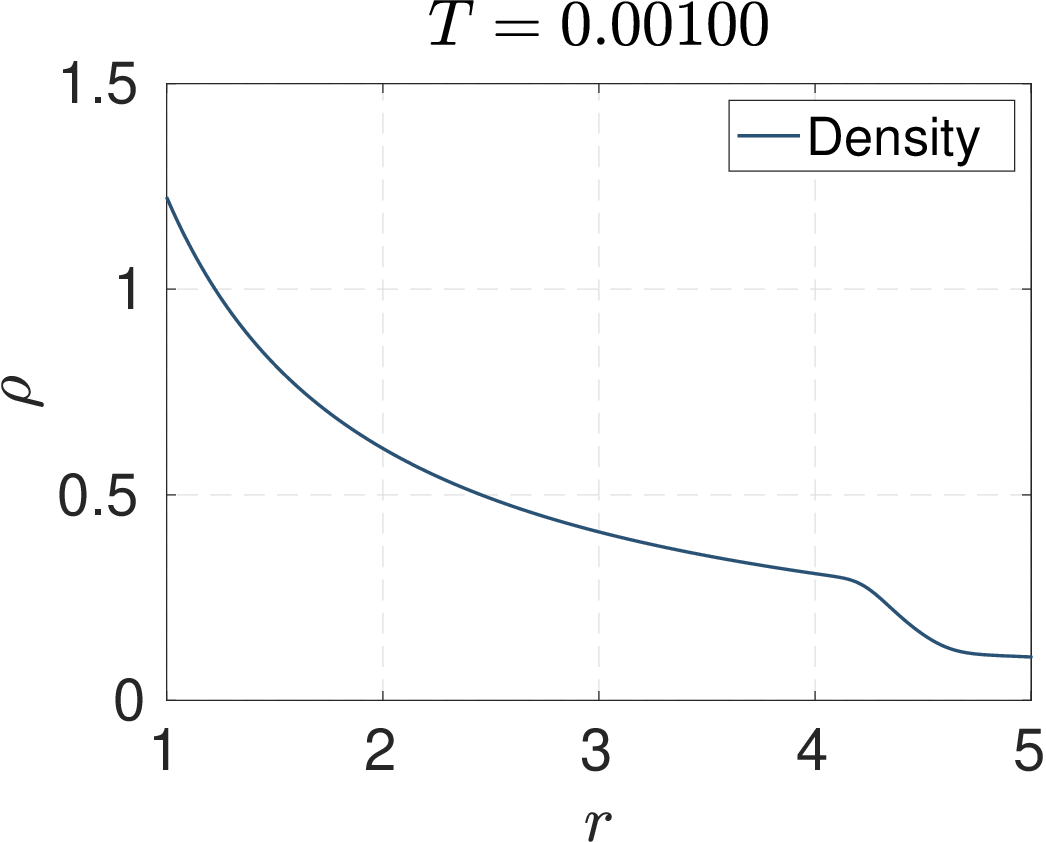}
  \end{subfigure}
  \caption{The initial \textit{density} is shown on the left and its time-evolved state on the center and on the right.}
  \label{fig:density_case5}
\end{figure}

% ------------------------------------------------------------------------

\begin{figure}[H]
  \centering
  \begin{subfigure}{0.32\textwidth}
    \centering
    \includegraphics[width=1.0\textwidth]{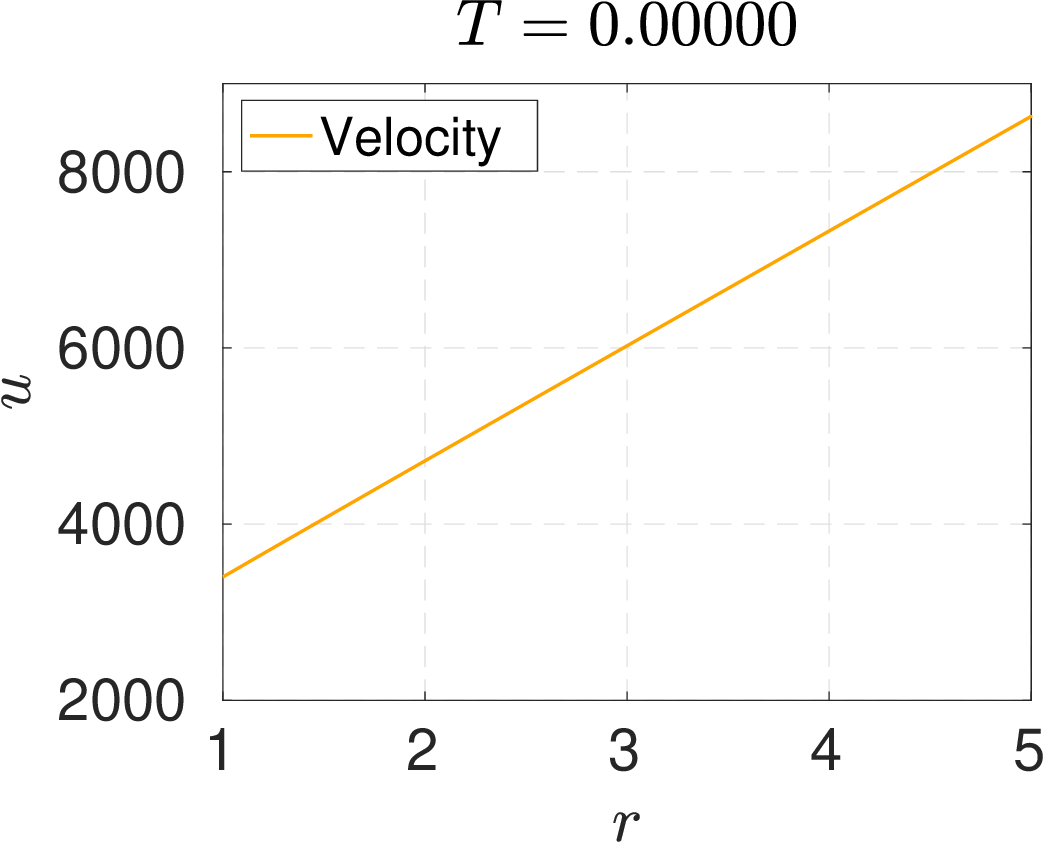}
  \end{subfigure}
    \begin{subfigure}{0.32\textwidth}
    \centering
    \includegraphics[width=1.0\textwidth]{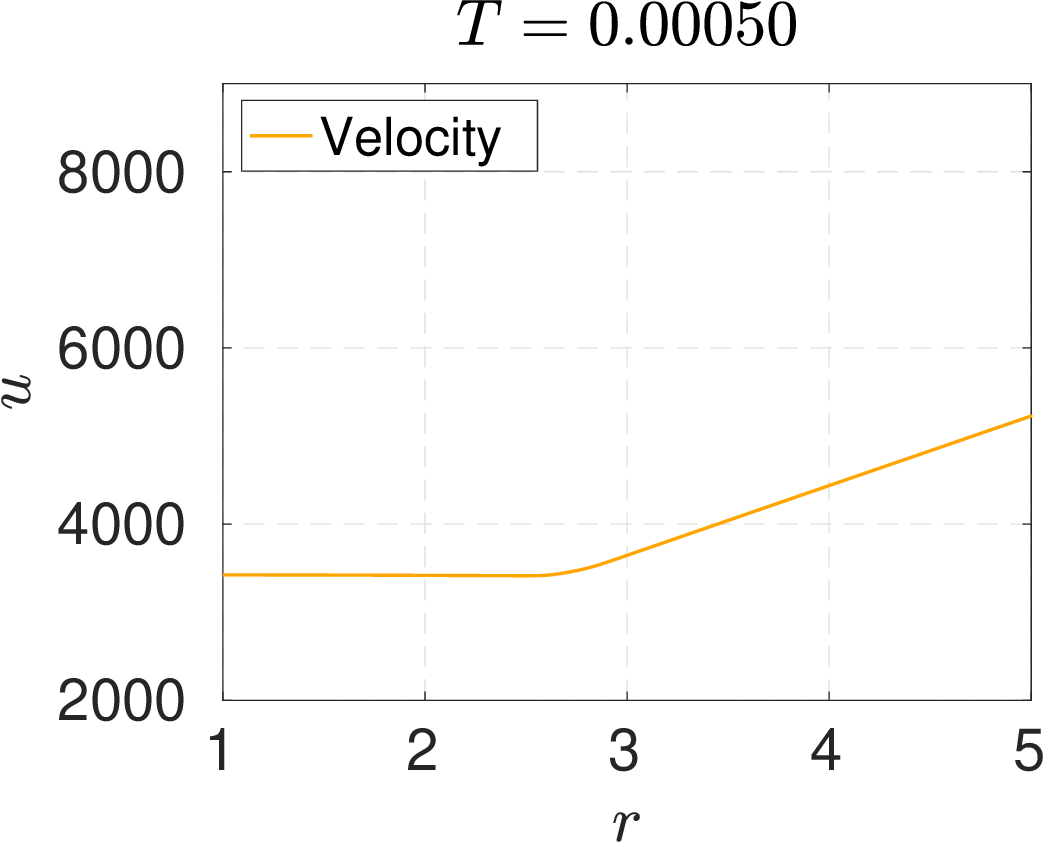}
  \end{subfigure}
  \begin{subfigure}{0.32\textwidth}
    \centering
    \includegraphics[width=1.0\textwidth]{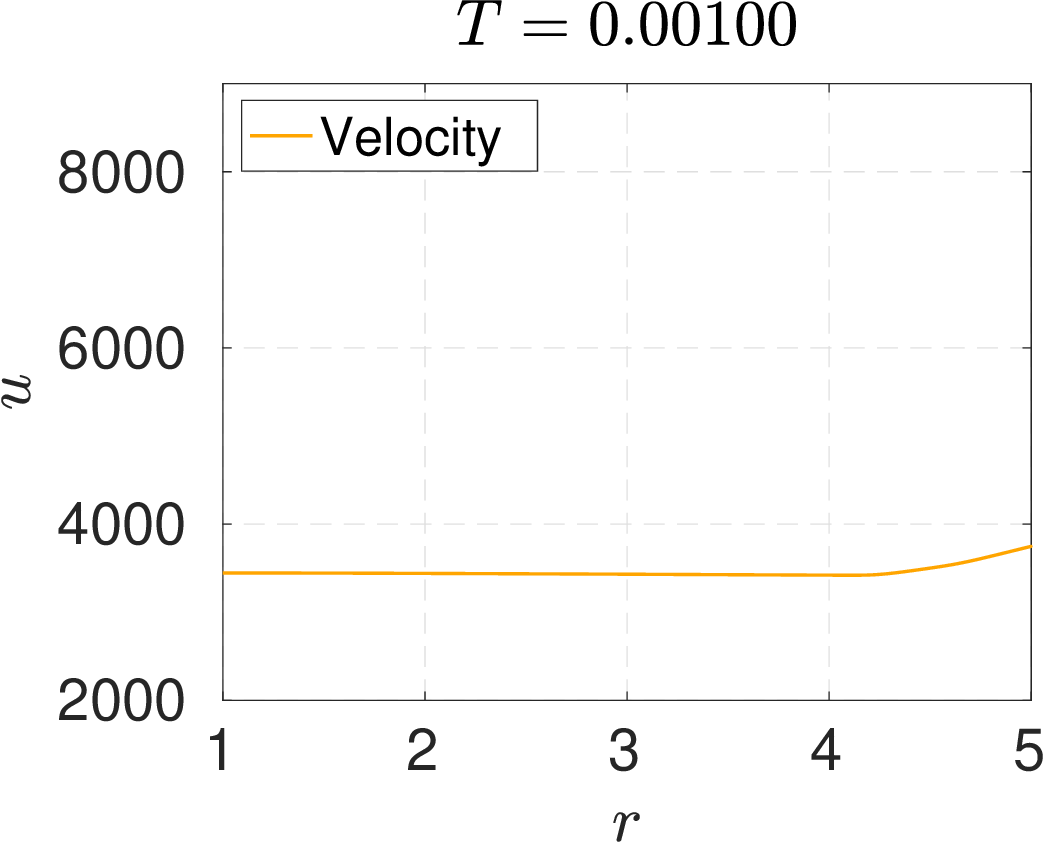}
  \end{subfigure}
  \caption{The initial \textit{velocity} is shown on the left and its time-evolved state on the center and on the right.}
  \label{fig:velocity_case5}
\end{figure}

% ------------------------------------------------------------------------

\begin{figure}[H]
  \centering
  \begin{subfigure}{0.32\textwidth}
    \centering
    \includegraphics[width=1.0\textwidth]{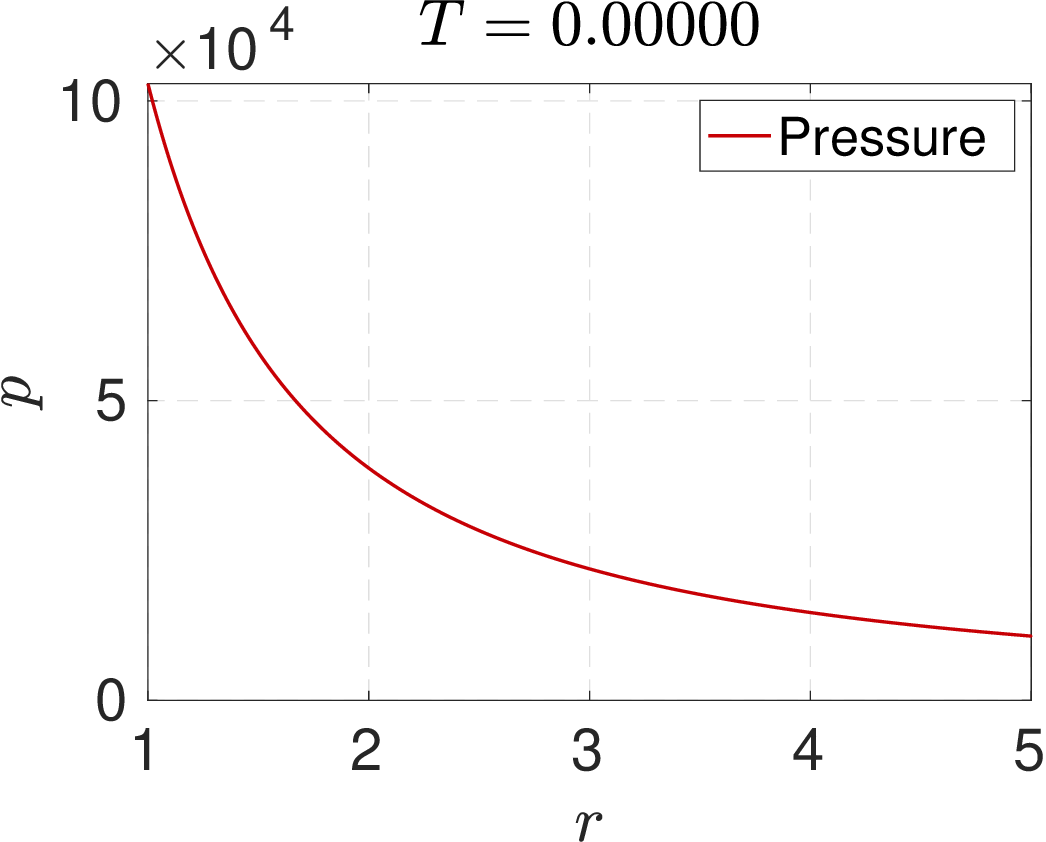}
  \end{subfigure}
    \begin{subfigure}{0.32\textwidth}
    \centering
    \includegraphics[width=1.0\textwidth]{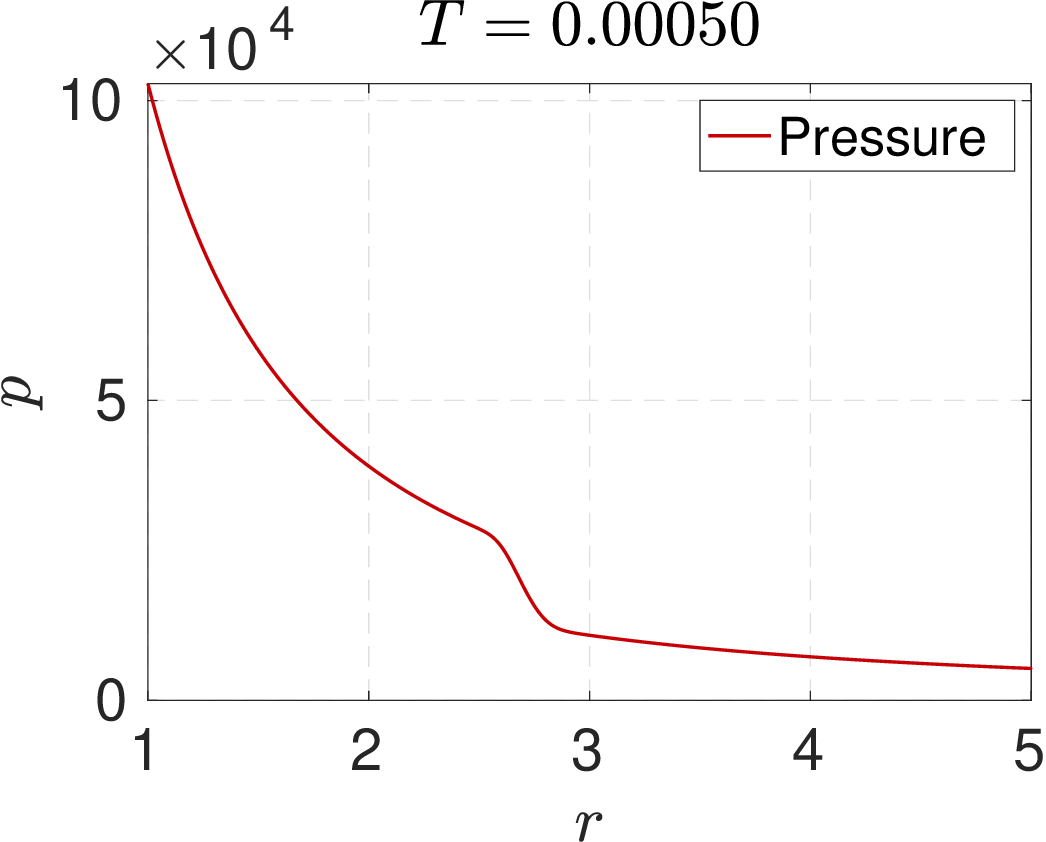}
  \end{subfigure}
  \begin{subfigure}{0.32\textwidth}
    \centering
    \includegraphics[width=1.0\textwidth]{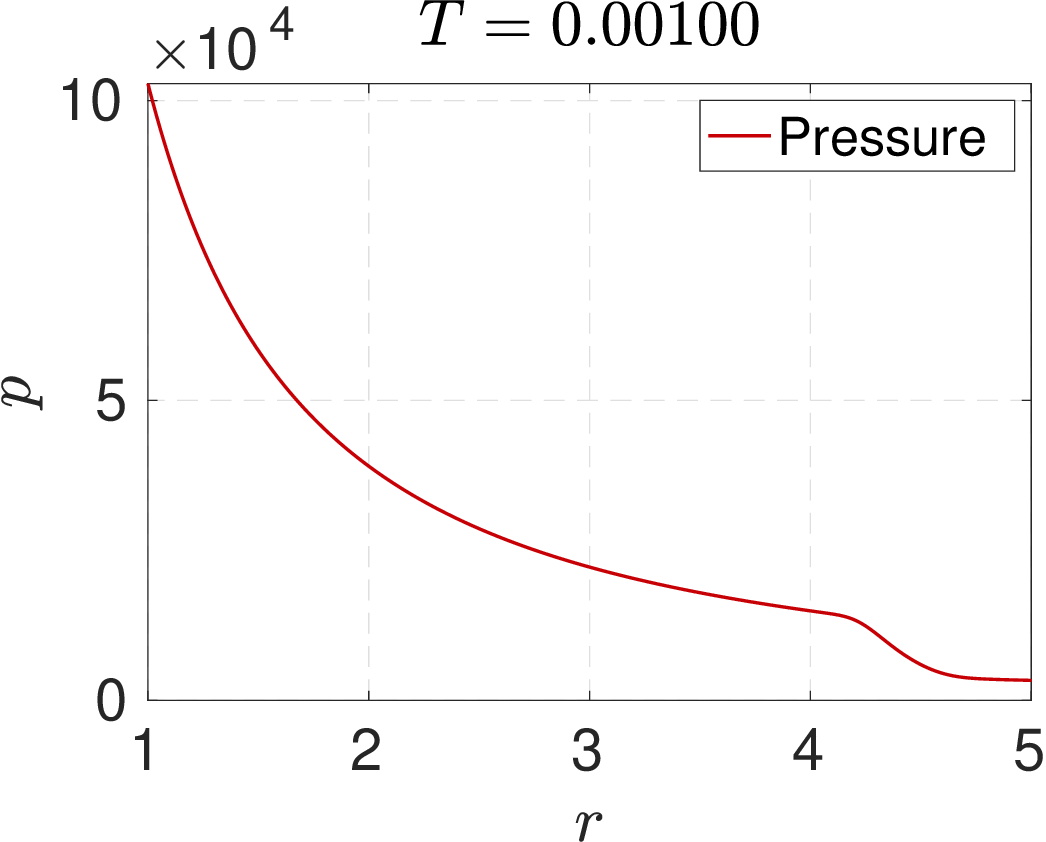}
  \end{subfigure}
  \caption{The initial \textit{pressure} is shown on the left and its time-evolved state on the center and on the right.}
  \label{fig:pressure_case5}
\end{figure}

% ------------------------------------------------------------------------

\begin{figure}[H]
  \centering
  \begin{subfigure}{0.32\textwidth}
    \centering
    \includegraphics[width=1.0\textwidth]{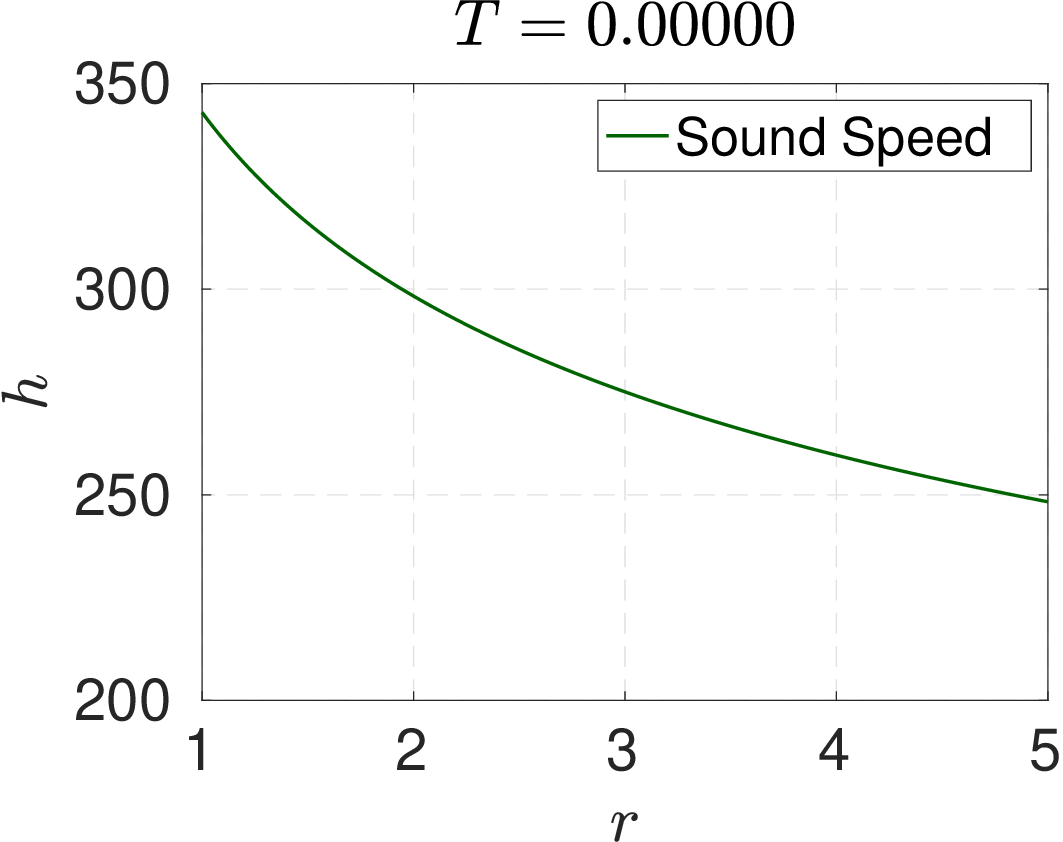}
  \end{subfigure}
    \begin{subfigure}{0.32\textwidth}
    \centering
    \includegraphics[width=1.0\textwidth]{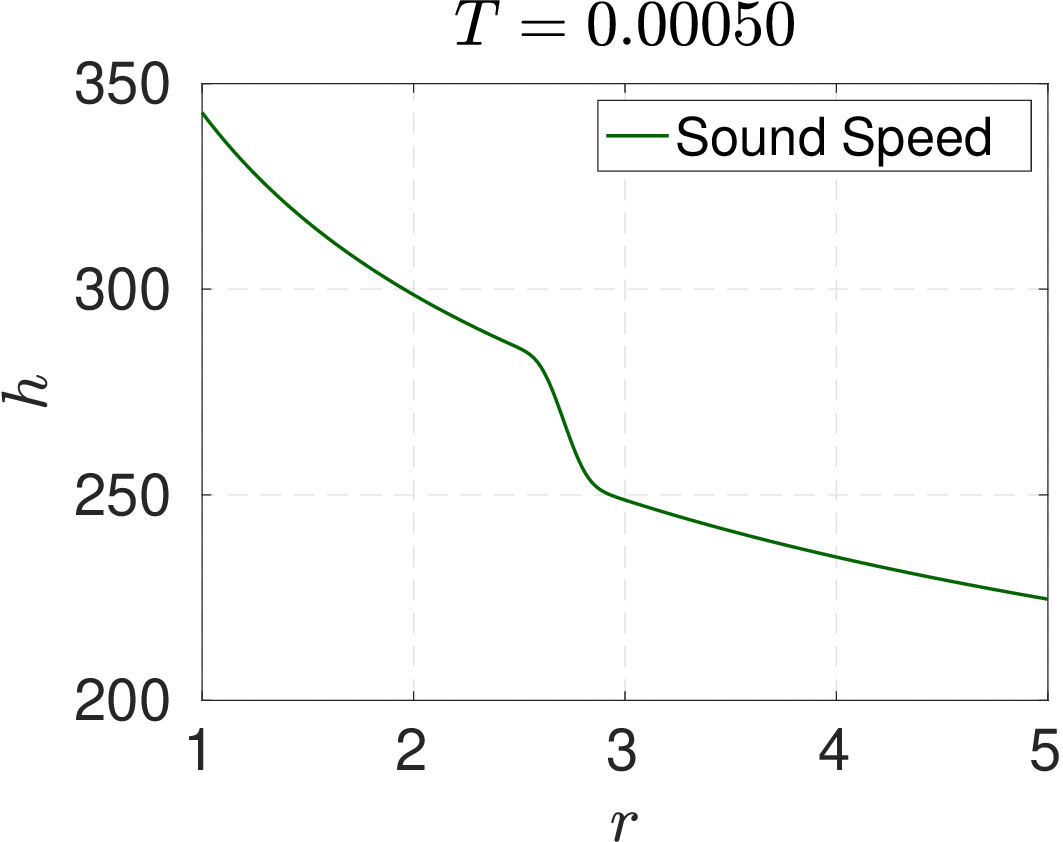}
  \end{subfigure}
  \begin{subfigure}{0.32\textwidth}
    \centering
    \includegraphics[width=1.0\textwidth]{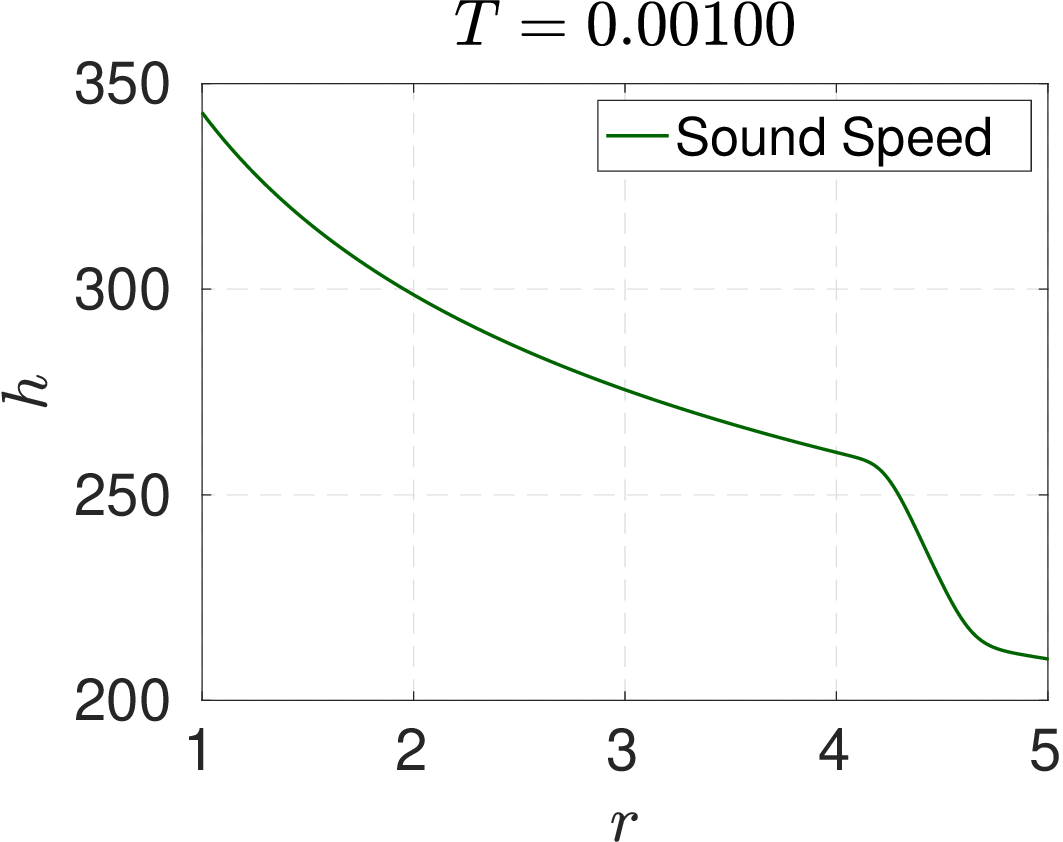}
  \end{subfigure}
  \caption{The initial \textit{sound speed} is shown on the left and its time-evolved state on the center and on the right.}
  \label{fig:sound-speed_case5}
\end{figure}

% ------------------------------------------------------------------------

\begin{figure}[H]
  \centering
  \begin{subfigure}{0.32\textwidth}
    \centering
    \includegraphics[width=1.0\textwidth]{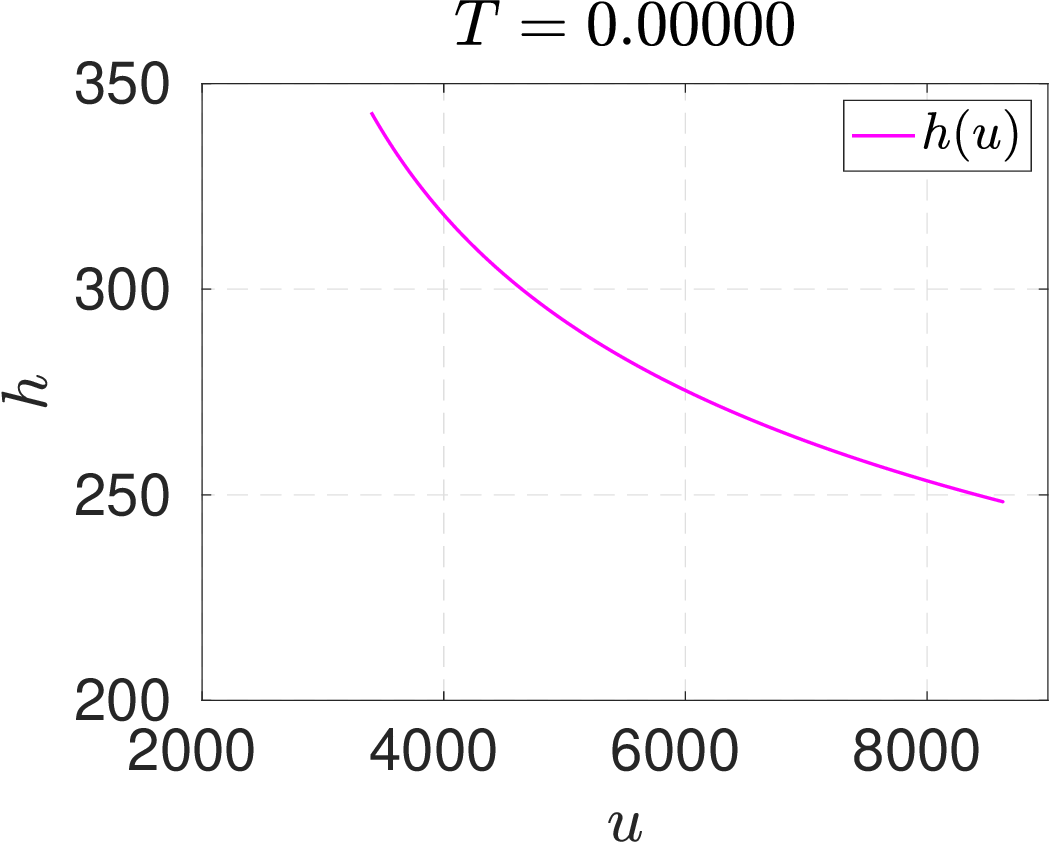}
  \end{subfigure}
    \begin{subfigure}{0.32\textwidth}
    \centering
    \includegraphics[width=1.0\textwidth]{new_cases/Case_5/graficos/state/mesh_8192/EPS_0001.eps}
  \end{subfigure}
  \begin{subfigure}{0.32\textwidth}
    \centering
    \includegraphics[width=1.0\textwidth]{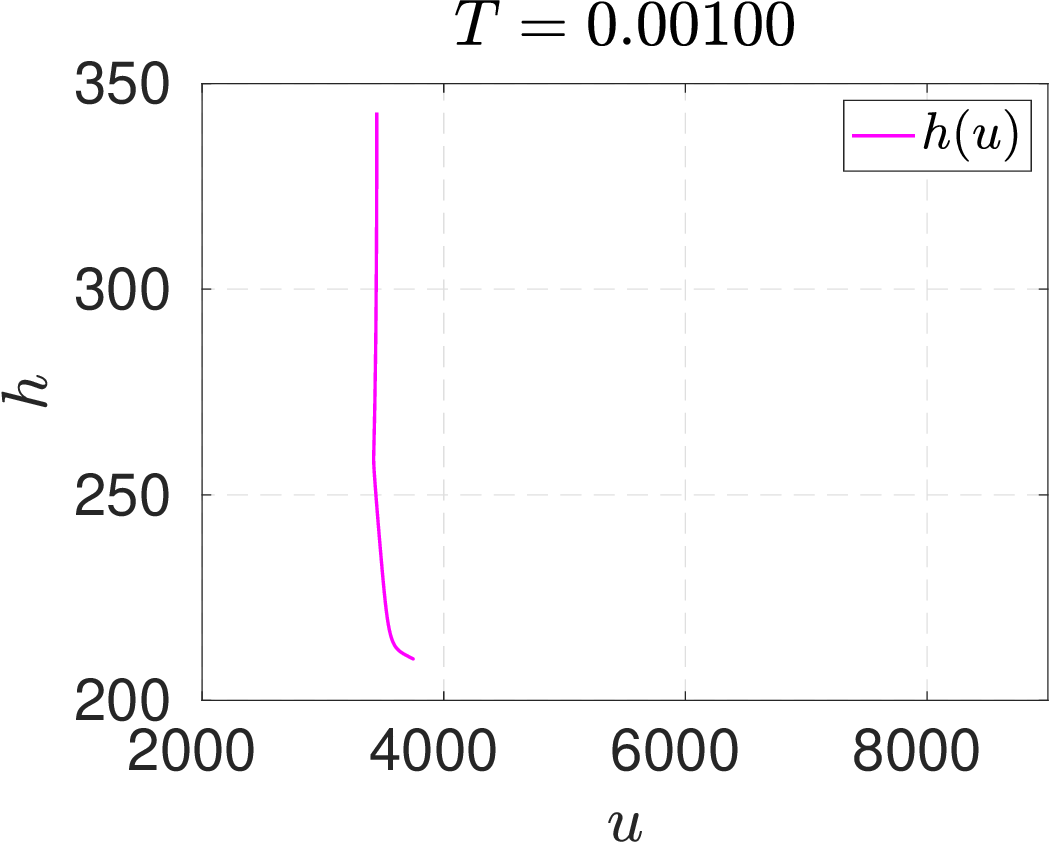}
  \end{subfigure}
  \caption{The initial \textit{invariant curve in \((u,h)-\)plane} is shown on the left and its time-evolved state on the center and on the right.}
  \label{fig:invariant-curve_case5}
\end{figure}

% ------------------------------------------------------------------------

\begin{figure}[H]
  \centering
  \begin{subfigure}{0.49\textwidth}
    \centering
    \includegraphics[width=1.0\textwidth]{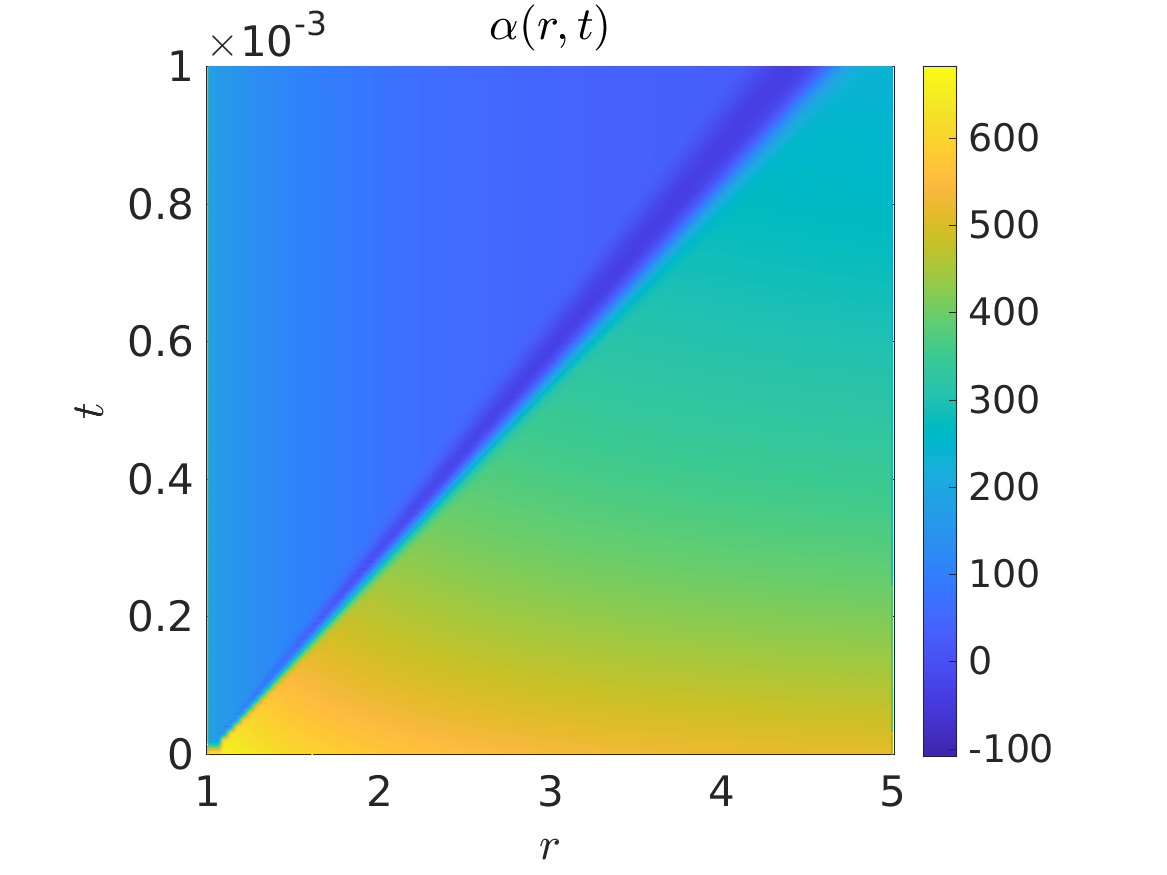}
  \end{subfigure}
  \begin{subfigure}{0.49\textwidth}
    \centering
	\includegraphics[width=1.0\textwidth]{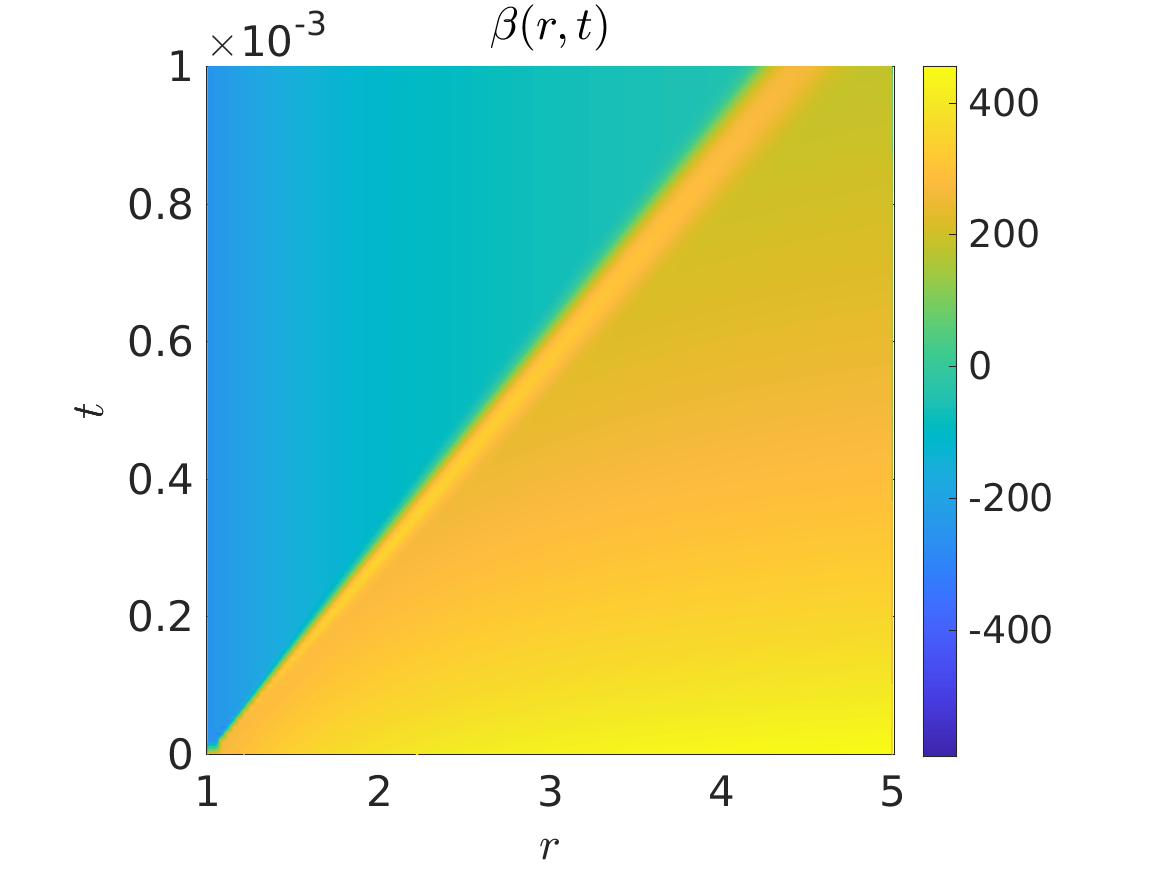}
  \end{subfigure}
  \caption{The \textit{heat map of $\alpha$ in \((r,t)-\)plane} is shown on the left and the heat map of $\beta$ on the right.}
  \label{fig:heat-map_case5}
\end{figure}

% ------------------------------------------------------------------------

\subsection*{Case 6: asymmetric implosion}

To investigate the dynamics of inward-directed transport dynamics, we consider a configuration characterized by a supersonic negative initial velocity field and mixed wave characters:
\[
K=7.75\times 10^{4},\ \ \gamma=1.4,\ \ h_{c}=343 \mathrm{m/s},\ \ v_a=-3400 \mathrm{m/s}, \ \ \alpha=1300, \ \ \beta=-1300,\ \ r\in[1,5].
\]
This initial state prescribes a flow field advecting toward the coordinate origin, inducing a complex asymmetric interplay between the rarefactive $(\alpha>0)$ and compressive $(\beta < 0)$ characteristic families. This setup drives the flow toward the center and produces an asymmetric interplay between compression and rarefaction. Indeed, the numerical solution exhibits inward displacement consistent with an implosive scenario; see Figures (\ref{fig:density_case6}-\ref{fig:invariant-curve_case6}).

The numerical evolution captures the resulting implosive scenario, where the inward displacement of the fluid mass is coupled with a geometric intensification of the gradients. The diagnostics presented in Figures (\ref{fig:density_case6}-\ref{fig:invariant-curve_case6}) reveal the persistence of this mixed-character structure, while the spatio-temporal heat maps (Figure \ref{fig:heat-map_case6}) delineate the distinct evolution of the gradient variables. This case highlights how the radial source terms modulate the stability of the wave families, demonstrating that the compressive character of $\beta$ is amplified by the geometric contraction as it approaches the center, while the rarefactive character of $\alpha$ acts to counteract total collapse.

% ------------------------------------------------------------------------

\begin{figure}[H]
  \centering
  \begin{subfigure}{0.32\textwidth}
    \centering
    \includegraphics[width=1.0\textwidth]{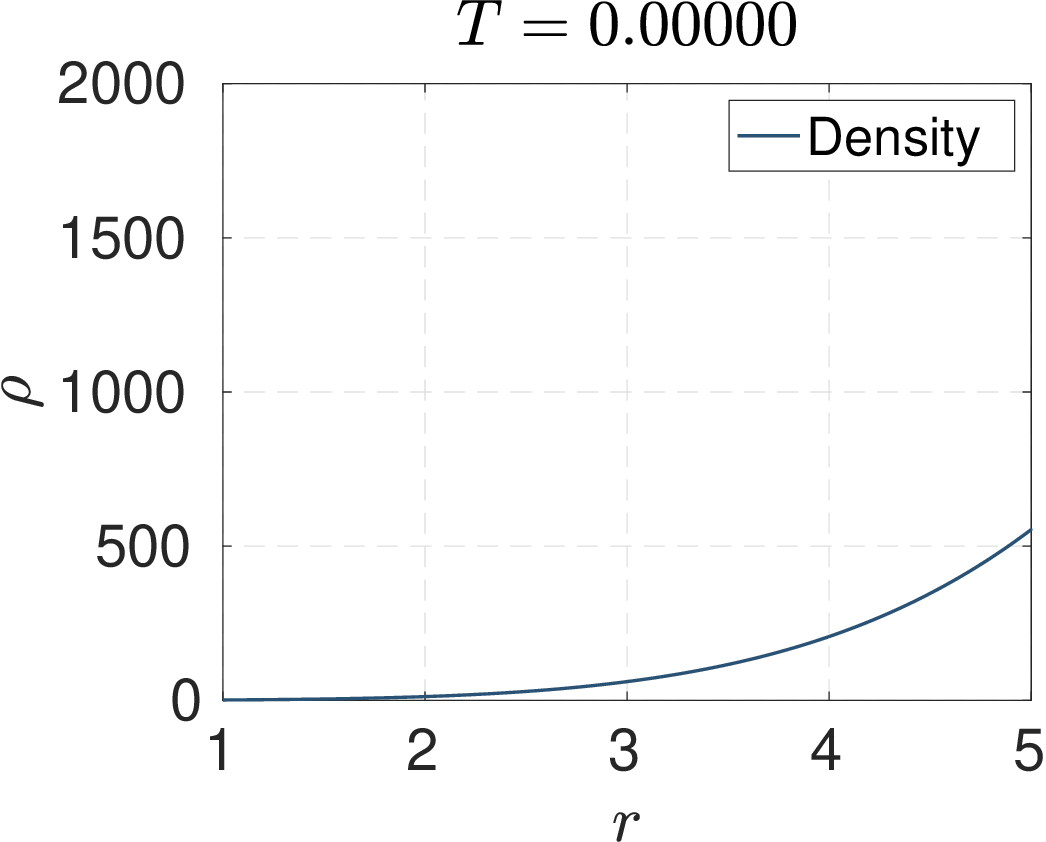}
  \end{subfigure}
    \begin{subfigure}{0.32\textwidth}
    \centering
    \includegraphics[width=1.0\textwidth]{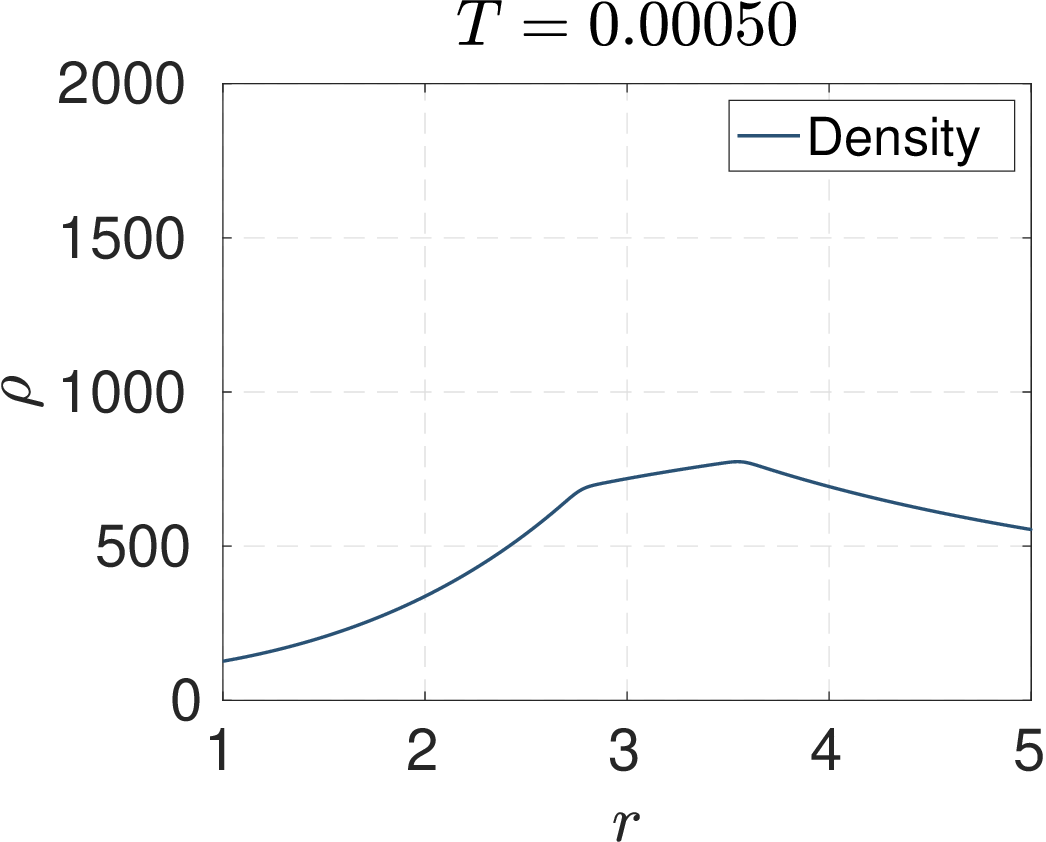}
  \end{subfigure}
  \begin{subfigure}{0.32\textwidth}
    \centering
    \includegraphics[width=1.0\textwidth]{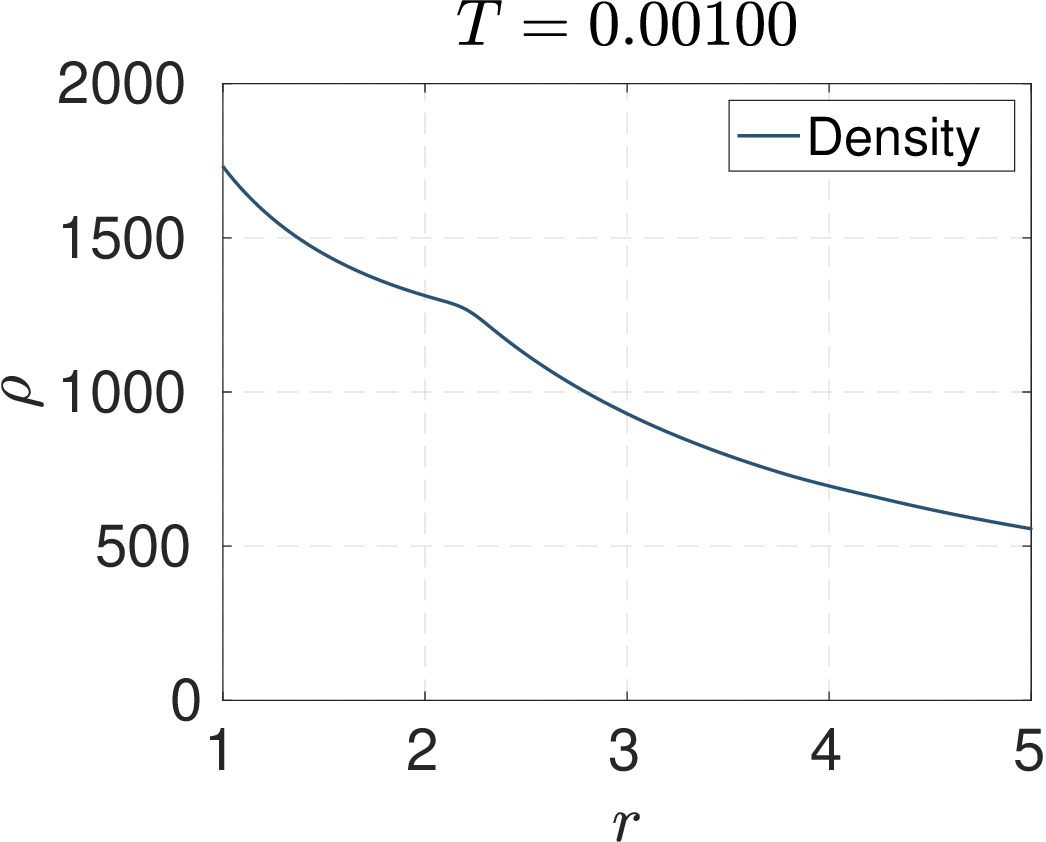}
  \end{subfigure}
  \caption{The initial \textit{density} is shown on the left and its time-evolved state on the center and on the right.}
  \label{fig:density_case6}
\end{figure}

% ------------------------------------------------------------------------

\begin{figure}[H]
  \centering
  \begin{subfigure}{0.32\textwidth}
    \centering
    \includegraphics[width=1.0\textwidth]{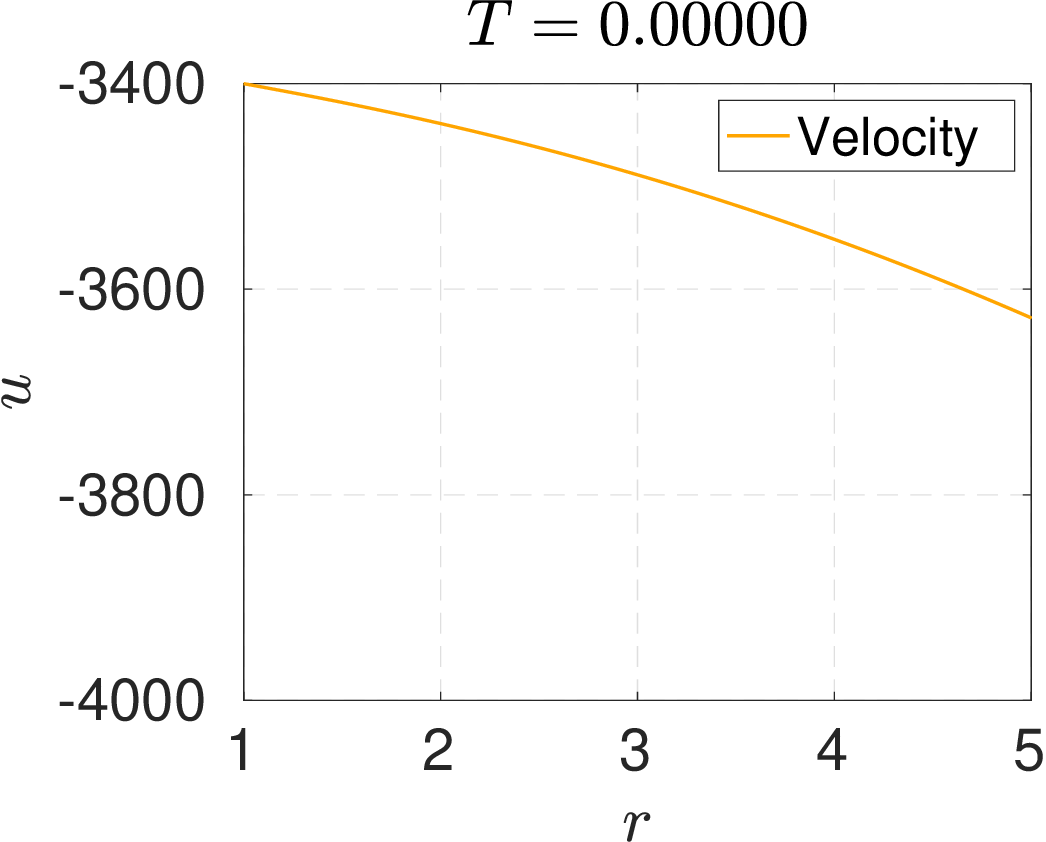}
  \end{subfigure}
    \begin{subfigure}{0.32\textwidth}
    \centering
    \includegraphics[width=1.0\textwidth]{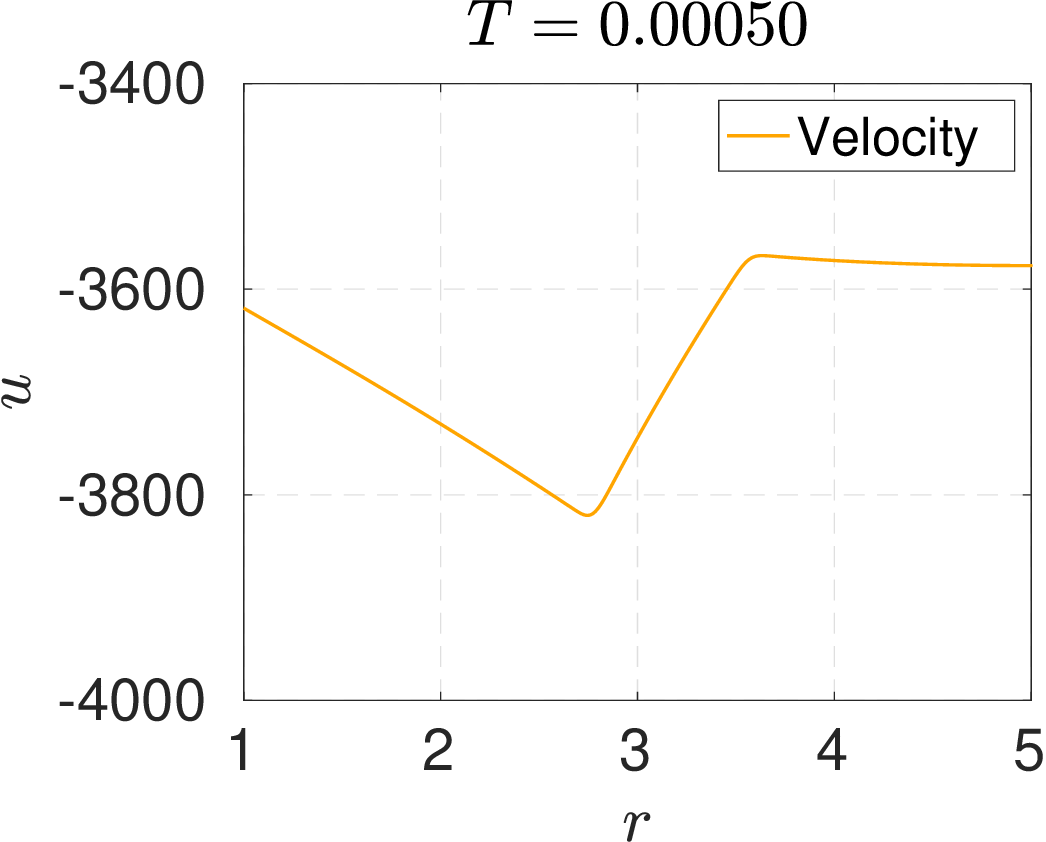}
  \end{subfigure}
  \begin{subfigure}{0.32\textwidth}
    \centering
    \includegraphics[width=1.0\textwidth]{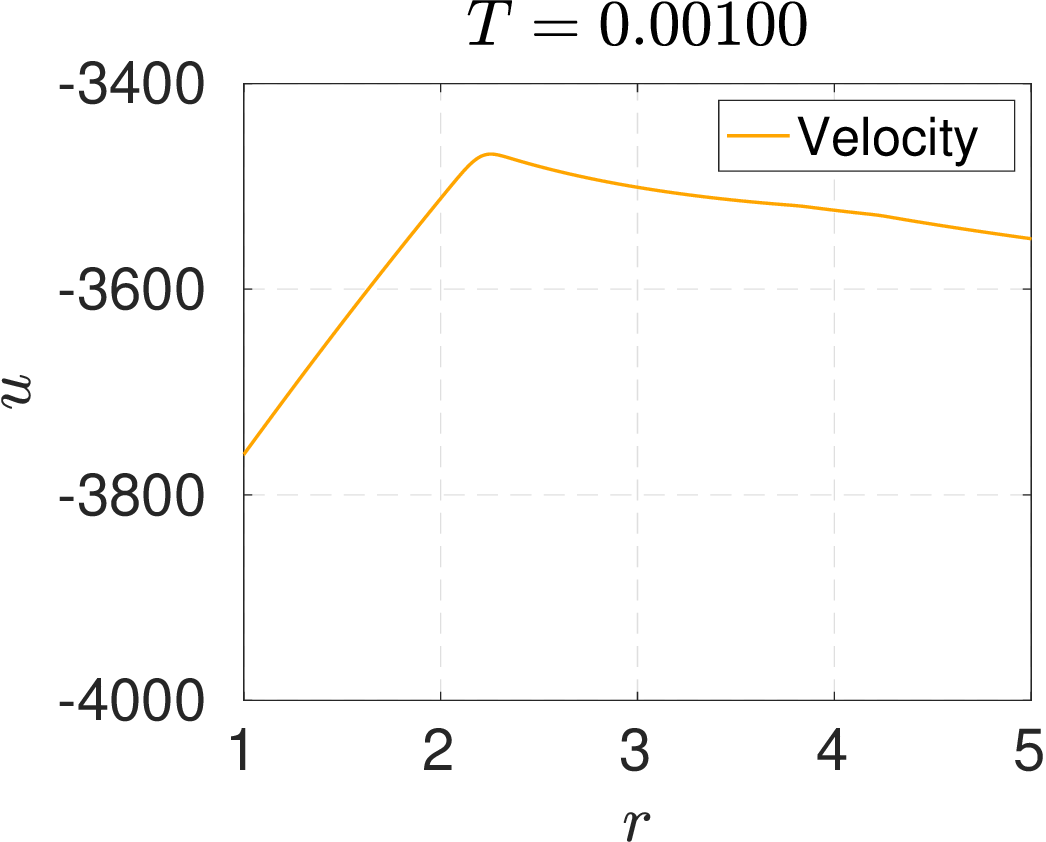}
  \end{subfigure}
  \caption{The initial \textit{velocity} is shown on the left and its time-evolved state on the center and on the right.}
  \label{fig:velocity_case6}
\end{figure}

% ------------------------------------------------------------------------

\begin{figure}[H]
  \centering
  \begin{subfigure}{0.32\textwidth}
    \centering
    \includegraphics[width=1.0\textwidth]{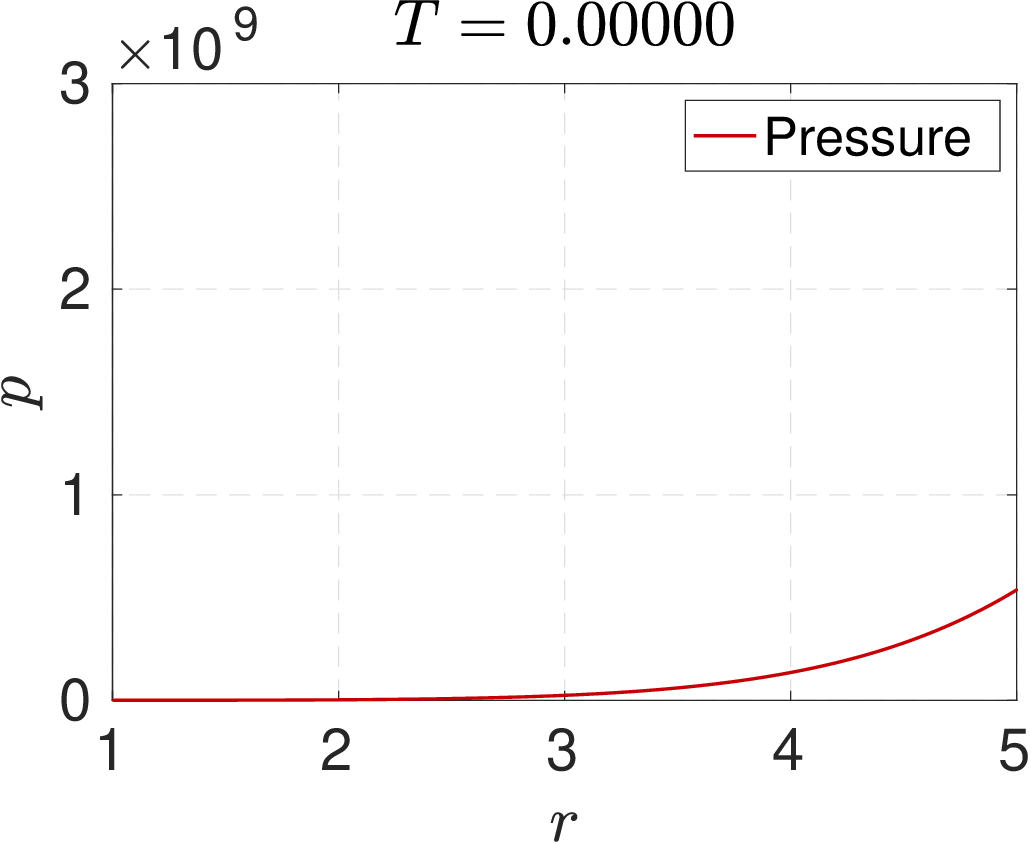}
  \end{subfigure}
    \begin{subfigure}{0.32\textwidth}
    \centering
    \includegraphics[width=1.0\textwidth]{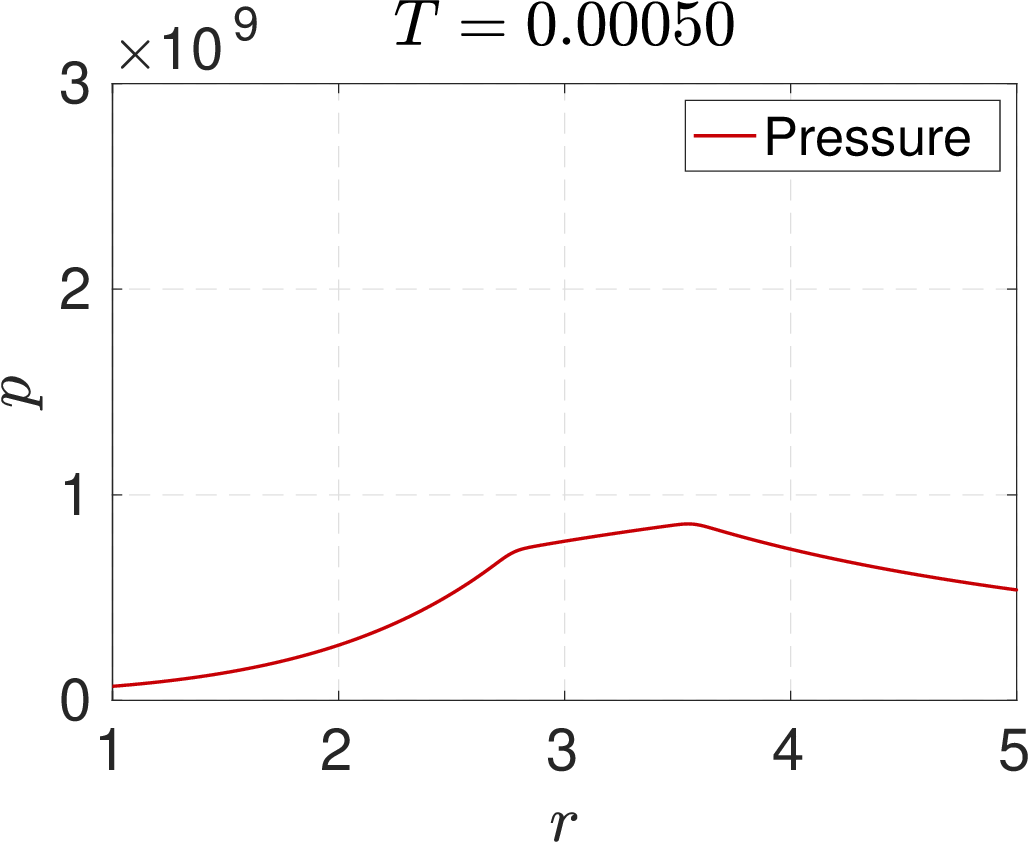}
  \end{subfigure}
  \begin{subfigure}{0.32\textwidth}
    \centering
    \includegraphics[width=1.0\textwidth]{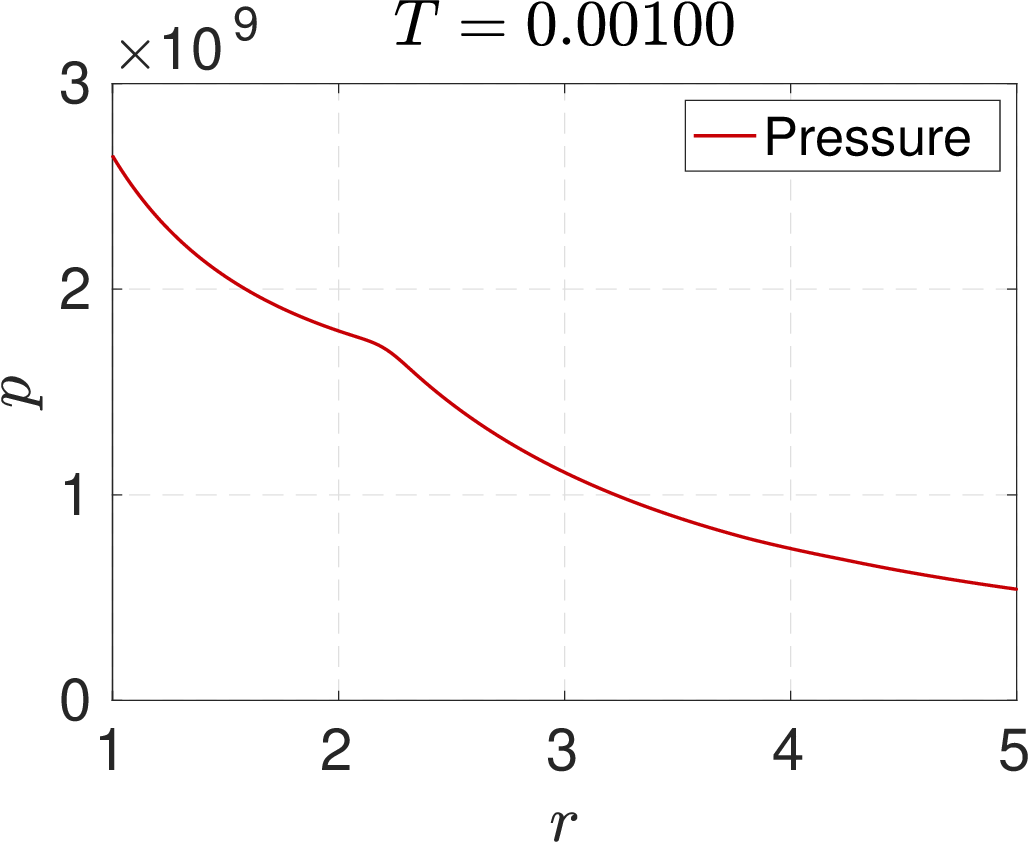}
  \end{subfigure}
  \caption{The initial \textit{pressure} is shown on the left and its time-evolved state on the center and on the right.}
  \label{fig:pressure_case6}
\end{figure}

% ------------------------------------------------------------------------

\begin{figure}[H]
  \centering
  \begin{subfigure}{0.32\textwidth}
    \centering
    \includegraphics[width=1.0\textwidth]{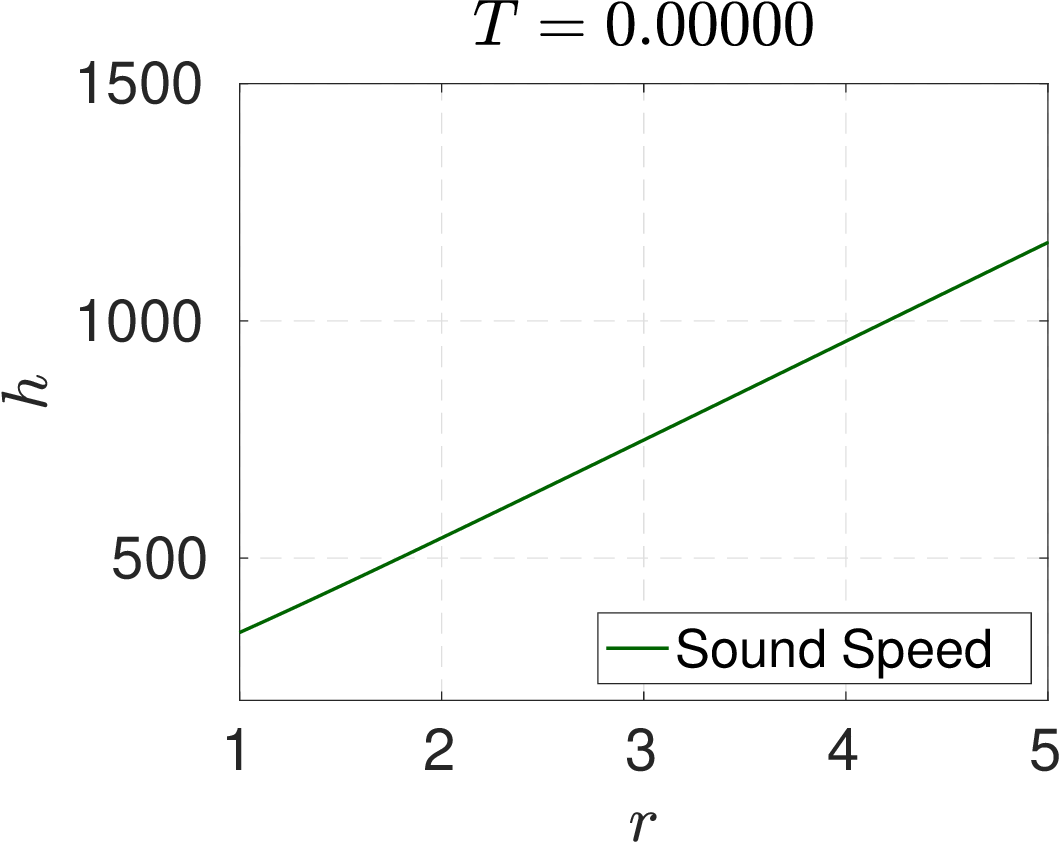}
  \end{subfigure}
    \begin{subfigure}{0.32\textwidth}
    \centering
    \includegraphics[width=1.0\textwidth]{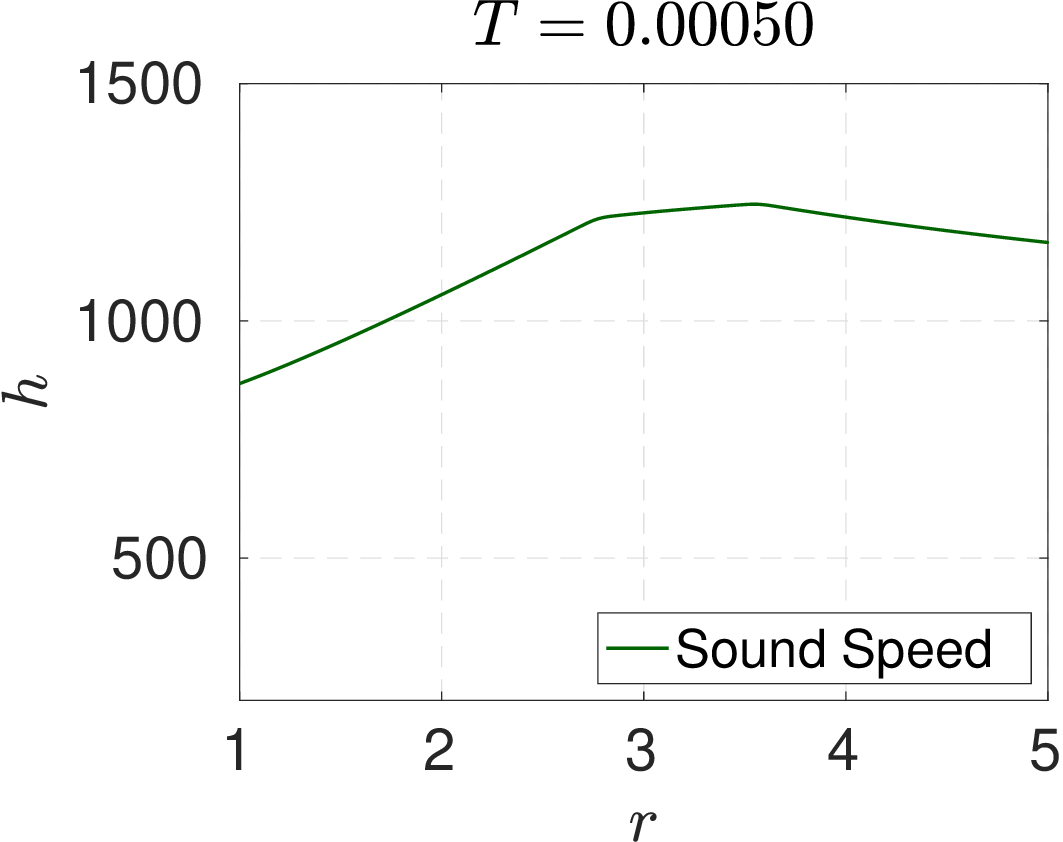}
  \end{subfigure}
  \begin{subfigure}{0.32\textwidth}
    \centering
    \includegraphics[width=1.0\textwidth]{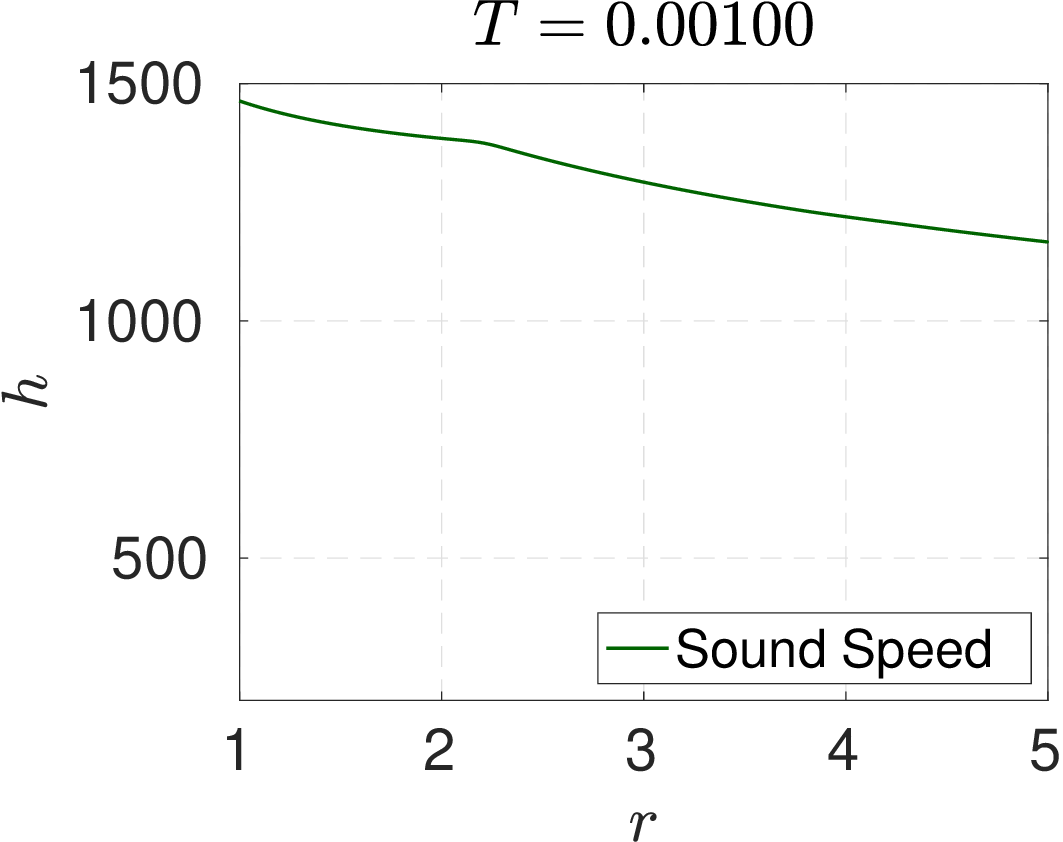}
  \end{subfigure}
  \caption{The initial \textit{sound speed} is shown on the left and its time-evolved state on the center and on the right.}
  \label{fig:sound-speed_case6}
\end{figure}

% ------------------------------------------------------------------------

\begin{figure}[H]
  \centering
  \begin{subfigure}{0.32\textwidth}
    \centering
    \includegraphics[width=1.0\textwidth]{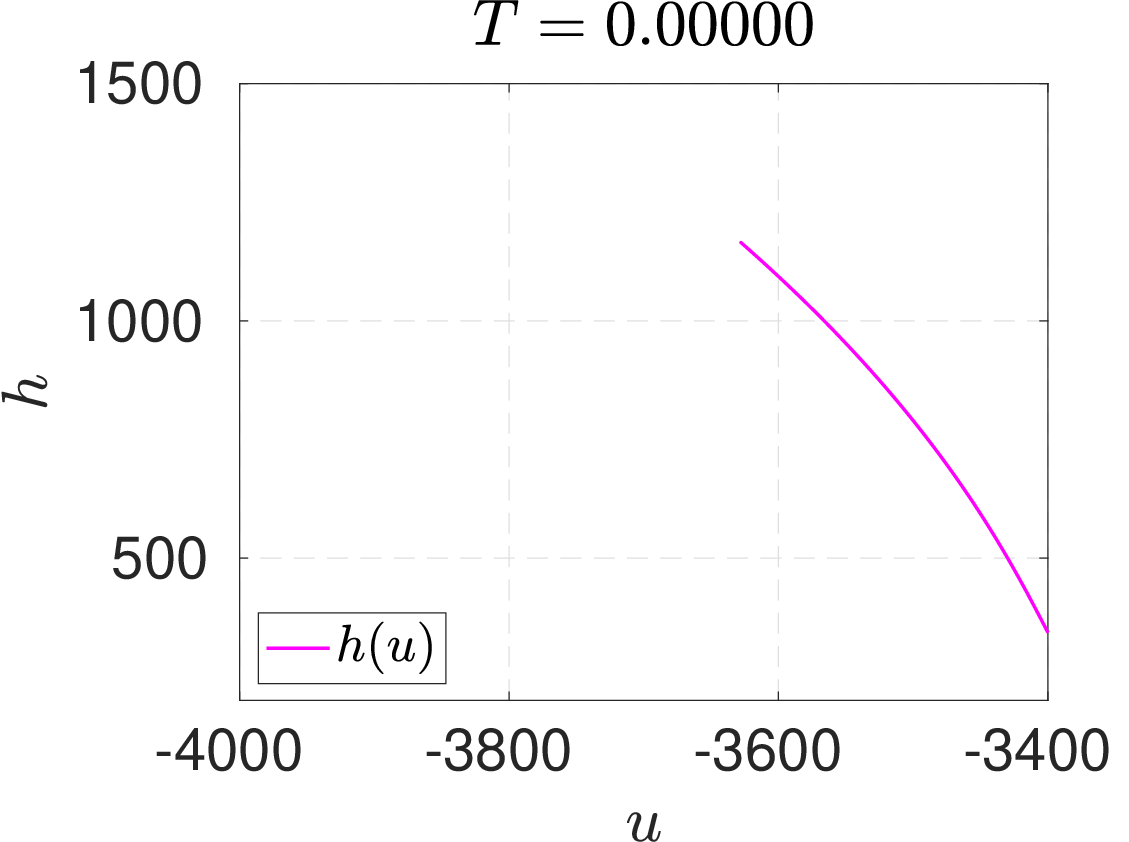}
  \end{subfigure}
    \begin{subfigure}{0.32\textwidth}
    \centering
    \includegraphics[width=1.0\textwidth]{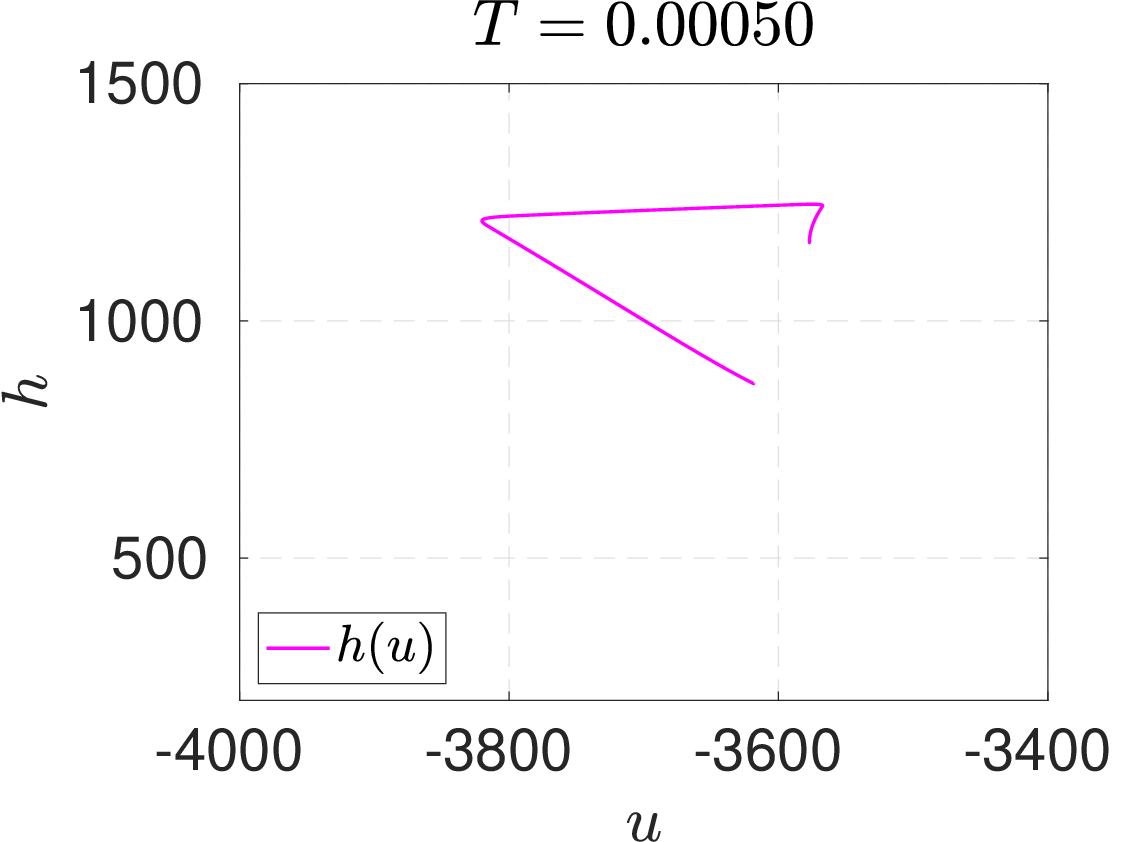}
  \end{subfigure}
  \begin{subfigure}{0.32\textwidth}
    \centering
    \includegraphics[width=1.0\textwidth]{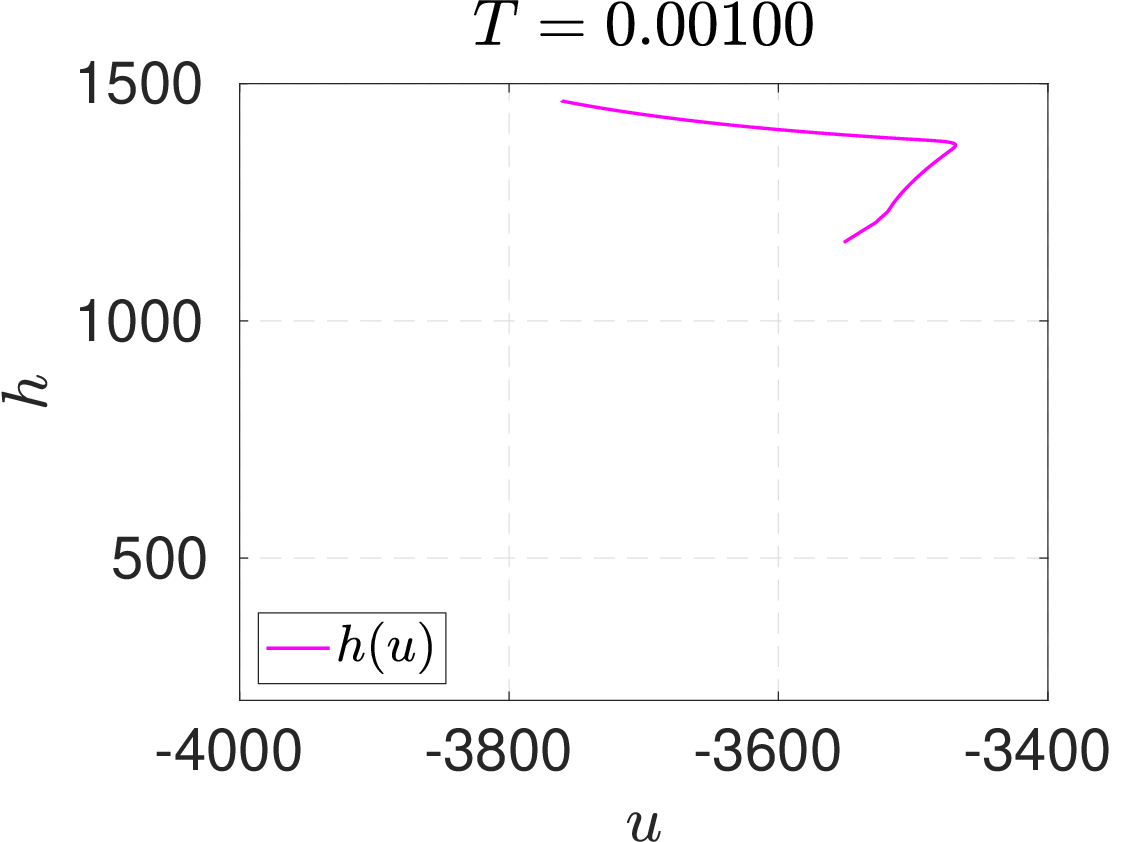}
  \end{subfigure}
  \caption{The initial \textit{invariant curve in \((u,h)-\)plane} is shown on the left and its time-evolved state on the center and on the right.}
  \label{fig:invariant-curve_case6}
\end{figure}

% ------------------------------------------------------------------------

\begin{figure}[H]
  \centering
  \begin{subfigure}{0.49\textwidth}
    \centering
    \includegraphics[width=1.0\textwidth]{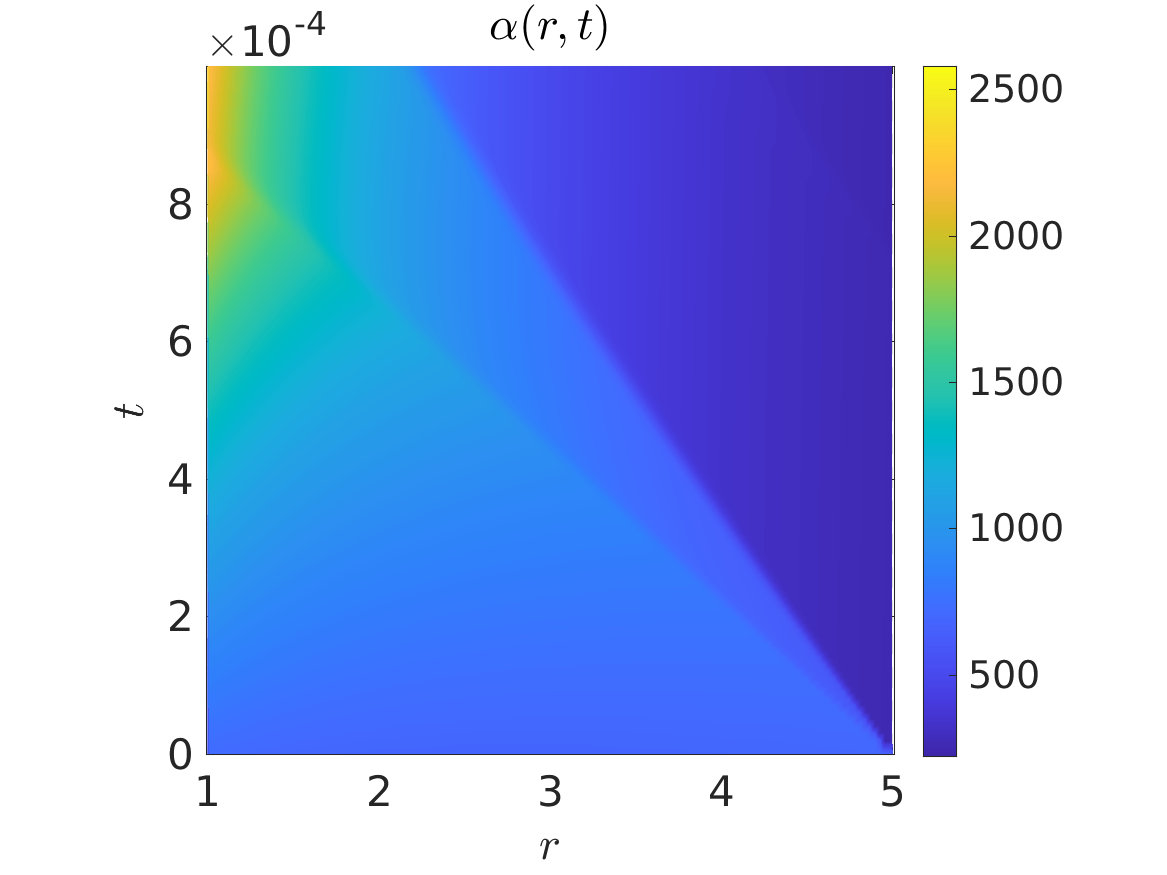}
  \end{subfigure}
  \begin{subfigure}{0.49\textwidth}
    \centering
	\includegraphics[width=1.0\textwidth]{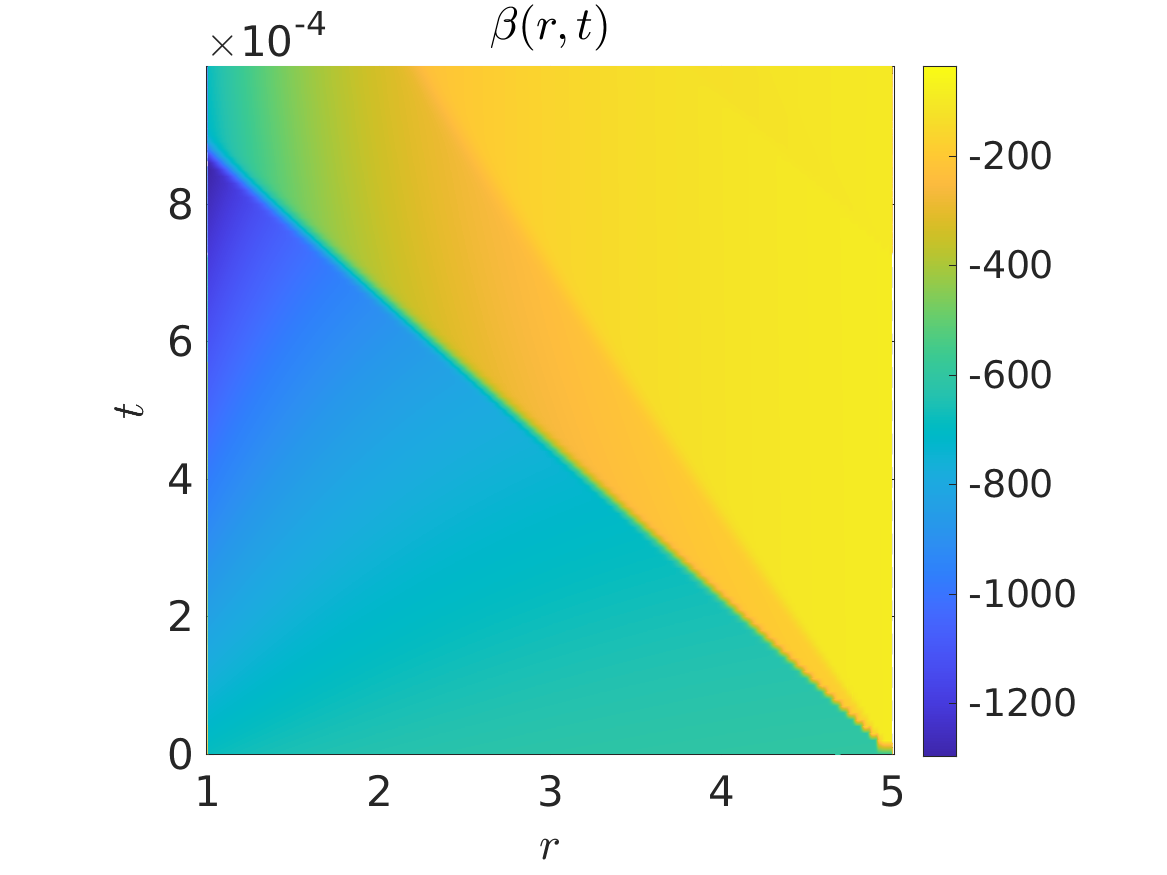}
  \end{subfigure}
  \caption{The \textit{heat map of $\alpha$ in \((r,t)-\)plane} is shown on the left and the heat map of $\beta$ on the right.}
  \label{fig:heat-map_case6}
\end{figure}

% ------------------------------------------------------------------------

\subsection*{Case 7: rarefactive implosion with sign change in $u$}

Finally, we examine an inward-directed rarefactive configuration characterized by,
\[
\alpha=\beta=1300,
\]
maintaining parametric consistency with Case 6. This configuration is distinguished by a velocity field that intersects the stagnant state $u=0$, inducing a topological bifurcation in the advective transport: the subdomain where $u>0$ is advected toward the coordinate origin, while the $u<0$ region propagates outward.

This case serves as a critical numerical probe into the persistence of regularity. Specifically, we investigate whether the confluence of geometric source terms and the bidirectional interaction of inward/outward characteristic families can precipitate a gradient catastrophe despite the strictly rarefactive initial character $(\alpha,\beta)$. Such a transition would highlight a fundamental departure from the one-dimensional Cartesian setting, where rarefactive waves typically preclude shock formation. The spatio-temporal evolution and phase-space diagnostics for the state variables $(\rho,u,p,h)$ and the gradient variables $\alpha,\beta)$, illustrated in figures \ref{fig:density_case7}--\ref{fig:heat-map_case7}, elucidate the nonlinear interplay between radial geometry and wave-character stability.
% ------------------------------------------------------------------------

\begin{figure}[H]
  \centering
  \begin{subfigure}{0.32\textwidth}
    \centering
    \includegraphics[width=1.0\textwidth]{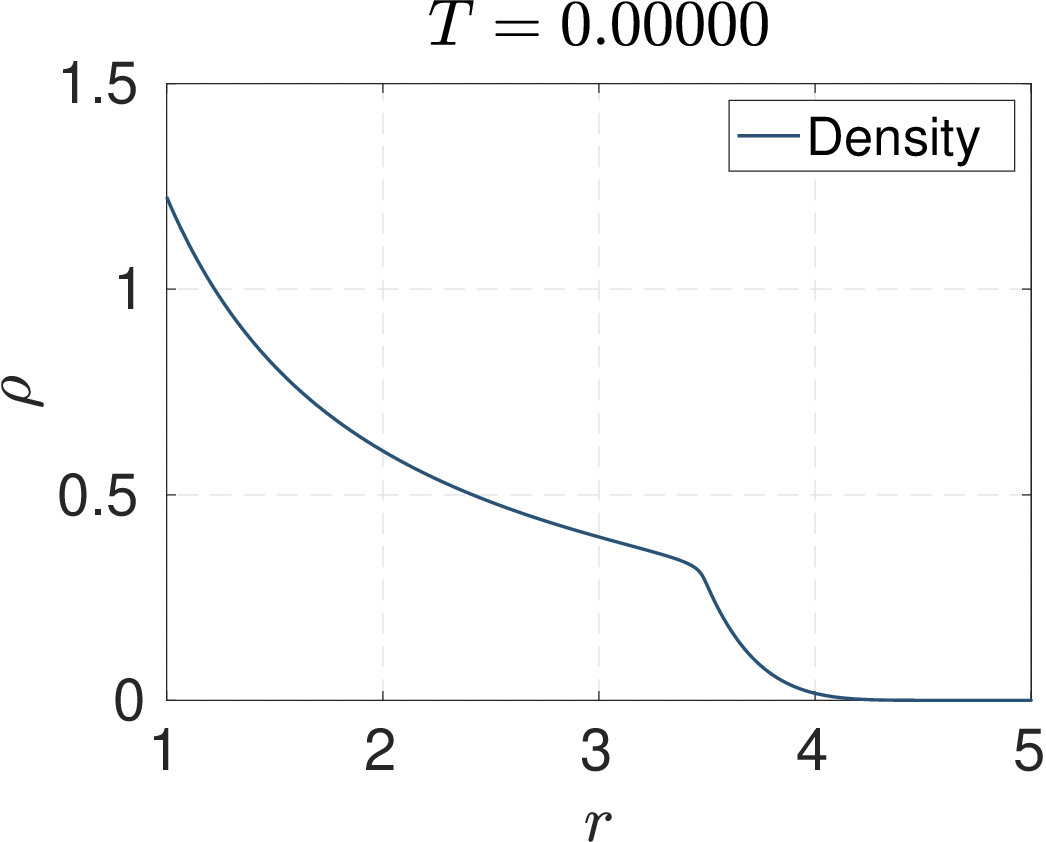}
  \end{subfigure}
    \begin{subfigure}{0.32\textwidth}
    \centering
    \includegraphics[width=1.0\textwidth]{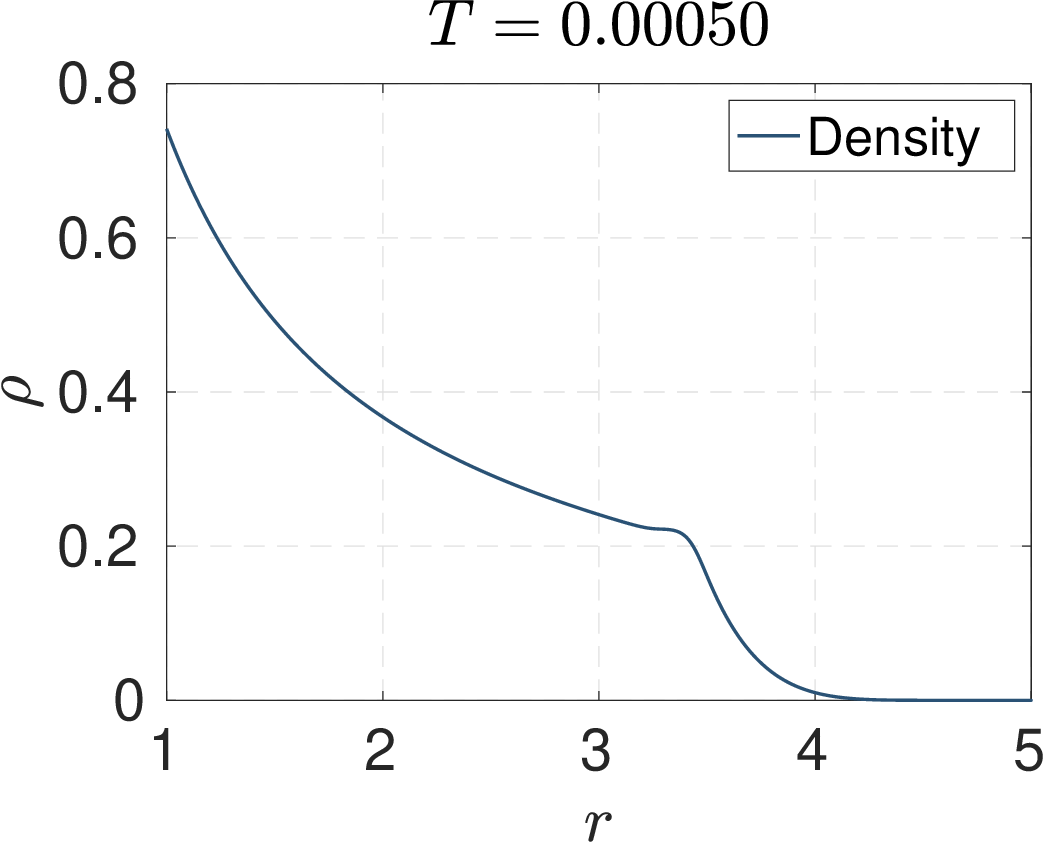}
  \end{subfigure}
  \begin{subfigure}{0.32\textwidth}
    \centering
    \includegraphics[width=1.0\textwidth]{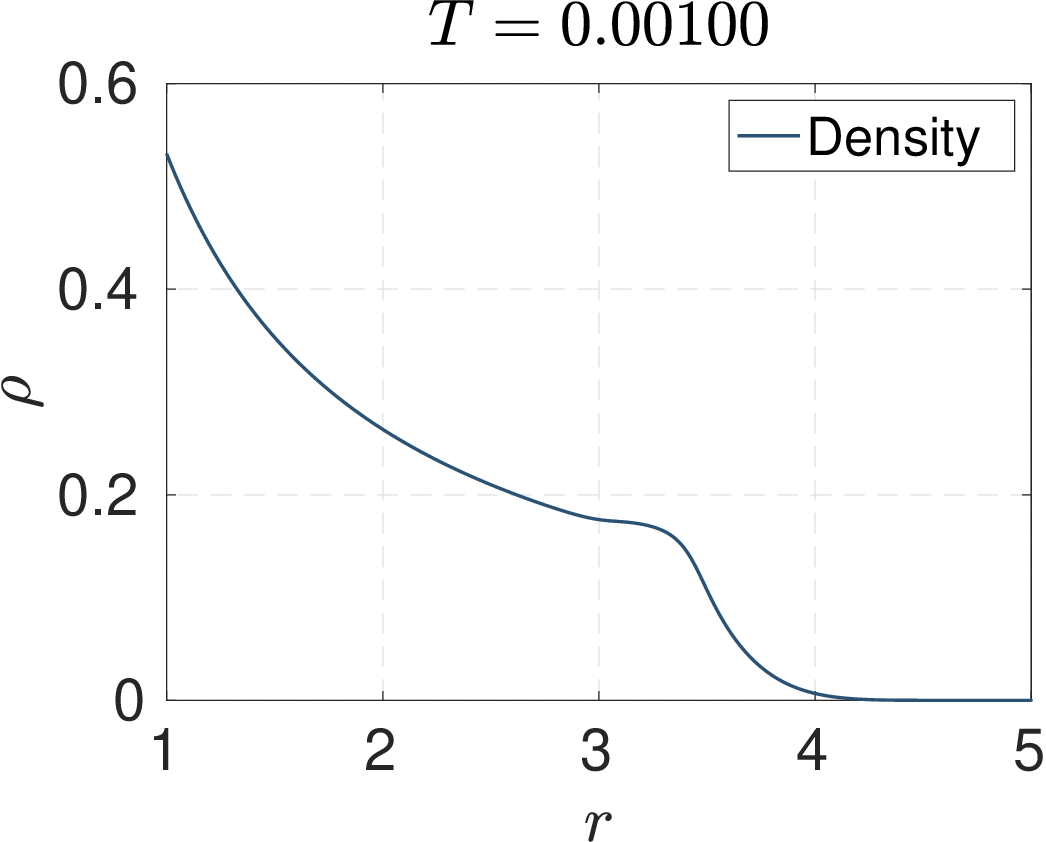}
  \end{subfigure}
  \caption{The initial \textit{density} is shown on the left and its time-evolved state on the center and on the right.}
  \label{fig:density_case7}
\end{figure}

% ------------------------------------------------------------------------

\begin{figure}[H]
  \centering
  \begin{subfigure}{0.32\textwidth}
    \centering
    \includegraphics[width=1.0\textwidth]{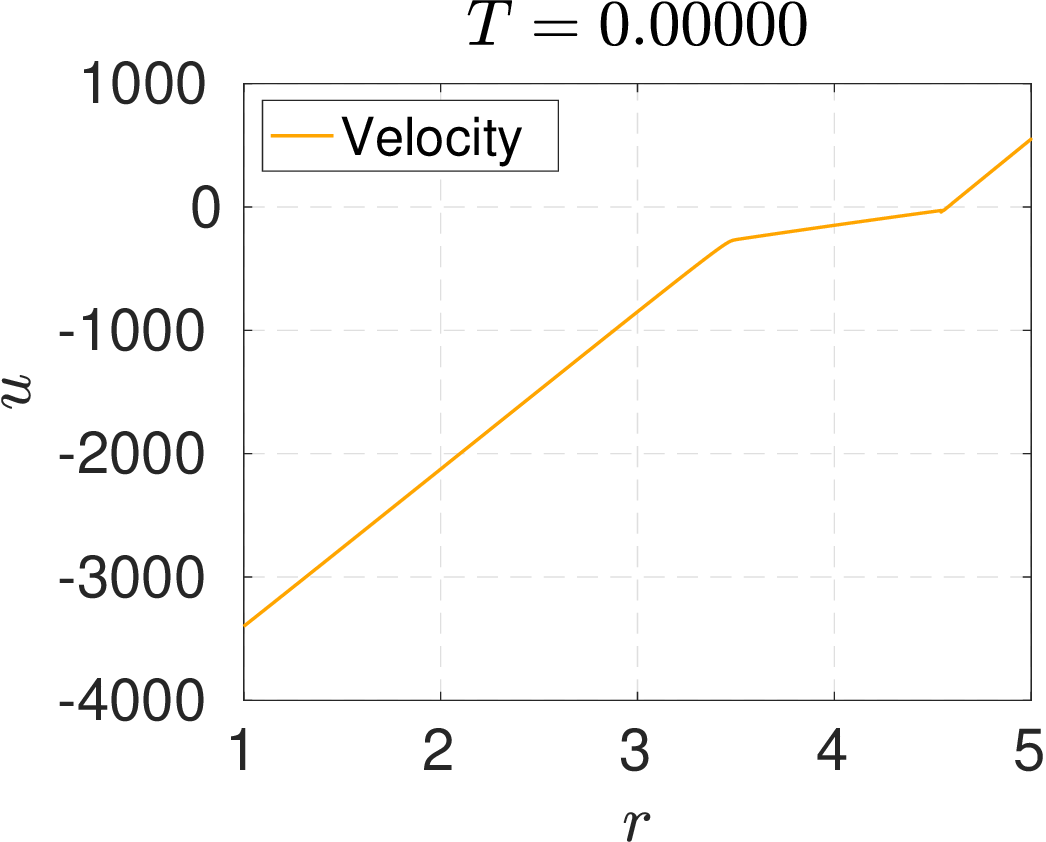}
  \end{subfigure}
    \begin{subfigure}{0.32\textwidth}
    \centering
    \includegraphics[width=1.0\textwidth]{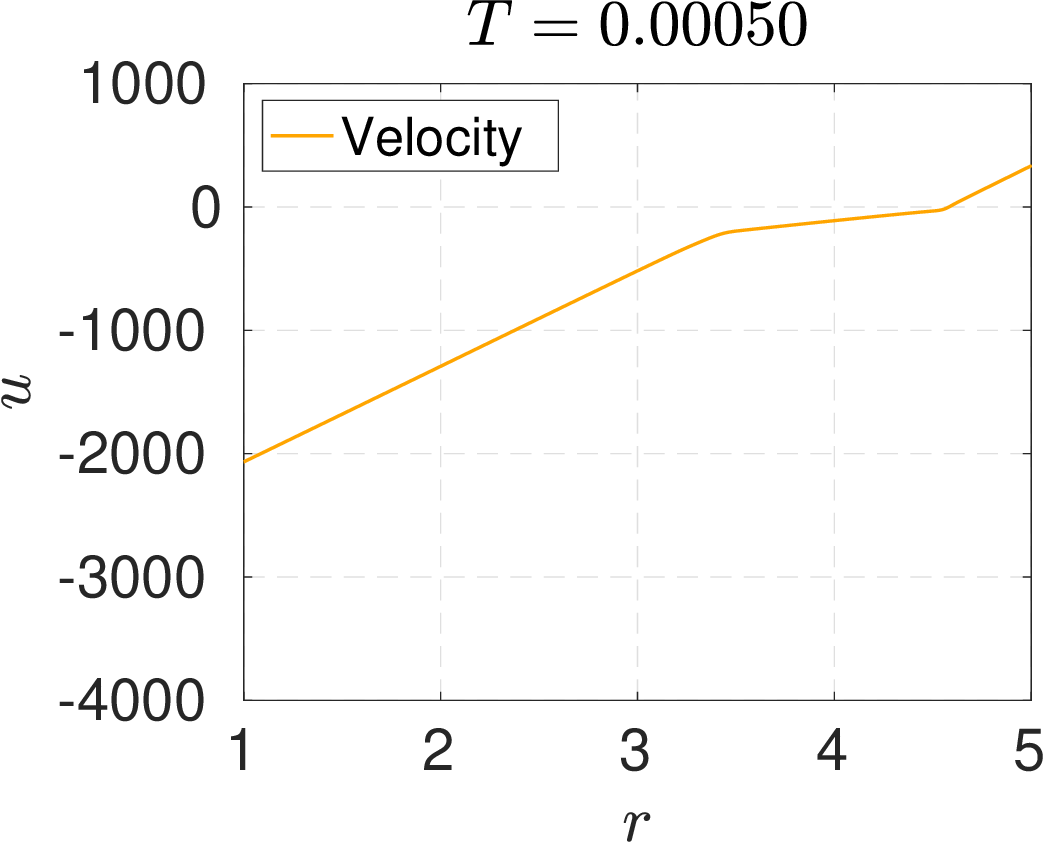}
  \end{subfigure}
  \begin{subfigure}{0.32\textwidth}
    \centering
    \includegraphics[width=1.0\textwidth]{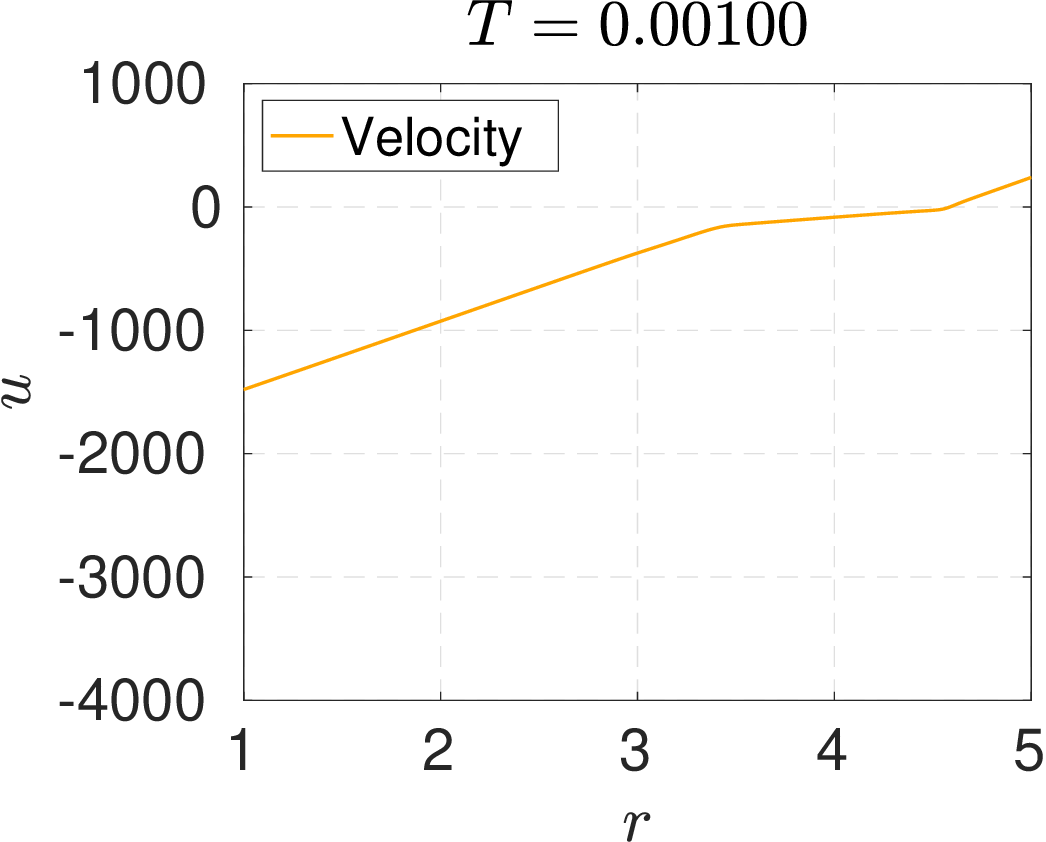}
  \end{subfigure}
  \caption{The initial \textit{velocity} is shown on the left and its time-evolved state on the center and on the right.}
  \label{fig:velocity_case7}
\end{figure}

% ------------------------------------------------------------------------

\begin{figure}[H]
  \centering
  \begin{subfigure}{0.32\textwidth}
    \centering
    \includegraphics[width=1.0\textwidth]{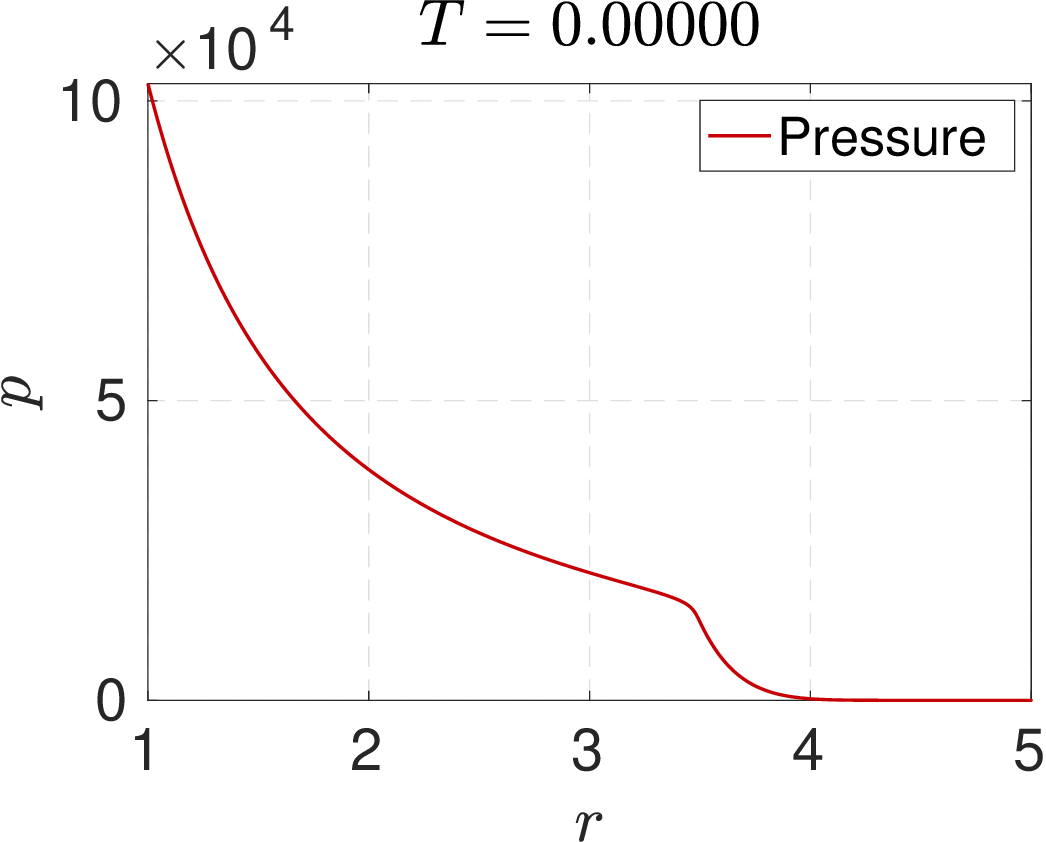}
  \end{subfigure}
    \begin{subfigure}{0.32\textwidth}
    \centering
    \includegraphics[width=1.0\textwidth]{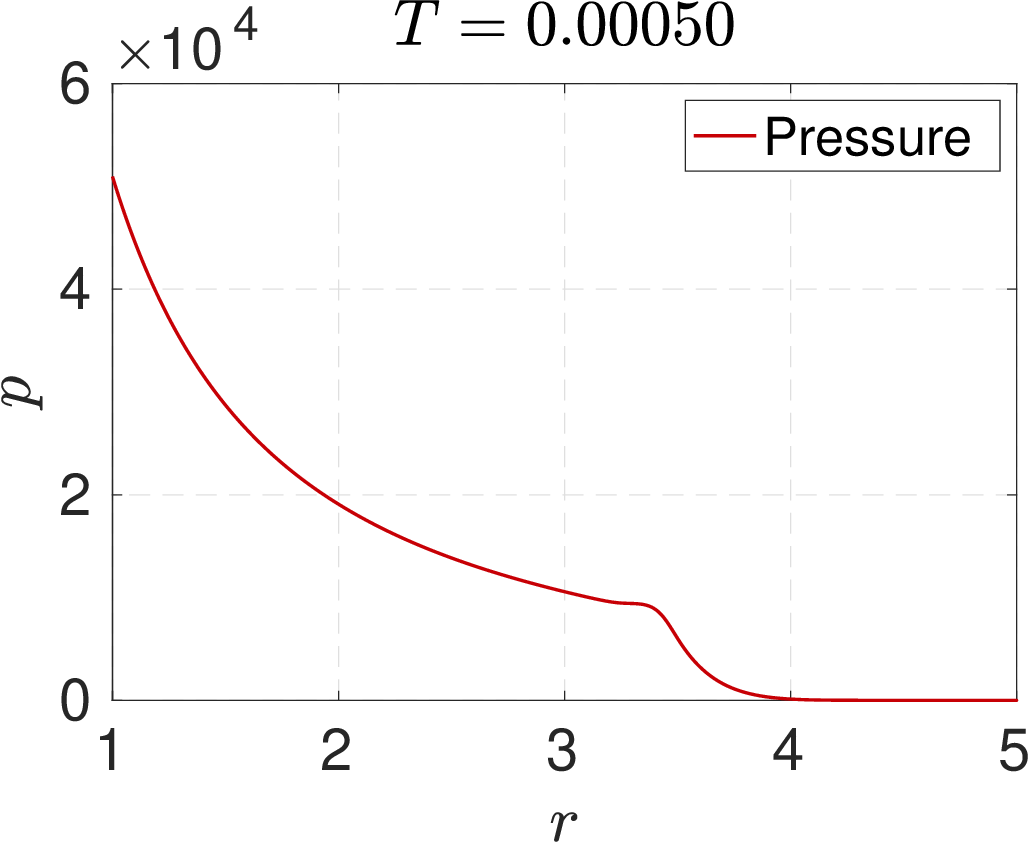}
  \end{subfigure}
  \begin{subfigure}{0.32\textwidth}
    \centering
    \includegraphics[width=1.0\textwidth]{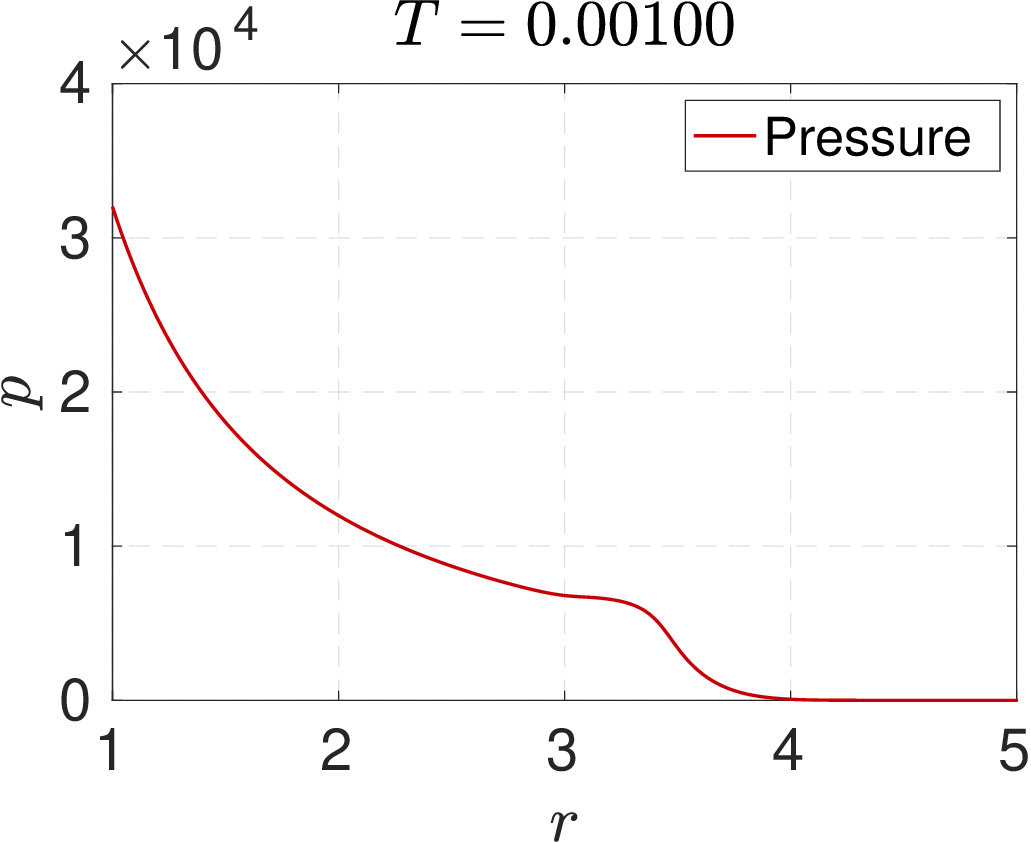}
  \end{subfigure}
  \caption{The initial \textit{pressure} is shown on the left and its time-evolved state on the center and on the right.}
  \label{fig:pressure_case7}
\end{figure}

% ------------------------------------------------------------------------

\begin{figure}[H]
  \centering
  \begin{subfigure}{0.32\textwidth}
    \centering
    \includegraphics[width=1.0\textwidth]{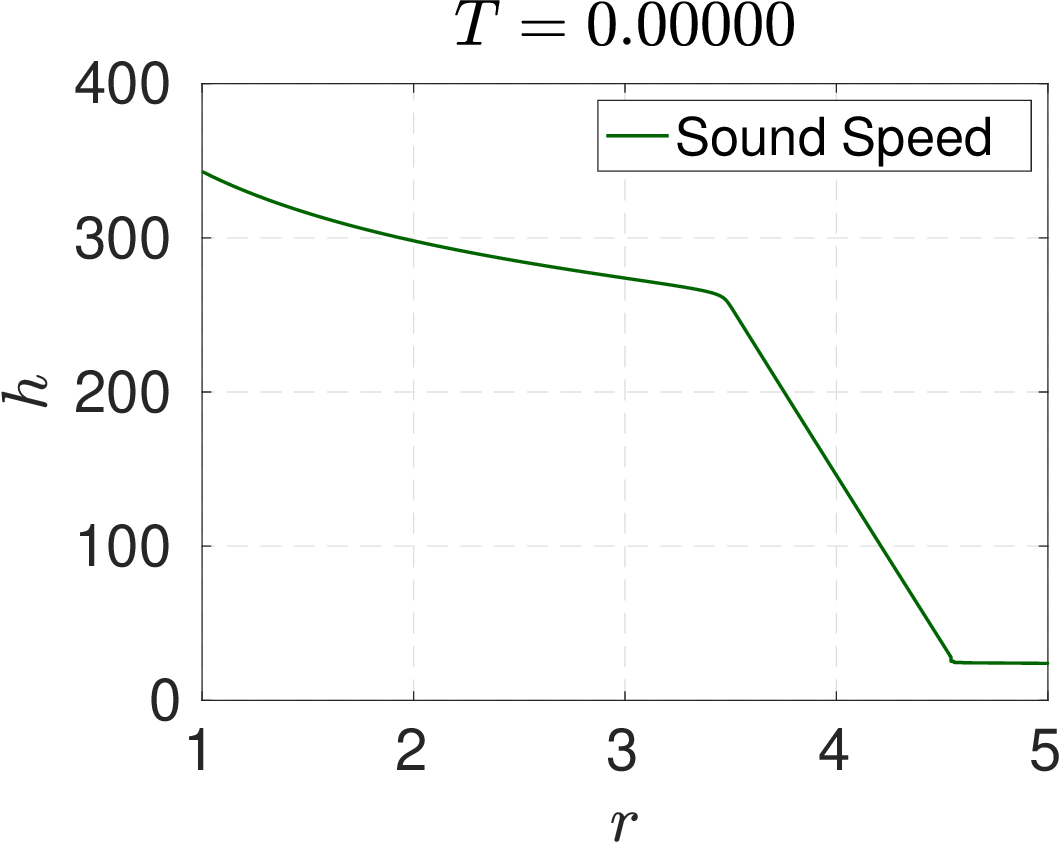}
  \end{subfigure}
    \begin{subfigure}{0.32\textwidth}
    \centering
    \includegraphics[width=1.0\textwidth]{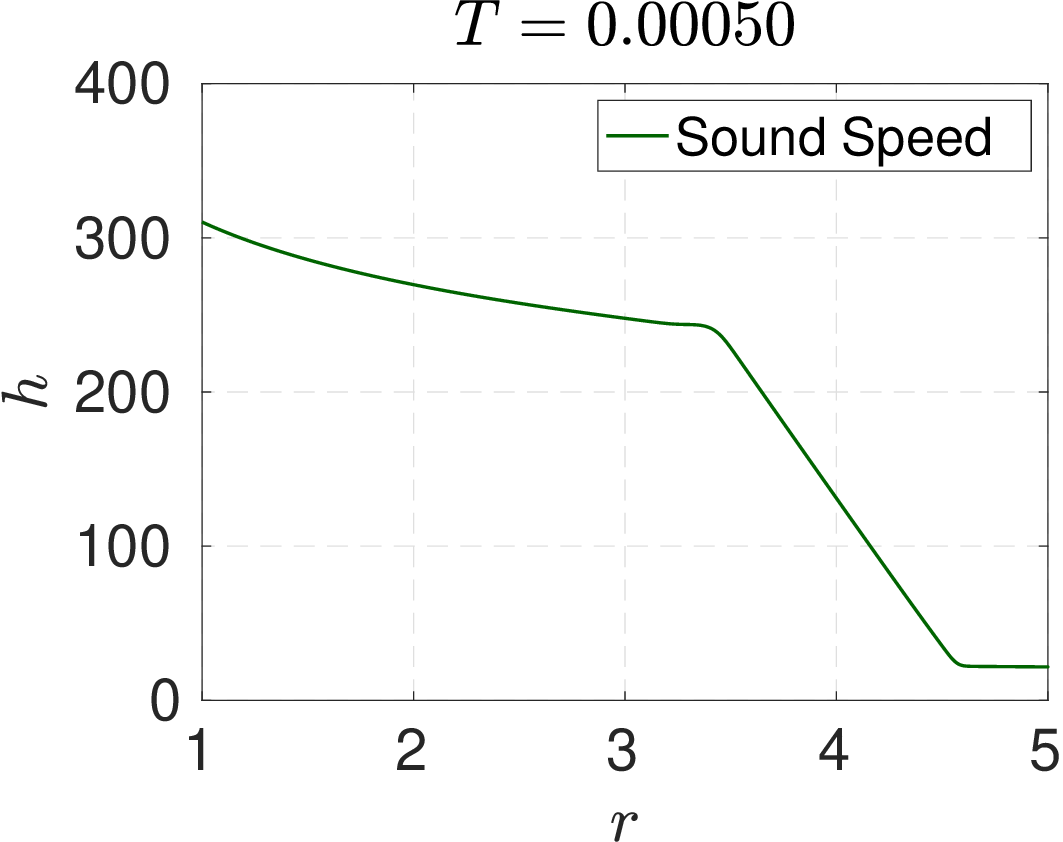}
  \end{subfigure}
  \begin{subfigure}{0.32\textwidth}
    \centering
    \includegraphics[width=1.0\textwidth]{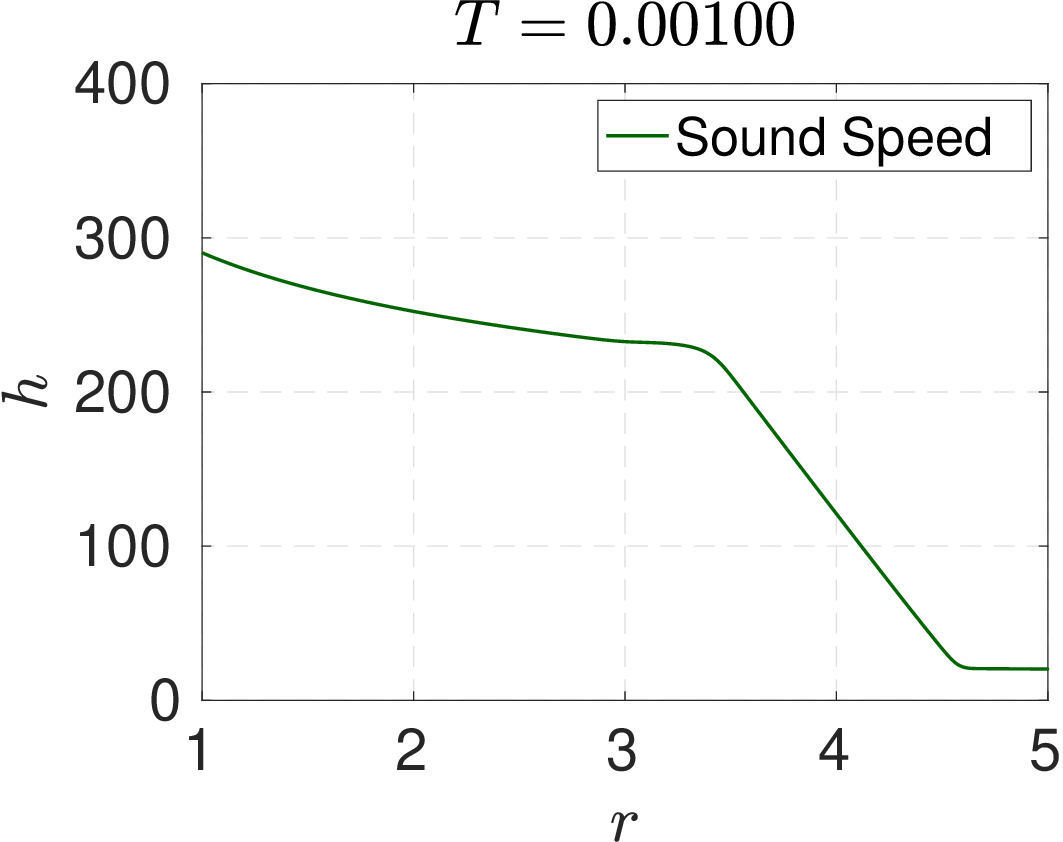}
  \end{subfigure}
  \caption{The initial \textit{sound speed} is shown on the left and its time-evolved state on the center and on the right.}
  \label{fig:sound-speed_case7}
\end{figure}

% ------------------------------------------------------------------------

\begin{figure}[H]
  \centering
  \begin{subfigure}{0.32\textwidth}
    \centering
    \includegraphics[width=1.0\textwidth]{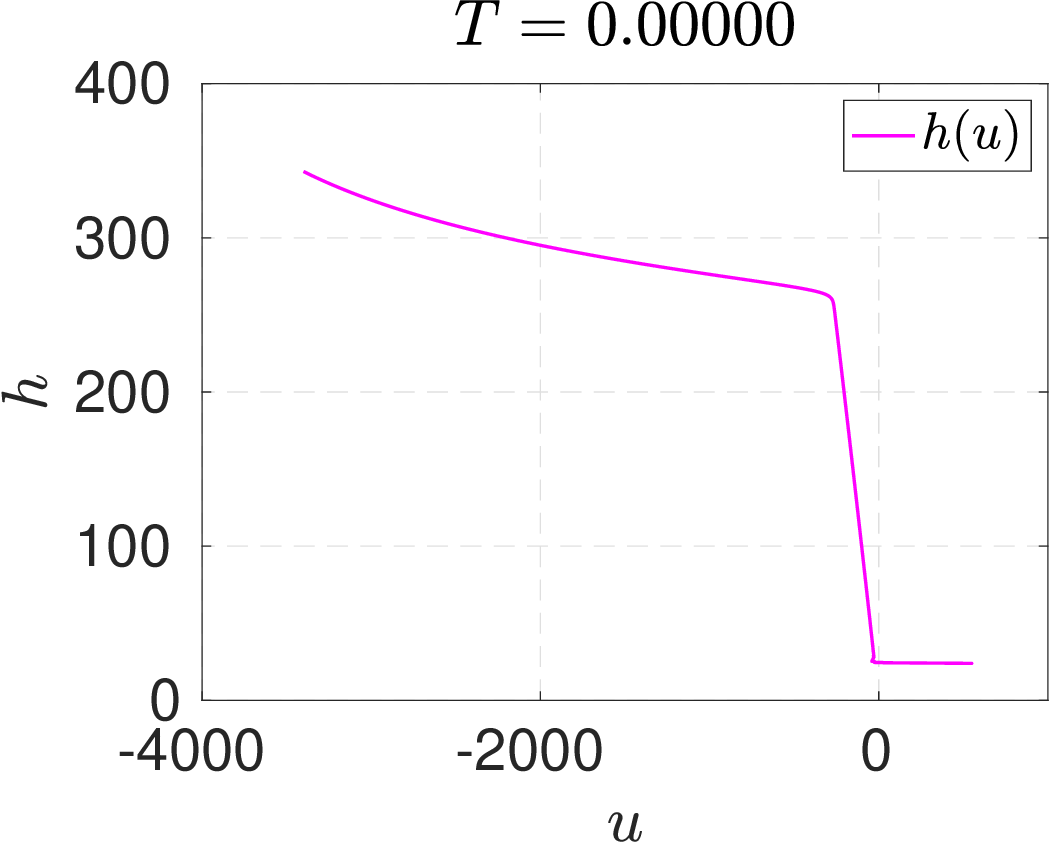}
  \end{subfigure}
    \begin{subfigure}{0.32\textwidth}
    \centering
    \includegraphics[width=1.0\textwidth]{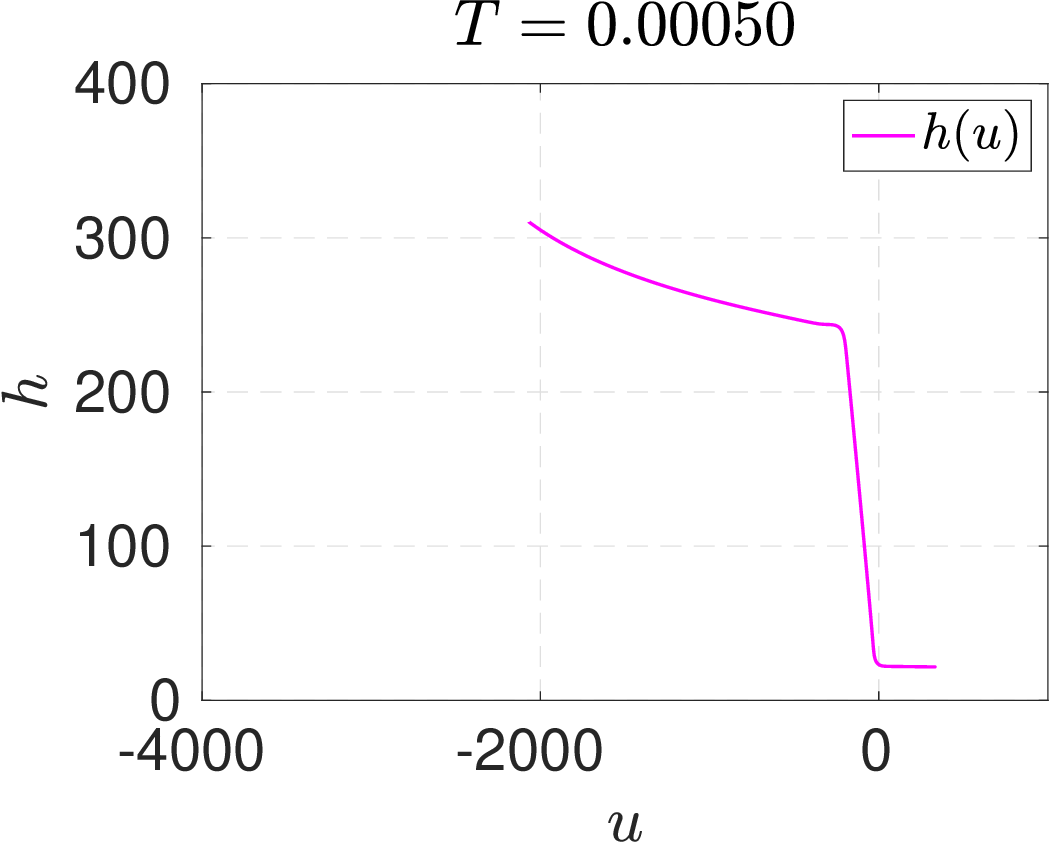}
  \end{subfigure}
  \begin{subfigure}{0.32\textwidth}
    \centering
    \includegraphics[width=1.0\textwidth]{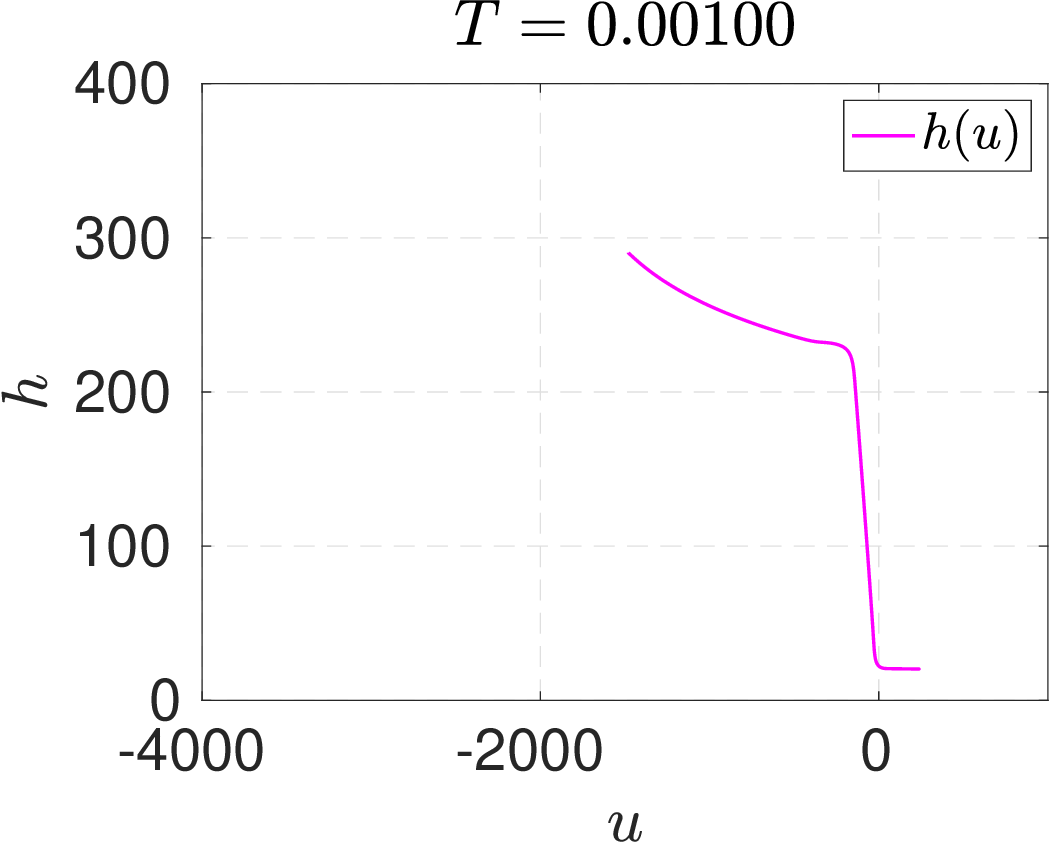}
  \end{subfigure}
  \caption{The initial \textit{invariant curve in \((u,h)-\)plane} is shown on the left and its time-evolved state on the center and on the right.}
  \label{fig:invariant-curve_case7}
\end{figure}

% ------------------------------------------------------------------------

\begin{figure}[H]
  \centering
  \begin{subfigure}{0.49\textwidth}
    \centering
    \includegraphics[width=1.0\textwidth]{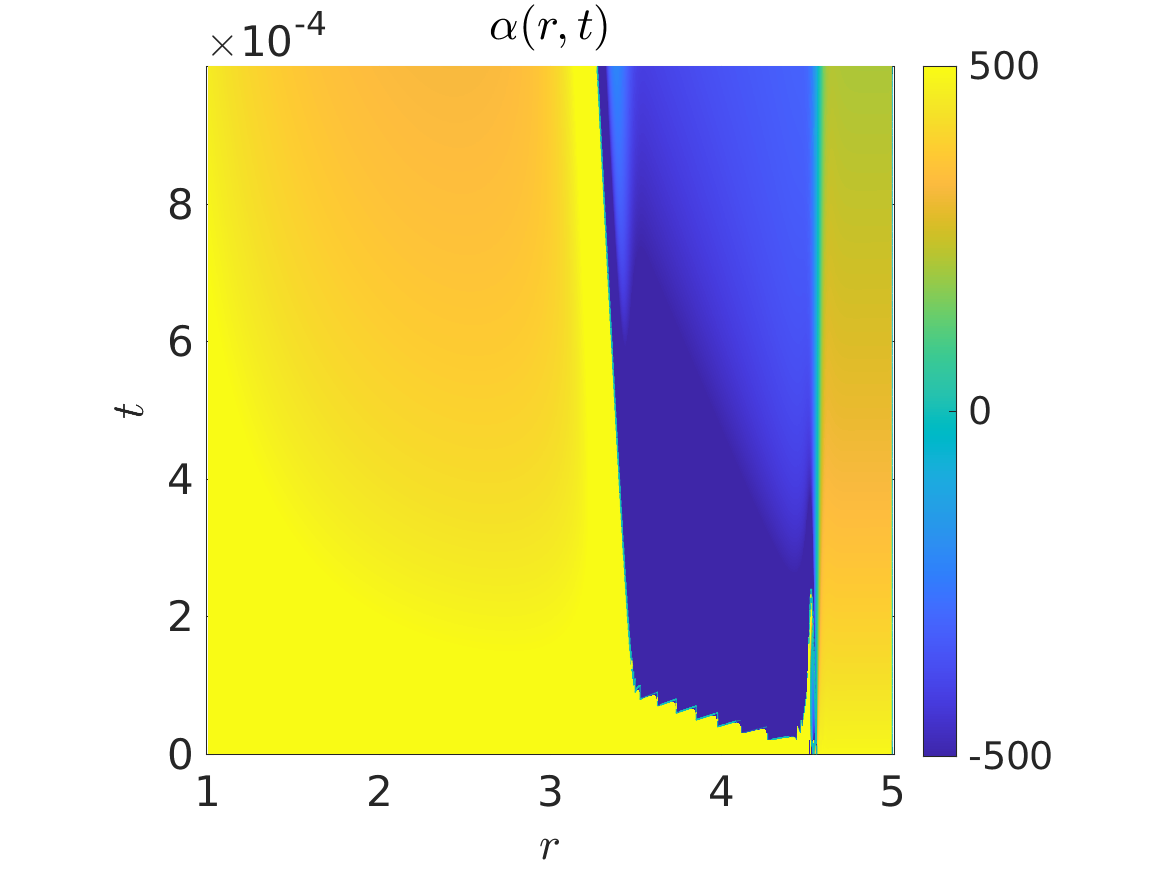}
  \end{subfigure}
  \begin{subfigure}{0.49\textwidth}
    \centering
	\includegraphics[width=1.0\textwidth]{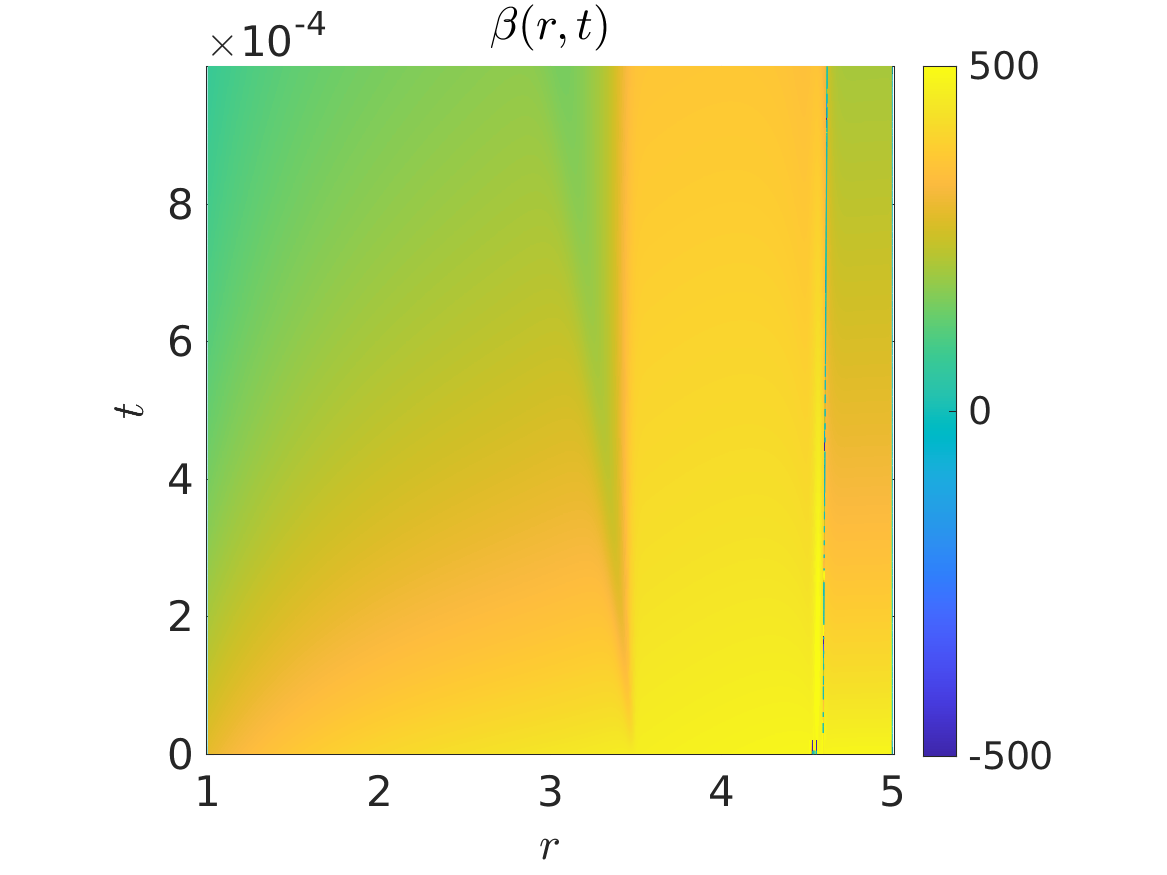}
  \end{subfigure}
  \caption{The \textit{heat map of $\alpha$ in \((r,t)-\)plane} is shown on the left and the heat map of $\beta$ on the right.}
  \label{fig:heat-map_case7}
\end{figure}

\medskip
%Overall, across the seven configurations, the SDLE simulations reproduce the expected qualitative dichotomy: strongly compressive supersonic data lead to finite-time shock formation (Cases~1 and 4), while outward supersonic rarefaction yields smooth evolution (Cases~2 and 5). The subsonic periodic setting (Case~3) exhibits a clear dependence on amplitude, with smooth oscillations for moderate perturbations and breakdown for sufficiently large amplitude. The inward-directed tests (Cases~6--7) highlight the asymmetric wave-character interactions emphasized in \cite{F1}, and illustrate how geometry and directionality influence the propagation of rarefactive/compressive domains in the $(r,t)$-plane.

Across the seven examined configurations, the SDLE simulations consistently recover the predicted qualitative dichotomy inherent in nonlinear hyperbolic systems. In the outward supersonic regimes, the results corroborate the theoretical threshold for regularity: strongly compressive initial data inevitably precipitate finite-time shock formation (Cases 1 and 4), whereas purely rarefactive data remain within smooth functional spaces, yielding globally well-posed evolution (Cases 2 and 5).

The subsonic periodic configuration (Case 3) demonstrates a critical dependence on the initial perturbation amplitude. For moderate amplitudes, the system exhibits stable, smooth oscillations; however, for data exceeding a specific threshold, the nonlinearity dominates, leading to a breakdown of the smooth solution. Finally, the inward-directed simulations (Cases 6-7) explicitly manifest the asymmetric wave-character interactions characteristic of radial geometry. These cases illustrate the profound influence of geometric source terms and radial directionality on the propagation and morphology of rarefactive and compressive domains within the $(r,t)$-plane.

% ------------------------------------------------------------------------

\section{Conclusions}\label{sec:conclusions}

In this study, we have investigated the qualitative dynamics of smooth solutions to the radially symmetric isentropic compressible Euler equations through a unified framework of characteristic analysis and numerical simulation. By employing a gradient-variable formulation, we characterized the evolution of rarefactive and compressive wave structures across supersonic and subsonic regimes, identifying the invariant manifolds and the mechanisms governing transitions in wave character.

From an analytical perspective, we established precise criteria under which the rarefactive or compressive nature of a wave is preserved along its respective characteristic family. Our results elucidate the structural disparities between outward supersonic, subsonic, and inward supersonic regimes, revealing inherent asymmetries in the interaction between characteristic families - complexities that do not manifest in the one-dimensional Cartesian limit. Furthermore, we derived sufficient conditions for the onset of a gradient catastrophe, providing a rigorous bound for finite-time singularity formation under specific initial data configurations.

On the numerical front, we implemented a Semi-Discrete Lagrangian-Eulerian (SDLE) scheme specifically adapted to the radial geometry. The formulation ensures the preservation of the system's conservative structure and accurately resolves the nonlinear transport of high-gradient wave structures. The computational results demonstrate a high degree of consistency with our analytical predictions, offering robust empirical validation for the invariant-domain analysis and the stability of smooth solutions within the identified regimes.

The synergy of rigorous analytical insight and high-fidelity numerical validation presented here offers a coherent description of the multidimensional effects induced by radial geometry in compressible gas dynamics. These findings open several avenues for further investigation, notably the long-time asymptotic behavior of subsonic solutions and the influence of geometric source terms on the global regularity of solutions in the presence of wave-front interactions.

\begin{itemize}
	\item[] {\bf DECLARATIONS}
	
	\vspace{-2mm}
	\item[] {\bf Conflicts of interest/Competing interests:} The authors have no competing interests to declare that are relevant to the content of this article.
	
	\vspace{-2mm}
	\item[] {\bf Acknowledgments:} E. Abreu thanks the University of Campinas - Unicamp (Brazil), E. Lima is grateful for the support provided by the Universidade Federal de Roraima and Unicamp. G. Chen and Faris El-Katri thank University of Kansas, EUA.
	
	\vspace{-2mm}
	\item[] {\bf Funding:} E. Abreu was supported by Grant No. 307641/2023-6 of the Brazilian National Council for Scientific and Technological Development (CNPq).
	
	\vspace{-2mm}
	\item[] {\bf Availability of data and material:} Not applicable.
	
	\vspace{-2mm}
	\item[] {\bf Code availability:} The authors commit to providing the code upon reasonable request and providing reasonable assistance for replication.
	
	\vspace{-2mm}
	\item[] {\bf Consent to participate:} Not applicable.
	
	\vspace{-2mm}
	\item[] {\bf Consent for publication:} Not applicable.
	
\end{itemize}

% ------------------------------------------------
\bibliography{ref}

\end{document}